\ifdefined\pdfoutput \pdfoutput=1 \fi

\documentclass[11pt]{amsart}

\usepackage[margin=1in]{geometry}

\usepackage{amssymb}
\usepackage{amsmath}
\usepackage{amsthm}

\usepackage{graphicx}
\usepackage{subfigure}
\usepackage{booktabs}
\usepackage{multirow}
\usepackage{float}

\usepackage[hidelinks]{hyperref}

\usepackage{microtype}

\newtheorem{theorem}{Theorem}
\newtheorem{problem}{Problem}
\newtheorem{lemma}{Lemma}
\newtheorem{remark}{Remark}

\title[An iterative approach for inverse acoustic obstacle scattering]{A highly efficient iterative approach for inverse acoustic obstacle scattering problems in three dimensions}

\author{Zhiyong Cheng}
\address{Jilin University, Qianjin Street 2699, Changchun, Jilin 130012, China}
\email{chengzy22@mails.jlu.edu.cn}

\author{Heping Dong}
\address{Jilin University, Qianjin Street 2699, Changchun, Jilin 130012, China}
\email{dhp@jlu.edu.cn}
\thanks{Corresponding author: Heping Dong (dhp@jlu.edu.cn)}

\author{Lu Zhao}
\address{Civil Aviation University of China, 2898 Jinbei Road, Tianjin 300300, China}
\email{zhaol@cauc.edu.cn}

\date{}

\subjclass[2020]{78A46}
\keywords{Helmholtz equation, boundary integral equation, iterative approach, inverse acoustic obstacle scattering, phaseless data}

\begin{document}

\begin{abstract}
This paper concerns a three-dimensional inverse acoustic obstacle scattering problem from scattered field or phased/phaseless far-field data. Based on the boundary integral defined on a homothetic surface, we propose a highly efficient iterative approach for obstacle reconstruction that completely avoids dealing with any singularity. Here, the injectivity and dense-range property of the Fr\'echet derivative have been proved to ensure the solvability of the linearized equivalent data equation. We also prove that the scattered field generated by the homothetic surface can arbitrarily approximate the exact one. Numerical experiments are presented to verify the superiority and robustness of the proposed approach.
\end{abstract}

\maketitle

\section{Introduction}
In this paper, we consider an inverse acoustic scattering problem of reconstructing a bounded obstacle from scattered field or phased/phaseless far-field data. This type of inverse problem is an important research topic with many practical applications such as geophysical exploration \cite{Borden}, biomedical imaging \cite{Ammari1} and nondestructive testing \cite{Ammari2}.

We begin by presenting the mathematical formulation of the direct acoustic scattering problem in three dimensions. Assume that $D \subset \mathbb{R}^3$ is a bounded, simply connected domain with sufficiently smooth boundary $\Gamma$, and that the exterior region $\mathbb{R}^3 \setminus \overline{D}$ is filled with a homogeneous medium.
 We further assume that the incident wave can be either a plane wave $u^{\rm inc}(\pmb{x},\pmb{d})=e^{\mathrm{i}\kappa\pmb{x}\cdot\pmb{d}}$ or a point source $u^{\rm inc}(\pmb{x},\pmb{z})=G(\pmb{x},\pmb{z}):=\frac{e^{\mathrm{i}\kappa|\pmb{x}-\pmb{z}|}}{4\pi|\pmb{x}-\pmb{z}|}$, where $\pmb{d}$ is an incident direction, $\pmb{z}\in\mathbb{R}^3\setminus \overline{D}$ is a location of the point source, and $G$ is the Green function of Helmholtz equation in three dimensions. Let $u^{\rm tot}$ denote the total field, that means $u^{\rm sc}=u^{\rm tot}-u^{\rm inc}$.
Then the direct scattering problem is to find the scattered field $u^{\rm sc}$ that satisfies the following Helmholtz system
\begin{equation}\label{helmholtz system}
	\left\lbrace 
	\begin{aligned}
		&\Delta u^{\rm sc}+\kappa^2u^{\rm sc}=0&&{\rm in}~\mathbb{R}^3\setminus\overline{D},\\
		&\mathcal{B}u^{\rm sc}=-\mathcal{B}u^{\rm inc}&&{\rm on}~\Gamma,\\
		&\lim\limits_{r\to\infty}r\left( \partial_r u^{\rm sc}-\mathrm{i}\kappa u^{\rm sc}\right) =0,&&r=|\pmb{x}|,
	\end{aligned}
	\right.
\end{equation}
where $\kappa>0$ is the real wavenumber. For any sufficiently smooth function $u$ defined on $\Gamma$, the boundary operator $\mathcal{B}$ is given by
\begin{equation}\label{boundary condition}
	\mathcal{B}u=\left\lbrace 
	\begin{aligned}
		&u, && {\rm Dirichlet~boundary~condition},\\
		&\partial_{\pmb{\nu}}u+\mathrm{i}\eta u, && {\rm impedance~boundary~condition}.
	\end{aligned}
	\right.
\end{equation}
Here, $\pmb{\nu}$ is the outward unit normal vector to $\Gamma$, and $\eta$ denotes the impedance function with $\Re(\eta)\geq 0$. The impedance condition reduces to the Neumann case when $\eta=0$.
The scattered field admits the following asymptotic behavior \cite{Colton}
\[
u^{\rm sc}(\pmb{x})=\frac{e^{\mathrm{i}\kappa|\pmb{x}|}}{|\pmb{x}|}\left\lbrace u^{\infty}\left(\hat{\pmb{x}} \right) +O\left( \frac{1}{|\pmb{x}|}\right) \right\rbrace ,\quad \hat{\pmb{x}}=\frac{\pmb{x}}{|\pmb{x}|}\in\mathbb{S}^2,\quad |\pmb{x}|\to\infty,
\]
where $u^{\infty}$ denotes the far-field pattern of $u^{\rm sc}$,  and $\mathbb{S}^2:=\{\pmb{x}\in\mathbb{R}^3:\ |\pmb{x}|=1\}$ is the unit sphere centered at the origin.

It is worth mentioning that the modulus of the far-field data has translation invariance for a shift domain, when an incident plane wave is employed \cite{ivanshyn10}. That is, the location of the obstacle cannot be determined from phaseless far-field data. 
To address this issue, we choose a point source as incident wave, and collect the phaseless far-field data $|u^{\mathrm{tot},\infty}|=|G^{\infty}+u^{\infty}|$,  where $u^{\mathrm{tot},\infty}$ denotes the far-field pattern of the total field $u^{\rm tot}$, and $G^{\infty}(\hat{\pmb{x}},\pmb{z})=\frac{1}{4\pi}e^{-\mathrm{i}\kappa \hat{\pmb{x}}\cdot \pmb{z}}$ is the far-field pattern of the point source incident field $G(\pmb{x},\pmb{z})$ with $\hat{\pmb{x}}\in\mathbb{S}^2$.
In this work, we concern with three kinds of inverse problems:
\begin{problem}\label{problem1}
	Given the scattered field $u^{\rm sc}$ measured on 
	\begin{align}\label{Gamma_B}
		\Gamma_{B_R}:=\big\lbrace \pmb{x}\in\mathbb{R}^3:|\pmb{x}|=R\big\rbrace
	\end{align}
	for the incident plane waves $u^{\rm inc}$, determine the shape and location of the obstacle under Dirichlet or impedance boundary conditions.
\end{problem}
\begin{problem}\label{problem2}
	Given the far-field data $u^{\infty}$ measured on $\mathbb{S}^2$ for the incident plane waves $u^{\rm inc}$,
	determine the shape and location of the obstacle under Dirichlet or impedance boundary conditions.
\end{problem}
\begin{problem}\label{problem3}
	Given the phaseless far-field data $|u^{\mathrm{tot},\infty}|$ measured on $\mathbb{S}^2$ for the incident point source $u^{\rm inc}$, determine the shape and location of the obstacle under Dirichlet or impedance boundary conditions.
\end{problem}

In recent years, a variety of numerical reconstruction methods have been developed for the above inverse obstacle scattering problems. Broadly speaking, these approaches can be categorized into: qualitative and quantitative methods. Qualitative techniques, such as linear sampling method \cite{Colton_ls}, factorization method \cite{Kirsch_f}, singular source method \cite{Potthast_book}, direct sampling method \cite{Zou_ds} and reverse time migration method \cite{RTM}, require no a priori information about the geometry or physical properties of the obstacle, and have been further developed and analyzed in various works \cite{Cakoni,kang,Louer1,Louer2,liukeyi24}. Specifically, they employ an appropriate indicator function $f(\pmb x)$, $\pmb{x}\in\mathbb{R}^3$, which can determine whether the point $\pmb x$ lies inside or outside the obstacle. However, these methods typically rely on scattered field or far-field data for a large number of incident waves. 

In contrast, the scattered field or far field generated by only one or few incident waves is required in quantitative methods, which are based primarily on linearization and optimization strategies to reconstruct the parameters of unknown obstacles.
Decomposition method \cite{Colton,dec} involves reconstructing the scattered wave from its far-field pattern based on an auxiliary surface, and then determining the obstacle's boundary by solving an optimization problem subject to the boundary condition. 
A regularized Newton method was introduced in \cite{farhat02,Hohage} to recover the obstacle's shape from phased or phaseless far-field data. Building upon the characteristics of Newton iterations and decomposition strategies, a hybrid method \cite{serranho07} was later proposed by using the far-field data for single incident wave. This method alternates between reconstructing the scattered field as a layer potential over the current boundary approximation and updating the boundary by solving a linearized equation, until a stopping criterion is fulfilled. Particularly, a method, named nonlinear integral equation method, was proposed in \cite{ivanshyn10}.
Given an initial approximation of the boundary, the density is solved from a field equation. By substituting this density into linearized data equation, the update of boundary can be obtained.
The process continues iteratively until the relative error of the scattered field (or far field) meets a predetermined tolerance. For the recursive linearization algorithm applied to inverse medium scattering of the Maxwell equations and to the inverse acoustic scattering by the axis-symmetric obstacles, we refer to \cite{bao05,borges20}. Additionally, the DeepSDF-based iterative method using phased/phaseless data is discussed in \cite{chen2024}. Moreover, for other quantitative reconstruction methods in three-dimensional elastic or electromagnetic inverse scattering problems, we refer to \cite{gauss-newton,chang,Hagemann,yuan}.

In this work, inspired by the iterative framework of the nonlinear integral equation and the decomposition methods, we focus on proposing a highly efficient iterative approach. Specifically, given an initial guess of boundary, a homothetic surface that contains complete information on the current boundary is constructed by using a fixed geometric contraction factor. Based on analytic continuation and the homothetic surface, we establish an equivalent field equation and solve it for density. The boundary update is then obtained by solving the corresponding equivalent linearized data equation.

The superiorities of our method can be described in the following three aspects. Firstly, it improves the classical nonlinear integral equation method by constructing the single-layer potential on a homothetic surface, so that only smooth kernels arise and no additional quadrature formula for handling singularities is required. Consequently, the proposed approach possesses high efficiency and accuracy, as well as simplicity in implementation. Secondly, compared with decomposition method, our method does not require a priori knowledge about the auxiliary surface being located strictly inside the true obstacle, and may therefore be better suited to the practical requirements of the reconstruction. Thirdly, our method can be applied to more types of measurement data, particularly suited for phaseless data, and its reconstruction is comparable with the phased case. In contrast, the hybrid and decomposition methods are difficult to extend to the phaseless case, due to their inability to reconstruct the scattered field from phaseless data directly. 
The goal of this work is threefold:
\begin{itemize}
	\item[(1)] propose a highly efficient iterative approach based on the homothetic surface to reconstruct the obstacle's shape and location by using scattered field or phased/phaseless far-field data;
	\item[(2)] prove that the scattered field generated by the homothetic surface can arbitrarily approximate the exact scattered field;
	\item[(3)] present the injectivity and dense-range property of the Fr$\acute{\rm e}$chet derivative to ensure the solvability of the inverse problem.
\end{itemize}

The paper is organized as follows. In Section 2, we propose a highly efficient iterative approach, and establish the approximation property for the scattered field generated by the homothetic surface. Section 3 presents the parametrization of the equivalent field and data equations, along with a discussion of the injectivity and dense-range property of the Fr$\acute{\rm e}$chet derivative. Numerical implementation and experiments are conducted in section 4 and 5 to demonstrate the effectiveness and robustness of the proposed approach. Finally, the paper is concluded in section 6. 

\section{A novel iterative approach}\label{novel approach} 
Building on the Helmholtz scattering formulation established in \eqref{helmholtz system}, we introduce a homothetic surface and propose a highly efficient iterative method to reconstruct both the shape and location of the obstacle without dealing with singularities. A key technique employed is the analytic continuation, which allows the scattered field generated by the homothetic surface to approximate the exact field with arbitrary accuracy. In the following, we denote the exact scattered field and far-field pattern by $u_D^{\rm sc}$ and $u_D^{\infty}$, respectively.

\subsection{Homothetic surface}
For simplicity, assume that the obstacle \(D\) is star-shaped with respect to a center \(\pmb c=(c_1,c_2,c_3)^\top\in\mathbb{R}^3\), such that
\begin{align}\label{parametrization}
	\Gamma
	=\bigl\{\pmb{p}_{D}(\hat{\pmb{x}})
	=\pmb{c}+r(\hat{\pmb{x}}) \hat{\pmb{x}}
	: \hat{\pmb{x}}\in\mathbb{S}^2\bigr\},
	\quad r(\hat{\pmb{x}})>0
\end{align}
with
\[\hat{\pmb{x}}(\theta,\phi)=(\sin\theta\cos\phi,\sin\theta\sin\phi,\cos\theta)^\top, ~\theta\in[0,\pi], ~\phi\in[0,2\pi).\]
Choose a geometric contraction factor $0<\varsigma<1$ and define the bounded domain $D'\subset D $ with homothetic boundary
\[
\Gamma'
=\bigl\{\pmb{p}_{D'}(\hat{\pmb{x}})
=\pmb{c}+\varsigma r(\hat{\pmb{x}}) \hat{\pmb{x}}
: \hat{\pmb{x}}\in\mathbb{S}^2\bigr\}.
\]
The scattered field $u_{D',g}^{\rm sc}$ and its far-field pattern $u_{D',g}^{\infty}$ can be expressed as
\begin{align}
u_{D',g}^{\rm sc}(\pmb{x})
&=\int_{\Gamma'} G(\pmb{x},\pmb{y})g(\pmb{y})\mathrm{d}s(\pmb{y}),\quad \pmb{x}\in\mathbb{R}^3\setminus\overline{D},\label{scattered}\\
u_{D',g}^{\infty}(\hat{\pmb{x}})
&=\frac{1}{4\pi}\int_{\Gamma'}e^{-\mathrm{i}\kappa\hat{\pmb{x}}\cdot\pmb{y}}
g(\pmb{y})\mathrm{d}s(\pmb{y}),\quad
 \hat{\pmb{x}}\in\mathbb{S}^2.\nonumber
\end{align}
According to the boundary conditions \eqref{boundary condition}, the unknown density $g$ on $\Gamma'$ is determined by the boundary integral equations
{\small
\begin{align}
	\int_{\Gamma'} G(\pmb{x},\pmb{y})g(\pmb{y})\mathrm{d}s(\pmb{y})
	&=f_1(\pmb{x}),
	&&\pmb{x}\in\Gamma,\quad\text{(Dirichlet)},\quad\label{dirichlet novel}\\
	\int_{\Gamma'}\bigl(\partial_{\nu(\pmb{x})} G(\pmb{x},\pmb{y})
	+\mathrm{i}\eta G(\pmb{x},\pmb{y})\bigr)
	g(\pmb{y})\mathrm{d}s(\pmb{y})
	&=f_2(\pmb x),
	&&\pmb{x}\in\Gamma,\quad\text{(impedance)}.\quad\label{impedance novel}
\end{align}}where $f_1(\pmb{x})=-u^{\mathrm{inc}}(\pmb{x})$ and $f_2(\pmb{x})=-\partial_{\nu(\pmb{x})}u^{\mathrm{inc}}(\pmb{x})
-\mathrm{i}\eta u^{\mathrm{inc}}(\pmb{x})$. 

The following two theorems demonstrate that the scattered field generated by the homothetic surface can arbitrarily approximate the exact one.
\begin{theorem}\label{injective and dense range}
	Let $\kappa^2$ not be a Dirichlet eigenvalue for $-\Delta$ in $D'$. Then, the boundary integral operator  $T_D: L^2(\Gamma')\to L^2(\Gamma)$, defined by
	\begin{equation*}\label{boundary integral operator T}
		\left( T_{D} g\right) (\pmb x):=\int_{\Gamma'}\left(\partial_{\nu(\pmb{x})} G(\pmb{x},\pmb{y})
		+\mathrm{i}\eta G(\pmb{x},\pmb{y})\right)
		g(\pmb{y})\mathrm{d}s(\pmb{y}), \quad \pmb x\in \Gamma,
	\end{equation*}
	is injective and has dense range.
\end{theorem}
\begin{proof}
	To establish the injectivity of $T_D: L^2(\Gamma')\to L^2(\Gamma)$, 
	we assume that $T_Dg=0$ in $L^2(\Gamma)$ for $g\in L^2(\Gamma')$. Define the function
	\[
	U(\pmb{x})=\int_{\Gamma'} G(\pmb{x},\pmb{y})g(\pmb{y})\mathrm{d}s(\pmb{y}), \quad \pmb x\in \mathbb{R}^3\setminus\Gamma'.
	\]
	Noting that the function \( U(\pmb{x}) \in L^2(\mathbb{R}^3\setminus\overline{D}) \) satisfies the Helmholtz equation and the Sommerfeld radiation condition. From the assumption that $T_Dg=0$ in $L^2(\Gamma)$, it follows that $\partial_{\pmb{\nu}}U(\pmb{x})+\mathrm{i}\eta U(\pmb{x})$ vanishes on $\Gamma$. By the uniqueness of the exterior impedance problem for the Helmholtz equation, we conclude that \( U \) vanishes identically in \( \mathbb{R}^3\setminus\overline{D} \). The unique continuation principle then implies that \( U \) vanishes in \( \mathbb{R}^3\setminus\overline{D'} \) as well. Applying the jump relation for the single-layer potential yields \( U = 0 \) on \( \Gamma' \). Furthermore, \( U \in L^2(D') \) satisfies the Helmholtz equation in \( D' \). Since \( \kappa^2 \) is not an interior Dirichlet eigenvalue for $-\Delta$ in \( D' \), the uniqueness of the interior Dirichlet problem implies that \( U = 0 \) in \( D' \). Finally, the jump relation gives $g=0$ on $\Gamma'$, which yields the injectivity of $T_D$.
	
	Let $T_D^*:L^2(\Gamma)\to L^2(\Gamma')$ denote the $L^2$-adjoint of $T_D$, i.e.,
	\begin{equation}\label{TDadjoint}
		\left( T_D^*\varphi\right) (\pmb{y})=\int_{\Gamma}\overline{\left(\partial_{\nu(\pmb{x})} G(\pmb{x},\pmb{y})+\mathrm{i}\eta G(\pmb{x},\pmb{y})\right)}\varphi(\pmb{x})\mathrm{d}s(\pmb{x}),\quad \pmb{y}\in\Gamma'.
	\end{equation}
	Assume that $T_D^*\varphi=0$ in $L^2(\Gamma')$ and define
	\[
	V(\pmb{y}):=\int_{\Gamma}\left(\partial_{\nu(\pmb{x})} G(\pmb{x},\pmb{y})+\mathrm{i}\eta G(\pmb{x},\pmb{y})\right)\overline{\varphi(\pmb{x})}\mathrm{d}s(\pmb{x}),\quad \pmb{y}\in\mathbb{R}^3\setminus\Gamma.
	\]
	The function $V$ represents a combined double- and single-layer potential that vanishes on $\Gamma'$. Consequently, $V$ satisfies the homogeneous interior Dirichlet problem in $D'$, which implies that $V=0$ in $D'$. By the unique continuation principle, $V$ also vanishes identically in $D$. The jump relations for double- and single- layer potentials \cite{Colton} yield
	\[
	V_+=\overline{\varphi},\quad \partial_{\pmb{\nu}}V_+=-\mathrm{i}\eta \overline{\varphi},\quad \text{on}~\Gamma,
	\]
	where
	\[
	V_+(\pmb{y}):=\lim\limits_{h\to +0}V(\pmb{y}+h\pmb{\nu}(\pmb{y})).
	\]
	Let $S_r$ denote the sphere of radius $r$ and center at the origin. For $r$ sufficiently large such that $D\subset S_r$, applying Green's formula \cite{Colton} in \[D_r:=\left\{\pmb{y}\in\mathbb{R}^3\setminus\overline{D}:|\pmb{y}|<r\right\},\] we have
	\begin{align*}
		\begin{aligned}
		\int_{\Gamma}\mathrm{i}\eta|\varphi|^2\mathrm{d}s&=-\int_{\Gamma}\overline{V}_+\partial_{\pmb{\nu}}V_+\mathrm{d}s\\
		&=-\int_{S_r}\overline{V}\partial_{\pmb{\nu}}V\mathrm{d}s(\pmb{x})+\int_{D_r}\left\{|\nabla V|^2-{\kappa}^2|V|^2\right\}\mathrm{d}\pmb{x}.
		\end{aligned}
	\end{align*}
	Taking the imaginary part of the above equation and employing the Sommerfeld radiation condition gives
	\begin{align*}
		\int_{\Gamma}\Re(\eta)|\varphi|^2\mathrm{d}s=-\kappa\lim\limits_{r\to\infty}\int_{S_r}|V|^2\mathrm{d}s.
	\end{align*}	
Under the assumption $\Re(\eta)\ge 0$ in the impedance boundary condition \eqref{boundary condition}, the left-hand side is non-negative while the right-hand side is non-positive, so both sides must vanish. In particular,
	\[
		\lim\limits_{r\to\infty}\int_{S_r}|V|^2\,\mathrm{d}s=0.
	\]
	By Rellich's lemma, $V$ vanishes in $\mathbb{R}^3 \setminus \overline{D}$. Consequently, the jump relations for double- and single-layer potentials yield $\varphi=0$ on $\Gamma$.
	Then $T_D^*$ is injective, which implies that $T_D$ has dense range.
	
\end{proof}

\begin{theorem}\label{approximation}
	Let $u_D^{\rm sc}$ denote the scattered field associated with a domain $D$ under the Dirichlet or impedance boundary conditions. Assume that $\kappa^2$ is not an interior Dirichlet eigenvalue for $-\Delta$ in $D'$. Then, $\forall\,\epsilon>0$, there exists $g\in L^2(\Gamma')$ such that
	\[
	\left\|u_{D',g}^{\rm sc}-u_D^{\rm sc}\right\| _{L^2(\Gamma_{B_R})}<\epsilon,
	\]
	where \(\Gamma_{B_R}\) is defined by \eqref{Gamma_B}.
\end{theorem}
\begin{proof}
	We first prove the case of impedance boundary condition. By Theorem \ref{injective and dense range}, for any $\epsilon>0$, there exists $g\in L^2(\Gamma')$ such that
	\[
	\left\| T_Dg+\partial_{\pmb{\nu}}u^{\rm inc}+\mathrm{i}\eta u^{\rm inc}\right\| _{L^2(\Gamma)}<\frac{\epsilon}{c}.
	\]
	where \(c>0\) is a constant independent of \(\epsilon\).
	Since the radiating solution of the Helmholtz equation depends continuously on the boundary data subject to the impedance boundary condition \cite[Theorem~3.16]{Colton}, we obtain the estimate
	\begin{align*}
		\left\| u_{D',g}^{\rm sc}-u_D^{\rm sc}\right\| _{L^2(\Gamma_{B_R})}&\leq c\left\| \left( \partial_{\pmb{\nu}}u_{D',g}^{\rm sc}+\mathrm{i}\eta u_{D',g}^{\rm sc}\right) |_{\Gamma}-\left( \partial_{\pmb{\nu}}u_{D}^{\rm sc}+\mathrm{i}\eta u_{D}^{\rm sc}\right) |_{\Gamma}\right\| _{L^2(\Gamma)}\\
		&= c\left\| T_Dg-\left( \partial_{\pmb{\nu}}u_{D}^{\rm sc}+\mathrm{i}\eta u_{D}^{\rm sc}\right) |_{\Gamma}\right\| _{L^2(\Gamma)},
	\end{align*}
	for some constant $c$. From $\partial_{\pmb{\nu}}u_{D}^{\rm sc}+\mathrm{i}\eta u_{D}^{\rm sc}=-\left( \partial_{\pmb{\nu}}u^{\rm inc}+\mathrm{i}\eta u^{\rm inc}\right) $ on $\Gamma$, it follows that
	\[
	\left\|u_{D',g}^{\rm sc}-u_D^{\rm sc}\right\| _{L^2(\Gamma_{B_R})}<\epsilon.
	\]
	
	For the case of Dirichlet boundary condition, we define the operator  
	\begin{equation}\label{single layer operator}
			\left( S_Dg\right) (\pmb{x}) := \int_{\Gamma'} G(\pmb{x}, \pmb{y}) g(\pmb{y})  \mathrm{d}s(\pmb{y}), \quad \pmb{x} \in \Gamma.
	\end{equation}
	By \cite[Theorem 5.21]{Colton}, the operator \(S_D: L^2(\Gamma') \to L^2(\Gamma)\) is injective and has dense range. Combined with the stability of exterior Dirichlet problem \cite[Theorem 5.17]{Colton}, the result follows analogously to the impedance boundary condition case.
	
\end{proof}

\subsection{Iterative framework}
We have demonstrated the approximation property of the scattered field generated by the homothetic surface, which provides the mathematical foundation for the novel iterative method to be proposed. 

It should be emphasized that the equivalent field equations \eqref{dirichlet novel}-\eqref{impedance novel} are ill-posed.
To address this issue, we employ Tikhonov regularization by solving the following regularized equivalent field equations:
\begin{align}
	\left( \alpha I+S_D^*S_D\right) g_\alpha&=S_D^*f_1,&& \text{(Dirichlet),} \label{field equation diri}\\
	\left( \alpha I+T_D^*T_D\right) g_\alpha&=T_D^*f_2,&& \text{(impedance),}\label{field equation imp}
\end{align}
where $\alpha$ is a regularization parameter, $I$ is an identity operator, and the $L^2$-adjoint operators $T^*_D$ and $S_D^*$ are defined in \eqref{TDadjoint} and 
\[
(S_D^* g_\alpha)(\pmb{y}) := \int_{\Gamma} \overline{G(\pmb{x},\pmb{y})}  g_\alpha(\pmb{x})  \mathrm{d}s(\pmb{x}), \quad \pmb{y} \in \Gamma',
\]
respectively.

The corresponding equivalent data equations for scattered field and far-field data are given by
\begin{align}
	{S}g_{\alpha}(\pmb{x}):=&\int_{\Gamma'} G(\pmb{x},\pmb{y})g_{\alpha}(\pmb{y})\mathrm{d}s(\pmb{y})=u_D^{\rm sc}(\pmb{x}), ~\quad\quad\pmb{x}\in\Gamma_{B_R},\label{data equation novel regularization}\\
	{S}_{\infty}g_{\alpha} (\hat{\pmb{x}}):=&\frac{1}{4\pi}\int_{\Gamma'}e^{-\mathrm{i}\kappa\hat{\pmb{x}}\cdot\pmb{y}}g_{\alpha}(\pmb{y})\mathrm{d}s(\pmb{y})=u_D^{\infty}(\hat{\pmb{x}}),\quad\hat{\pmb{x}}\in\mathbb{S}^2.\label{far data equation novel regularization}
\end{align}

Therefore, we present the iterative approach as follows:
\begin{itemize}
	\item Given an initial guess for $\Gamma$, solve the regularized equivalent field equation \eqref{field equation diri} or \eqref{field equation imp} to compute $g_\alpha$;
	\item For a fixed $g_\alpha$, linearize the equivalent data equation \eqref{data equation novel regularization} or \eqref{far data equation novel regularization} with respect to $\Gamma$, since the homothetic surface $\Gamma'$ contains the geometric information of $\Gamma$. The linearized equation yields a boundary update $\Delta \Gamma$, and the boundary is updated by \( \Gamma^{(k+1)} = \Gamma^{(k)} + \Delta \Gamma^{(k)}\).
\end{itemize}
The two steps are iterated alternately until the relative error of the scattered field (or far field) meets a predetermined tolerance. A complete description of the algorithmic procedure is given in Table~\ref{algorithm} of Section \ref{sec3}.

\section{The property of Fr\'echet derivative} \label{sec3}
In this section, we will establish the injectivity and dense-range property of the associated Fr\'echet derivative to guarantee the solvability of the linearized equivalent data equation. 
Then, we provide a detailed presentation of the iterative algorithm.

 For clarity, we denote the parametrized integral operators by $\mathcal{S}_D$, $\mathcal{T}_D$, $\mathcal{S}$ and $\mathcal{S}_\infty$, corresponding to the original operators $S_D$, $T_D$, $S$ and $S_\infty$, respectively.
\begin{align*}
	(\mathcal{S}_D \psi_\alpha)(\hat{\pmb{x}}) 
	&= \int_{\mathbb{S}^2} \frac{e^{\mathrm{i}\kappa|\pmb p_D(\hat{\pmb{x}}) - \pmb p_{D'}(\hat{\pmb{y}})|}}{4\pi|\pmb p_D(\hat{\pmb{x}}) - \pmb p_{D'}(\hat{\pmb{y}})|} J_{D'}(\hat{\pmb{y}})\psi_{\alpha}(\hat{\pmb{y}}) \mathrm{d}s(\hat{\pmb{y}}), \\
	(\mathcal{T}_D \psi_\alpha)(\hat{\pmb{x}}) 
	&= \int_{\mathbb{S}^2} \left( \widetilde{K}(\hat{\pmb{x}},\hat{\pmb{y}}) + \mathrm{i}\eta \frac{e^{\mathrm{i}\kappa|\pmb p_D(\hat{\pmb{x}}) - \pmb p_{D'}(\hat{\pmb{y}})|}}{4\pi|\pmb p_D(\hat{\pmb{x}}) - \pmb p_{D'}(\hat{\pmb{y}})|} \right) J_{D'}(\hat{\pmb{y}})\psi_{\alpha}(\hat{\pmb{y}}) \mathrm{d}s(\hat{\pmb{y}}), \\
	(\mathcal{S}(\pmb{p}_D)\psi_\alpha)(\hat{\pmb{x}}) 
	&= \int_{\mathbb{S}^2} \frac{e^{\mathrm{i}\kappa|\pmb p_{B_R}(\hat{\pmb{x}}) - \pmb p_{D'}(\hat{\pmb{y}})|}}{4\pi|\pmb p_{B_R}(\hat{\pmb{x}}) - \pmb p_{D'}(\hat{\pmb{y}})|} J_{D'}(\hat{\pmb{y}})\psi_{\alpha}(\hat{\pmb{y}}) \mathrm{d}s(\hat{\pmb{y}}), \\
	(\mathcal{S}_\infty(\pmb{p}_D) \psi_\alpha)(\hat{\pmb{x}}) 
	&= \int_{\mathbb{S}^2} \frac{1}{4\pi} e^{-\mathrm{i}\kappa \hat{\pmb{x}} \cdot \pmb{p}_{D'}(\hat{\pmb{y}})} J_{D'}(\hat{\pmb{y}})\psi_{\alpha}(\hat{\pmb{y}}) \mathrm{d}s(\hat{\pmb{y}}),
\end{align*}
where $\pmb p_{B_R}(\hat{\pmb{x}}) = R\hat{\pmb{x}}, ~\hat{\pmb{x}} \in \mathbb{S}^2$, $J_{D'}(\hat{\pmb{x}})$ is the Jacobian of transformation and
{\small\[
\widetilde{K}(\hat{\pmb{x}},\hat{\pmb{y}}) =  (\pmb{p}_D(\hat{\pmb{x}}) - \pmb{p}_{D'}(\hat{\pmb{y}})) \cdot \pmb{\nu}(\hat{\pmb{x}})  \left( \mathrm{i}\kappa - \frac{1}{|\pmb{p}_D(\hat{\pmb{x}}) - \pmb{p}_{D'}(\hat{\pmb{y}})|} \right) \frac{e^{\mathrm{i}\kappa|\pmb p_D(\hat{\pmb{x}}) - \pmb p_{D'}(\hat{\pmb{y}})|}}{4\pi|\pmb p_D(\hat{\pmb{x}}) - \pmb p_{D'}(\hat{\pmb{y}})|^2}
\]}with $\psi_\alpha(\hat{\pmb{y}}) = g_\alpha(\pmb{p}_{D'}(\hat{\pmb{y}}))$. Here, we define $\pmb{\nu}(\hat{\pmb{x}}): = \pmb{\nu}(\pmb{p}_D(\hat{\pmb{x}}))=\pmb{\nu}(\pmb{p}_{D'}(\hat{\pmb{x}}))$. In addition, since $\pmb{p}_{D'}$ is uniquely determined by $\pmb{p}_D$ through the homothetic relation, we use $\pmb{p}_D$ as the primary variable in the parametrized operator notation.

Consequently, the equivalent field equations \eqref{field equation diri}-\eqref{field equation imp} can be reformulated as parametrized integral equations
\begin{align}
	(\alpha I + \mathcal{S}_D^* \mathcal{S}_D) \psi_\alpha &= \mathcal{S}_D^* \tilde{f_1}, \label{parametrized field equation diri}\\
	(\alpha I + \mathcal{T}_D^* \mathcal{T}_D) \psi_\alpha &= \mathcal{T}_D^* \tilde{f_2},\label{parametrized field equation imp}
\end{align}
where $\tilde{f_1}(\hat{\pmb{x}}) = f_1(\pmb{p}_D(\hat{\pmb{x}}))$, $\tilde{f_2}(\hat{\pmb{x}}) = f_2(\pmb{p}_D(\hat{\pmb{x}}))$. The parametrized adjoint operators $\mathcal{S}_D^*$ and $\mathcal{T}_D^*$ are given by
\begin{align*}
	\left( \mathcal{S}_D^*\varphi_\alpha\right) (\hat{\pmb{y}})&=\int_{\mathbb{S}^2}\frac{e^{-\mathrm{i}\kappa|\pmb p_D(\hat{\pmb{x}}) - \pmb p_{D'}(\hat{\pmb{y}})|}}{4\pi|\pmb p_D(\hat{\pmb{x}}) - \pmb p_{D'}(\hat{\pmb{y}})|} J_{D}(\hat{\pmb{x}})\varphi_{\alpha}(\hat{\pmb{x}}) \mathrm{d}s(\hat{\pmb{x}}),\\
	\left( \mathcal{T}_D^*\varphi_\alpha\right) (\hat{\pmb{y}})&=\int_{\mathbb{S}^2}\left( \overline{\widetilde{K}(\hat{\pmb{x}},\hat{\pmb{y}})}-\mathrm{i}\overline{\eta}\frac{e^{-\mathrm{i}\kappa|\pmb p_D(\hat{\pmb{x}}) - \pmb p_{D'}(\hat{\pmb{y}})|}}{4\pi|\pmb p_D(\hat{\pmb{x}}) - \pmb p_{D'}(\hat{\pmb{y}})|} \right)J_{D}(\hat{\pmb{x}})\varphi_{\alpha}(\hat{\pmb{x}}) \mathrm{d}s(\hat{\pmb{x}}),
\end{align*}
where $J_D$ is the Jacobian. Similarly, the equivalent data equations \eqref{data equation novel regularization}-\eqref{far data equation novel regularization} in parametrized form are shown by
\begin{align}
	\mathcal{S}(\pmb{p}_D)\psi_\alpha &= w_D^{\mathrm{sc}}, \label{data equation sc} \\
	\mathcal{S}_\infty(\pmb{p}_D) \psi_\alpha &= w_D^\infty \label{data equation far}
\end{align}
with $w_D^{\rm sc}(\hat{\pmb{x}})=u_D^{\rm sc}(\pmb p_{B_R}(\hat{\pmb{x}}))$ and $w_D^{\infty}(\hat{\pmb{x}})=u_D^\infty(\hat{\pmb{x}})$.

The linearization of \eqref{data equation sc} and \eqref{data equation far} with respect to $\pmb{p}_D$ requires the Fr$\acute{\rm e}$chet derivative of the parametrized integral operators $\mathcal{S}$ and $\mathcal{S}_\infty$, which can be deduced explicitly as
\begin{align}
\left( \mathcal{S}'[\pmb{p}_D,\psi_\alpha]\pmb{q}\right) (\hat{\pmb{x}})&=\int_{\mathbb{S}^2}\widetilde{S}(\hat{\pmb{x}},\hat{\pmb{y}})J_{D'}(\hat{\pmb{y}})\psi_\alpha(\hat{\pmb{y}})\mathrm{d}s(\hat{\pmb{y}}),\label{frechet sc}\\
\left( \mathcal{S}_\infty'[\pmb{p}_D,\psi_\alpha]\pmb{q}\right) (\hat{\pmb{x}})&=\frac{-\mathrm{i}\kappa}{4\pi}\int_{\mathbb{S}^2}e^{-\mathrm{i}\kappa\hat{\pmb{x}}\cdot\pmb{p}_{D'}(\hat{\pmb{y}})}\hat{\pmb{x}}\cdot\tilde{\pmb{q}}(\hat{\pmb{y}})J_{D'}(\hat{\pmb{y}})\psi_\alpha(\hat{\pmb{y}})\mathrm{d}s(\hat{\pmb{y}}),\label{frechet far}
\end{align}
where
\begin{equation}\label{updateq}
	\pmb{q}(\hat{\pmb{x}})=\Delta\pmb{c}+\Delta r(\hat{\pmb{x}})\hat{\pmb{x}},
\qquad
\tilde{\pmb{q}}(\hat{\pmb{x}})=\Delta\pmb{c}+\varsigma \Delta r(\hat{\pmb{x}})\hat{\pmb{x}},
\end{equation}
denote the updates of $\pmb p_D$ and $\pmb p_{D'}$, respectively, with $\Delta \pmb{c}$ and $\Delta r$ being the updates of the center $\pmb{c}$ and the radial function $r$ in the parametrization \eqref{parametrization} of $\pmb p_D$. Moreover,
\[
\widetilde{S}(\hat{\pmb{x}},\hat{\pmb{y}})=(\pmb{p}_{D'}(\hat{\pmb{y}})-\pmb p_{B_R}(\hat{\pmb{x}}))\cdot\tilde{\pmb{q}}(\hat{\pmb{y}})\left( \mathrm{i}\kappa-\frac{1}{|\pmb p_{B_R}(\hat{\pmb{x}})-\pmb{p}_{D'}(\hat{\pmb{y}})|}\right) \frac{e^{\mathrm{i}\kappa|\pmb p_{B_R}(\hat{\pmb{x}})-\pmb{p}_{D'}(\hat{\pmb{y}})|}}{4\pi|\pmb p_{B_R}(\hat{\pmb{x}})-\pmb{p}_{D'}(\hat{\pmb{y}})|^2}.
\]

For the shape reconstruction, we assume that $\Delta \pmb c=\pmb 0$, which implies $\tilde{\pmb{q}}=\varsigma\pmb{q}$. Then, we will establish the injectivity of the Fr$\acute{\rm e}$chet derivative $\mathcal{S}'$. The following lemma can be found in \cite{Ivanyshyn2010}.

\begin{lemma}\label{lemma 1}
Let $\Gamma$ be a smooth closed surface, and let 
$\pmb p_D$ be a parametrization of  $\Gamma$ as in \eqref{parametrization}. Suppose that
\(
\pmb p_{D}^{new}(\hat{\pmb{x}}) = \pmb p_D(\hat{\pmb{x}})+\pmb q(\hat{\pmb{x}})
\)
is a sufficiently small smooth perturbation of $\pmb p_D$. Then, there exists a smooth scalar function $q(\hat{\pmb{x}})$ such that
\[
\pmb p_{D}^{new}(\hat{\pmb{x}})
=
\pmb p_D(\hat{\pmb{x}})
+
q(\hat{\pmb{x}})\pmb\nu(\hat{\pmb{x}}).
\]
\end{lemma}

We now turn to establishing the injectivity of the operator $\mathcal{S}'\left[ \pmb{p}_D, \psi_{\alpha}\right]$.
\begin{theorem} \label{Frechet_injective}
Let \(\kappa^2\) not be an interior Neumann eigenvalue for $-\Delta$ in \(D'\). Then the Fr\'echet operator $\mathcal{S}'\left[ \pmb{p}_D, \psi_{\alpha}\right]$, defined by \eqref{frechet sc}, is injective.
\end{theorem}

\begin{proof}
To prove the injectivity of $\mathcal{S}'\left[ \pmb{p}_D, \psi_{\alpha}\right]$, we assume that \(\mathcal{S}'\left[ \pmb{p}_D, \psi_{\alpha}\right]\pmb{q} = 0\) on \(\mathbb{S}^2\). 
Given the boundary update parametrization
\[
\pmb{q}(\hat{\pmb{x}}) = q(\hat{\pmb{x}}) \pmb{\nu}(\hat{\pmb{x}}),
\]
where \(q\) is a continuous scalar function, the Fr$\acute{\rm e}$chet derivative $\mathcal{S}'\left[ \pmb{p}_D, \psi_{\alpha}\right]\pmb{q}$ is reduced to
\[
\left( \mathcal{S}'\left[ \pmb{p}_D, \psi_{\alpha}\right]\pmb{q} \right)(\hat{\pmb{x}}) = 
\varsigma \int_{\mathbb{S}^2}
\frac{\partial G(\pmb p_{B_R}(\hat{\pmb{x}}), \pmb{p}_{D'}(\hat{\pmb{y}}))}{\partial \pmb{\nu}(\hat{\pmb{y}})} 
q(\hat{\pmb{y}})  J_{D'}(\hat{\pmb{y}})\psi_{\alpha}(\hat{\pmb{y}})   \mathrm{d}s(\hat{\pmb{y}}).
\]

Define the potential function
\[
U(\pmb{x}) =\int_{\mathbb{S}^2}
\frac{\partial G(\pmb{x}, \pmb{p}_{D'}(\hat{\pmb{y}}))}{\partial \pmb{\nu}(\hat{\pmb{y}})} 
q(\hat{\pmb{y}})  J_{D'}(\hat{\pmb{y}})\psi_{\alpha}(\hat{\pmb{y}})   \mathrm{d}s(\hat{\pmb{y}}), 
\quad \pmb{x} \in \mathbb{R}^3 \setminus \Gamma'.
\]
This function $U(\pmb{x})$ is a standard double layer potential, which satisfies the Helmholtz equation in \(\mathbb{R}^3 \setminus \overline{D'}\) and the Sommerfeld radiation condition.

Since $\mathcal{S}'[\pmb{p}_D,\psi_{\alpha}]\pmb{q}=0$ on $\mathbb{S}^2$, it is clear that 
$U(\pmb p_{B_R}(\hat{\pmb{x}}))=0$ on $\mathbb{S}^2$, and hence $U=0$ on $\Gamma_{B_R}$. 
By the uniqueness of the exterior Helmholtz problem, we deduce that $U$ vanishes in the exterior of $\Gamma_{B_R}$. The unique continuation principle shows that \(U = 0\) throughout \(\mathbb{R}^3 \setminus \overline{D'}\).

Furthermore, the jump relation for the normal derivative of the double-layer potential yields 
\[
\partial_{\pmb{\nu}} U = 0\quad \text{on}~ \Gamma'.
\]
Since \(U\) satisfies the Helmholtz equation in \(D'\) and \(\kappa^2\) is not an interior Neumann eigenvalue for $-\Delta$ in \(D'\), the uniqueness of interior problem implies that \(U = 0\) in \(D'\).

Applying the jump relation for the double layer potential yields
\[
q(\hat{\pmb{y}})J_{D'}(\hat{\pmb{y}})\psi_{\alpha}(\hat{\pmb{y}}) = 0 \quad \text{on }~ \mathbb{S}^2.
\]
By Lemma~\ref{lemma 1} and the fact that $J_{D'}\psi_{\alpha}$ cannot vanish identically on any open subset of \(\mathbb{S}^2\) \cite{Ivanyshyn2010}, we conclude that \(\pmb{q} = \pmb{0}\) on $\Gamma'$, which completes the proof.

\end{proof}

Noting that the proof of Theorem \ref{Frechet_injective} depends on the assumption \(\pmb{q}(\hat{\pmb{x}}) = q(\hat{\pmb{x}}) \pmb{\nu}(\hat{\pmb{x}})\). In fact, the assumption can be omitted in establishing the dense-range property of the Fr\'echet derivative $\mathcal{S}'$.

\begin{theorem}\label{Frechet_dense}
Let \(\kappa^2\) not be an interior Neumann eigenvalue for $-\Delta$ in \(D'\). Then the Fr\'echet operator
\[
\mathcal{S}'\left[ \pmb{p}_D, \psi_{\alpha}\right] : \left(L^2(\mathbb{S}^2)\right)^3 \to L^2(\mathbb{S}^2)
\] 
has dense range.
\end{theorem}

\begin{proof}
To prove that the range of \(\mathcal{S}'\left[ \pmb{p}_D, \psi_{\alpha}\right]\) is dense, it suffices to show that the \(L^2\)-adjoint 
\(
\left( \mathcal{S}'\left[ \pmb{p}_D, \psi_{\alpha}\right]  \right)^* \colon L^2(\mathbb{S}^2) \to \left( L^2(\mathbb{S}^2) \right)^3
\)
denoted by
	\[
	\left( \left( \mathcal{S}'\left[ \pmb{p}_D, \psi_{\alpha}\right]  \right)^* \varPsi \right)(\hat{\pmb{y}}) 
	= \varsigma\overline{\psi_{\alpha}(\hat{\pmb{y}})}
	\nabla_{\pmb{p}_{D'}} \int_{\mathbb{S}^2}
	\overline{G(\pmb p_{B_R}(\hat{\pmb{x}}), \pmb{p}_{D'}(\hat{\pmb{y}}))}J_B(\hat{\pmb{x}})\varPsi(\hat{\pmb{x}})  \mathrm{d}s(\hat{\pmb{x}}),
	\]
is injective, where $J_B$ is the Jacobian and \(\varPsi \in L^2(\mathbb{S}^2)\).

Suppose that \(\left( \mathcal{S}'\left[ \pmb{p}_D, \psi_{\alpha}\right] \right)^* \varPsi =\pmb 0\) in \((L^2(\mathbb{S}^2))^3\), and define
\[
U(\pmb{y}) = \int_{\mathbb{S}^2}
\overline{G(\pmb p_{B_R}(\hat{\pmb{x}}), \pmb{y})}J_B(\hat{\pmb{x}}) \varPsi(\hat{\pmb{x}})  \mathrm{d}s(\hat{\pmb{x}}), 
\quad \pmb{y} \in \mathbb{R}^3 \setminus \Gamma_{B_R}.
\]
By assumption, we have $\varsigma  \overline{\psi_{\alpha}(\hat{\pmb{y}})}
\nabla U(\pmb{p}_{D'}(\hat{\pmb{y}}))= \pmb{0}$
in \((L^2(\mathbb{S}^2))^3\). Since \(0 < \varsigma < 1\) and $\psi_\alpha$ cannot vanish identically on any open subset of \(\mathbb{S}^2\), it follows that
\[
\nabla U(\pmb{p}_{D'}(\hat{\pmb{x}})) = \pmb{0}, \quad \hat{\pmb{x}} \in \mathbb{S}^2.
\]
Hence, the normal derivative of \(U\) vanishes on \(\Gamma'\). 
Since $U$ is a parametrized single-layer potential that satisfies the Helmholtz equation in $D'$, and $\kappa^2$ is not an interior Neumann eigenvalue of $-\Delta$ in $D'$, then $U$ vanishes  in $D'$. The unique continuation principle then implies that \(U = 0\) in the interior of $\Gamma_{B_R}$.

By continuity of the single-layer potential, we have $U=0$ on $\Gamma_{B_R}$. The injectivity of the single-layer potential operator then yields \(\varPsi(\hat{\pmb{x}}) = 0\) in \(L^2(\mathbb{S}^2)\), which implies that \(\mathcal{S}'\left[ \pmb{p}_D, \psi_{\alpha}\right]\) has dense range. 

\end{proof}

Regarding the injectivity and dense-range property of \(\mathcal{S}_{\infty}'[\pmb{p}_D,\psi_{\alpha}]\), we refer to \cite[Theorem 3.1]{Ivanyshyn2010} for a complete proof.  
Since the Fr$\acute{\rm e}$chet derivatives $\mathcal{S}'[\pmb{p}_D,\psi_{\alpha}]$ and $\mathcal{S}_{\infty}'[\pmb{p}_D,\psi_{\alpha}]$ are injective with dense range, the linearized equivalent data equation is required to compute the boundary update. 

Let $\pmb{p}_D^{(l)}$ denote the parametrization of the current approximation of the boundary $\Gamma$ at the $l$-th iteration. The linearization of \eqref{data equation sc}--\eqref{data equation far} leads to the equations
\begin{align}
\mathcal{S}'[\pmb{p}_D^{(l)},\psi_{\alpha}]\,\pmb{q} &= f_D^{\rm sc}, \label{parametrization of frechet scattering}\\
\mathcal{S}_\infty'[\pmb{p}_D^{(l)},\psi_{\alpha}]\,\pmb{q} &= f_D^{\infty}, \label{parametrization of frechet far field}
\end{align}
where 
\(
f_D^{\rm sc} = w_D^{\rm sc} - \mathcal{S}(\pmb{p}_D^{(l)}) \psi_{\alpha}
\) and \(
f_D^{\infty} = w_D^{\infty} - \mathcal{S}_{\infty}(\pmb{p}_D^{(l)})\psi_{\alpha}.
\)
To terminate the iterative procedure, we employ the relative error estimators 
\begin{align}
E_l^{\rm sc} &= 
\frac{\|w_D^{\rm sc} - \mathcal{S}(\pmb{p}_D^{(l)})\psi_{\alpha}\|_{L^2}}
     {\|w_D^{\rm sc}\|_{L^2}} \leq \epsilon, 
\label{error sc}\\
E_l^{\infty} &= 
\frac{\|w_D^{\infty} - \mathcal{S}_\infty(\pmb{p}_D^{(l)}) \psi_{\alpha}\|_{L^2}}
     {\|w_D^{\infty}\|_{L^2}} \leq \epsilon, 
\label{error far}
\end{align}
where \(\epsilon\) is a small positive constant depending on the noise level.

In the iterative reconstruction procedure, the obstacle boundary is updated through its parametrization $\pmb p_D$, and the corresponding update is described by $\pmb q$ defined in \eqref{updateq}. Thus, the parametrization update reduces to determining the vector $\Delta\pmb{c}$ and the scalar function $\Delta r(\hat{\pmb{x}})$. Here, $\Delta r(\hat{\pmb{x}})$ is approximated by a truncated spherical harmonic expansion.

From \cite{ganesh2004,ivanshyn10}, the spherical harmonics
\[
Y_{k,j}(\theta,\phi)=c_k^jP_k^{|j|}(\cos\theta)e^{\mathrm{i}j\phi},\quad k\in\mathbb{N},\quad |j|\leq k,
\]
form a complete orthonormal system in $L^2(\mathbb{S}^2)$, where the coefficients are given by
\[ c_k^j=(-1)^{(|j|-j)/2}\sqrt{\frac{2k+1}{4\pi}\frac{(k-|j|)!}{(k+|j|)!}},
\]
and $P_k^{|j|}$ denotes the associated Legendre function with degree $k$ and order $|j|$. Noting that the radial function $r$ is real-valued, its update can be approximated in the form
\begin{equation*}
	\Delta r(\theta,\phi)=\sum_{k=1}^{M}\sum_{j=1}^{k}\alpha_{kj}\Im (Y_{k,j}(\theta,\phi))+\sum_{k=0}^{M}\sum_{j=0}^{k}\beta_{kj}\Re (Y_{k,j}(\theta,\phi)),
\end{equation*}
where the integer $M \geq 0$ is the truncation number.
In particular, when $M=0$, the expression simplifies to $\Delta r(\theta,\phi) = \beta_{00} \Re(Y_{0,0}(\theta,\phi))$, corresponding to a constant radial perturbation.

\begin{remark}
	Due to the severe ill-posedness of the inverse problem, the reconstructions are sensitive to the initial guess. To address this issue, our numerical implementation employs an adaptive strategy: starting from the truncation number $M=0$, we progressively increase $M$ through $1,2,\cdots,M_{\rm max}$ until the stopping criterion is satisfied. This process not only provides the update of the density but also ensures efficient use of the available data. Consequently, a better approximation of the position and scale can be achieved.
\end{remark}

Now, we present the iterative algorithm in Table \ref{algorithm} for the inverse problem.
\begin{table}[!htbp]
	\caption{}\label{algorithm}
	\vspace{0.5ex}
	\centering
	\begin{tabular}{cp{.8\textwidth}}
		\toprule
		\multicolumn{2}{l}{{\bf Algorithm:}\quad Iterative approach for the phased inverse scattering problem} \\
		\midrule
		Step 1 & Given an incident plane wave $ u^{\rm inc}$, collect the scattered field  $ u_D^{\rm sc}$ or far-field data $ u_D^{\infty}$.\\
		Step 2 & Choose contraction factor $\varsigma$, initial surface $\boldsymbol p_{D}^{(0)}$, tolerance $\epsilon$, and truncation set $\{M_i=i\}_{i\ge0}$. For each $M_i$, set a maximum iteration count $\textit{loop}(M_i)\in\mathbb{N}$. Initialize $i=0$ and $l=0$.\\
		Step 3 & For current surface $\boldsymbol p_{D}^{(l)}$ at truncation $M_i$, compute the density ${\psi}_\alpha^{(l,M_i)}$ from the parametrized regularized field equation~\eqref{parametrized field equation diri} or \eqref{parametrized field equation imp}.\\
		Step 4 & Calculate error $E_{l,M_i}^{\chi}(\chi\in\left\lbrace \mathrm{sc},\infty\right\rbrace )$. If $E_{l,M_i}^{\chi}\le\epsilon$, stop the iteration and output $ \boldsymbol p_{D}^{(l)}$. Else if $l\le \textit{loop}(M_i)$, solve \eqref{parametrization of frechet scattering} or \eqref{parametrization of frechet far field} to obtain an update $\boldsymbol q^{(l,M_i)}$, and set $\boldsymbol p_{D}^{(l+1)}=\boldsymbol p_{D}^{(l)}+\boldsymbol q^{(l,M_i)}$, $l= l+1$. Go to Step 3.\\
		Step 5 & If $l=\textit{loop}(M_i)$ and $E_{l,M_i}^{\rm sc}>\epsilon$, increase truncation to the next value $M_{i+1}$, and choose the surface with the minimal error for $M_i$ as the initial guess $\boldsymbol p_{D}^{(0)}$. Set $i= i+1$ and $l= 0$, then {return to Step 3}. \\
		\bottomrule
	\end{tabular}
\end{table}

\begin{remark}
The iterative approach in Table~\ref{algorithm} can be naturally extended to the case of multiple incident waves. For incident waves \( u_{1}^{\mathrm{inc}}, \dots, u_{k}^{\mathrm{inc}} \), we collect the corresponding scattered fields \( u_{D,1}^{\mathrm{s}}, \dots, u_{D,k}^{\mathrm{s}} \) or far-field data \( u_{D,1}^{\infty}, \dots, u_{D,k}^{\infty} \). The iterative approach is modified as follows: 
\begin{itemize}
	\item[(1)] Calculate the densities $\psi_{\alpha,1}, \dots, \psi_{\alpha,k}$ by equivalent field equations;
	\item[(2)] Replace the calculation of $E_{l}^{\chi}$ in equations \eqref{error sc}-\eqref{error far} with the vector form
	\[
	E_{l}^{\chi} = \frac{ \left\| \left( w^{\chi}_{D,1}, \dots, w^{\chi}_{D,k} \right)^\top - \left( \mathcal{S}_\chi(\pmb{p}_D^{(l)}) \psi_{\alpha,1}, \dots, \mathcal{S}_\chi(\pmb{p}_D^{(l)}) \psi_{\alpha,k} \right)^\top \right\|_{L^2} }{ \left\| \left( w^{\chi}_{D,1}, \dots, w^{\chi}_{D,k} \right)^\top \right\|_{L^2} } \leq \epsilon, 
	\]
	for $\chi \in \{ \mathrm{sc}, \infty \}$ with $\mathcal{S}_{\mathrm{sc}}(\pmb{p}_D^{(l)}) := \mathcal{S}(\pmb{p}_D^{(l)})$;
	\item[(3)] Transform the linearized equivalent data equations \eqref{parametrization of frechet scattering}-\eqref{parametrization of frechet far field} into the following vectorized systems in Step 5:
	\[
	\left( \mathcal{S}_\chi'[\pmb{p}_D^{(l)},\psi_{\alpha,1}], \dots, \mathcal{S}_\chi'[\pmb{p}_D^{(l)},\psi_{\alpha,k}] \right)^\top  \pmb{q} = \left(f^{\chi}_{D,1}, \dots, f^{\chi}_{D,k}\right)^\top,
	\]
	with $\mathcal{S}_{\mathrm{sc}}'[\pmb{p}_D^{(l)},\psi_{\alpha,1}]:=\mathcal{S}'[\pmb{p}_D^{(l)},\psi_{\alpha,1}]$.
\end{itemize}
As shown in Section~\ref{num ex}, employing multiple incident waves can further improve the reconstruction accuracy and robustness in some cases.
\end{remark}

%
%
%
\begin{remark}
    It is worth mentioning that the proposed method generally exhibits high computational efficiency. All numerical tests were carried out in MATLAB 2025a on a personal laptop (48 GB RAM; 4.50 GHz Apple M4 Pro) using Apple's BLAS library. For the Dirichlet boundary condition cases (e.g., Examples 1, 2, 4--6) in Section~\ref{num ex}, 1000 iterations require only about 3.240732 seconds of computation time, and satisfactory reconstructions are typically achieved within 200 iterations. The computational cost may increase for more complex geometries, such as in Example~7, as the number of computational nodes grows significantly.
\end{remark}

\section{Numerical implementation}\label{numerical section}
The full discretization of the equivalent field equations \eqref{parametrized field equation diri}-\eqref{parametrized field equation imp} and the linearized equivalent data equations \eqref{parametrization of frechet scattering}-\eqref{parametrization of frechet far field} will be discussed in this section, and then a discrete formulation for the phaseless case is also developed.

\subsection{Discretization}\label{Discretization}
In the following, we denote $f(\hat{\pmb{x}}(\theta,\phi))$ simply as $f(\theta,\phi)$ for brevity.
The Gauss-trapezoidal product rule for numerical integration of continuous functions over $\mathbb{S}^2$ is given by
\[
\int_{\mathbb{S}^2} f(\hat{\pmb{x}}) \mathrm{d}s(\hat{\pmb{x}}) \approx \frac{\pi}{n+1} \sum_{j=0}^{2n+1} \sum_{i=1}^{n+1} \varpi_i f(\theta_i,\phi_j)
\]
with quadrature points
\begin{equation}\label{discretization node}
\phi_j = {j\pi}/{(n+1)},\quad \theta_i = \arccos z_i
\end{equation}
 and weights 
\[
\varpi_i = \frac{2(1-z_i^2)}{(n+1)^2 [P_n(z_i)]^2},
\] where $z_i$ denotes the zeros of the Legendre polynomial $P_{n+1}$.

Crucially, the equivalent field equations \eqref{parametrized field equation diri}-\eqref{parametrized field equation imp} do not contain singularities, which enables us to obtain approximate solutions through the Gaussian quadrature scheme with high efficiency.

Define
\begin{align*}
L_{1,l,q}(\zeta,\xi,\theta,\phi):=&\widetilde{K}(\zeta,\xi,\theta,\phi)\Im(Y_{l,q}(\theta,\phi)),\\[1ex]
L_{2,l,q}(\zeta,\xi,\theta,\phi):=&\widetilde{K}(\zeta,\xi,\theta,\phi)\Re(Y_{l,q}(\theta,\phi)),
\end{align*}
where
\begin{align}
\widetilde{K}(\zeta,\xi,\theta,\phi):=&\mathcal{A}_1(\zeta,\xi,\theta,\phi)\sin\theta\cos\phi+\mathcal{A}_2(\zeta,\xi,\theta,\phi)\sin\theta\sin\phi\notag\\
&+\mathcal{A}_3(\zeta,\xi,\theta,\phi)\cos\theta,\notag\\[1ex]
\mathcal{A}_1(\zeta,\xi,\theta,\phi):=&K(\zeta,\xi,\theta,\phi)(c_1+\varsigma r(\theta,\phi)\sin\theta\cos\phi-R\sin\zeta\cos\xi),\label{A1}\\[1ex]
\mathcal{A}_2(\zeta,\xi,\theta,\phi):=&K(\zeta,\xi,\theta,\phi)(c_2+\varsigma r(\theta,\phi)\sin\theta\sin\phi-R\sin\zeta\sin\xi),\label{A2}\\[1ex]
\mathcal{A}_3(\zeta,\xi,\theta,\phi):=&K(\zeta,\xi,\theta,\phi)(c_3+\varsigma r(\theta,\phi)\cos\theta-R\cos\zeta),\label{A3}
\end{align}
and
{\small
\[
K(\zeta,\xi,\theta,\phi)=\left(\mathrm{i}\kappa- \frac{1}{|\pmb p_{B_R}(\zeta,\xi)-\pmb p_{D'}(\theta,\phi)|}\right)\frac{e^{\mathrm{i}\kappa|\pmb p_{B_R}(\zeta,\xi)-\pmb p_{D'}(\theta,\phi)|}}{4\pi|\pmb p_{B_R}(\zeta,\xi)-\pmb p_{D'}(\theta,\phi)|^2}J_{D'}(\theta,\phi).
\]}
Substituting these definitions into \eqref{parametrization of frechet scattering} and applying the Gauss--trapezoidal product rule yields the fully discrete linear system
\begin{equation}\label{overdetermined system}
\small
\begin{aligned}
	\sum_{{p}=1}^{3}B_{{p}}(\zeta_r,\xi_s)\Delta c_{{p}}+\sum_{l=1}^{M}\sum_{q=1}^{l}\alpha_{lq}B_{1,l,q}(\zeta_r,\xi_s)+\sum_{l=0}^{M}\sum_{q=0}^{l}\beta_{lq}B_{2,l,q}(\zeta_r,\xi_s)=f_D^{\rm sc}(\zeta_r,\xi_s),
\end{aligned}
\end{equation}
for $r=0,\cdots,\tilde{n}-1$ and $s=0,\cdots,2\tilde{n}-1$. This overdetermined system is solved for the real coefficients $\Delta c_1$, $\Delta c_2$, $\Delta c_3$, $\alpha_{lq}$ and $\beta_{lq}$, where the system matrices are defined by
\begin{align*}
B_{p}(\zeta_r,\xi_s) = &\frac{\pi}{n+1}\sum_{{j}=0}^{2n+1}\sum_{{i}=1}^{n+1}\varpi_{{i}}\mathcal{A}_{p}(\zeta_r,\xi_s,\theta_{{i}},\phi_{{j}})\psi_{\alpha}(\theta_{{i}},\phi_{{j}}),\quad {p}=1,2,3,\\
B_{{p},l,q}(\zeta_r,\xi_s) = &\varsigma\frac{\pi}{n+1}\sum_{{j}=0}^{2n+1}\sum_{{i}=1}^{n+1}\varpi_{{i}}L_{{p},l,q}(\zeta_r,\xi_s,\theta_{i},\phi_{j})\psi_{\alpha}(\theta_{{i}},\phi_{{j}}),\quad {p}=1,2.
\end{align*}

In the above discrete system and matrix definitions, $M$ denotes the truncation number. The points $\pmb p_{B_R}(\zeta_r,\xi_s)$ are the measurement nodes on the observation surface $\Gamma_{B_R}$, while $\pmb p_{D'}(\theta_i,\phi_j)$ are the quadrature nodes on the homothetic auxiliary surface $\Gamma'$. More precisely,
\[
\zeta_r=\frac{\pi r}{\tilde n},\quad \xi_s=\frac{\pi s}{\tilde n},
\quad r=0,\cdots,\tilde n-1,\quad s=0,\cdots,2\tilde n-1,
\]
so that the total number of measurement points is $2\tilde n^2$. For the Gauss--trapezoidal product rule, $\theta_i$ ($i=1,\cdots,n+1$) and $\phi_j$ ($j=1,\cdots,2n+1$) are defined in \eqref{discretization node}. Hence the total number of quadrature points is $2(n+1)^2$.

Typically, the number of unknown coefficients $(M+1)^2+3\ll 2\tilde{n}^2$ . Due to the severely ill-posed nature of the problem, we solve the overdetermined system \eqref{overdetermined system} via Tikhonov regularization. Accordingly, the linear system \eqref{overdetermined system} is expressed as an optimization problem aimed at minimizing the Tikhonov functional \cite{Tikhonov}
\begin{equation}\label{minimization}
\begin{aligned}
	\sum_{s=0}^{2\tilde{n}-1}&\sum_{r=0}^{\tilde{n}-1}\Big|\sum_{{p}=1}^{3}B_{{p}}(\zeta_r,\xi_s)\Delta c_{{p}}+\sum_{l=1}^{M}\sum_{q=1}^{l}\alpha_{lq}B_{1,l,q}(\zeta_r,\xi_s)\\
	&+\sum_{l=0}^{M}\sum_{q=0}^{l}\beta_{lq}B_{2,l,q}(\zeta_r,\xi_s)-f_D^{\rm sc}(\zeta_r,\xi_s)\Big|^2+\lambda \Bigg(|\Delta c_1|^2+|\Delta c_2|^2+|\Delta c_3|^2\\
	&+\frac{1}{2}\sum_{l=1}^{M}\sum_{q=1}^{l}(1+l(l+1))^3\alpha_{lq}^2+\beta_{00}^2+\frac{1}{2}\sum_{l=1}^{M}\sum_{q=0}^{l}(1+l(l+1))^3\beta_{lq}^2\Bigg),
\end{aligned}
\end{equation}
where $\lambda > 0$ is the regularization parameter and the last term denotes the $H^3$-penalty. Applying the least squares principle, the minimizer of \eqref{minimization} satisfies the normal equation
\begin{equation}\label{the final}
	\left(\lambda\tilde{\pmb I}+\Re(\pmb B^*\pmb B)\right)\pmb \varUpsilon=\Re(\pmb B^*\pmb f_D^{\rm sc})
\end{equation}
with the system matrix $\pmb{B}$ defined by
\[
\pmb{B} = \left( \pmb {B}_1, \pmb{B}_2, \pmb{B}_3, \pmb{B}_{1,1,1}, \cdots, \pmb{B}_{1,M,M}, \pmb{B}_{2,0,0}, \cdots, \pmb{B}_{2,M,M} \right).
\]
The column blocks of $\pmb{B}$ are given by
\begin{align*}
\pmb B_{{p}}&=\left(B_{{p}}(\pmb\zeta,\xi_0),\cdots,B_{{p}}(\pmb\zeta,\xi_{2\tilde{n}-1})\right)^\top, \quad {p}=1,2,3,\\
\pmb B_{1,l,q}&=\left(B_{1,l,q}(\pmb\zeta,\xi_0),\cdots,B_{1,l,q}(\pmb\zeta,\xi_{2\tilde{n}-1})\right)^\top,\quad l=1,\cdots,M,~ q=1,\cdots,l,\\
\pmb B_{2,l,q}&=\left(B_{2,l,q}(\pmb\zeta,\xi_0),\cdots,B_{2,l,q}(\pmb\zeta,\xi_{2\tilde{n}-1})\right)^\top,\quad l=0,\cdots,M,~ q=0,\cdots,l,
\end{align*}
where 
\[
B_\chi(\pmb{\zeta},\xi_s) = \left( B_\chi(\zeta_0,\xi_s), \cdots, B_\chi(\zeta_{\tilde{n}-1},\xi_s) \right) \quad \text{for} \quad \chi \in \{p, (1,l,q), (2,l,q)\}.
\]
The solution vector and regularization matrix are shown by
\begin{align*}
	&\pmb \varUpsilon=(\Delta c_1,\Delta c_2,\Delta c_3,\alpha_{11},\cdots,\alpha_{M1},\cdots,\alpha_{MM},\beta_{00},\cdots,\beta_{M0},\cdots,\beta_{MM})^\top,\\
	&\tilde {\pmb I}
	=\mathrm{diag}\big(
	\underbrace{1,1,1,}_{\Delta c_j,\;j=1,2,3}
	\underbrace{\varrho_1,}_{\alpha_{11}}\;
	\underbrace{\varrho_2,\varrho_2}_{\alpha_{2j},j=1,2},\;\ldots,\;
	\underbrace{\varrho_M,\cdots,\varrho_M}_{\alpha_{Mj},j=1,\ldots,M},\\
	&\qquad\qquad\underbrace{1}_{\beta_{00}},\;
	\underbrace{\varrho_1,\varrho_1}_{\beta_{1j},j=0,1},\;\ldots,\;
	\underbrace{\varrho_M,\cdots,\varrho_M}_{\beta_{Mj},j=0,\ldots,M}
	\big),\\
	&\pmb{f}_D^{\rm sc}=\left(f_D^{\rm sc}(\pmb\zeta,\xi_0),\cdots,f_D^{\rm sc}(\pmb\zeta,\xi_{2\tilde{n}-1})\right)^\top,
\end{align*}
where
\[
\varrho_k=\bigl(1+k(k+1)\bigr)^3/2,\quad k=1,\cdots,M,
\]
\[
f_D^{\rm sc}(\pmb\zeta,\xi_s)=\left( f_D^{\rm sc}(\zeta_0,\xi_s),\cdots,f_D^{\rm sc}(\zeta_{\tilde{n}-1},\xi_s)\right).
\]
Thus, the updated approximation is given by
\[
\pmb p_D^{new}=(\pmb c+\Delta \pmb c)+(r(\hat{\pmb x})+\Delta r(\hat{\pmb x}))\hat{\pmb x},
\]
where $\pmb c=(c_1,c_2,c_3)^\top$ and $\Delta\pmb c=(\Delta c_1,\Delta c_2, \Delta c_3)^\top$.

For the specific discretization of \eqref{parametrization of frechet far field}, we only need to reformulate $\mathcal{A}_{p}$ for ${p}=1,2,3$ (given in \eqref{A1}-\eqref{A3}) as follows
\begin{align*}
	\mathcal{A}_1(\zeta,\xi,\theta,\phi)&=K(\zeta,\xi,\theta,\phi)\sin\zeta\cos\xi,\\
	\mathcal{A}_2(\zeta,\xi,\theta,\phi)&=K(\zeta,\xi,\theta,\phi)\sin\zeta\sin\xi,\\
	\mathcal{A}_3(\zeta,\xi,\theta,\phi)&=K(\zeta,\xi,\theta,\phi)\cos\zeta,
\end{align*}
where 
\[
K(\zeta,\xi,\theta,\phi)=-\mathrm{i}\kappa e^{-\mathrm{i}\kappa\hat{\pmb x}(\zeta,\xi)\cdot\pmb p_{D'}(\theta,\phi)}J_{D'}(\theta,\phi)/4\pi.
\]
Correspondingly, the right hand $f_D^{\rm sc}$ and $\pmb f_D^{\rm sc}$ are replaced by $f_D^{\infty}$ and 
\(
\pmb f_D^{\infty}=\left( f_D^\infty(\pmb\zeta,\xi_0),\cdots,f_D^\infty(\pmb\zeta,\xi_{2\tilde{n}-1})\right) ^\top.
\)

\subsection{The inverse scattering problem with phaseless data}

We now consider the inverse scattering problem using phaseless data. 
Within the iterative framework developed in Section~\ref{novel approach}, 
the equivalent field equations remain unchanged as given in \eqref{parametrized field equation diri}-\eqref{parametrized field equation imp}, 
while the equivalent data equation is rewritten as
\begin{equation}\label{phaseless}
\bigl|G^{\infty}(\zeta,\xi) + (\mathcal{S}_{\infty}(\pmb{p}_D)\psi_{\alpha})(\zeta,\xi)\bigr|
= \bigl|u_D^{\rm tot,\infty}(\zeta,\xi)\bigr|,
\end{equation}
with $\zeta\in[0,\pi]$ and $\xi\in[0,2\pi)$.

The linearization of \eqref{phaseless} with respect to $\pmb{p}_D$ yields
\begin{align*}
\mathcal{S}'_P[\pmb{p}_D,\psi_{\alpha}]\pmb{q} = f_D^{\rm tot,\infty},
\end{align*}
where
\begin{align*}
\mathcal{S}_P'[\pmb{p}_D, \psi_{\alpha}]\pmb{q}&= 2\Re\left( \overline{\left(G^{\infty} + \mathcal{S}_\infty(\pmb{p}_D) \psi_{\alpha}\right)} \mathcal{S}_\infty'[\pmb{p}_D, \psi_{\alpha}]\pmb{q} \right),\\[1ex]
f_D^{\mathrm{tot},\infty}&= |u_D^{\mathrm{tot},\infty}|^2 - |G^{\infty} + \mathcal{S}_\infty(\pmb{p}_D)\psi_{\alpha}|^2,
\end{align*}
and $\mathcal{S}'_\infty[\pmb{p}_D,\psi_{\alpha}]$ is defined in \eqref{frechet far}.


Introducing
{\small
\begin{equation*}
{B}_{p}(\zeta_r, \xi_s) = \frac{2\pi}{n+1} \Re\left\{ \overline{\widetilde{M}(\zeta_r, \xi_s)} \sum_{{j}=0}^{2n+1} \sum_{{i}=1}^{n+1} \varpi_{{i}} \mathcal{A}_{p}(\zeta_r, \xi_s, \theta_{{i}}, \phi_{{j}}) \psi_{\alpha}(\theta_{{i}}, \phi_{{j}}) \right\}, \quad p=1,2,3,
\end{equation*}
\begin{equation*}
{B}_{p,l,q}(\zeta_r, \xi_s) = \varsigma \frac{2\pi}{n+1} \Re\left\{ \overline{\widetilde{M}(\zeta_r, \xi_s)} \sum_{{j}=0}^{2n+1} \sum_{{i}=1}^{n+1} \varpi_{{i}} L_{p,l,q}(\zeta_r, \xi_s, \theta_{{i}}, \phi_{{j}}) \psi_{\alpha}(\theta_{{i}}, \phi_{{j}}) \right\}, \quad p=1,2,
\end{equation*}}where 
\begin{align*}
	\widetilde{M}(\zeta_r, \xi_s) = &~G^{\infty}(\zeta_r, \xi_s) \\
	&+ \frac{1}{4(n+1)} \sum_{j=0}^{2n+1} \sum_{i=1}^{n+1} \varpi_i e^{-\mathrm{i}\kappa \hat{\pmb{x}}(\zeta_r, \xi_s) \cdot \pmb{p}_{D'}(\theta_i,\phi_j)} \psi_{\alpha}(\theta_i,\phi_j)J_{D'}(\theta_{i},\phi_{j}).
\end{align*}
We obtain the discretized linear system
\begin{equation} \label{phaseless overdetermined system}
\begin{aligned}
	\sum_{p=1}^{3} {B}_{p}(\zeta_r, \xi_s) \Delta c_{p} 
	&+ \sum_{l=1}^{M} \sum_{q=1}^{l} \alpha_{lq} {B}_{1,l,q}(\zeta_r, \xi_s) \\
	&+ \sum_{l=0}^{M} \sum_{q=0}^{l} \beta_{lq} {B}_{2,l,q}(\zeta_r, \xi_s) 
	= f_D^{\mathrm{tot},\infty}(\zeta_r, \xi_s).
\end{aligned}
\end{equation}

\noindent The discretization and regularization of \eqref{phaseless overdetermined system} follow the same procedure as in the previous subsection, replacing \(f_D^{\rm sc}\) and \(\pmb{f}_D^{\rm sc}\) in \eqref{the final} by \(f_D^{\mathrm{tot},\infty}\) and
\[
\pmb{f}_D^{\mathrm{tot},\infty} = 
\left(
f_D^{\mathrm{tot},\infty}(\pmb\zeta, \xi_0),
\cdots,
f_D^{\mathrm{tot},\infty}(\pmb\zeta, \xi_{2\tilde{n}-1})
\right)^\top.
\]

The iterative procedure for the phaseless inverse scattering problem will terminate when relative error
\[
E_l^{\rm tot}=\frac{\big\| |w_D^{\mathrm{tot}, \infty}|^2-|G^{ \infty}+\mathcal{S}_\infty(\pmb p_D^{(l)})\psi_{\alpha}|^2\big\|_{L^2}}{\big\||w_D^{\mathrm{tot}, \infty}|^2\big\|_{L^2}}\leq\epsilon
\]
holds, where $w_D^{\rm tot, \infty}(\hat{\pmb{x}})=u_D^{\mathrm{tot},\infty}(\hat{\pmb{x}})$. 

\section{Numerical experiments}\label{num ex}

This section is organized into two parts. We first numerically verify the approximation property of homothetic surfaces, and then demonstrate the performance of the proposed inversion algorithm through several representative examples.

\subsection{Numerical verification of the approximation property for homothetic surfaces}

We begin with a simple forward scattering experiment to illustrate the approximation property of homothetic surfaces. For simplicity, we restrict ourselves to the Dirichlet case. The purpose of this test is to show that the scattered field generated by a suitably chosen homothetic surface can approximate the exact scattered field to the desired accuracy.
We consider a smooth cushion-shaped obstacle \(D\) with a star-shaped boundary parameterized by
\[
\pmb{p}_D(\theta,\phi)
= \sqrt{0.3+0.1(\cos 2\phi -1)(\cos 4\theta -1)}\,\hat{\pmb{x}}(\theta,\phi),
\quad
(\theta,\phi)\in[0,\pi]\times[0,2\pi).
\]

Following the validation strategy in Section~6 of \cite{jcp22}, we construct the exact solution in the form
\[
u_D^{\rm sc}(\pmb{x}) = G(\pmb{x},\pmb{z}_0), 
\quad \pmb{z}_0=(0,\,0.05,\,0.05)^{\top},
\]
where $G$ is the fundamental solution of the three-dimensional Helmholtz equation.

The numerical density $g$ on $\Gamma'$ is obtained by solving the boundary integral equation~\eqref{dirichlet novel} with the boundary data
\[
f_1(\pmb{x}) = G(\pmb{x},\pmb{z}_0),\quad \pmb{x}\in\Gamma,
\]
and then the numerical solution $u^{\rm sc}_{D'}$ generated by the homothetic surface $\Gamma'$ can be computed directly by \eqref{scattered}. 

Since the boundary integral equation~\eqref{dirichlet novel} is ill-posed, we employ Tikhonov regularization as shown in \eqref{field equation diri} by solving the normal equation
\begin{equation}\label{regular}
	\left( \alpha I + S_D^{*} S_D \right) g_{\alpha} = S_D^{*} f_1, 
\end{equation}
where $\alpha>0$.
As another regularized formulation, we may also consider the augmented system
\begin{align}\label{zengguang}
\begin{bmatrix}
S_D \\[2mm]
\breve{\alpha} I
\end{bmatrix}
g_{\alpha}
=
\begin{bmatrix}
f_1 \\[2mm]
0
\end{bmatrix},
\end{align}
which can be efficiently solved by a standard QR decomposition algorithm. 
When the regularization parameter $\breve{\alpha}=\sqrt{\alpha}$, the augmented system \eqref{zengguang} yields the same normal equation as \eqref{regular}; otherwise, the two formulations can be viewed as different regularization strategies for \eqref{dirichlet novel}.

The detailed discretization of the boundary integral equations is presented in Section~\ref{numerical section}. The relative $L^2$ errors are evaluated over 800 observation points on $\Gamma_{B_R}$ with $R=1$ (equally distributed in both the polar angle $\zeta$ and the azimuthal angle $\xi$) according to
\[
\epsilon_{\rm rel} 
= 
\frac{\|u^{\rm sc}_{D} - u^{\rm sc}_{D'}\|_{L^2(\Gamma_{B_R})}}{\|u^{\rm sc}_{D}\|_{L^2(\Gamma_{B_R})}} 
\approx
\frac{
\left(
\sum_{j=1}^{800} |u^{\rm sc}_{D}(\pmb{x}_j) - u^{\rm sc}_{D'}(\pmb{x}_j)|^2
\right)^{1/2}
}{
\left(
\sum_{j=1}^{800} |u^{\rm sc}_{D}(\pmb{x}_j)|^2
\right)^{1/2}
}.
\]

We test the method at wavenumber \(\kappa = 1\) with geometric contraction factor \(\varsigma = 0.5\) and fixed regularization parameters \(\alpha = 10^{-10}\), \(\breve{\alpha}=10^{-12}\).
The surface \(\Gamma'\) is discretized using \(n+1\) points in \(\theta\) and \(2(n+1)\) points in \(\phi\), as defined by equation~\eqref{discretization node} in Section~\ref{numerical section}, and the same discretized nodes are used on \(\Gamma\).

Table~\ref{tab:cushion} compares the results for the augmented system~\eqref{zengguang} and the normal equation~\eqref{regular}.
For the augmented system, the relative error $\epsilon_{\rm rel}$ decays rapidly as \(n\) increases and reaches machine precision around \(n=20\), maintaining a level of \(10^{-15}\) thereafter.
In contrast, the normal equation stagnates at a relative error of about \(10^{-8}\) for all \(n\), due to its squared condition number and increased sensitivity to round-off errors. 
Although its accuracy is lower than that of the augmented system, it offers faster computation.
Therefore, we adopt the normal equation formulation to solve the equivalent field equations as described in Section~\ref{numerical section}. Its accuracy is sufficient for the inverse problem, and its lower computational cost makes it more efficient when the equivalent field equations are solved repeatedly throughout the reconstruction process.

\begin{table}[!htbp]
\centering
\caption{Scattering by a cushion-shaped obstacle (Dirichlet case, $\kappa=1$).
Relative errors and computational time for the augmented system and normal equation formulations.}
\label{tab:cushion}
\vspace{0.5ex}
\begin{tabular}{c c c c c}
\hline
\multirow{2}{*}{$n$}
 & \multicolumn{2}{c}{Augmented system}
 & \multicolumn{2}{c}{Normal equation} \\
\cline{2-5}
 & $\epsilon_{\rm rel}$ & Time (s)
 & $\epsilon_{\rm rel}$ & Time (s) \\
\hline
\vspace{-2.2ex} \\
5 & $1.3261\times10^{-5}$ & 0.00056 & $1.3087\times10^{-5}$ & 0.00021 \\
10 & $1.2804\times10^{-7}$ & 0.00619 & $1.8153\times10^{-7}$ & 0.00750 \\
15 & $2.9817\times10^{-9}$ & 0.06171 & $1.0396\times10^{-7}$ & 0.01699 \\
20 & $4.1626\times10^{-11}$ & 0.32119 & $4.2147\times10^{-8}$ & 0.01984 \\
25 & $5.9130\times10^{-13}$ & 1.22071 & $4.2802\times10^{-8}$ & 0.05332 \\
30 & $1.5393\times10^{-14}$ & 3.71026 & $4.6967\times10^{-8}$ & 0.12984 \\
35 & $2.4155\times10^{-15}$ & 5.71834 & $4.6384\times10^{-8}$ & 0.24029 \\
40 & $1.6422\times10^{-15}$ & 11.40335 & $4.6245\times10^{-8}$ & 0.37619 \\
\hline
\end{tabular}
\end{table}

To further illustrate the influence of the regularization parameters in the forward approximation test, we compare the normal equation~\eqref{regular} and the augmented system~\eqref{zengguang} by varying $\alpha$ and $\breve{\alpha}$ independently under 
the setting $\kappa=1$, $\varsigma=0.5$, and $n=30$.
The same numerical parameter values are tested in the two formulations, but they should not be interpreted as equivalent regularization levels. 
Indeed, solving the augmented system~\eqref{zengguang} is equivalent to solving the normal equation with parameter $\breve{\alpha}^{\,2}$; hence the two formulations coincide when $\alpha=\breve{\alpha}^{\,2}$.

\begin{table}[!htbp]\centering

\caption{Sensitivity of the relative $L^2$-error $\epsilon_{\rm rel}$ to $\alpha$ in~\eqref{regular} and $\breve{\alpha}$ in~\eqref{zengguang}.}
\label{alpha-breve-alpha-relation}
\vspace{1.2ex}
\begin{tabular}{c|cc}
\toprule
$\alpha$ or $\breve{\alpha}$ 
& $\epsilon_{\rm rel}(\alpha)$
& $\epsilon_{\rm rel}(\breve{\alpha})$ \\
\midrule
$10^{-2}$  & $5.24\!\times\!10^{-3}$  & $9.33\!\times\!10^{-4}$ \\
$10^{-4}$  & $9.33\!\times\!10^{-4}$  & $8.81\!\times\!10^{-7}$ \\
$10^{-6}$  & $1.90\!\times\!10^{-5}$  & $2.65\!\times\!10^{-9}$ \\
$10^{-8}$  & $8.81\!\times\!10^{-7}$  & $2.27\!\times\!10^{-11}$ \\
$10^{-10}$ & $4.70\!\times\!10^{-8}$  & $1.75\!\times\!10^{-13}$ \\
$10^{-12}$ & $2.65\!\times\!10^{-9}$  & $1.54\!\times\!10^{-14}$ \\
$10^{-14}$ & $1.73\!\times\!10^{-10}$ & $2.46\!\times\!10^{-14}$ \\
\bottomrule
\end{tabular}
\end{table}

Table~\ref{alpha-breve-alpha-relation} shows that the augmented system~\eqref{zengguang} is more stable for small values of $\breve{\alpha}$, while the normal equation~\eqref{regular} is more sensitive to $\alpha$ because the formation of $S_D^*S_D$ squares the condition number. 
The relation $\alpha=\breve{\alpha}^{\,2}$ is also reflected in the table; for instance, the error from~\eqref{zengguang} with $\breve{\alpha}=10^{-6}$ is consistent with the error from~\eqref{regular} with $\alpha=10^{-12}$.

\subsection{Numerical experiments for the inverse approach}

We now present several numerical examples to examine the accuracy, stability, and robustness of the proposed inversion algorithm.
The scattered field or far-field data are numerically simulated at \(400\) measurement points, i.e., \(\tilde{n} = 20\). 
To avoid the ``inverse crime'', the forward scattering problem is solved via a Galerkin method based on the single-layer potential formulation~\cite{ganesh2004}, while the inverse problem is addressed by recovering the density via a regularization scheme based on homothetic surfaces.
To test stability, the synthetic noisy data are considered as follows:
\begin{align*}
u_\delta^{\rm sc}&=(1+\delta\Theta)u_D^{\rm sc},\\
u_\delta^\infty&=(1+\delta\Theta)u_D^\infty,\\
|u_\delta^{\mathrm{tot},\infty}|^2&=(1+\delta\Theta_1)|u_D^{\mathrm{tot},\infty}|^2,
\end{align*}
where $\Theta=\Theta_1+\mathrm{i}\Theta_2$, $\Theta_1$ and $\Theta_2$ are independent and normally distributed random variables supported in $[-1,1]$, and $\delta>0$ denotes the relative noise level.

Similar to the iteration process in \cite{zhao22}, for fixed scaling factor $\rho > 0$, the update $\varUpsilon$ is computed by
\vspace{-1ex}
\[
\pmb \varUpsilon
= \rho\left(\lambda\tilde{\pmb I}+\Re(\pmb{B}^*\pmb{B})\right)^{-1}
\Re(\pmb{B}^*\pmb{f}_D^{\rm sc}),
\]
\vspace{-0.5ex}
for the scattered field data case, while the same strategy can be employed for other types of data.

Throughout all the numerical examples, unless otherwise stated, we adopt the sphere of radius \(R\) centered at \((0,0,0)^\top\) as the observation surface.
The initial guess is also a sphere with center \(\pmb{c}^{(0)}\) and radius \(r^{(0)}\). 
The regularization parameters $\alpha$ and $\lambda$ are chosen as $10^{-8}$.
The scaling factor and geometric contraction factor are fixed at \(\rho=0.2\) and \(\varsigma=0.9\), respectively.
The maximum truncation number is \(M_{\rm max}=5\), and the maximum number of iterations for each truncation number is \(\textit{loop}=20\).
The number of quadrature points on \(\Gamma'\) is \(162\) (i.e., \(n=8\)), while the number of discrete nodes on the current approximate boundary \(\Gamma\) is \(338\).
In all numerical examples, $\epsilon$ denotes the stopping tolerance and it is chosen empirically. For each combination of data type and noise level, except for the non-star-shaped obstacle and limited-aperture cases, the same value is used consistently throughout all examples.

\begin{table}[!htbp]
	\caption{}\label{obstacle}
	\centering
	\vspace{0.5ex}
	\begin{tabular}{lll}
		\toprule  
		Type           &Parametrization of the exact boundary curves\\
		\midrule  
		pinched ball-shaped   & $\pmb p_D(\theta, \phi)=\sqrt{0.3+0.12 \cos 2 \phi(\cos 2 \theta-1)} \hat{\pmb{x}}(\theta,\phi),$ \\
		~\\
		cushion-shaped  & 
		$\pmb p_D(\theta, \phi)=\sqrt{0.3+0.1(\cos 2 \phi-1)(\cos 4 \theta-1)} \hat{\pmb{x}}(\theta,\phi)$.\\   
		\bottomrule 
	\end{tabular}
\end{table}


\begin{figure}[!htbp]
	\centering 
	\subfigure[True shape with multiple views of pinched ball-shaped obstacle.]{
		\begin{tabular}{cccc}
			\includegraphics[width=0.21\textwidth]{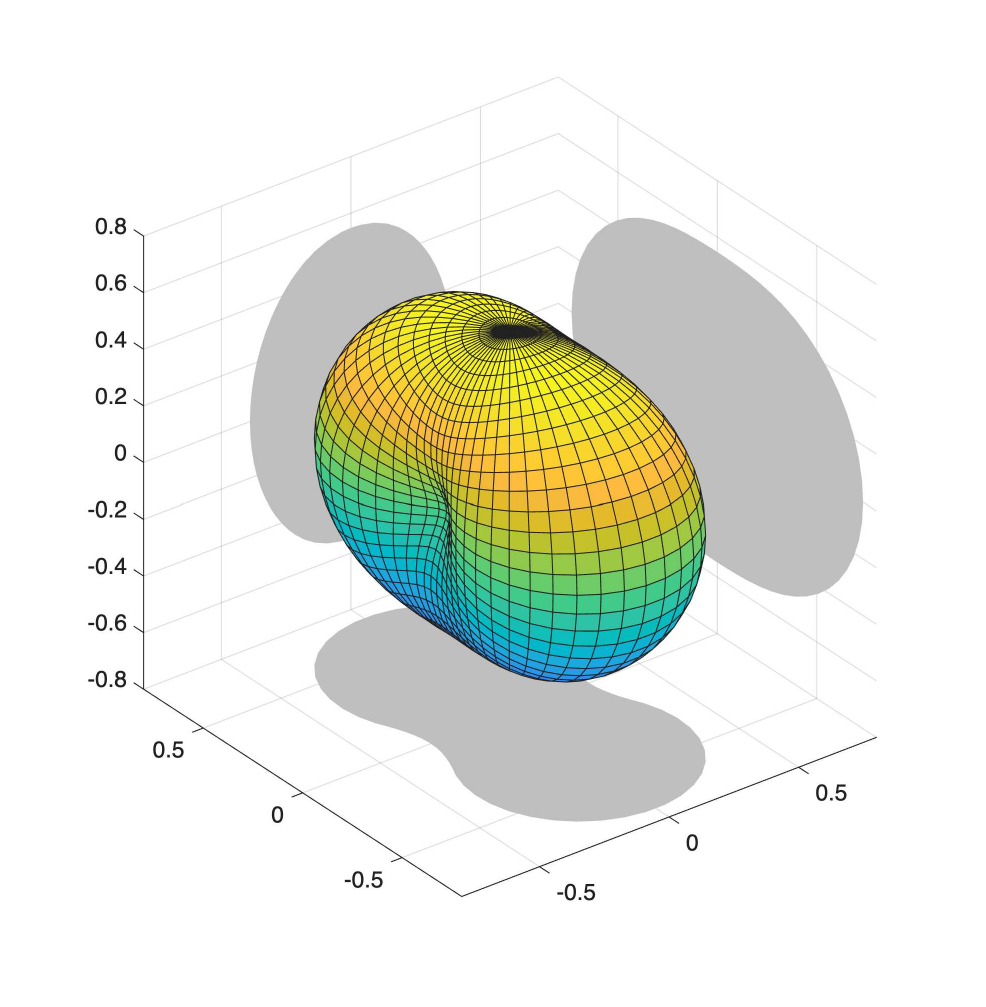} &
			\includegraphics[width=0.21\textwidth]{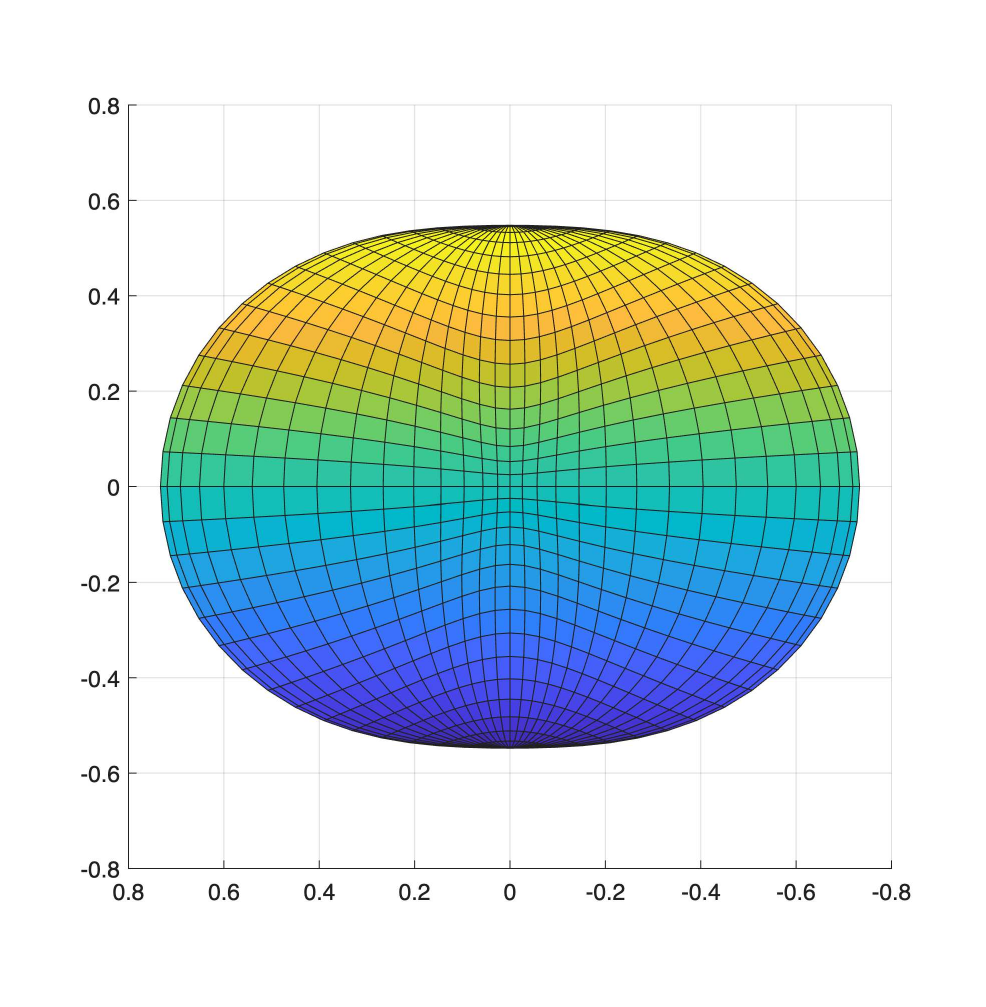} &
			\includegraphics[width=0.21\textwidth]{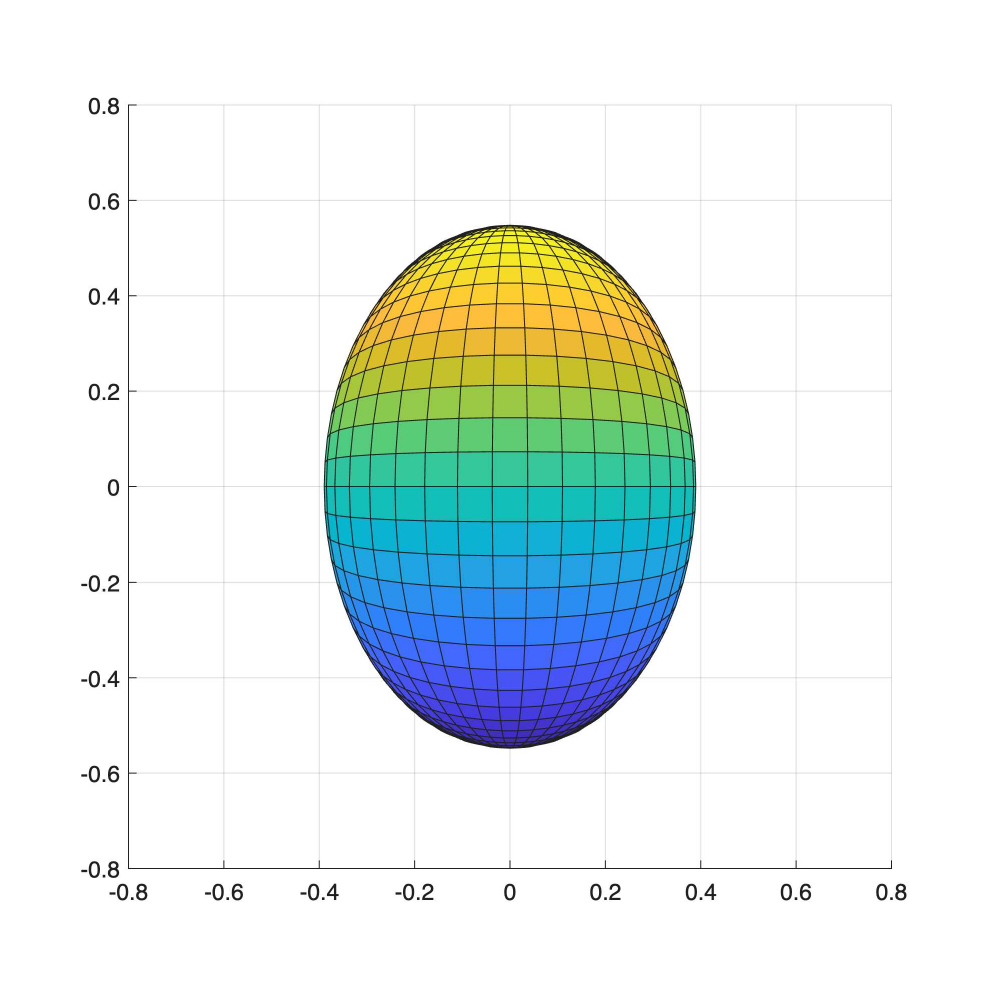} &
			\includegraphics[width=0.21\textwidth]{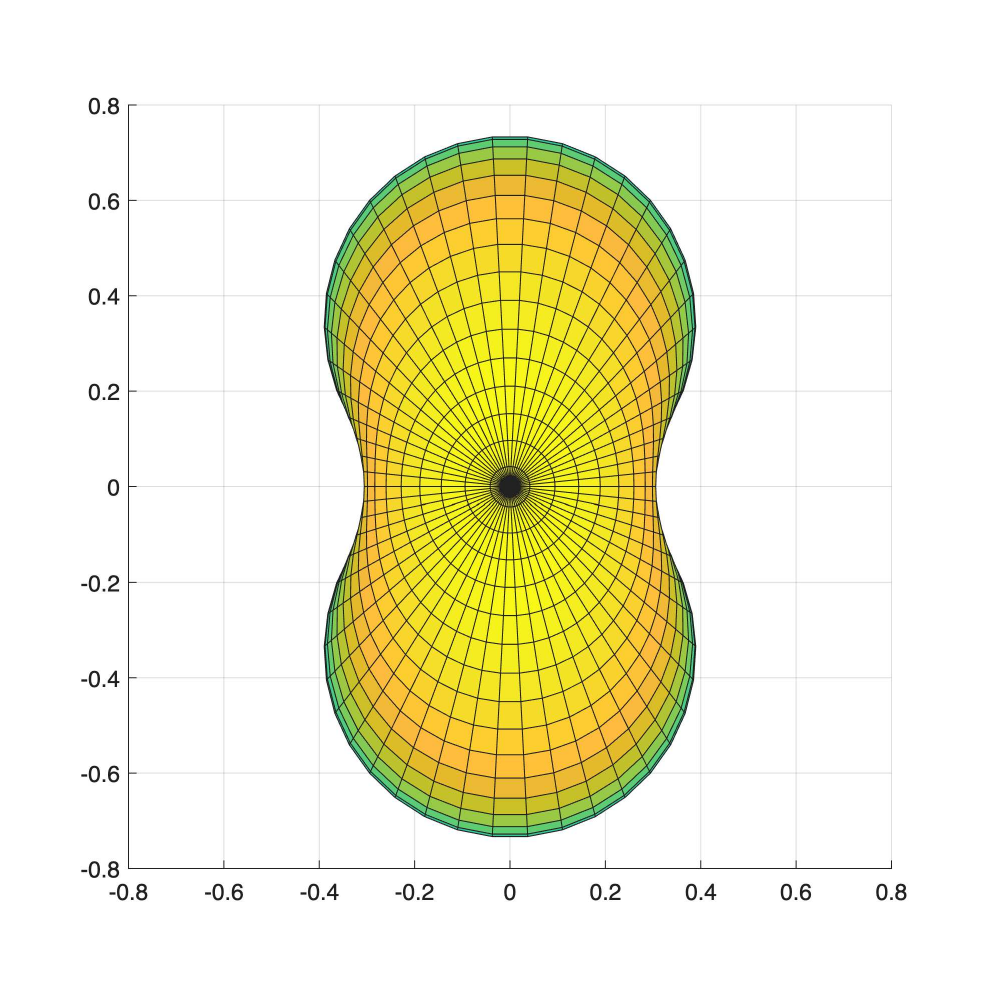}
		\end{tabular}
	} 
	
	\subfigure[True shape with multiple views of cushion-shaped obstacle.]{
		\begin{tabular}{cccc}
			\includegraphics[width=0.21\textwidth]{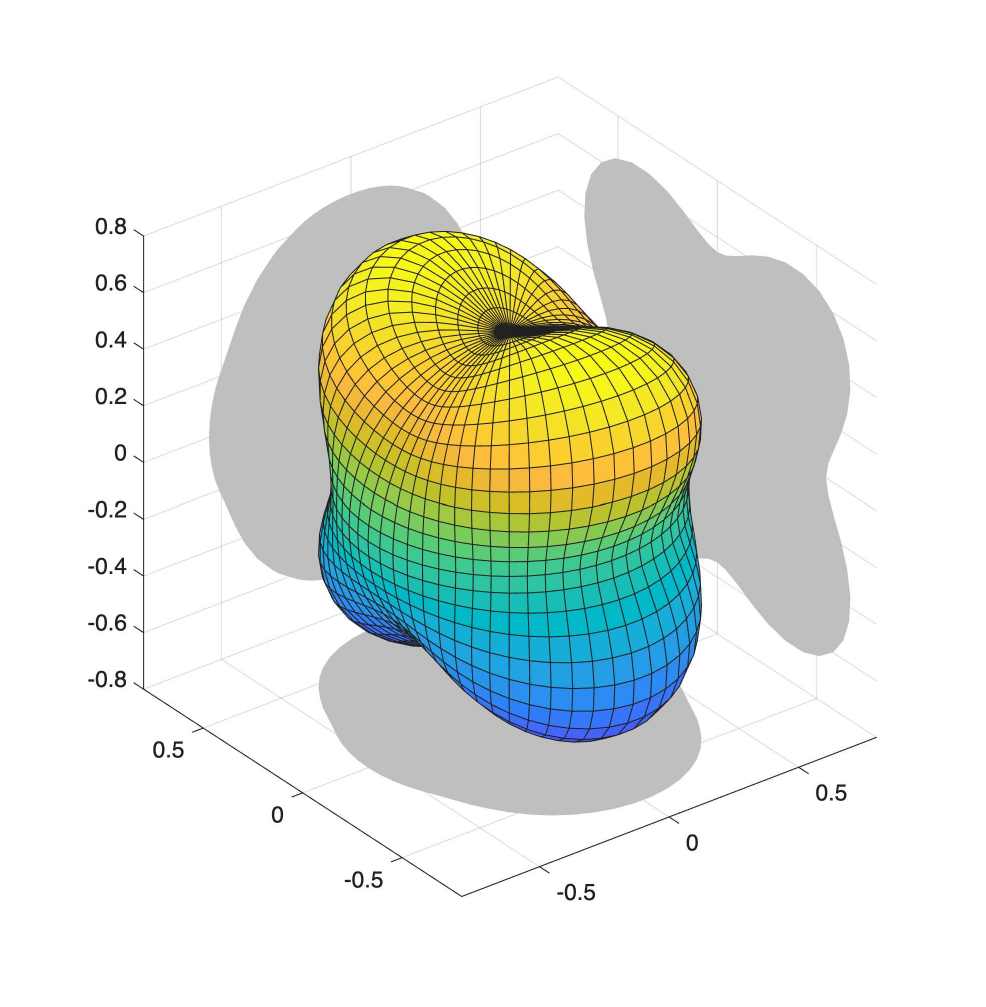} &
			\includegraphics[width=0.21\textwidth]{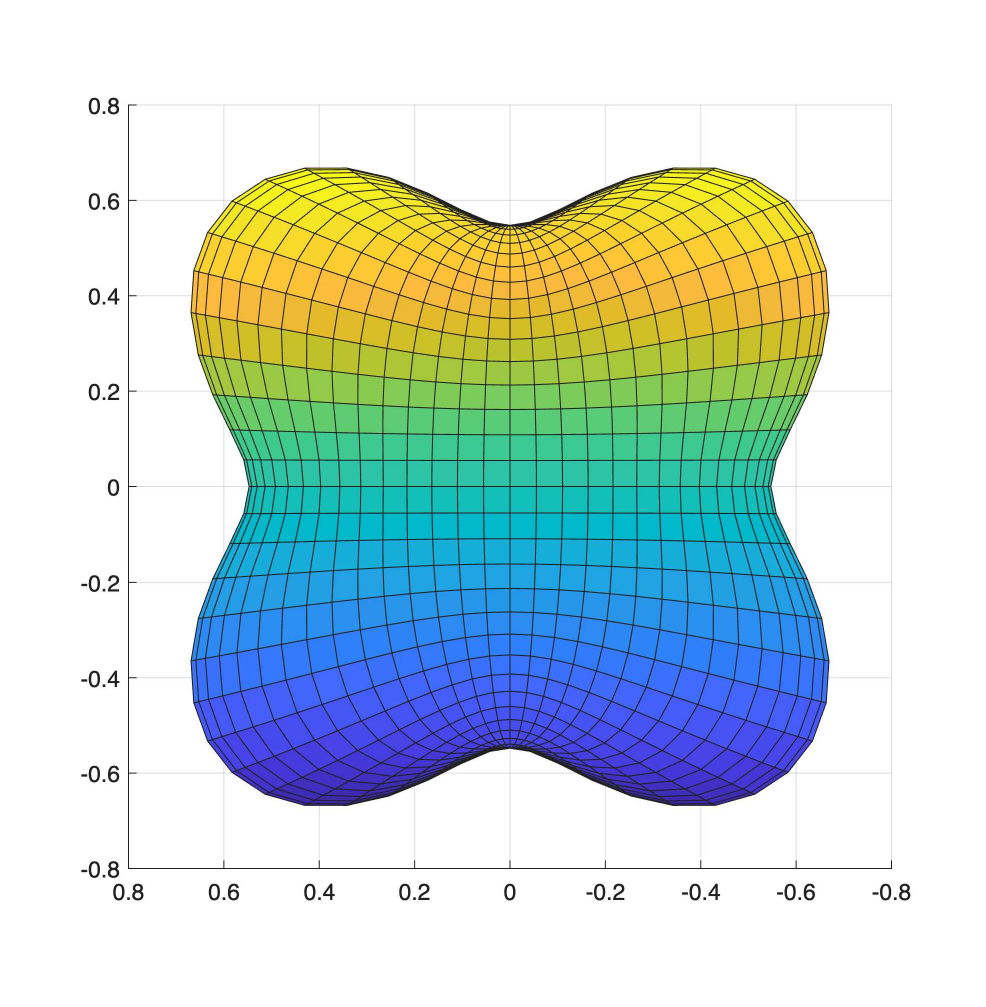} &
			\includegraphics[width=0.21\textwidth]{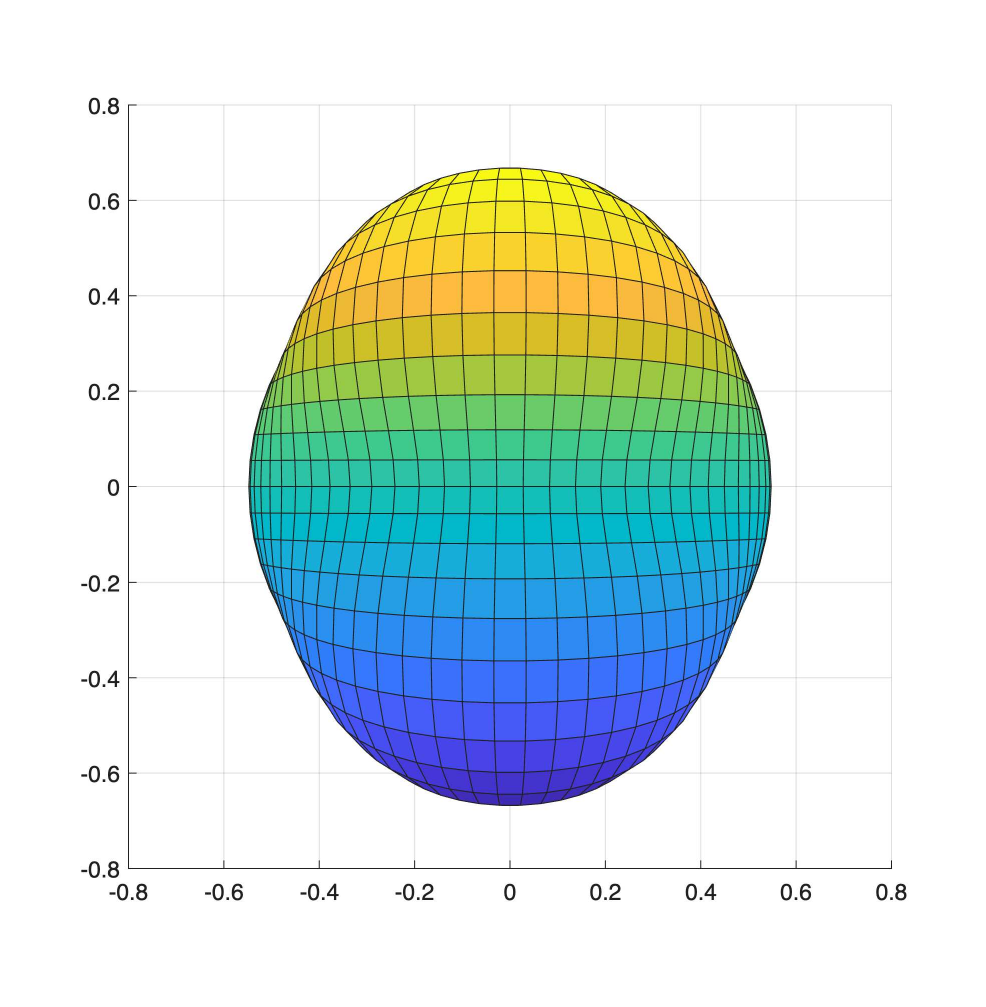} &
			\includegraphics[width=0.21\textwidth]{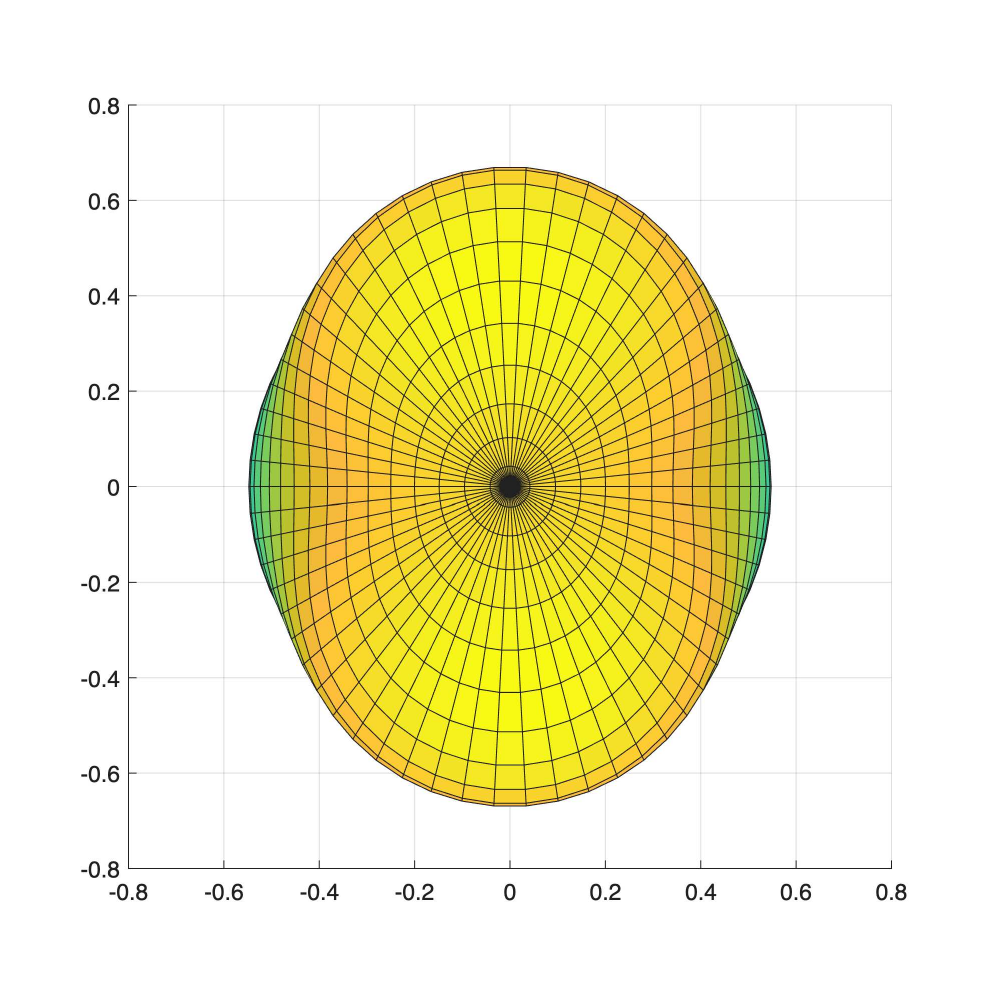}
		\end{tabular}
		
	} 
	\caption{True shape of 3D plots with front, side and top views.}\label{true shape}
\end{figure}

\begin{figure}[!htbp]
\centering 
\subfigure[$\pmb c^{(0)} = (-0.1, 0.4, -0.3)^\top$, $r^{(0)} = 0.3$, $\kappa=3.5$, $R=5$, $\pmb d=(0, 0, 1)^\top$.]{
    \begin{tabular}{ccc}
        {\includegraphics[width=0.285\textwidth]{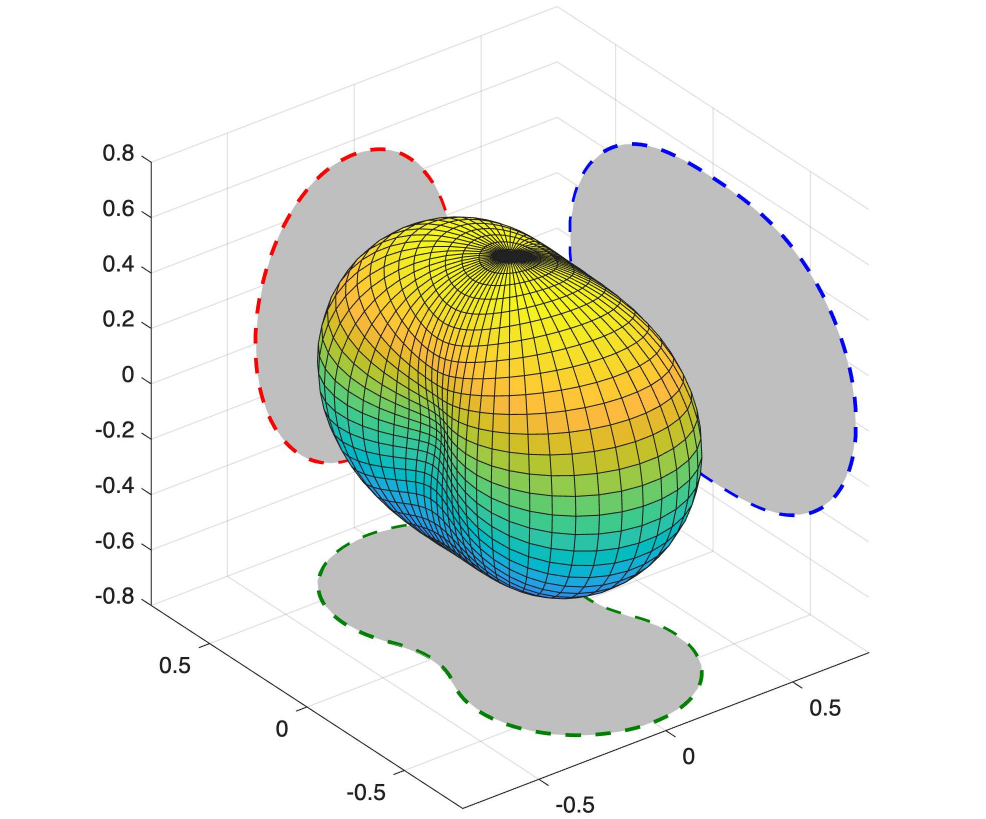}} &
        {\includegraphics[width=0.285\textwidth]{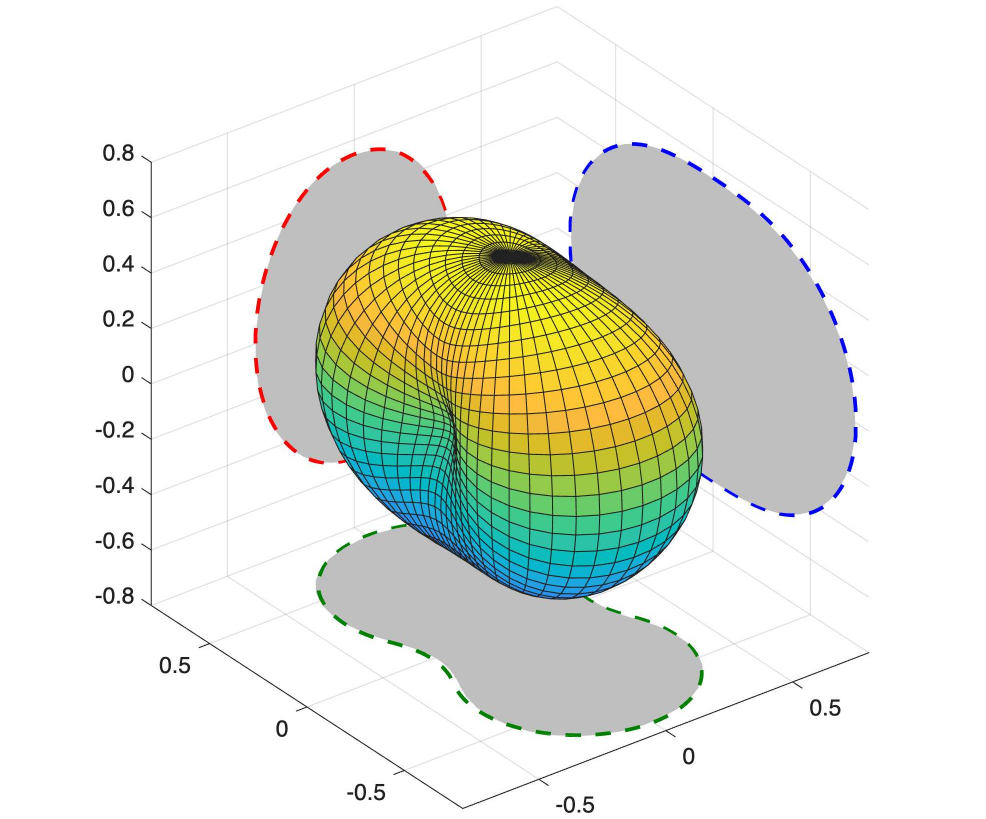}} &
        {\includegraphics[width=0.285\textwidth]{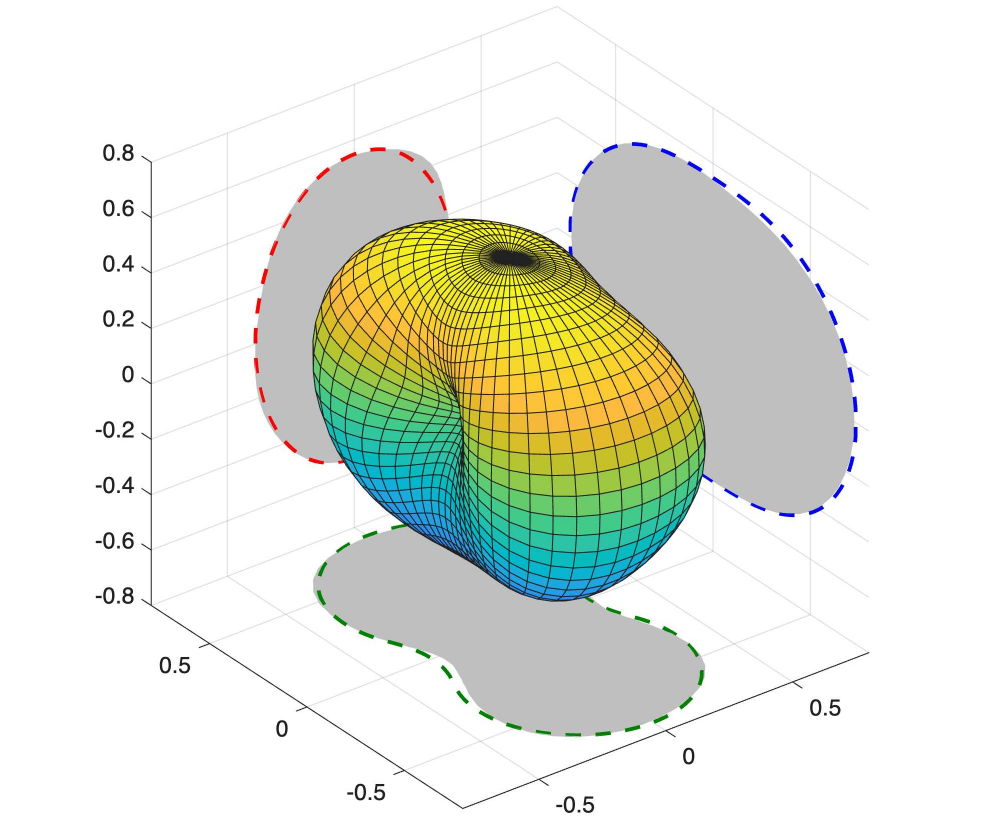}}
    \end{tabular}
}

\subfigure[$\pmb c^{(0)} = (0.3, 0.4, 0.2)^\top$, $r^{(0)} = 0.6$, $\kappa=4.5$, $\pmb d=(0, 0, 1)^\top$.]{
    \begin{tabular}{ccc}
        {\includegraphics[width=0.285\textwidth]{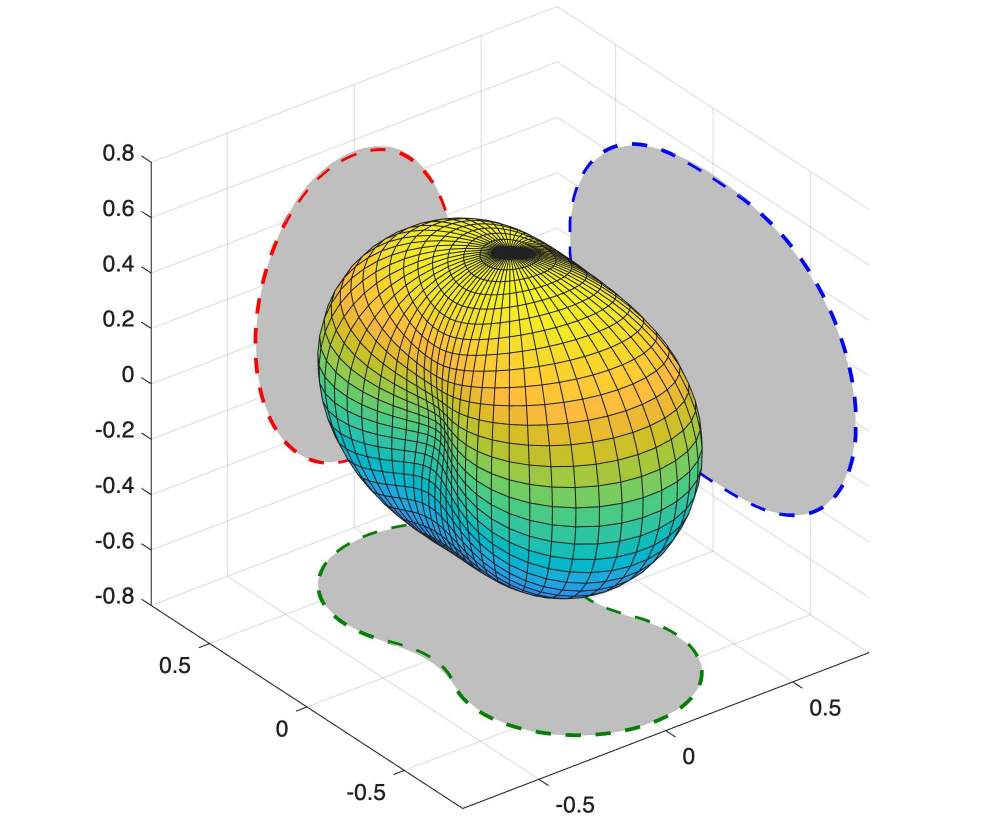}} &
        {\includegraphics[width=0.285\textwidth]{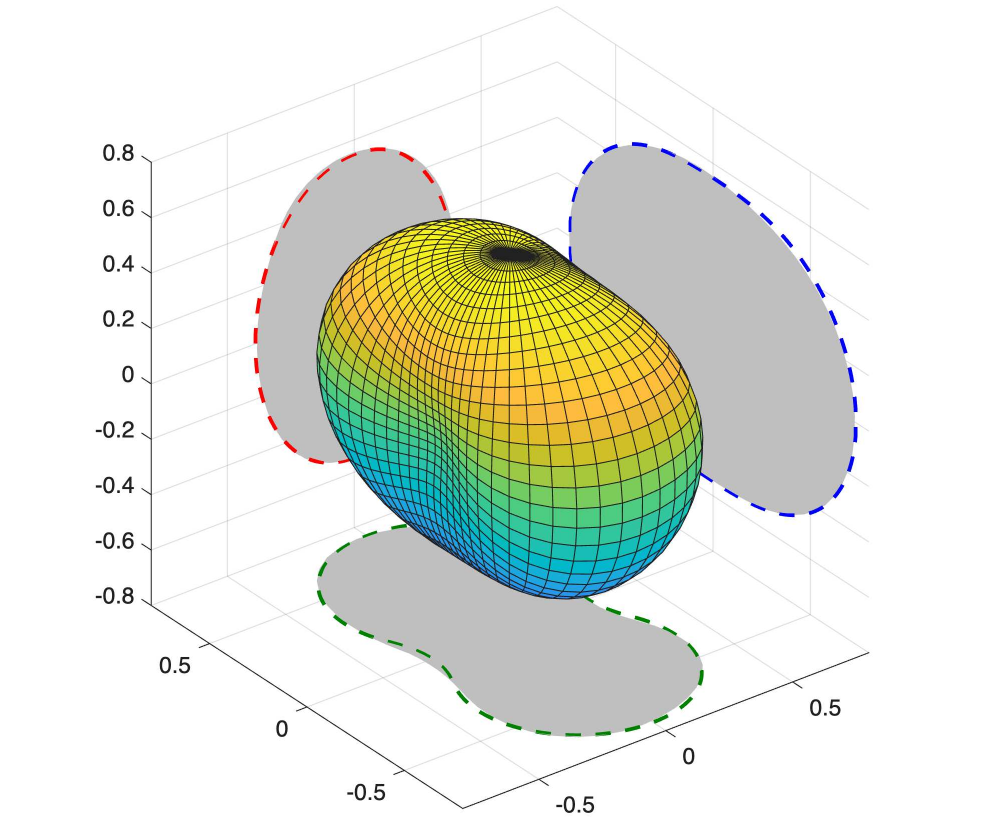}} &
        {\includegraphics[width=0.285\textwidth]{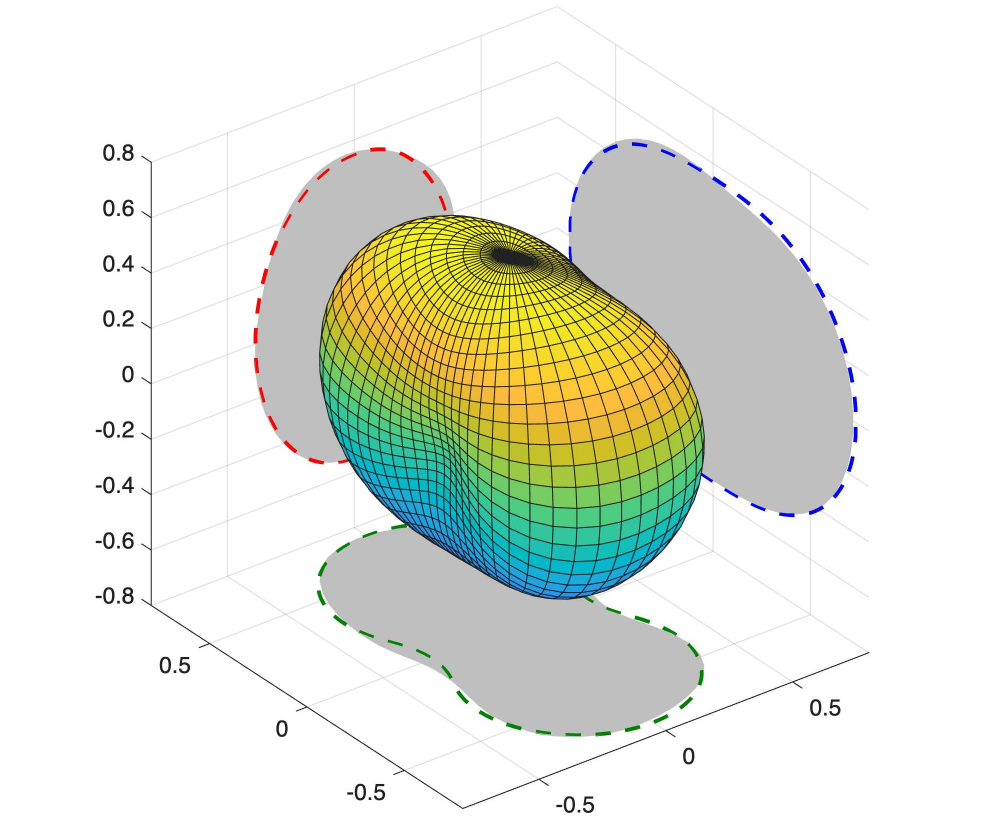}}
    \end{tabular}
}

\subfigure[$\pmb c^{(0)} = (0.5, -0.5, 0.1)^\top$, $r^{(0)} = 0.3$, $\kappa=4.5$, $\pmb z=(0, 0, 4)^\top$.]{
    \begin{tabular}{ccc}
        {\includegraphics[width=0.285\textwidth]{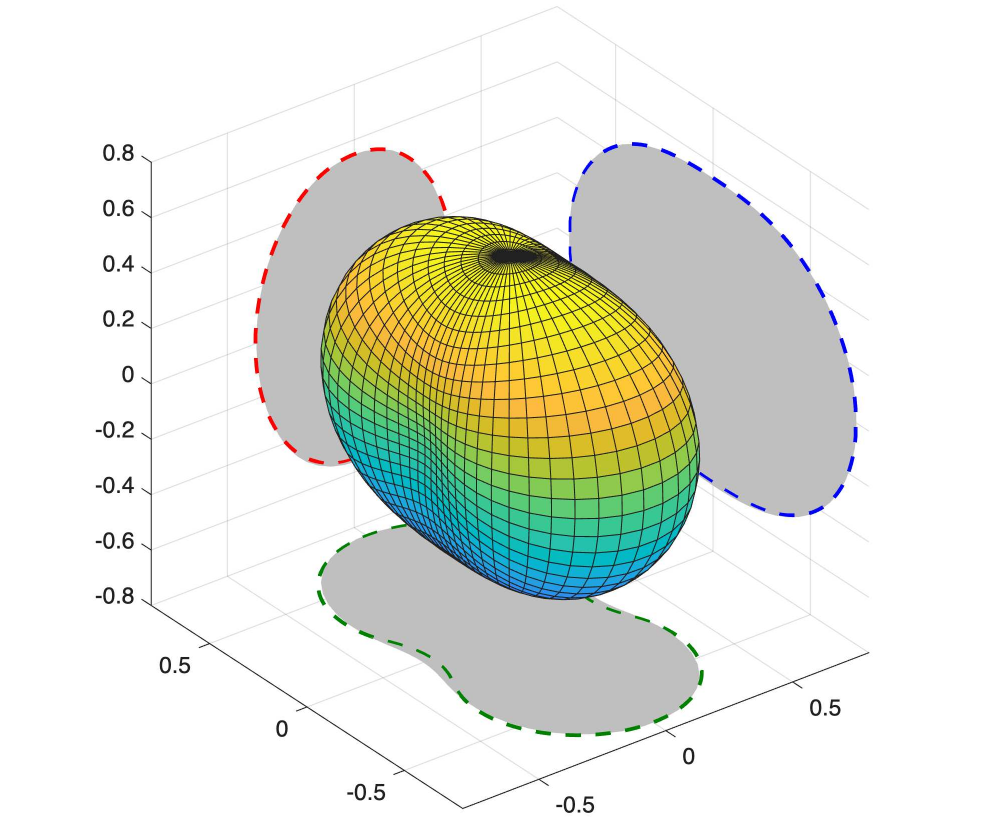}} &
        {\includegraphics[width=0.285\textwidth]{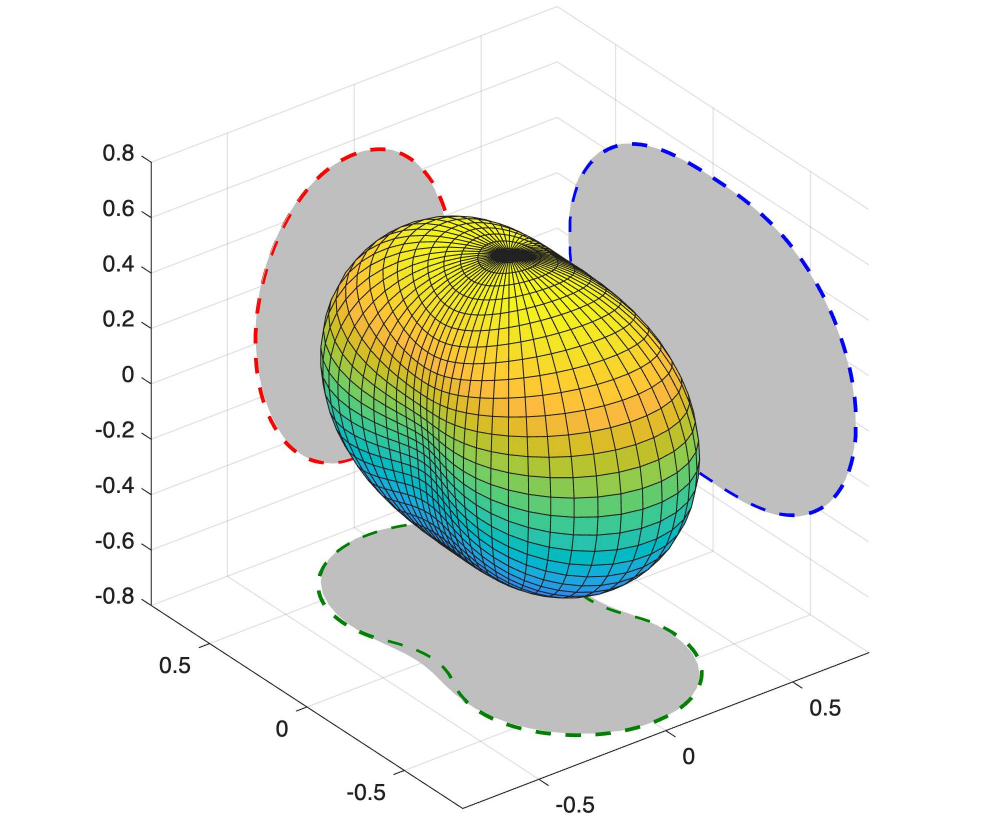}} &
        {\includegraphics[width=0.285\textwidth]{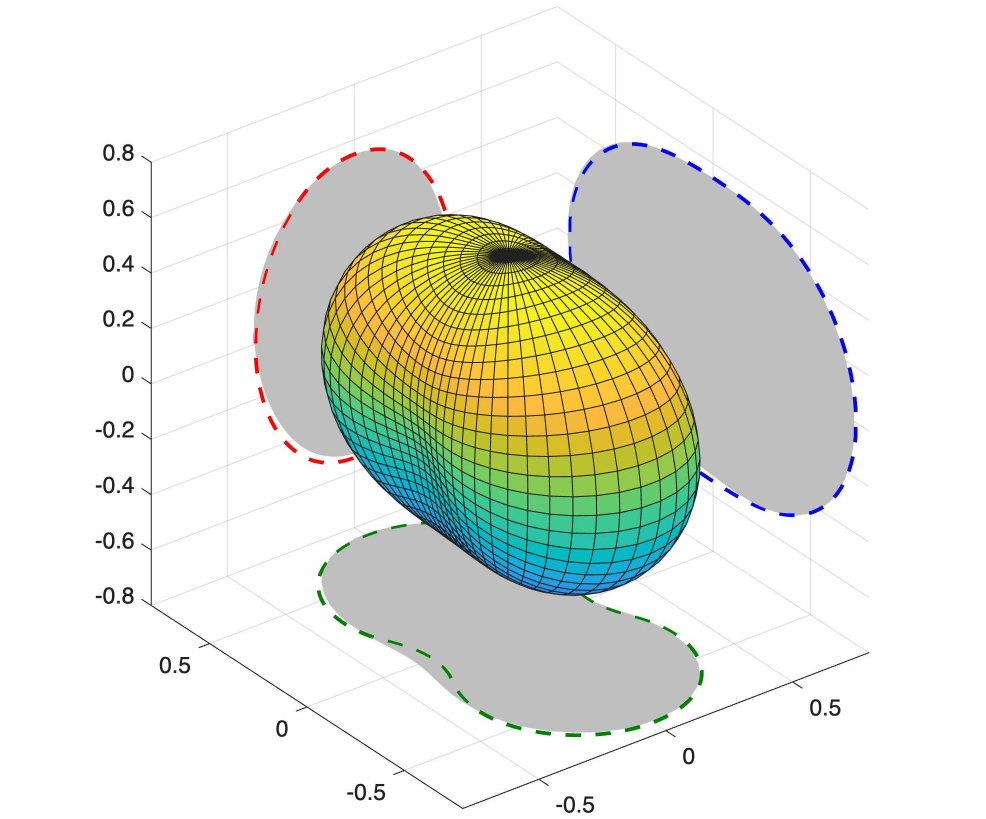}}
    \end{tabular}
}
\caption{Reconstructions of a pinched ball-shaped obstacle with single incident wave. Subfigures~(a)-(c) show results from phased scattered field, phased far-field, and phaseless far-field data, respectively, each with $1\%$, $5\%$ and $10\%$ noise. The stopping parameters $\epsilon=(0.009,0.023,0.045)$ for (a), $\epsilon=(0.018,0.032,0.050)$ for (b), and $\epsilon=(0.004,0.015,0.030)$ for (c).}\label{fig_ex1.1.1}
\end{figure}

\begin{figure}[!htbp]
\centering 
\subfigure[$\pmb c^{(0)} = (0.1, -0.7, 0.1)^\top$, $r^{(0)} = 0.3$, $\kappa=5$, $R=5$, $\pmb d=(0, 0, 1)^\top$.]{
    \begin{tabular}{ccc}
        {\includegraphics[width=0.285\textwidth]{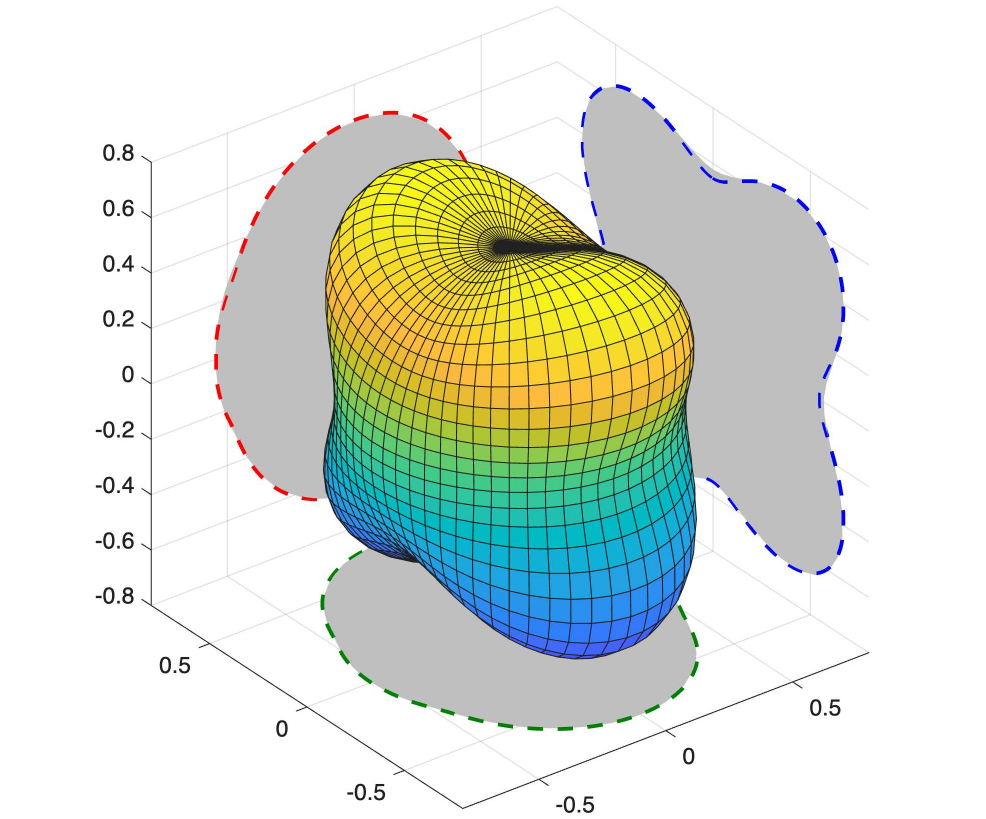}} &
        {\includegraphics[width=0.285\textwidth]{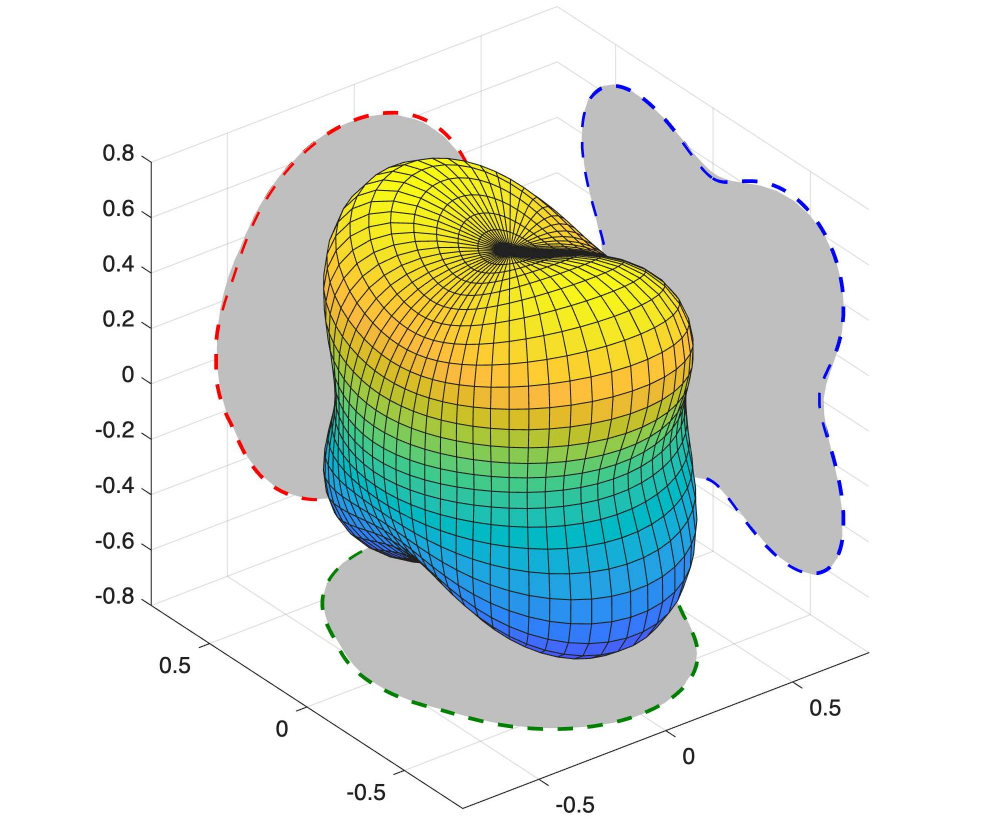}} &
        {\includegraphics[width=0.285\textwidth]{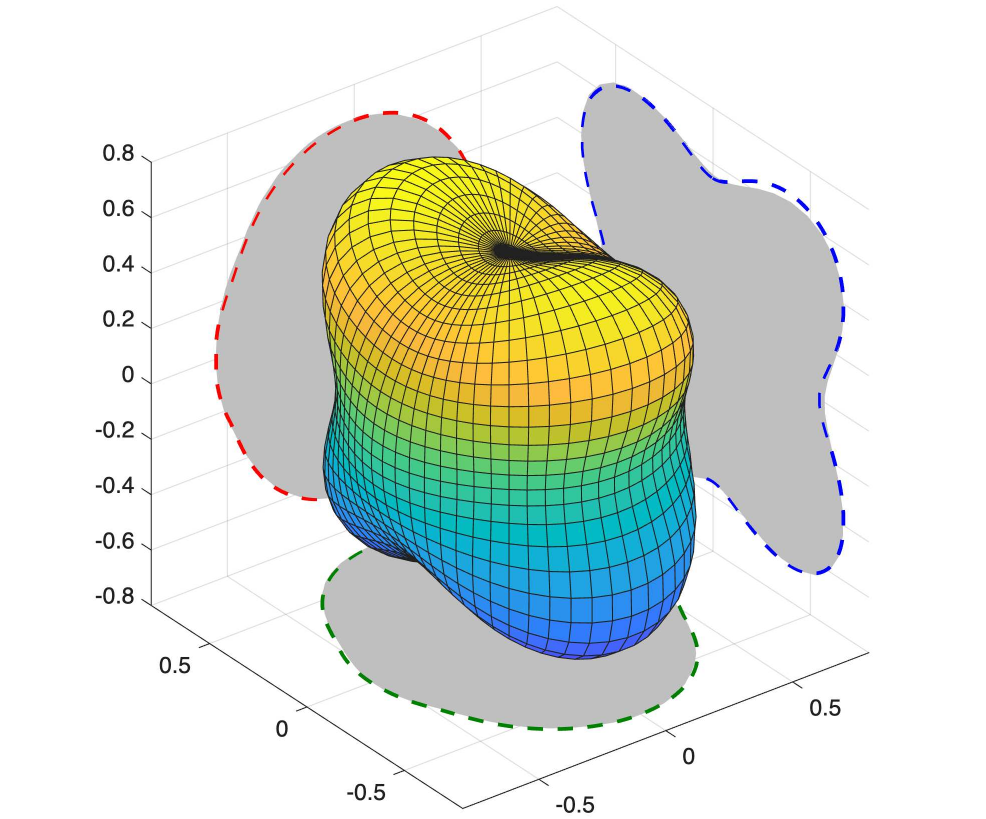}}
    \end{tabular}
}

\subfigure[$\pmb c^{(0)} = (-0.6, -0.3, 0.1)^\top$, $r^{(0)} = 0.3$, $\kappa=4.5$, $\pmb d=(0, 0, 1)^\top$.]{
    \begin{tabular}{ccc}
        {\includegraphics[width=0.285\textwidth]{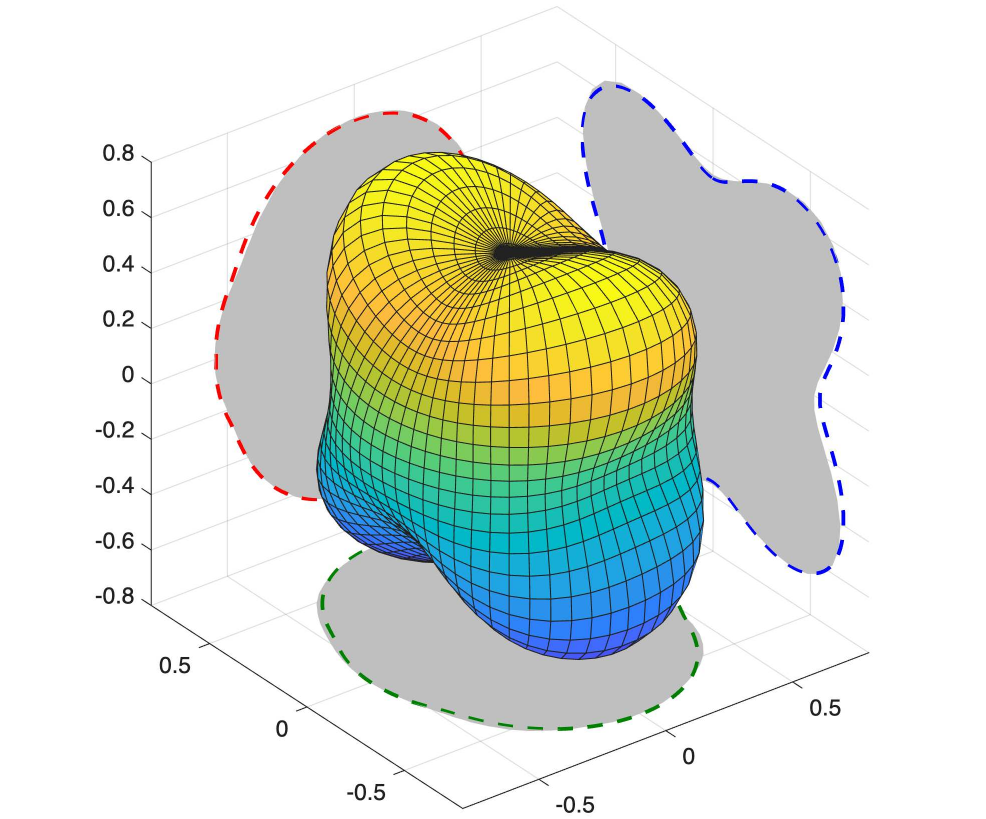}} &
        {\includegraphics[width=0.285\textwidth]{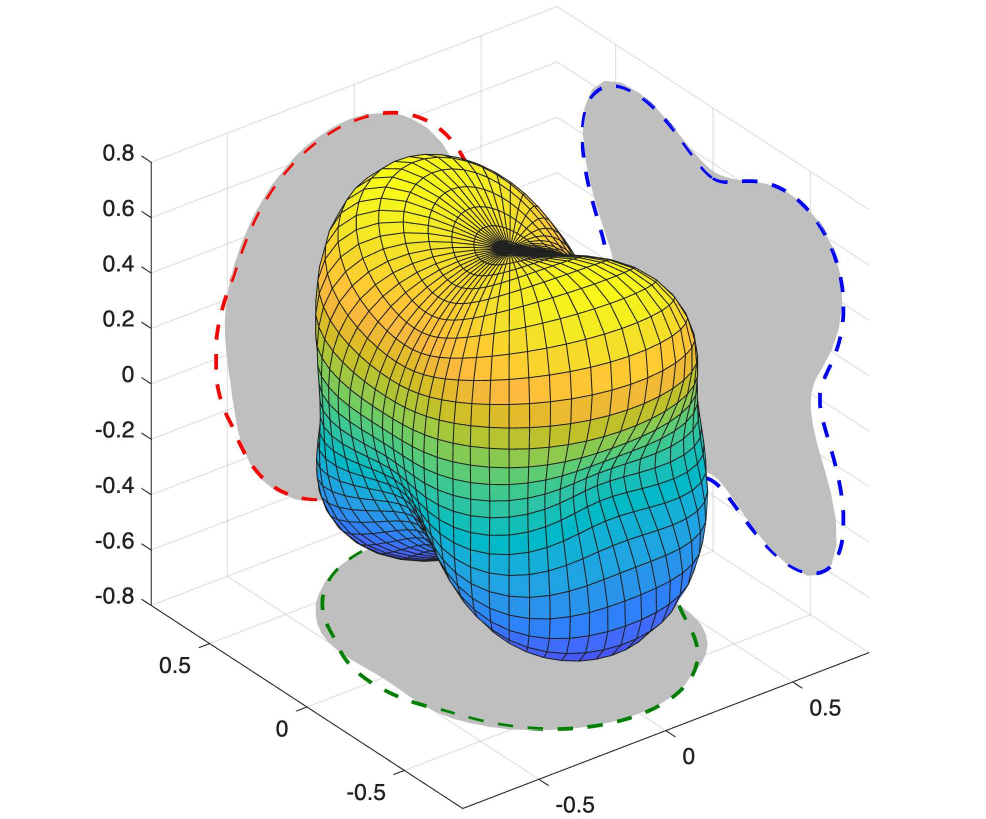}} &
        {\includegraphics[width=0.285\textwidth]{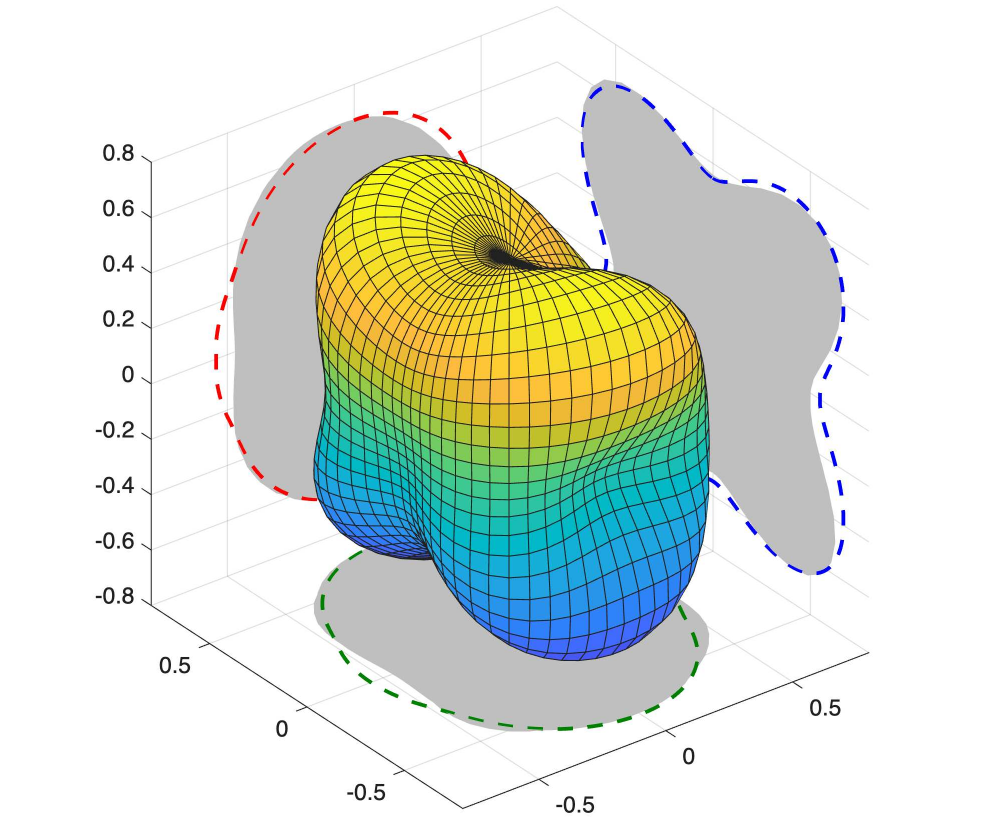}}
    \end{tabular}
}

\subfigure[$\pmb c^{(0)} = (0.1, -0.5, -0.1)^\top$, $r^{(0)} = 0.6$, $\kappa=4.5$, $\pmb z=(0, 0, 4)^\top$.]{
    \begin{tabular}{ccc}
        {\includegraphics[width=0.285\textwidth]{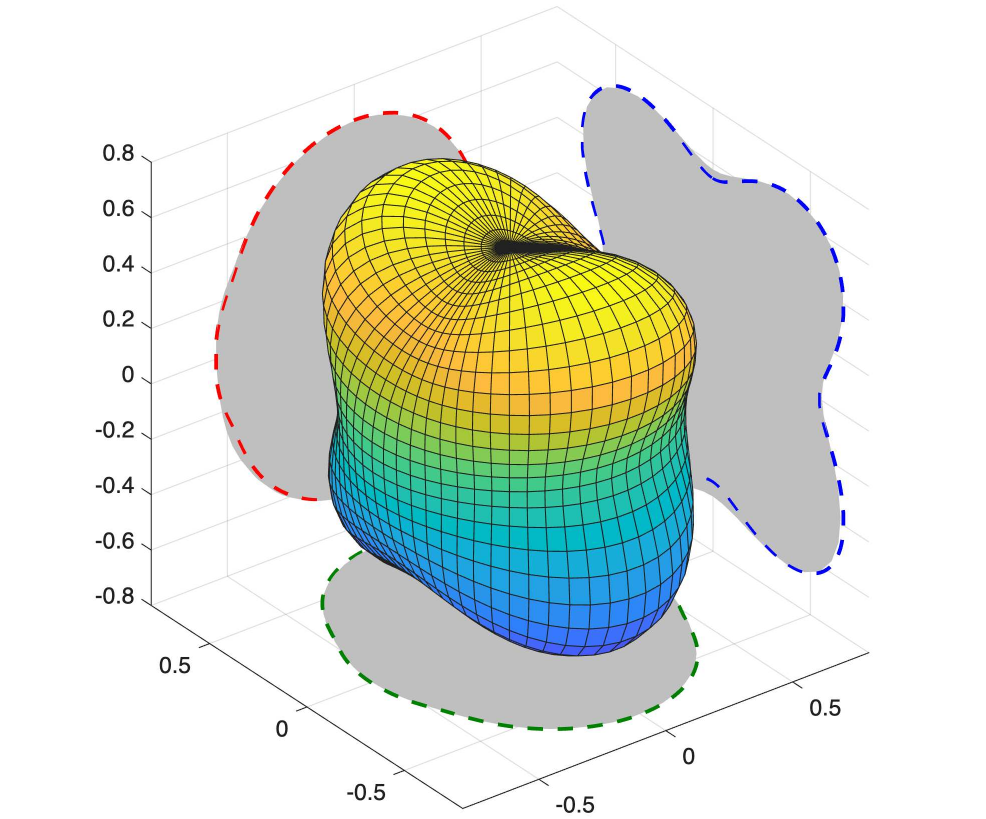}} &
        {\includegraphics[width=0.285\textwidth]{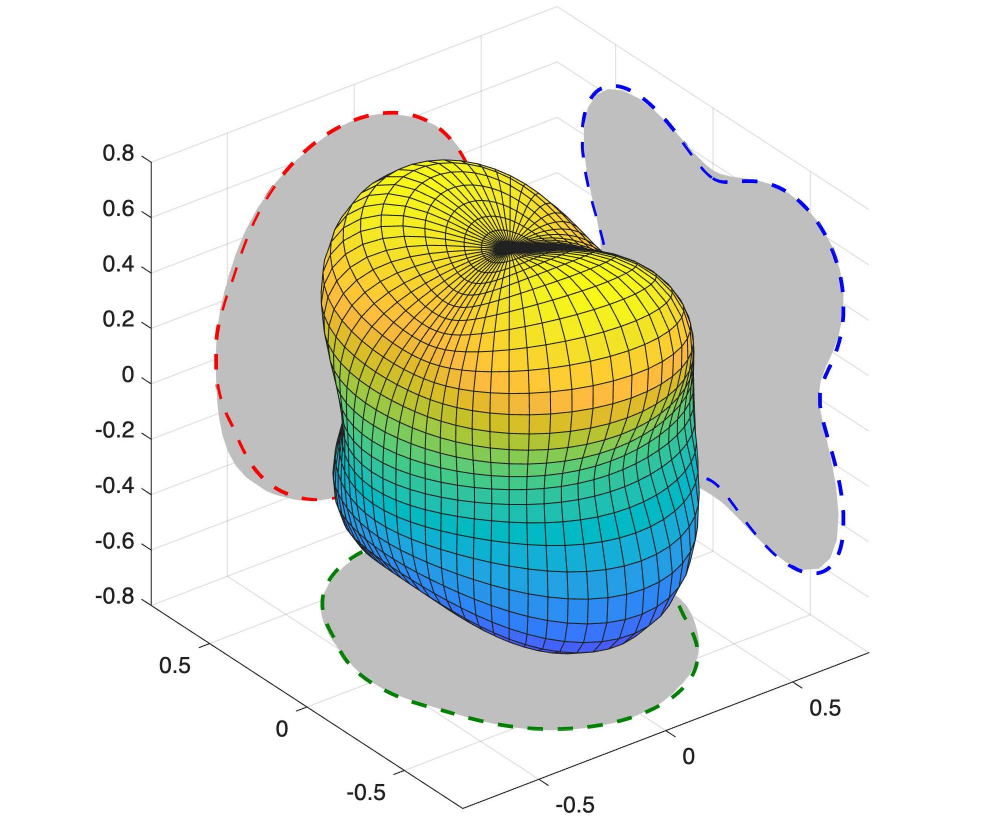}} &
        {\includegraphics[width=0.285\textwidth]{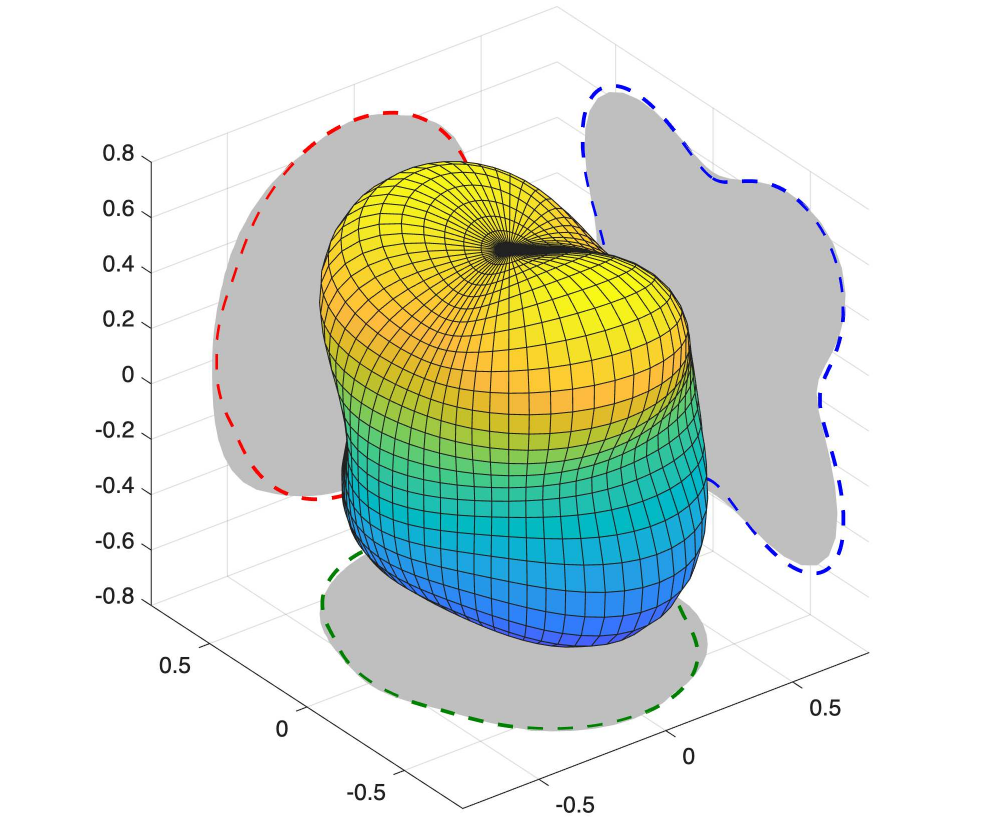}}
    \end{tabular}
}
\caption{Reconstructions of a cushion-shaped obstacle with single incident wave. Subfigures~(a)-(c) show results from phased scattered field, phased far-field, and phaseless far-field data, respectively, each with $1\%$, $5\%$ and $10\%$ noise. The stopping parameters $\epsilon=(0.009,0.023,0.045)$ for (a), $\epsilon=(0.018,0.032,0.050)$ for (b), and $\epsilon=(0.004,0.015,0.030)$ for (c).}\label{fig_ex1.2.1}
\end{figure}

\begin{figure}[htbp!]
\centering
\begin{tabular}{ccc}
    \subfigure[$\pmb d=(\pm 1,0,0)^\top$, $\epsilon=0.023$.]
    {\includegraphics[width=0.295\textwidth]{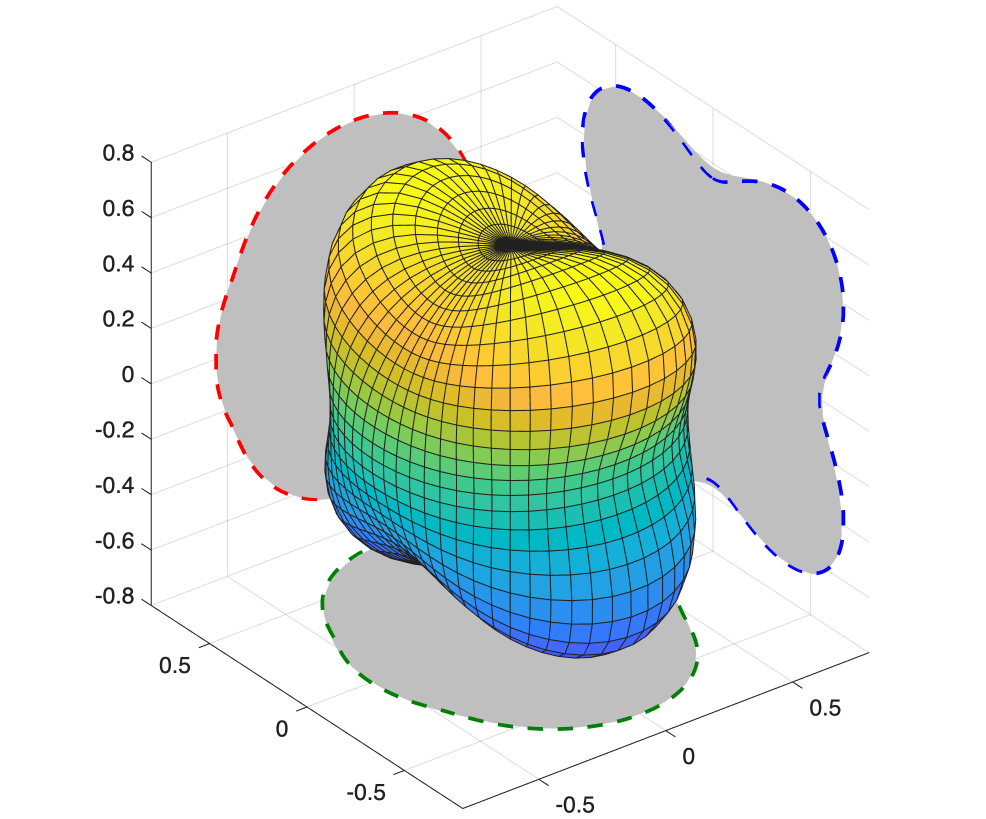}} &
    \subfigure[$\pmb d=(0,0,\pm 1)^\top$, $\epsilon=0.032$.]
    {\includegraphics[width=0.295\textwidth]{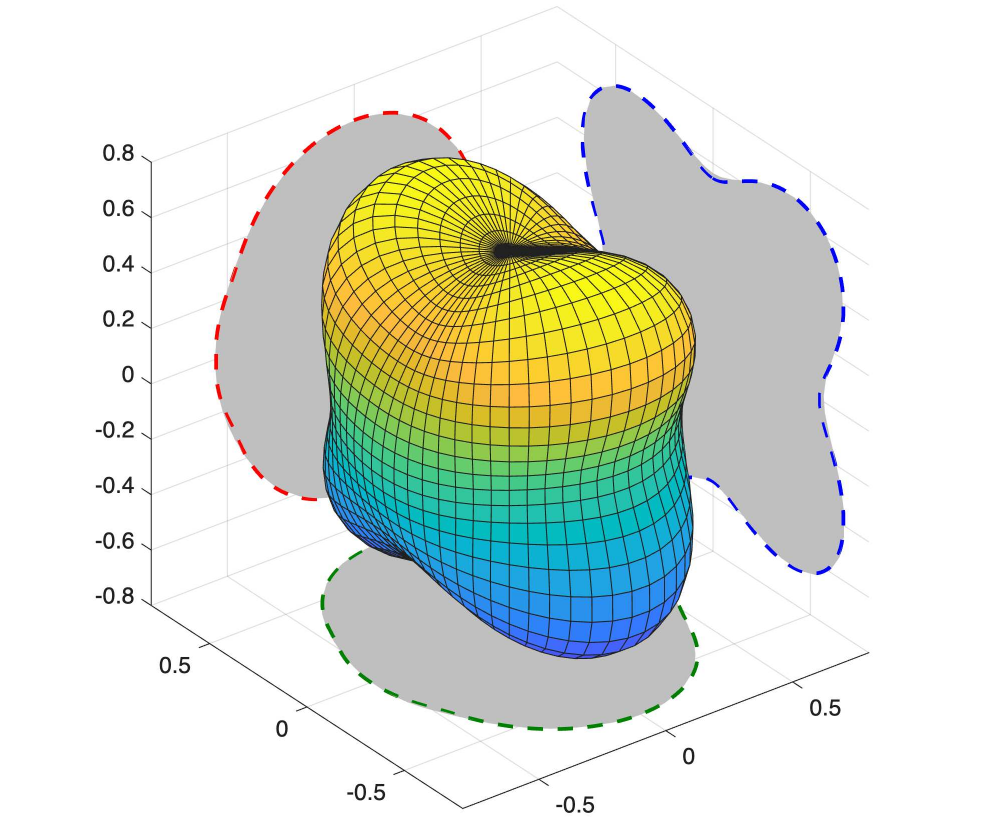}} &
    \subfigure[$\pmb z=(\pm 4,0,0)^\top$, $\epsilon=0.015$.]
    {\includegraphics[width=0.295\textwidth]{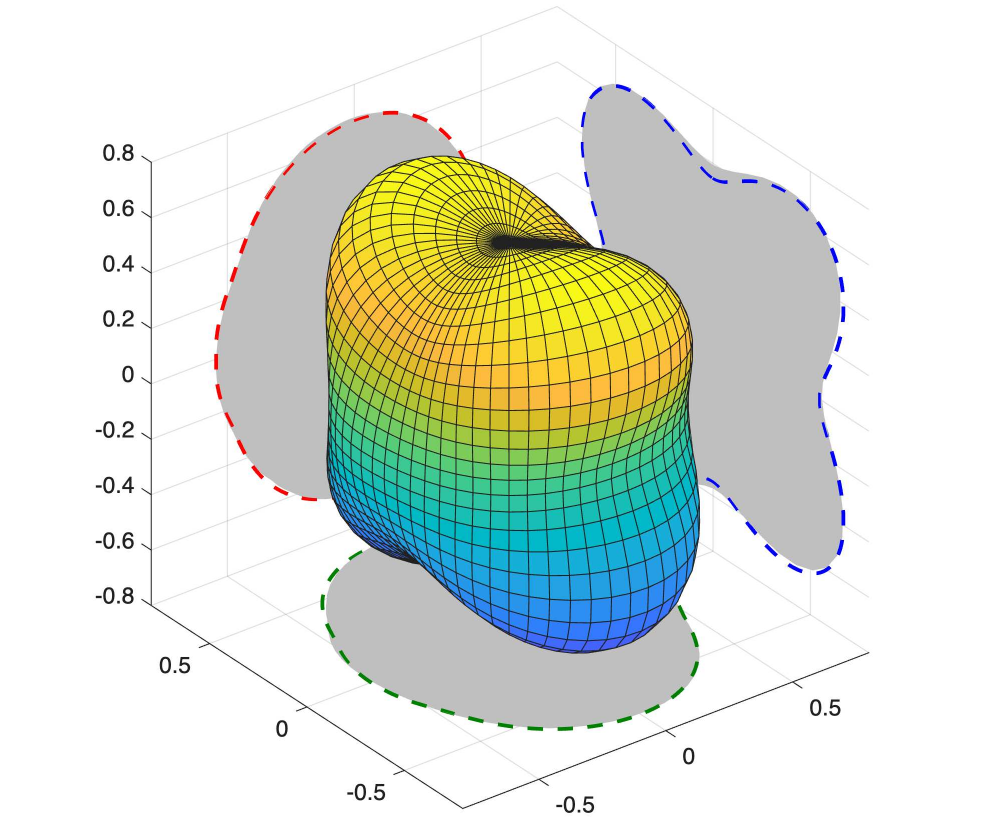}}
\end{tabular}
\begin{tabular}{ccc}
    \subfigure[$\pmb d$=$(\pm 1,0,0)^\top$,$(0,0,\pm 1)^\top$, $\epsilon=0.023$.]
    {\includegraphics[width=0.295\textwidth]{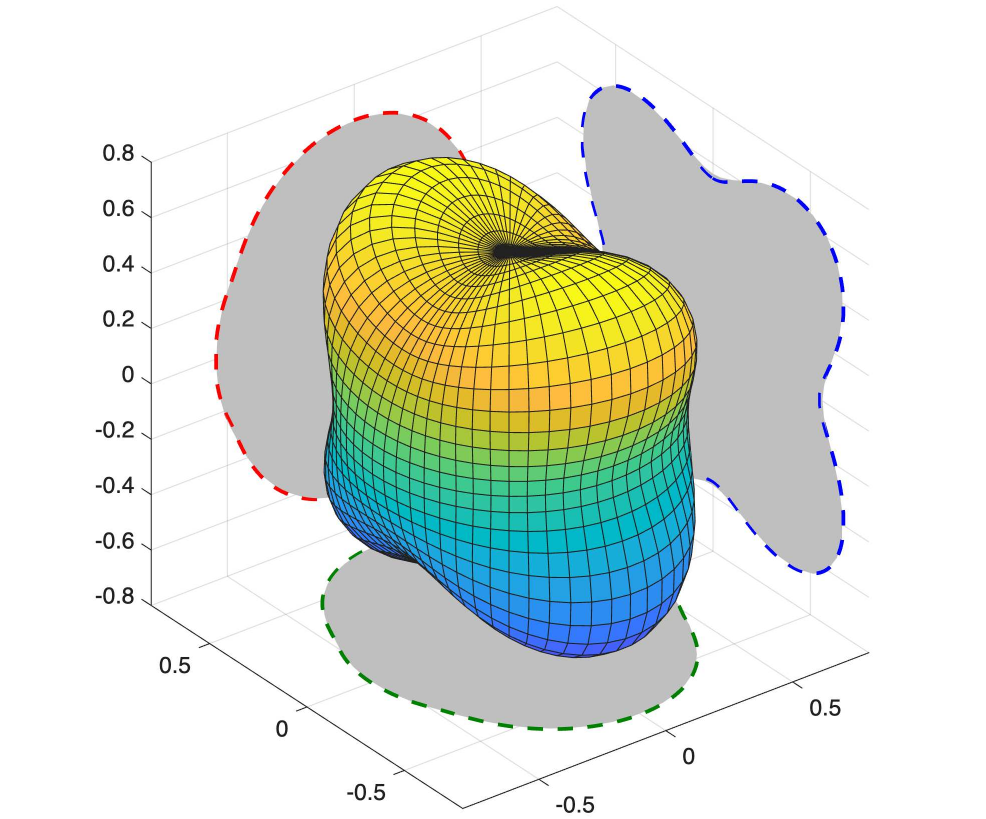}} &
    \subfigure[$\pmb d$=$(\pm 1,0,0)^\top$,$(0,0,\pm 1)^\top$, $\epsilon=0.032$.]
    {\includegraphics[width=0.295\textwidth]{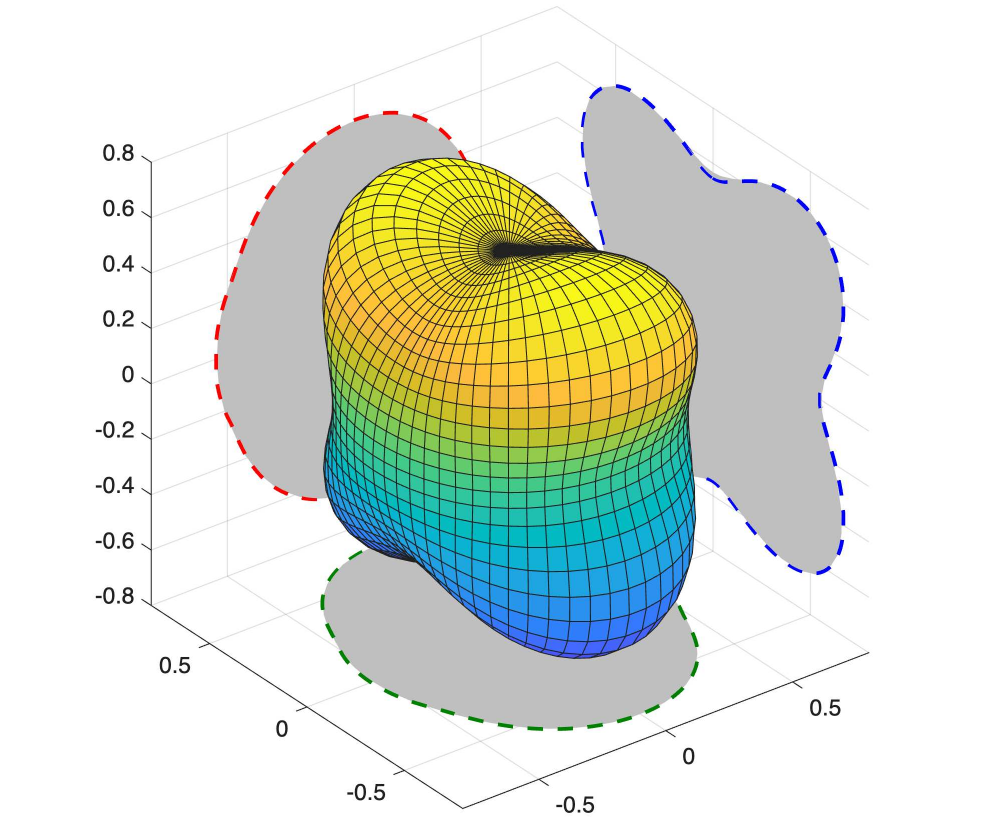}} &
    \subfigure[$\pmb z$=$(\pm 4,0,0)^\top$,$(0,0,\pm 4)^\top$, $\epsilon=0.015$.]
    {\includegraphics[width=0.295\textwidth]{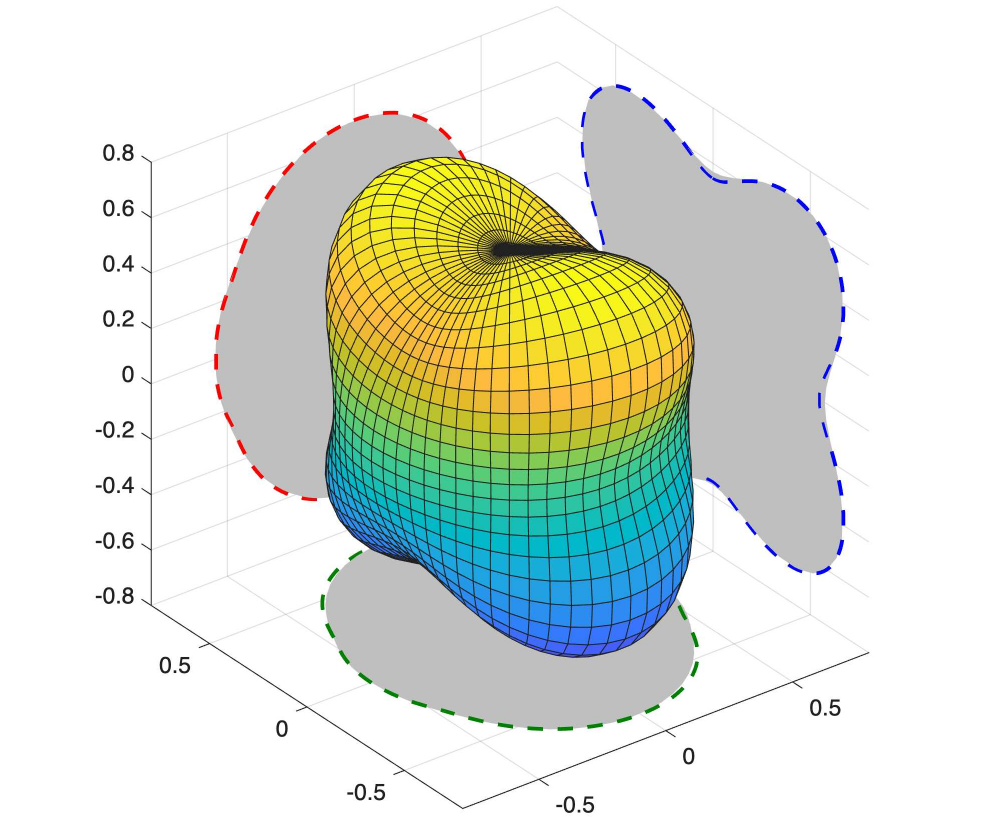}}
\end{tabular}
\caption{Reconstructions of a cushion shaped obstacle with multiple incident waves under 5\% noise. Subfigures~(a), (b), and (c) correspond to reconstructions with two incident waves using phased scattered field, phased far field, and phaseless far field data, respectively. Subfigures~(d), (e), and (f) correspond to reconstructions with four incident waves under the same three data types. The initial guesses, wavenumbers, and the other parameters for each data type are the same as shown in Figure \ref{fig_ex1.2.1} under $5\%$ noise.}\label{fig_ex1.2.2}
\end{figure}

\begin{figure}[!htbp]
\centering
\subfigure[1 point source $\pmb z=(0,0,4)^{\top}$, $\epsilon=0.015$.]{
\begin{tabular}{cccc}
\includegraphics[width=0.21\textwidth]{ex1_17} &
\includegraphics[width=0.21\textwidth]{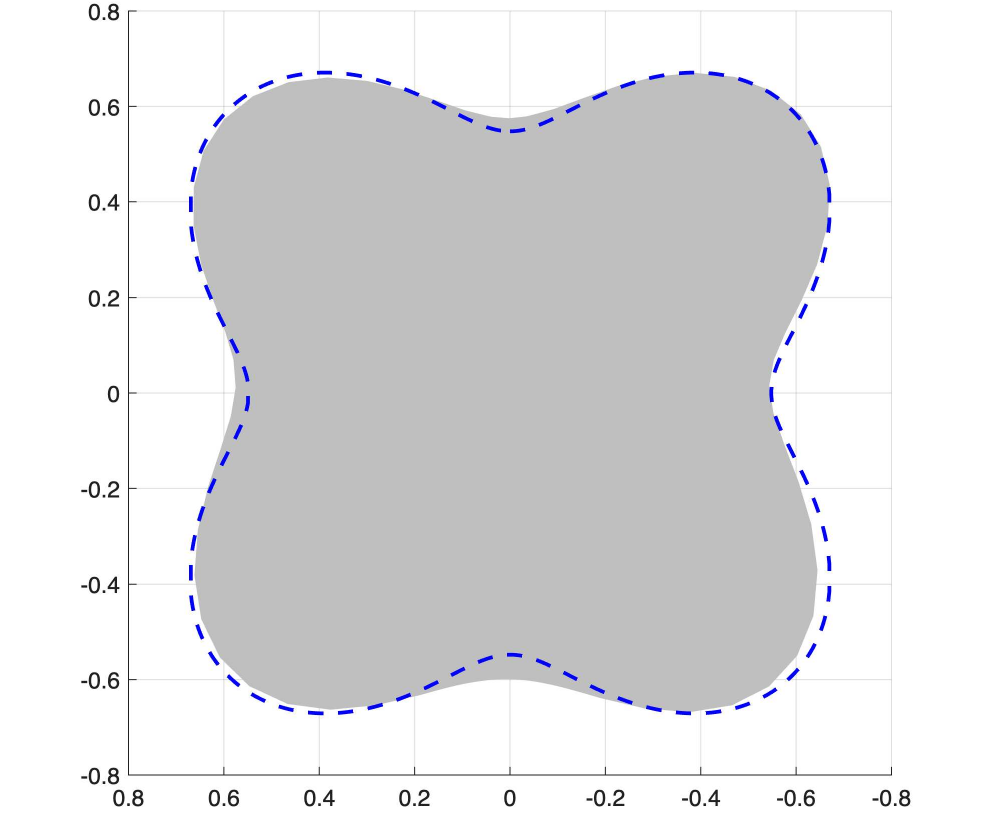} &
\includegraphics[width=0.21\textwidth]{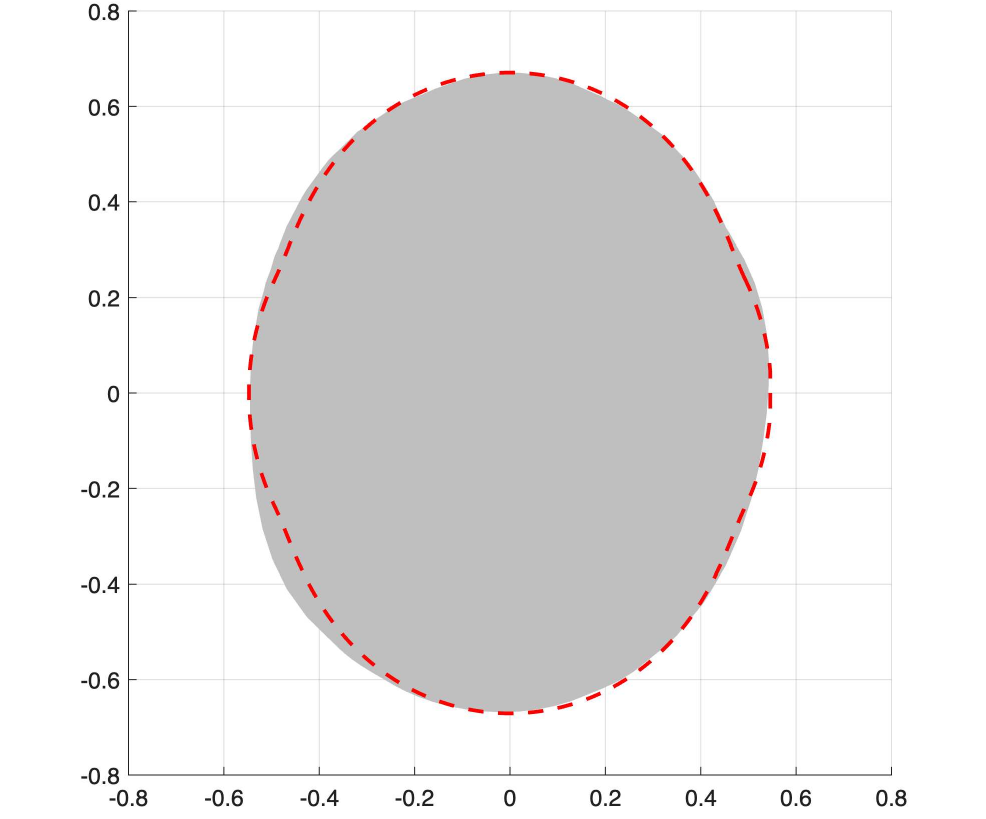} &
\includegraphics[width=0.21\textwidth]{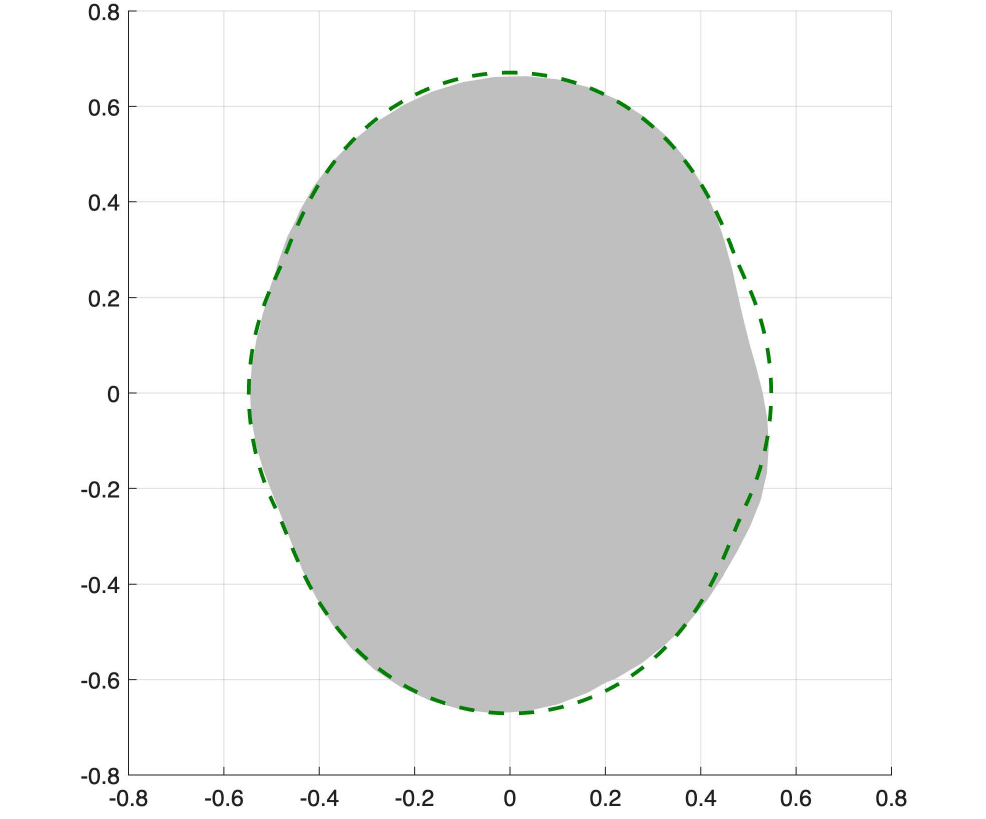}
\end{tabular}
}\\[-1.2ex]
\subfigure[2 point sources $\pmb z=(\pm 4,0,0)^{\top}$, $\epsilon=0.015$.]{
\begin{tabular}{cccc}
\includegraphics[width=0.21\textwidth]{ex1_21} &
\includegraphics[width=0.21\textwidth]{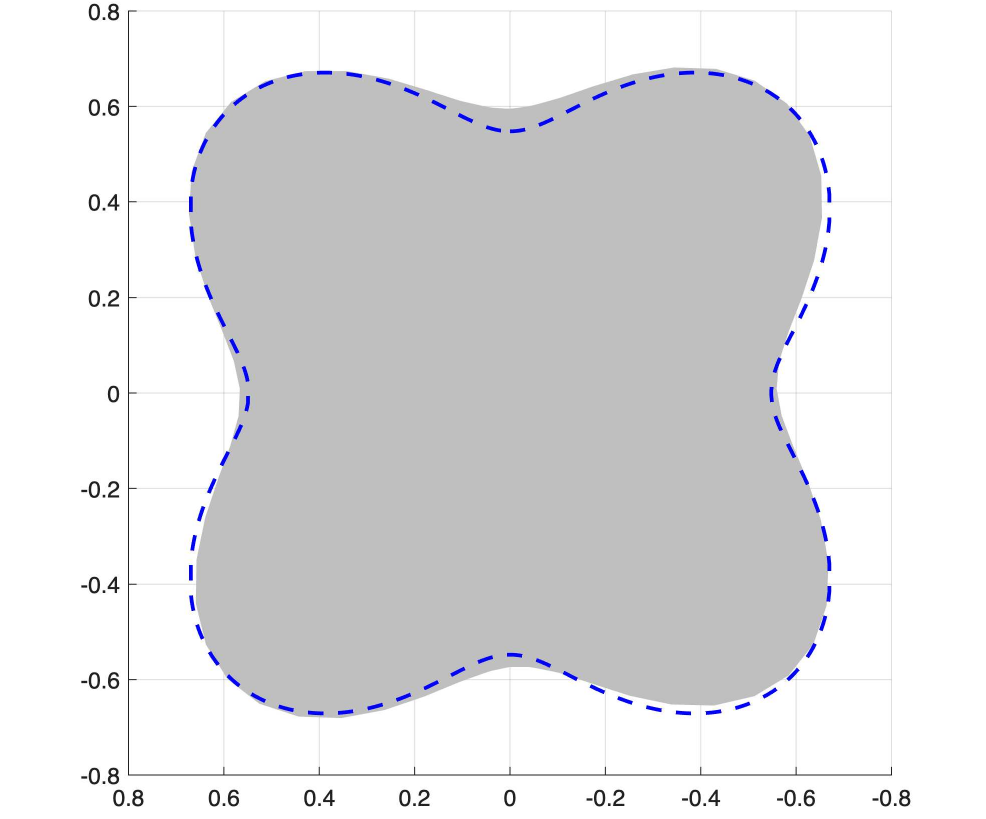} &
\includegraphics[width=0.21\textwidth]{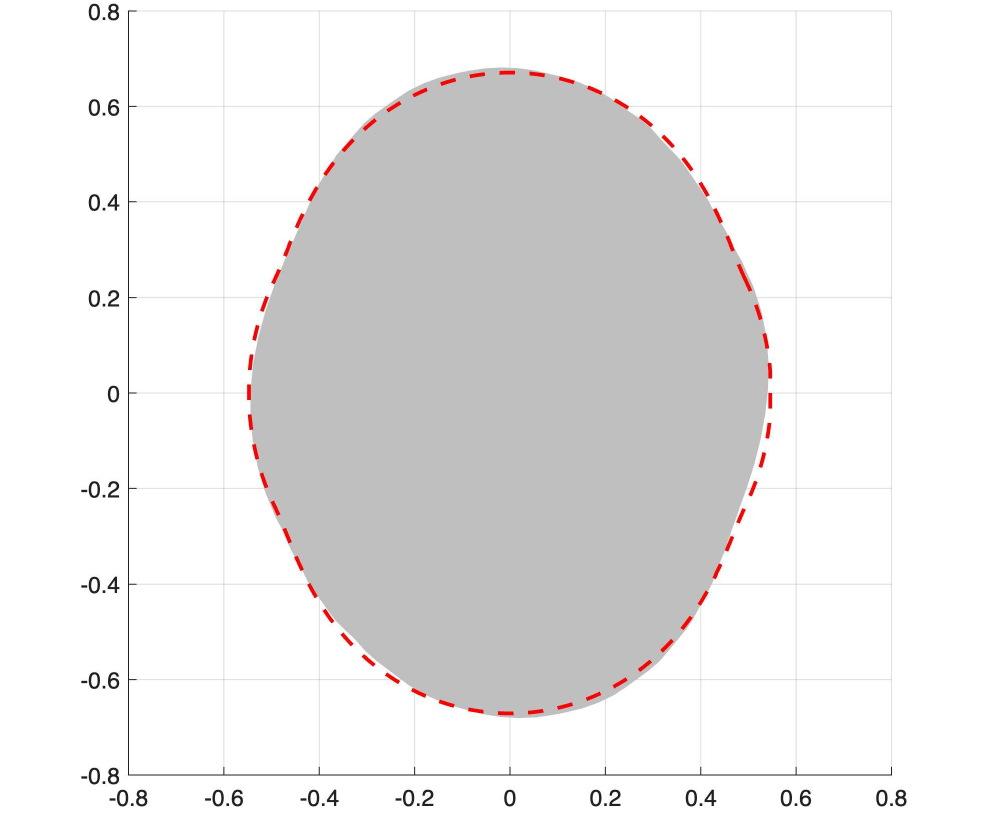} &
\includegraphics[width=0.21\textwidth]{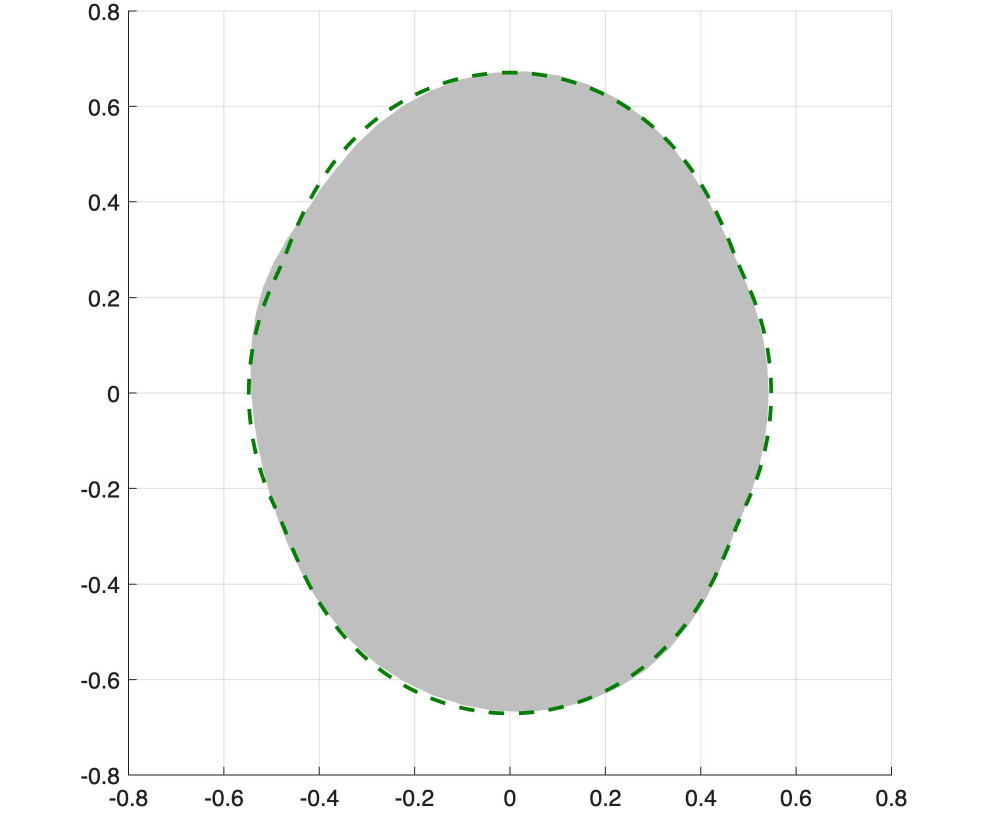}
\end{tabular}
}\\[-1.2ex]
\subfigure[4 point sources $\pmb z=(\pm 4,0,0)^{\top},(0,0,\pm 4)^{\top}$, $\epsilon=0.015$.]{
\begin{tabular}{cccc}
\includegraphics[width=0.21\textwidth]{ex1_24} &
\includegraphics[width=0.21\textwidth]{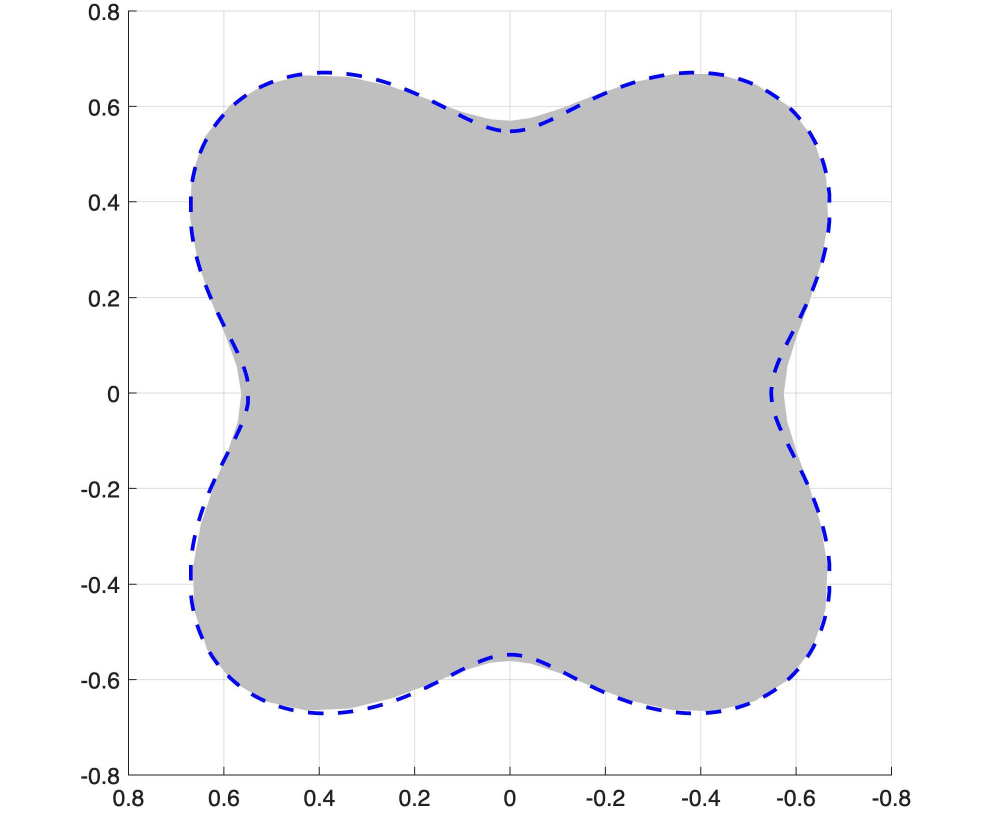} &
\includegraphics[width=0.21\textwidth]{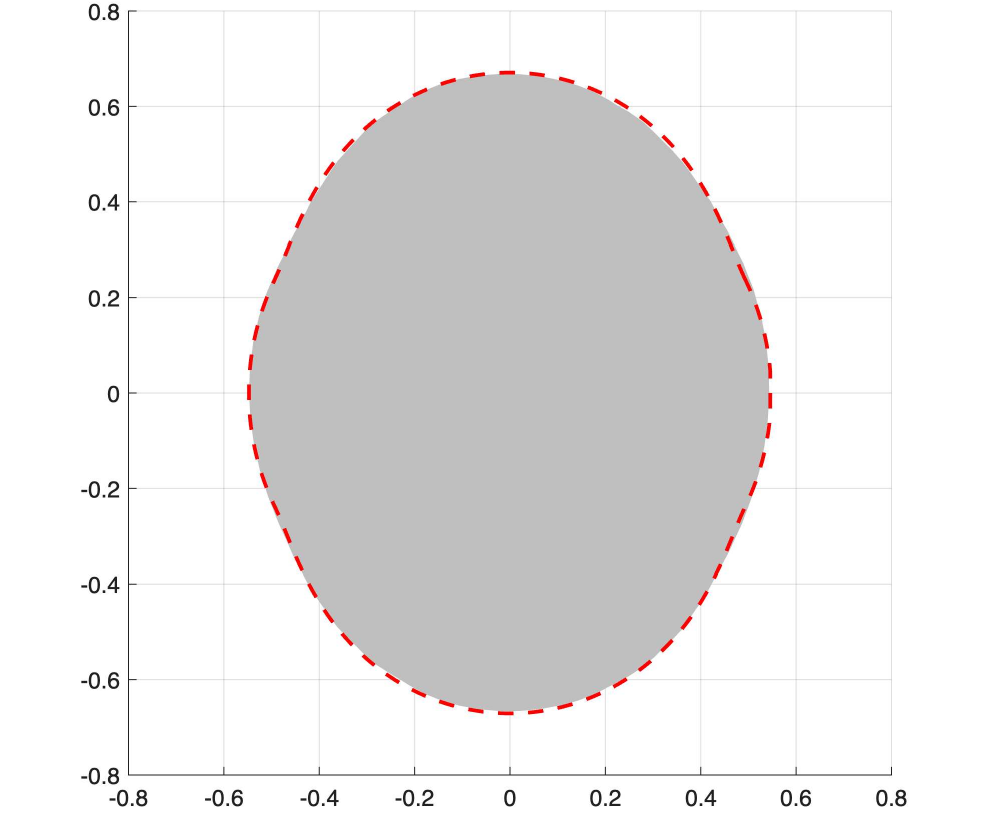} &
\includegraphics[width=0.21\textwidth]{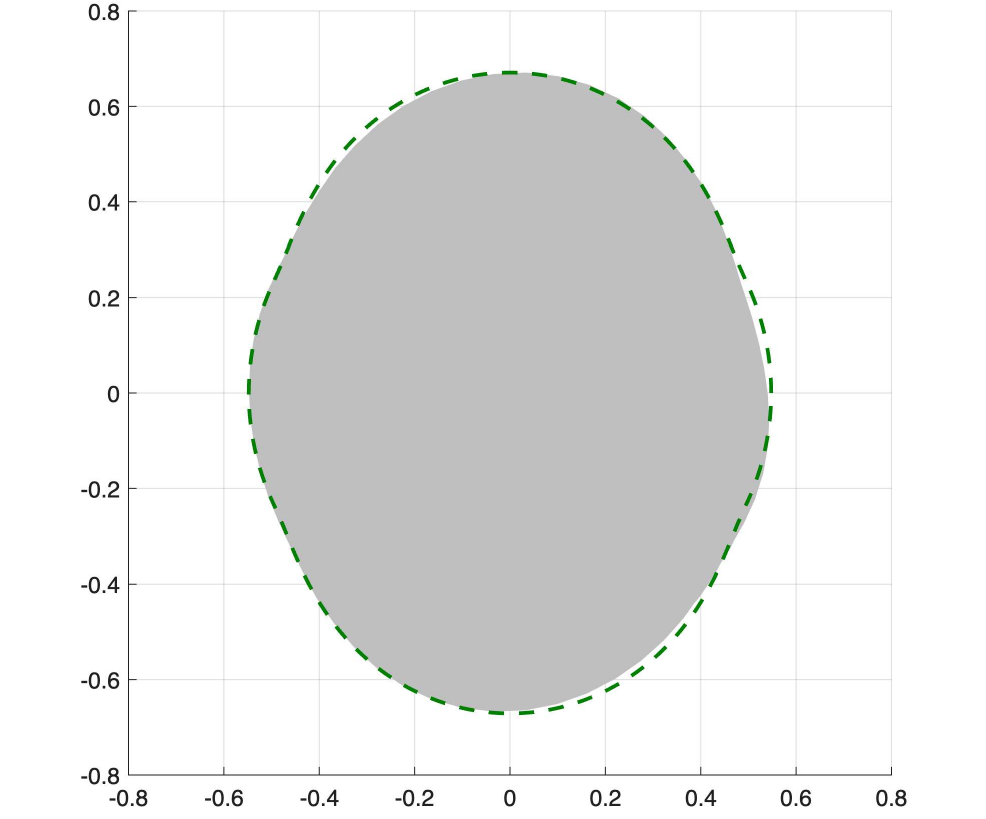}
\end{tabular}
}
\caption{Reconstructions of a cushion-shaped obstacle from phaseless far-field data with multiple point sources under $5\%$ noise. Subfigures~(a)--(c) show the results obtained with one, two, and four point sources, respectively. In each subfigure, the four panels show the three-dimensional reconstruction and its projection views along the directions $(1,0,0)^\top$, $(0,1,0)^\top$, and $(0,0,1)^\top$. In all cases, $\pmb c^{(0)}=(0.1,-0.5,-0.1)^\top$, $r^{(0)}=0.6$, $\kappa=4.5$, and $\epsilon=0.015$.}
\label{fig_ex1.2.3}
\end{figure}

In all figures, the shaded regions represent the projections of the reconstructions, whereas the colored dashed curves denote the true obstacle boundaries in the corresponding projection directions. This visualization allows a direct comparison between the exact geometry and the recovered shape.

We begin with two standard test obstacles: a pinched ball-shaped obstacle and a cushion-shaped obstacle.
The parameterized boundary curves are shown in Table~\ref{obstacle} and illustrated in figure~\ref{true shape}.

\vspace{2ex}
{\noindent\bf Example 1: Reconstructions with one or more incident waves.}
\vspace{1ex}

We consider the inverse problems~\ref{problem1}--\ref{problem3} under the Dirichlet boundary condition. Figures~\ref{fig_ex1.1.1} and~\ref{fig_ex1.2.1} show the reconstructions of the pinched-ball and cushion-shaped obstacles from a single incident wave under three noise levels. For each obstacle and data type, all reconstruction parameters are fixed except the noise level and the corresponding stopping tolerance. In all cases, both the shape and location of the obstacles are satisfactorily recovered.

Figure~\ref{fig_ex1.2.2} further examines the effect of multiple incident waves for the cushion-shaped obstacle under $5\%$ noise. Compared with the corresponding result for a single incident wave in Figure~\ref{fig_ex1.2.1}, all settings are kept unchanged except the number of incident waves, or equivalently, the number of source locations for the phaseless far-field data. The results show that one incident wave already yields satisfactory reconstructions, while two or four incident waves lead to slight improvements. This effect is seen more clearly in the projection views in Figure~\ref{fig_ex1.2.3}, which make the small improvements in the phaseless far-field reconstructions more visible.
These results indicate that the proposed method is effective with a single illumination and can further benefit from richer data.

\vspace{2ex}
{\noindent\bf Example 2: Reconstructions with different parameters.}
\vspace{1ex}

For two incident waves, Figures~\ref{fig_ex3.1}--\ref{fig_ex3.3} present the reconstruction results under different parameters, corresponding to scattered field, phased far-field, and phaseless far-field data, respectively. 
In each figure, the first and second rows show reconstructions with different initial guesses, while the second and third rows illustrate results for different wavenumbers. 
These results demonstrate that the proposed approach remains effective across a broad range of parameter settings, yielding consistently meaningful reconstructions despite variations in the initial guess and wavenumber.

\begin{figure}[!htbp]
\centering
\subfigure[$\pmb c^{(0)} = (-0.7, -0.5, 0.2)^\top$, $r^{(0)} = 0.4$, $R=5$, $\kappa=4$.]{
    \begin{tabular}{cccc}
        {\includegraphics[width=0.21\textwidth]{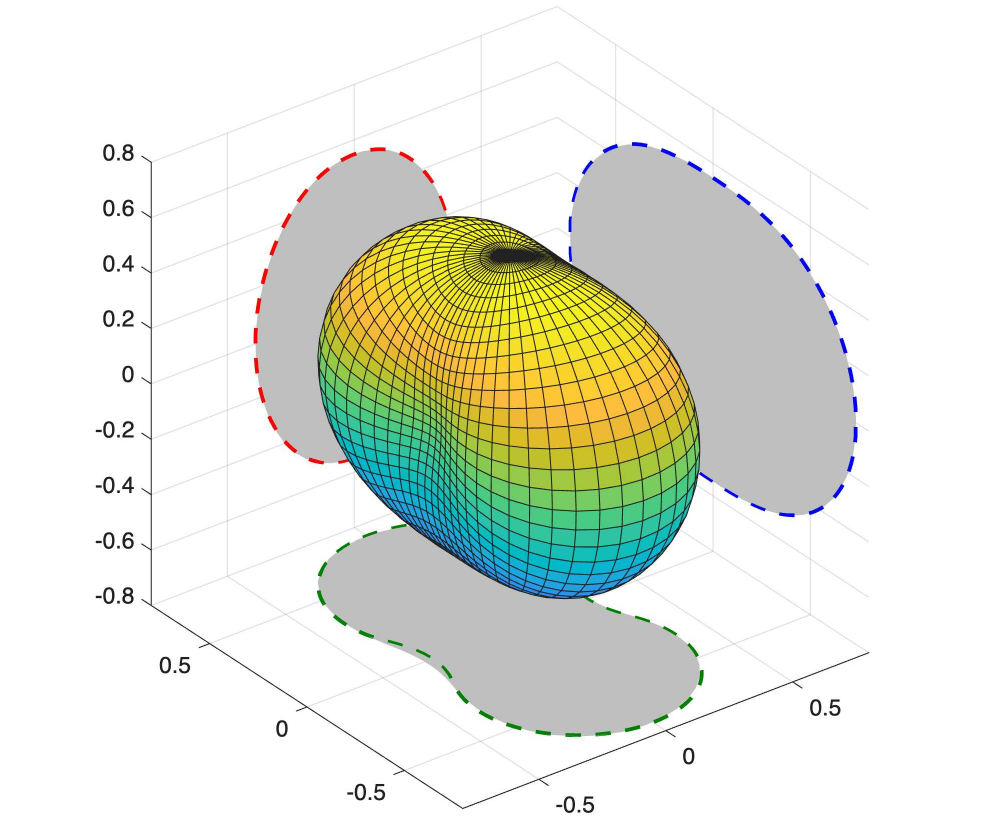}} &
        {\includegraphics[width=0.21\textwidth]{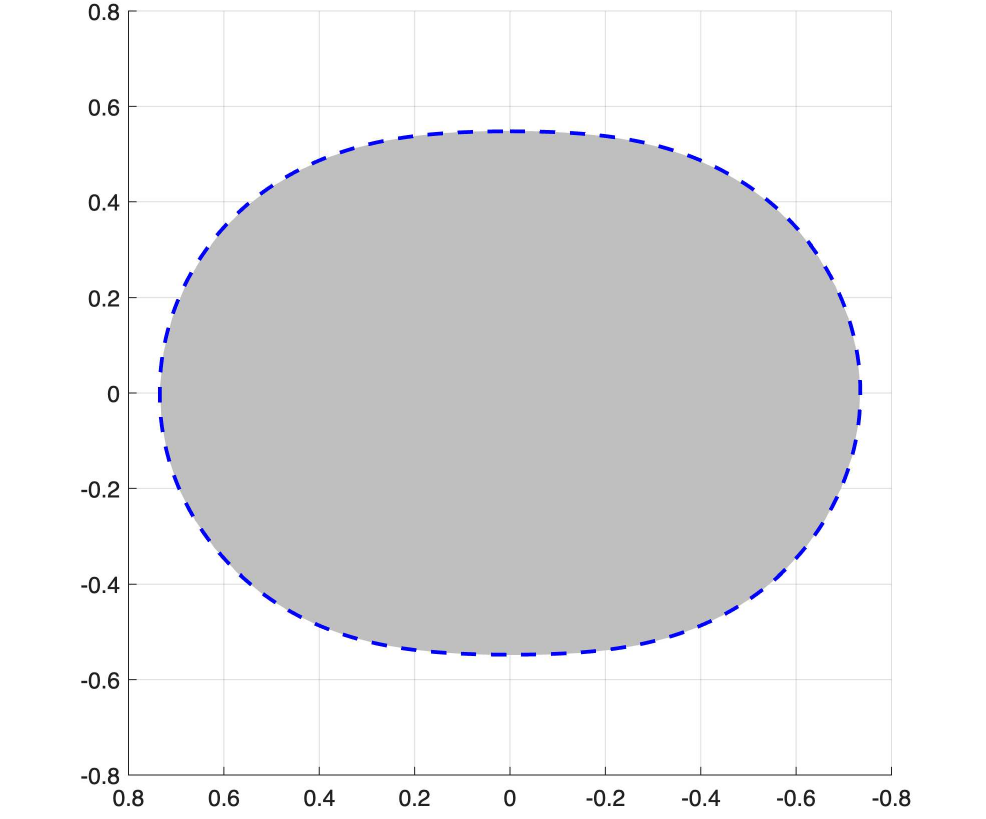}} &
        {\includegraphics[width=0.21\textwidth]{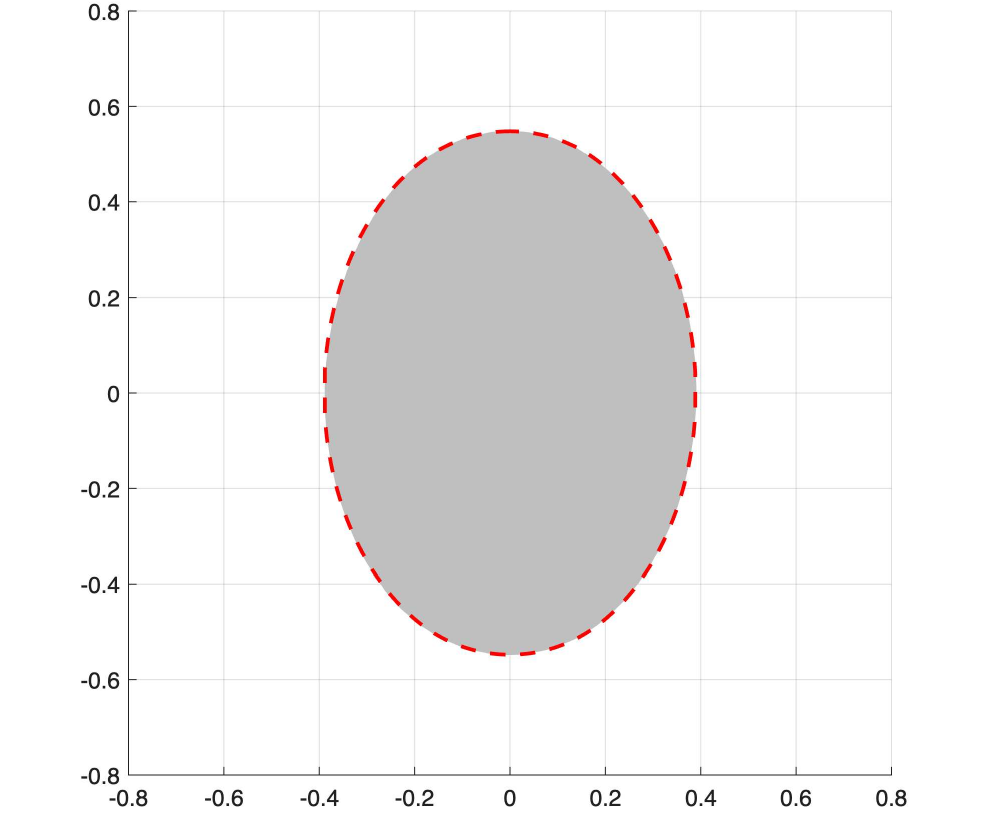}} &
        {\includegraphics[width=0.21\textwidth]{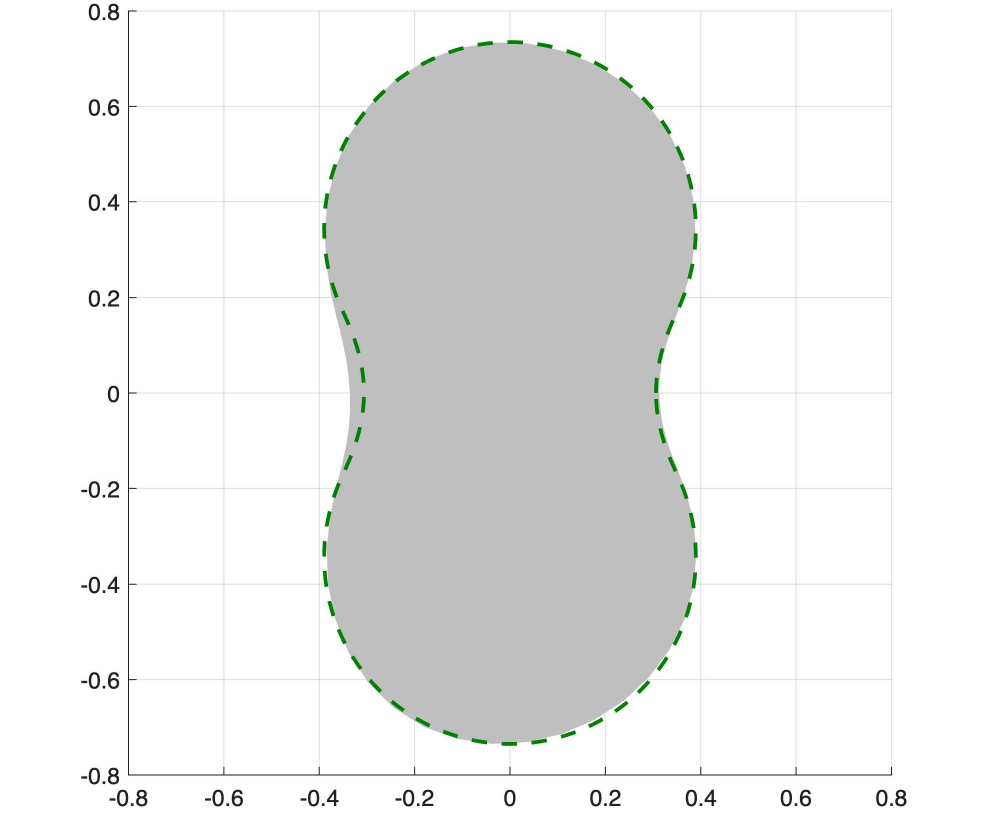}}
    \end{tabular}
}

\subfigure[$\pmb c^{(0)} = (0.3, 0.4, -0.1)^\top$, $r^{(0)} = 0.6$, $R=5$, $\kappa=4$.]{
    \begin{tabular}{cccc}
        {\includegraphics[width=0.21\textwidth]{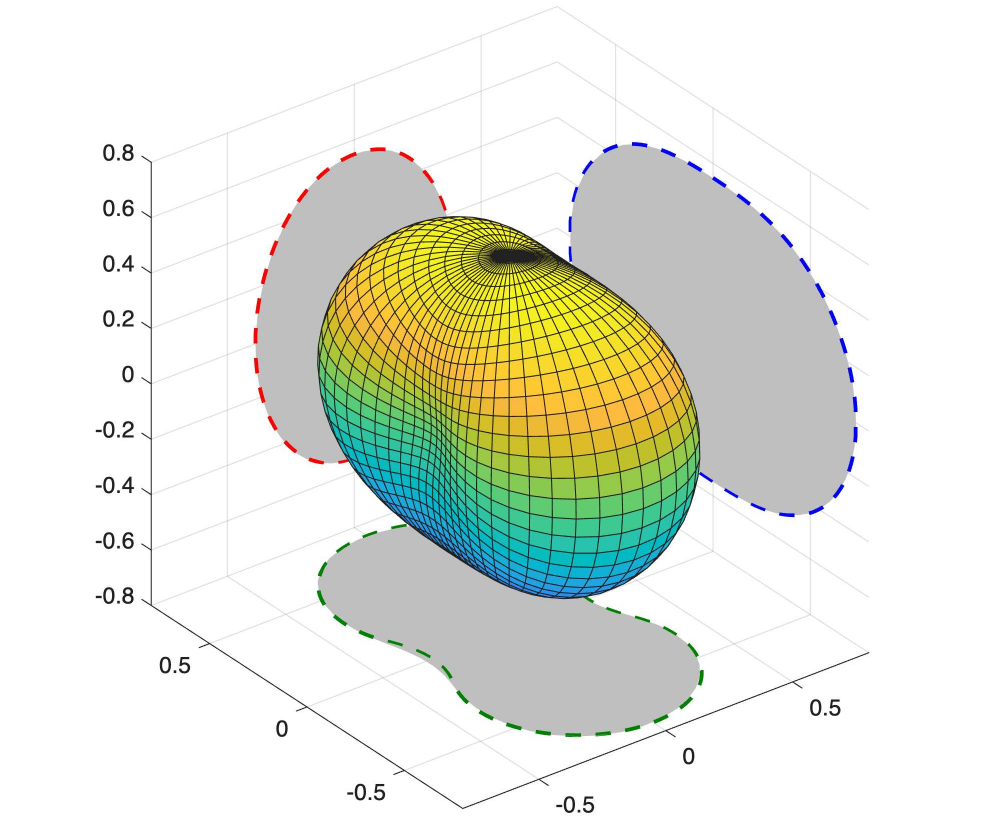}} &
        {\includegraphics[width=0.21\textwidth]{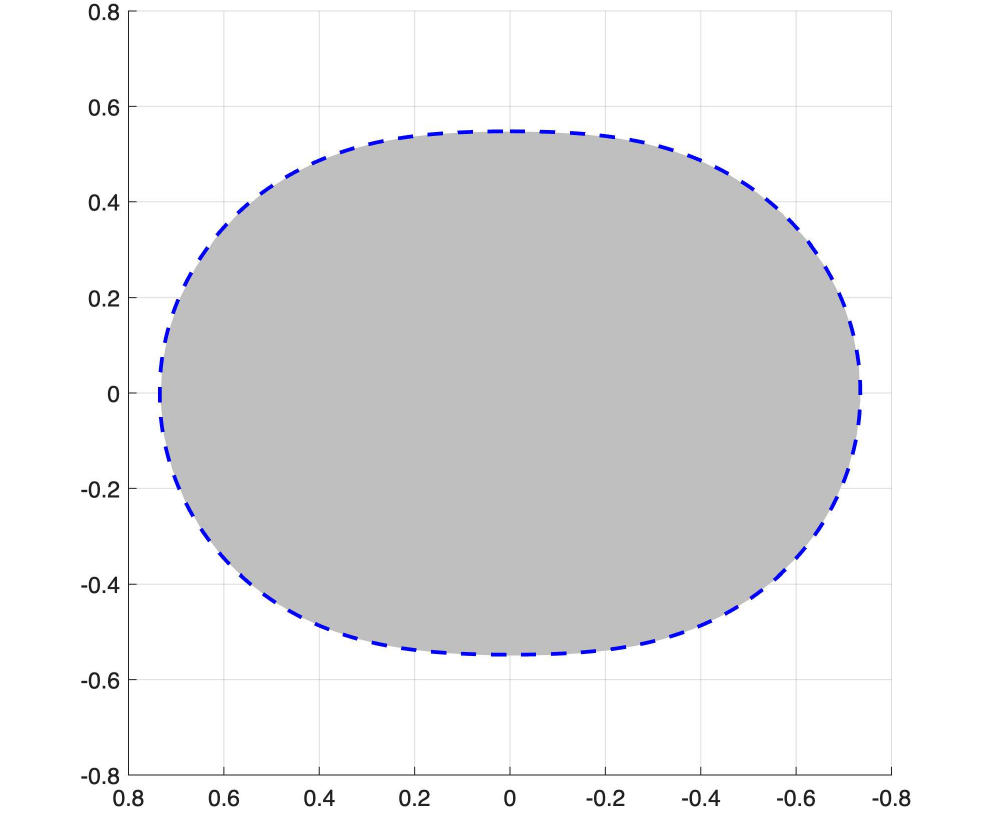}} &
        {\includegraphics[width=0.21\textwidth]{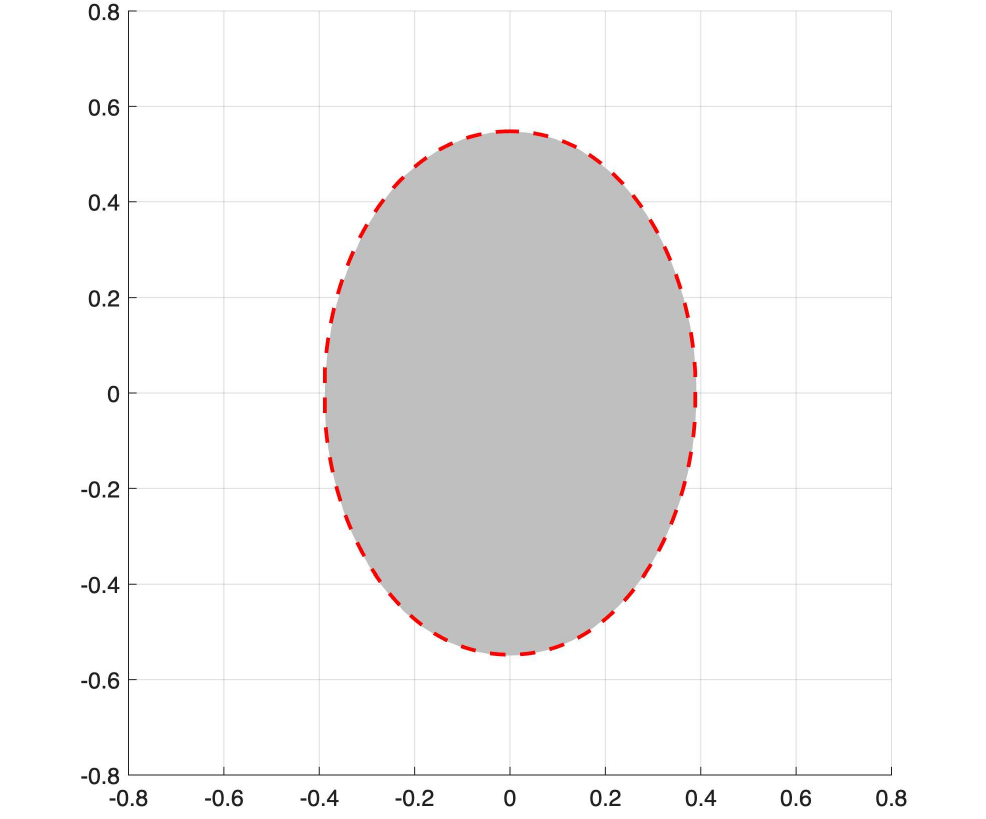}} &
        {\includegraphics[width=0.21\textwidth]{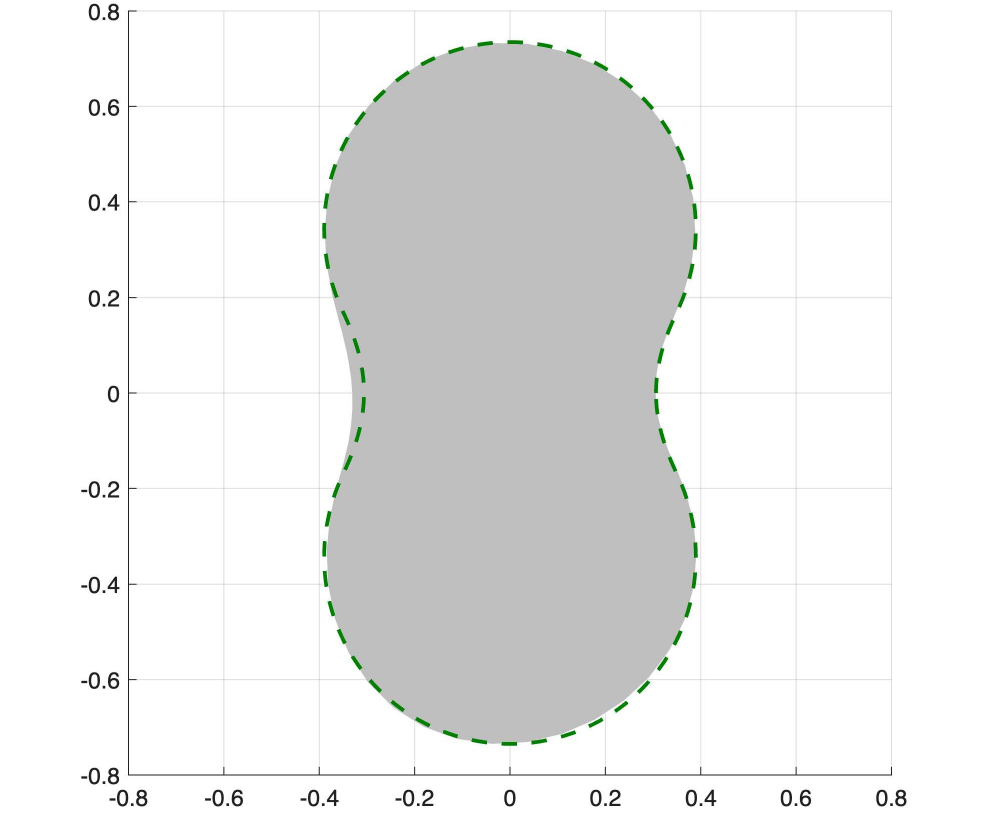}}
    \end{tabular}
}

\subfigure[$\pmb c^{(0)} = (0.3, 0.4, -0.1)^\top$, $r^{(0)} = 0.6$, $R=5$, $\kappa=7$.]{
    \begin{tabular}{cccc}
        {\includegraphics[width=0.21\textwidth]{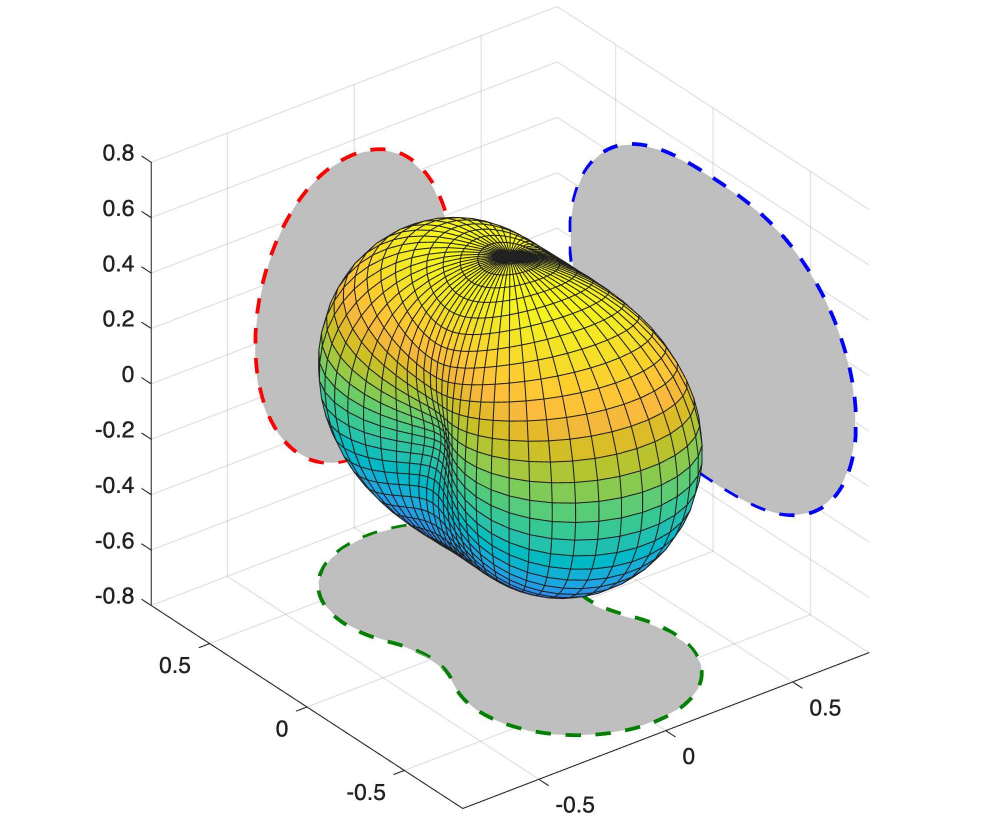}} &
        {\includegraphics[width=0.21\textwidth]{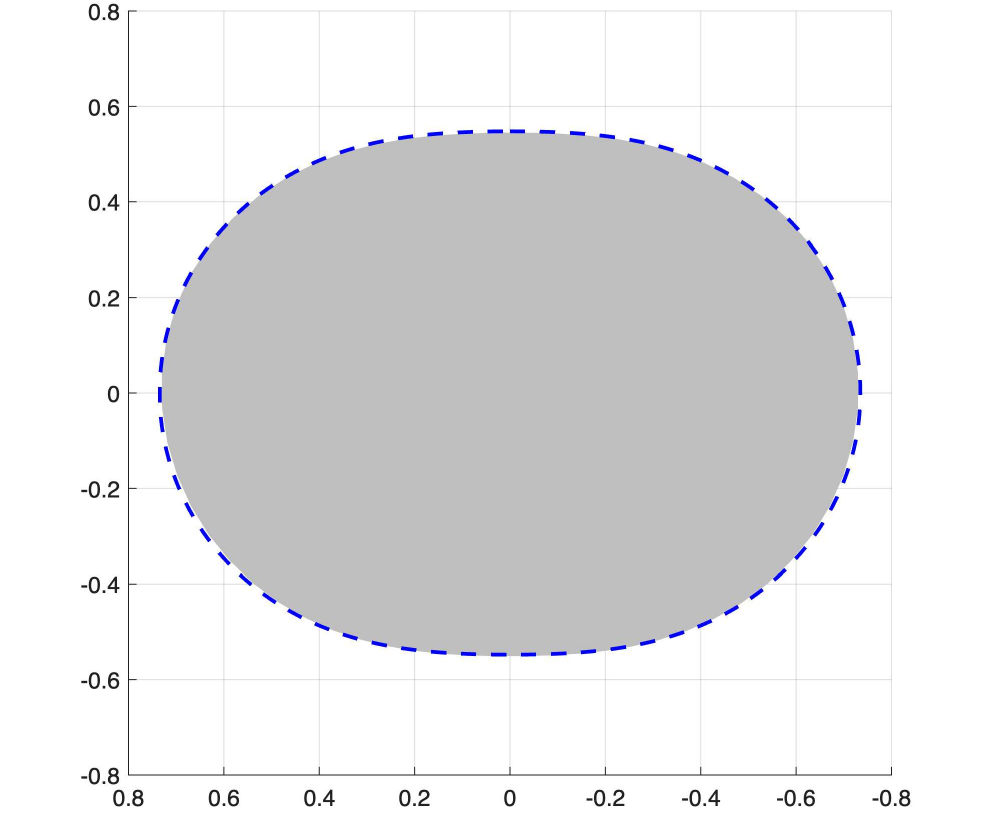}} &
        {\includegraphics[width=0.21\textwidth]{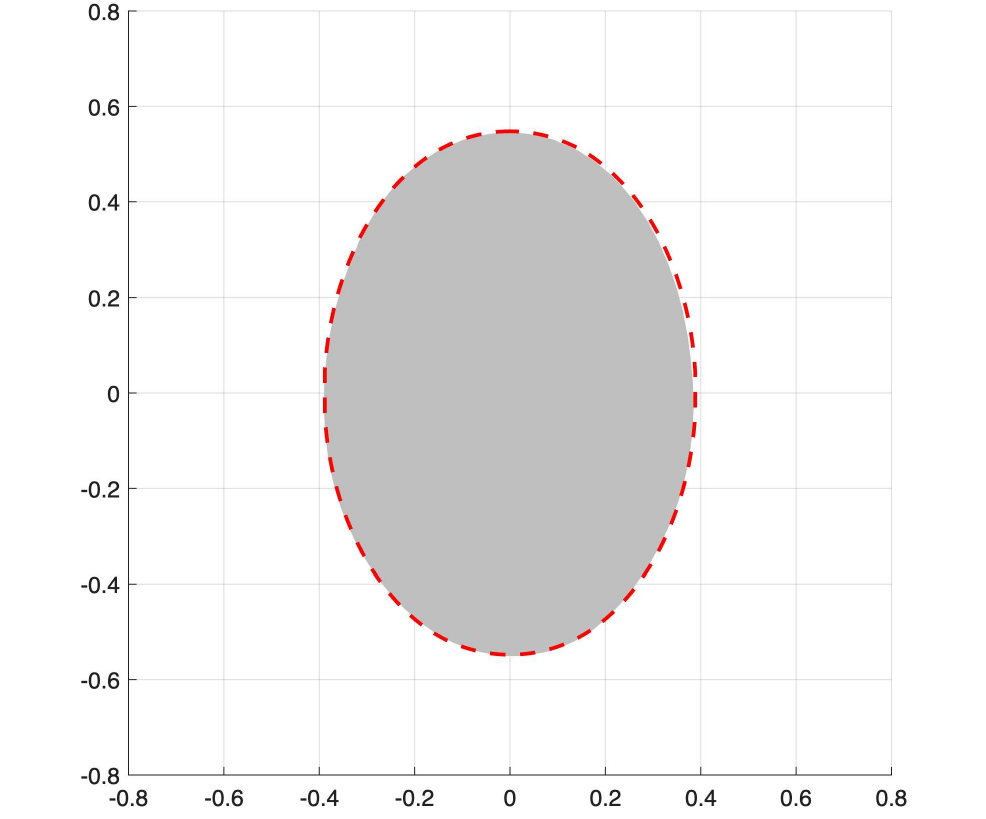}} &
        {\includegraphics[width=0.21\textwidth]{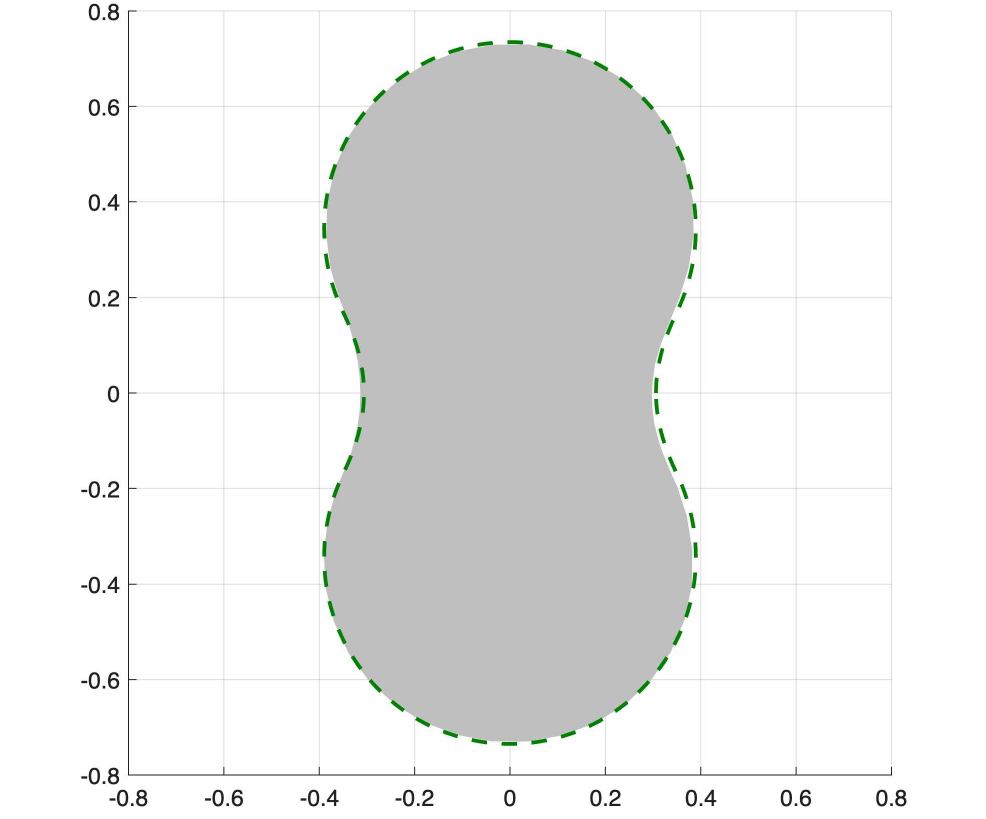}}
    \end{tabular}
}
\caption{Reconstruction of a pinched ball-shaped obstacle from phased scattered field data generated by two incident waves in the directions \((0,0,1)^\top\) and \((0,0,-1)^\top\) with \(5\%\) noise, $\epsilon=0.023$.}\label{fig_ex3.1}
\end{figure}

\begin{figure}[!htbp]
\centering
\subfigure[$\pmb c^{(0)}=(0.1,-0.6,0.2)^\top$, $r^{(0)}=0.4$, $\kappa=5$.]{
    \begin{tabular}{cccc}
        \includegraphics[width=0.21\textwidth]{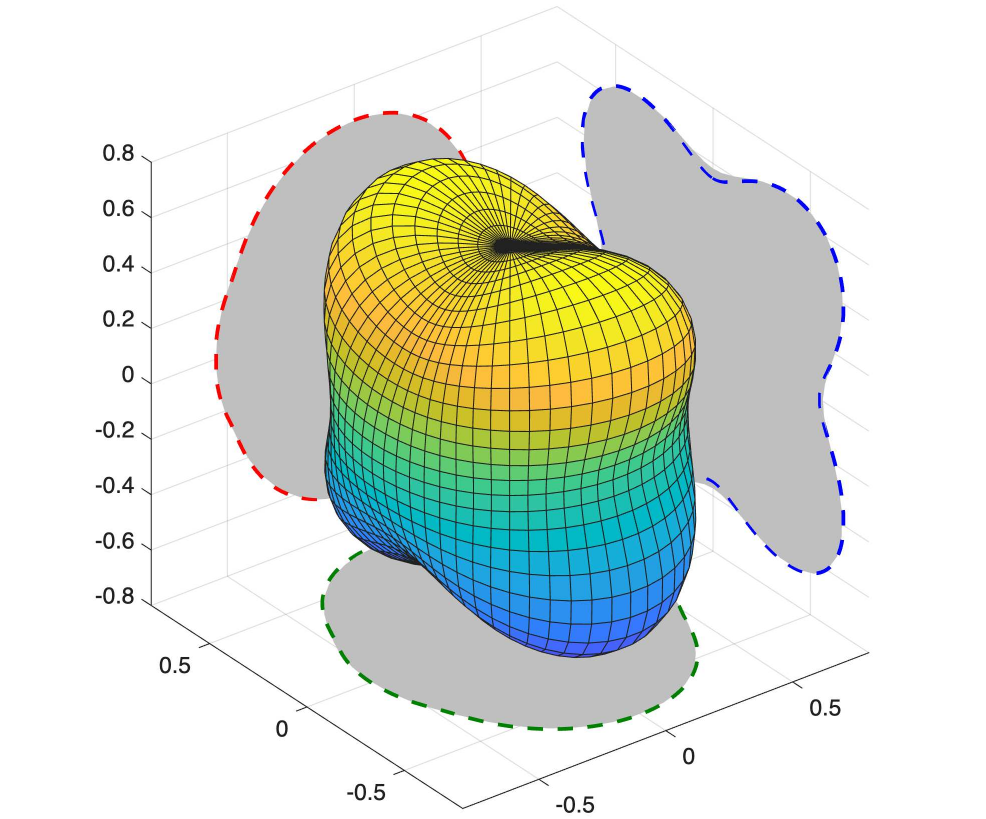} &
        \includegraphics[width=0.21\textwidth]{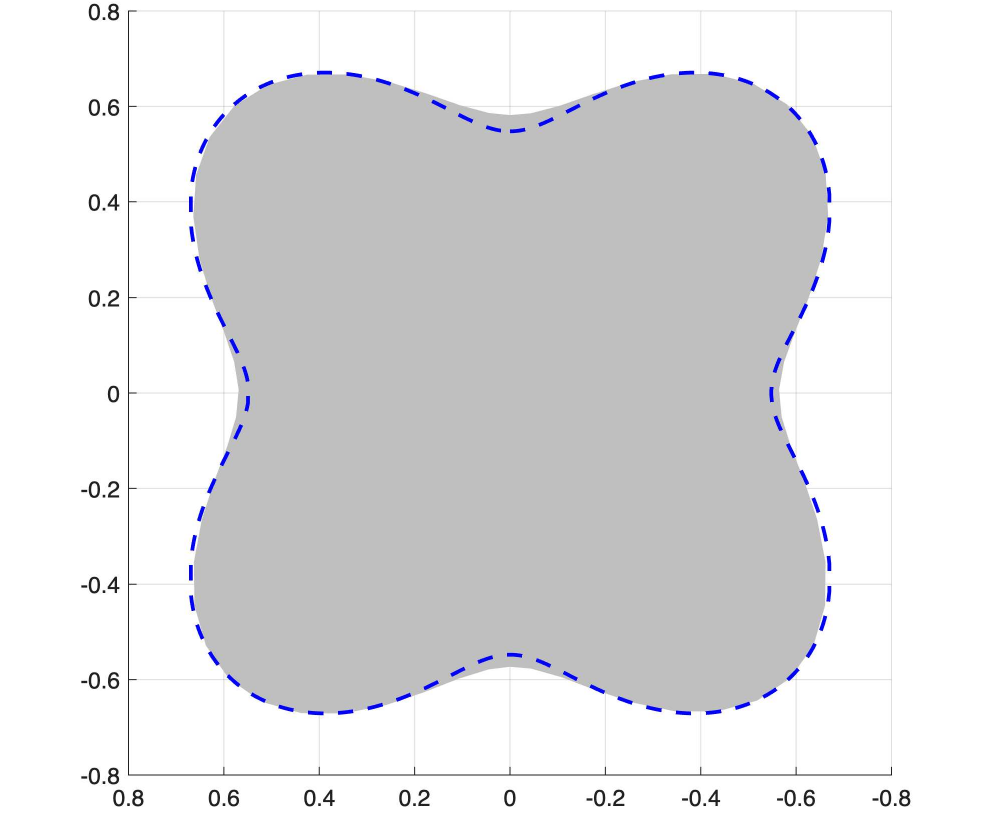} &
        \includegraphics[width=0.21\textwidth]{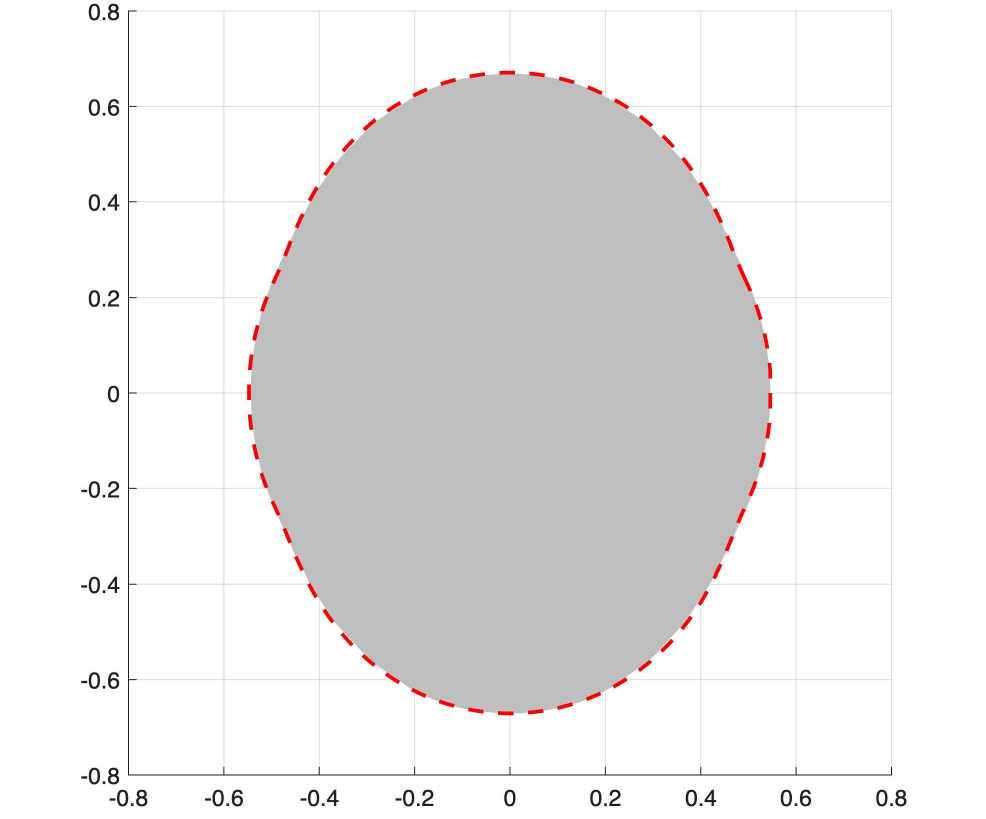} &
        \includegraphics[width=0.21\textwidth]{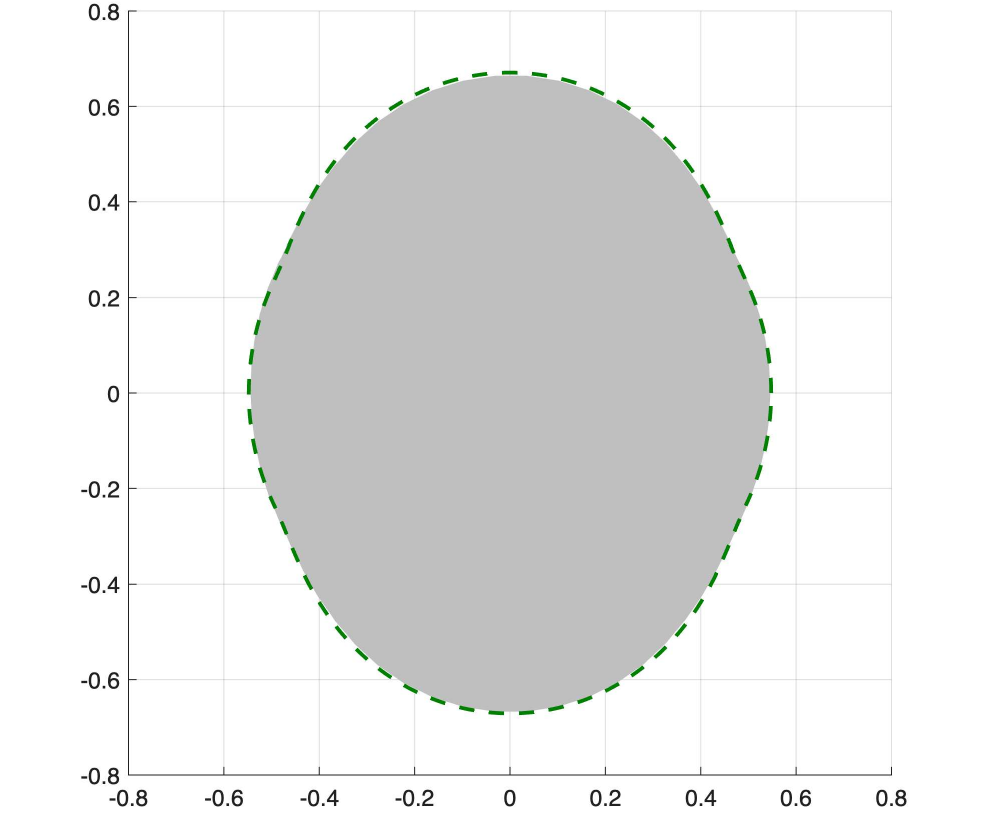}
    \end{tabular}
}

\subfigure[$\pmb c^{(0)}=(0.1,-0.5,-0.3)^\top$, $r^{(0)}=0.3$, $\kappa=5$.]{
    \begin{tabular}{cccc}
        \includegraphics[width=0.21\textwidth]{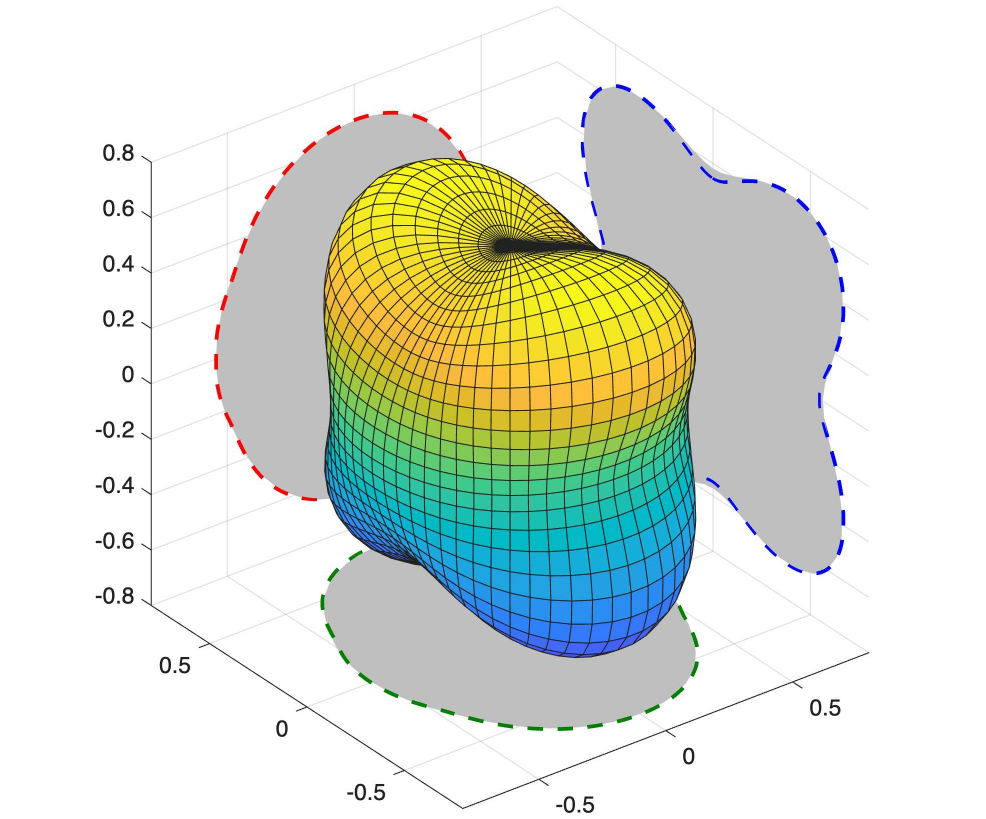} &
        \includegraphics[width=0.21\textwidth]{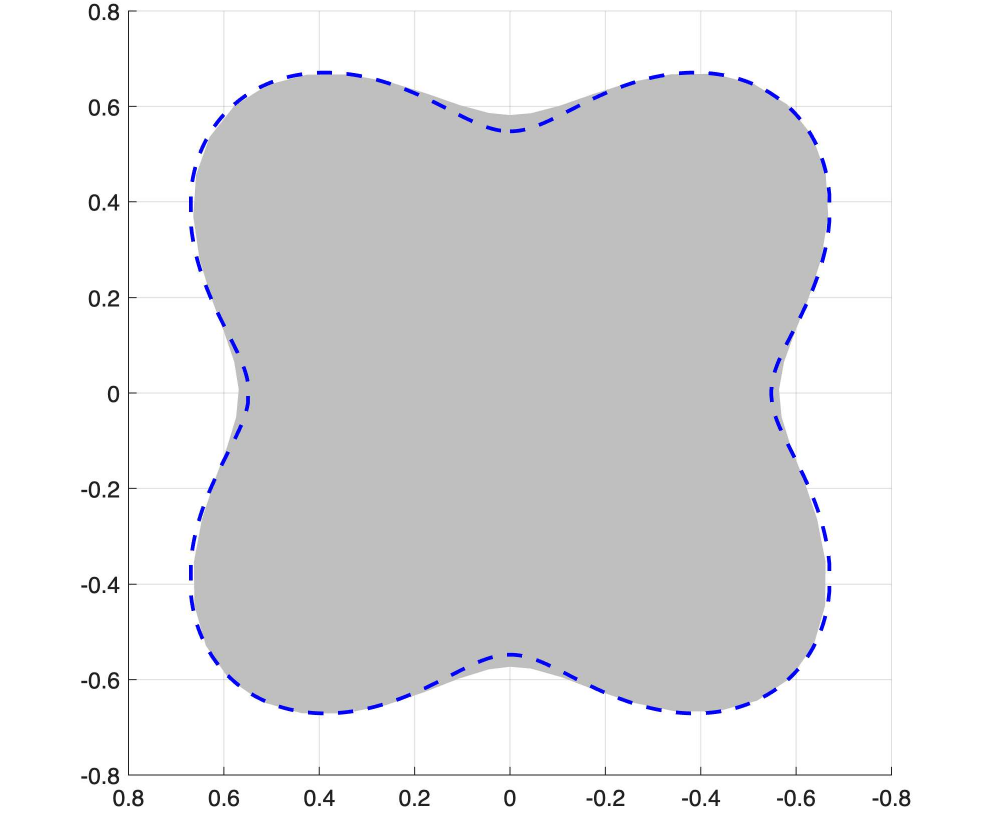} &
        \includegraphics[width=0.21\textwidth]{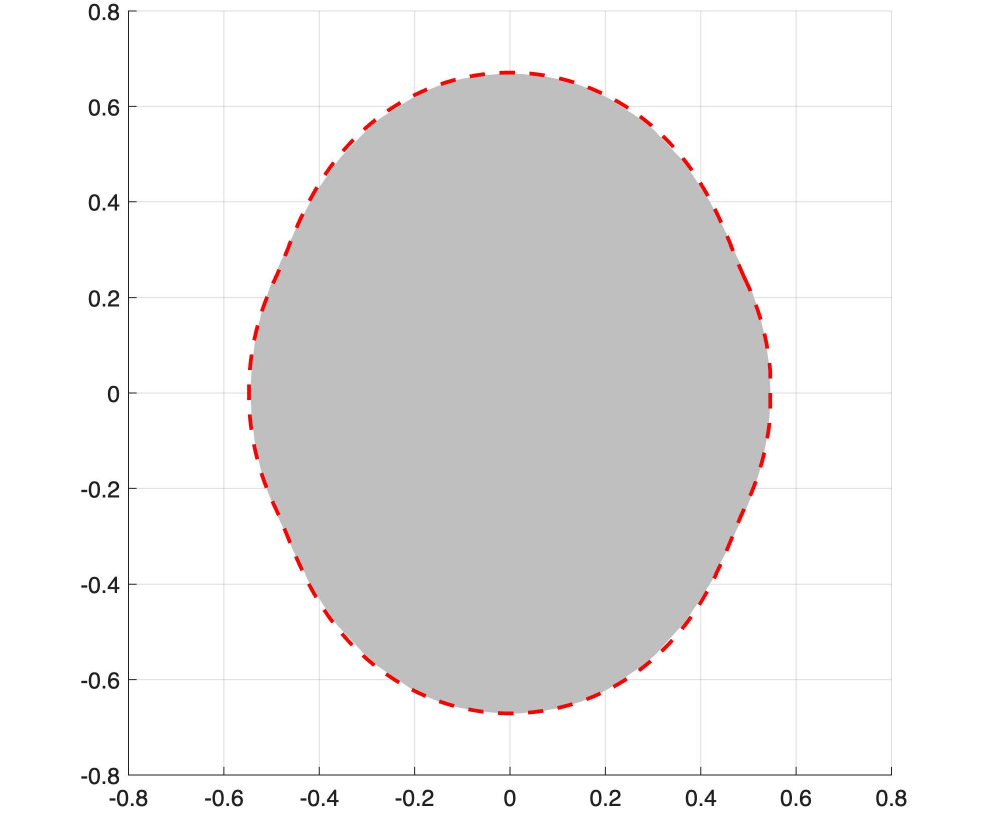} &
        \includegraphics[width=0.21\textwidth]{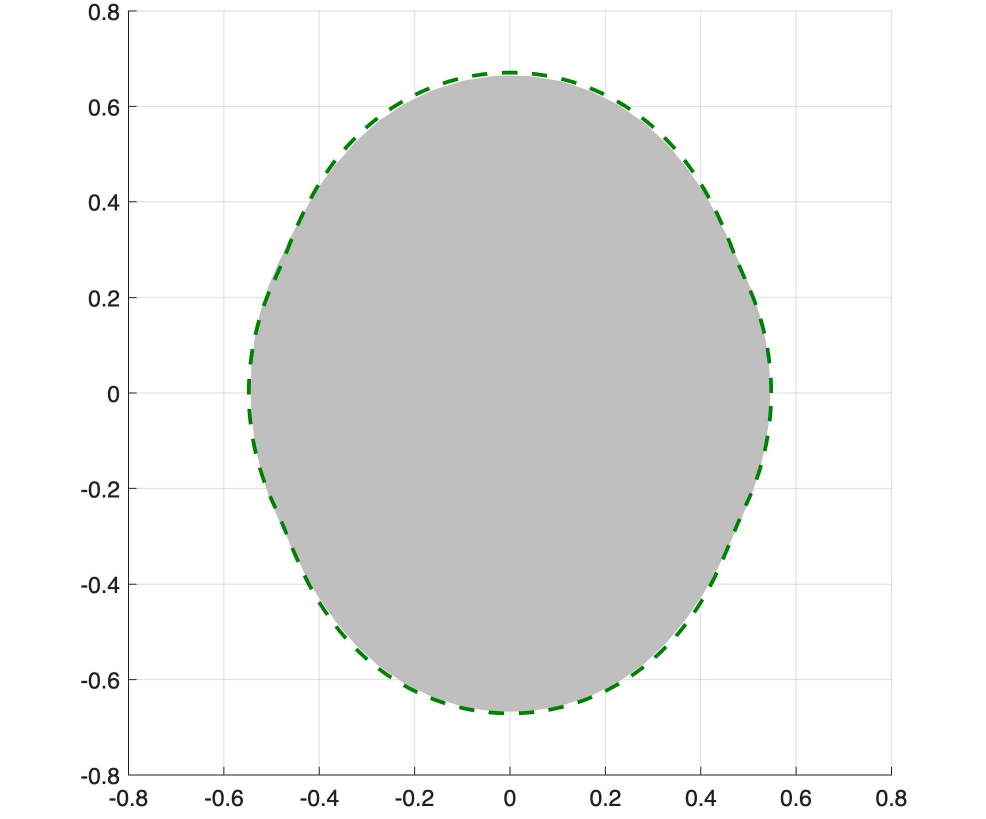}
    \end{tabular}
}

\subfigure[$\pmb c^{(0)}=(0.1,-0.5,-0.3)^\top$, $r^{(0)}=0.3$, $\kappa=3.6$.]{
    \begin{tabular}{cccc}
        \includegraphics[width=0.21\textwidth]{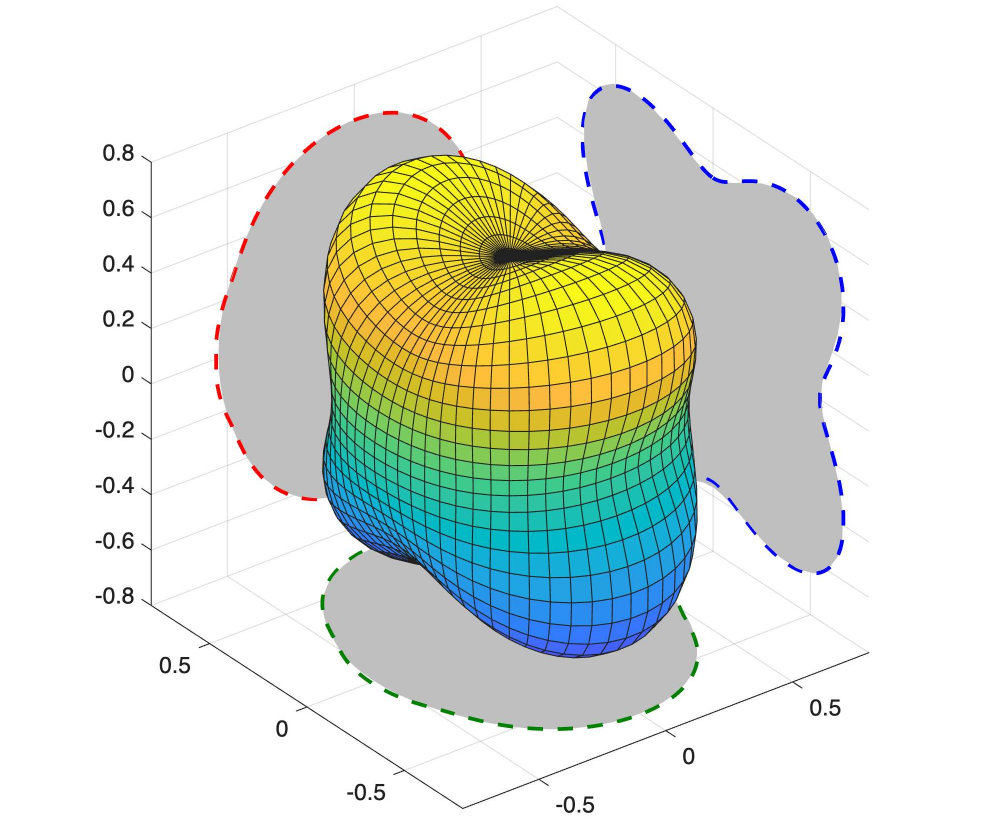} &
        \includegraphics[width=0.21\textwidth]{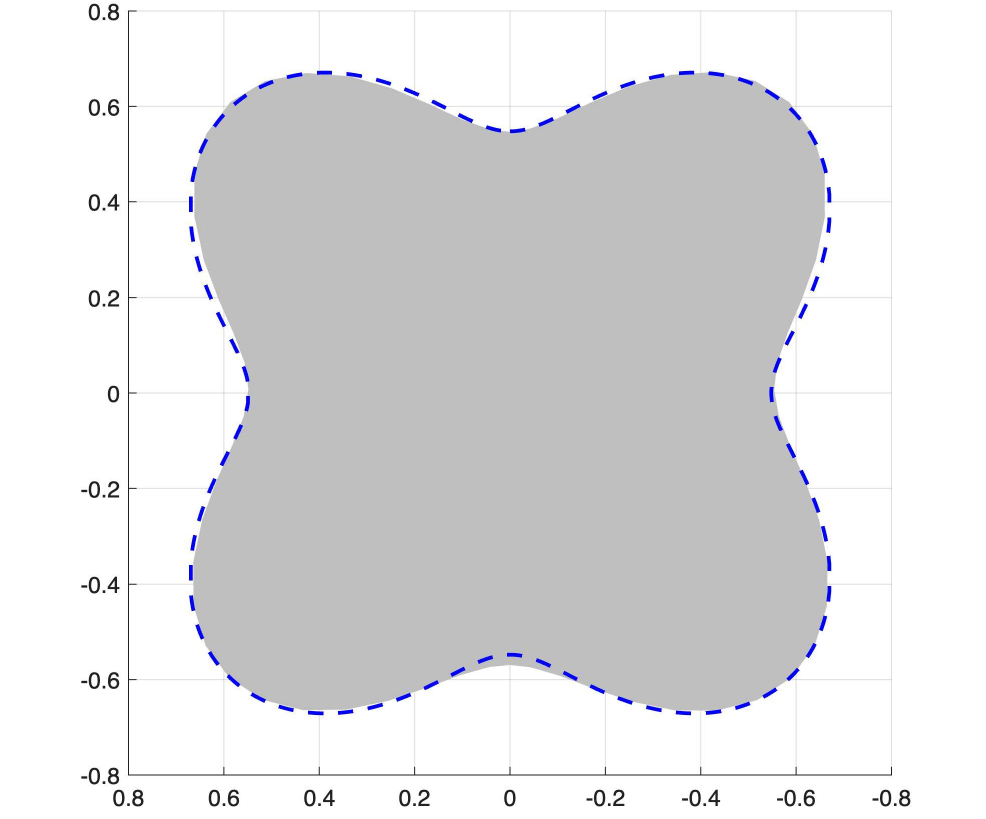} &
        \includegraphics[width=0.21\textwidth]{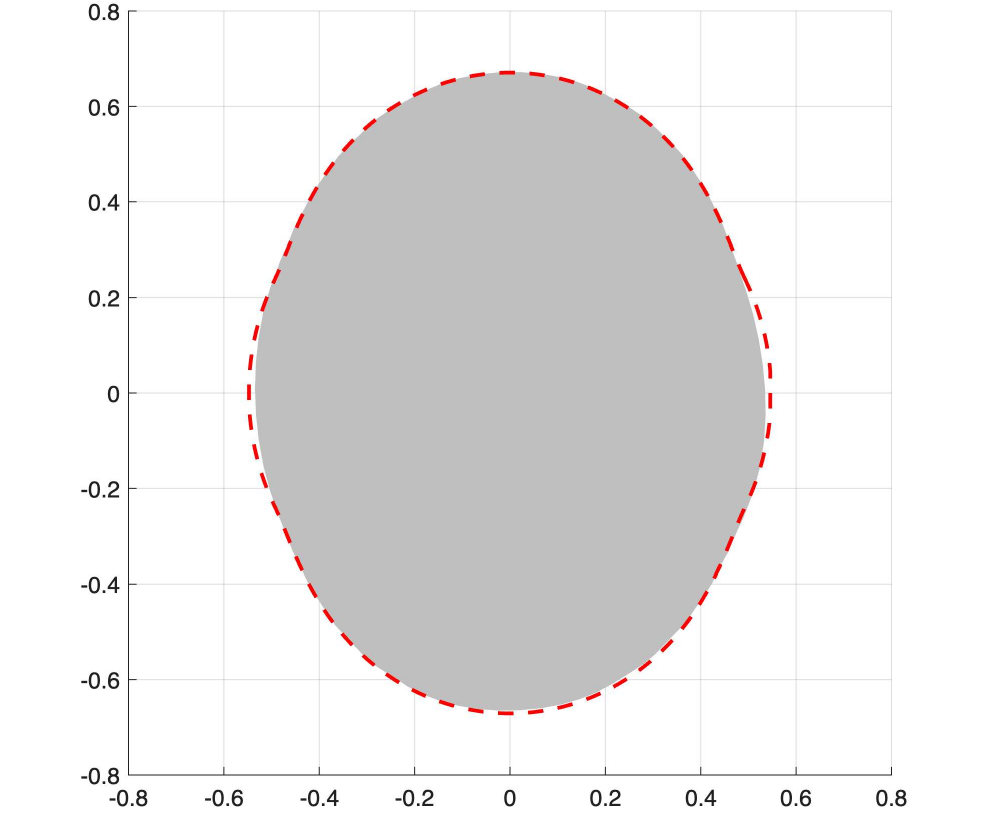} &
        \includegraphics[width=0.21\textwidth]{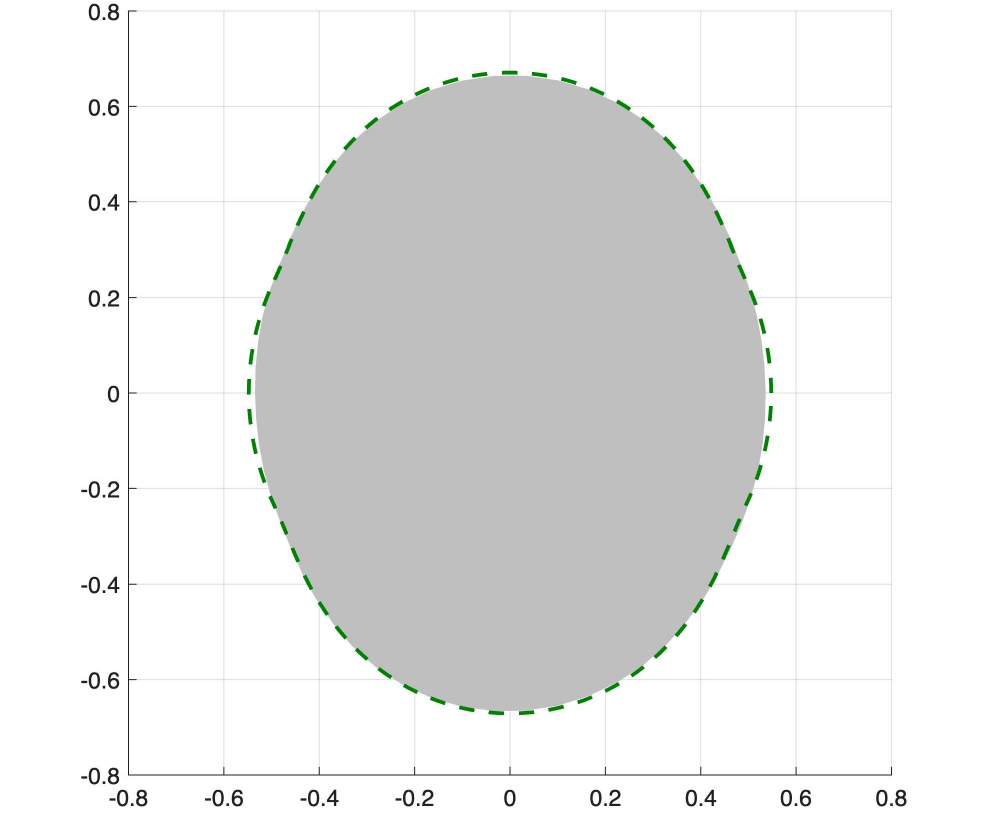}
    \end{tabular}
}
\caption{Reconstruction of a cushion-shaped obstacle from phased far-field data generated by two incident waves $(1,0,0)^\top$, $(-1,0,0)^\top$ with $5\%$ noise, $\epsilon=0.032$.}\label{fig_ex3.2}
\end{figure}

\begin{figure}[!htbp]
\centering
\subfigure[$\pmb c^{(0)} = (-0.1, 0.6, 0.2)^\top$, $r^{(0)} = 0.3$, $\kappa=3$.]{
    \begin{tabular}{cccc}
        {\includegraphics[width=0.21\textwidth]{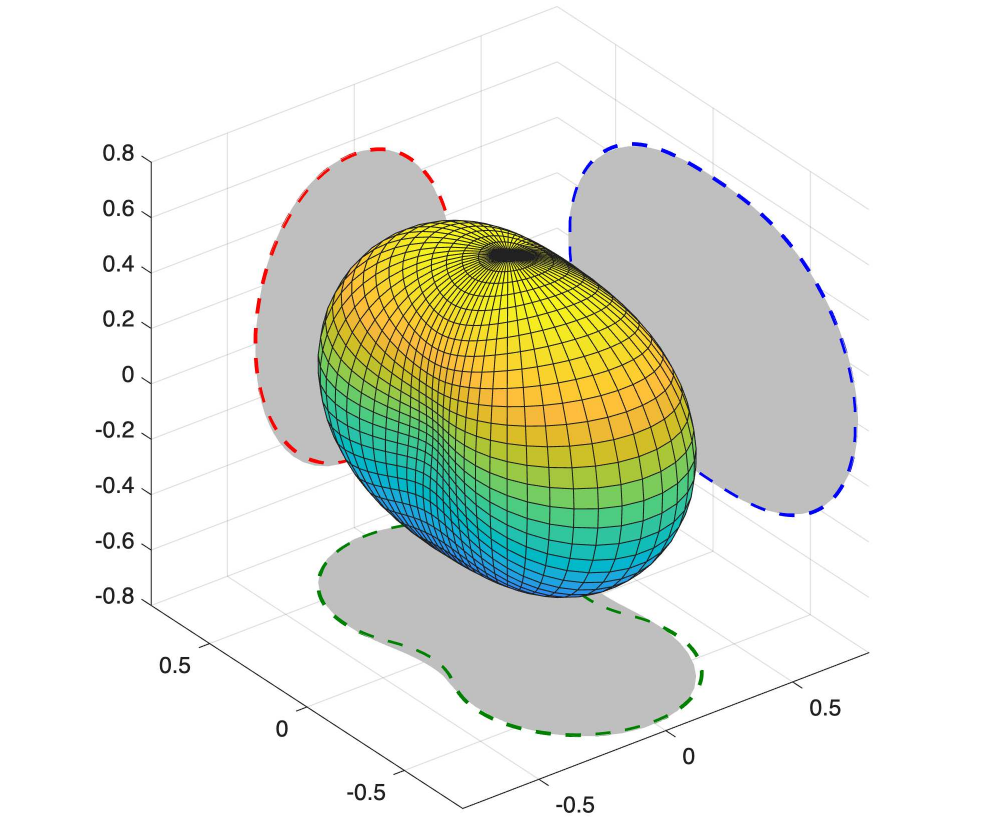}} &
        {\includegraphics[width=0.21\textwidth]{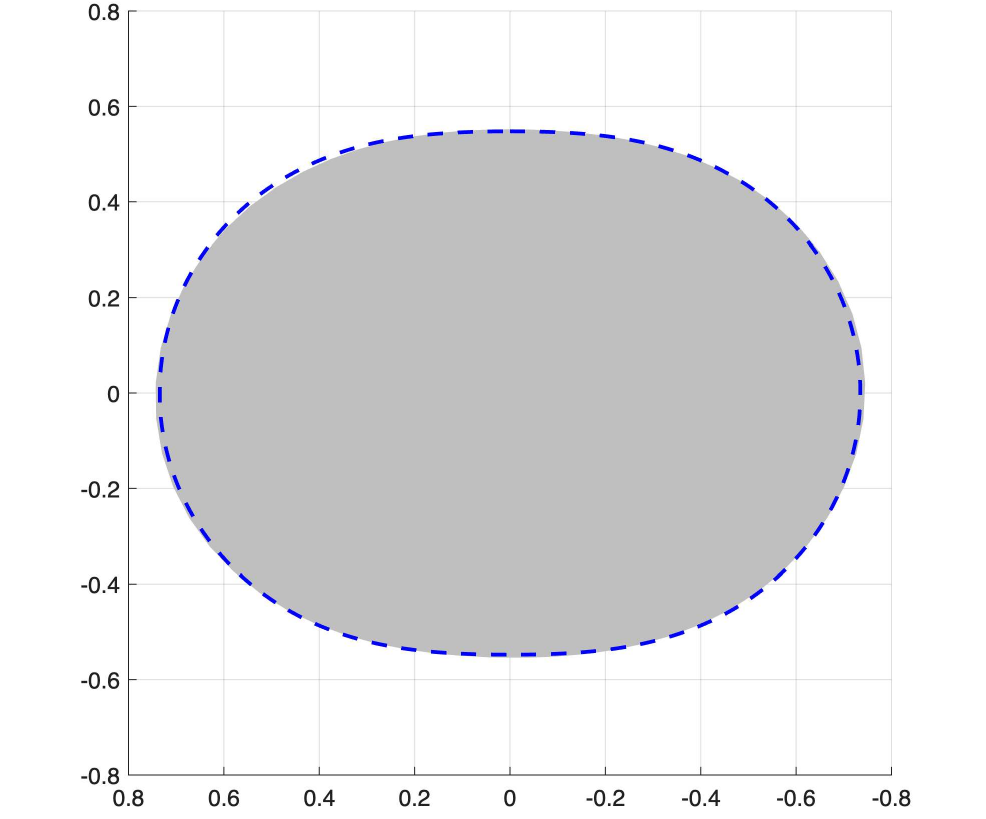}} &
        {\includegraphics[width=0.21\textwidth]{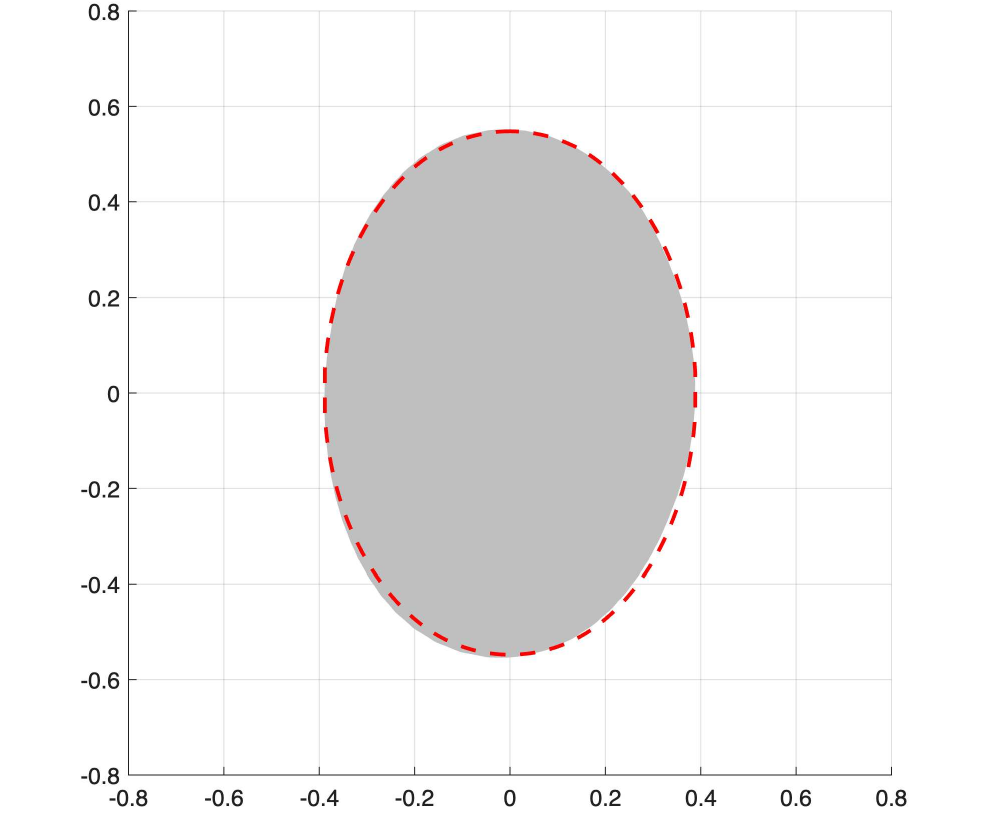}} &
        {\includegraphics[width=0.21\textwidth]{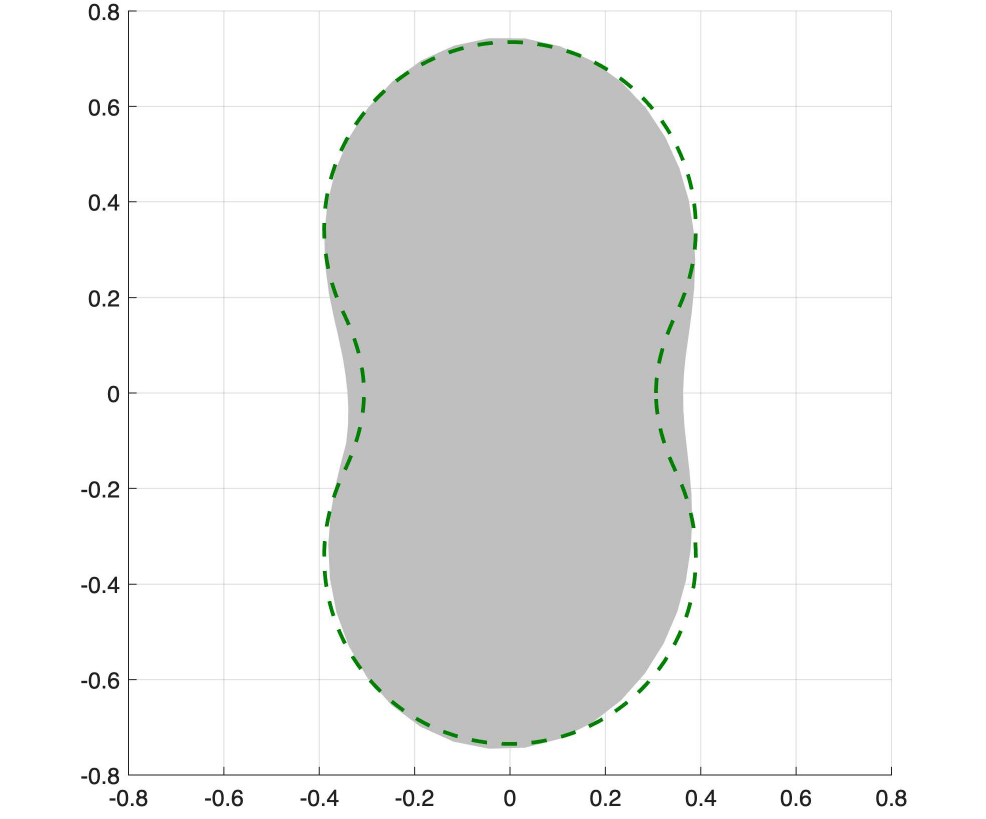}}
    \end{tabular}
}

\subfigure[$\pmb c^{(0)} = (-0.2, -0.4, 0.3)^\top$, $r^{(0)} = 0.8$, $\kappa=3$.]{
    \begin{tabular}{cccc}
        {\includegraphics[width=0.21\textwidth]{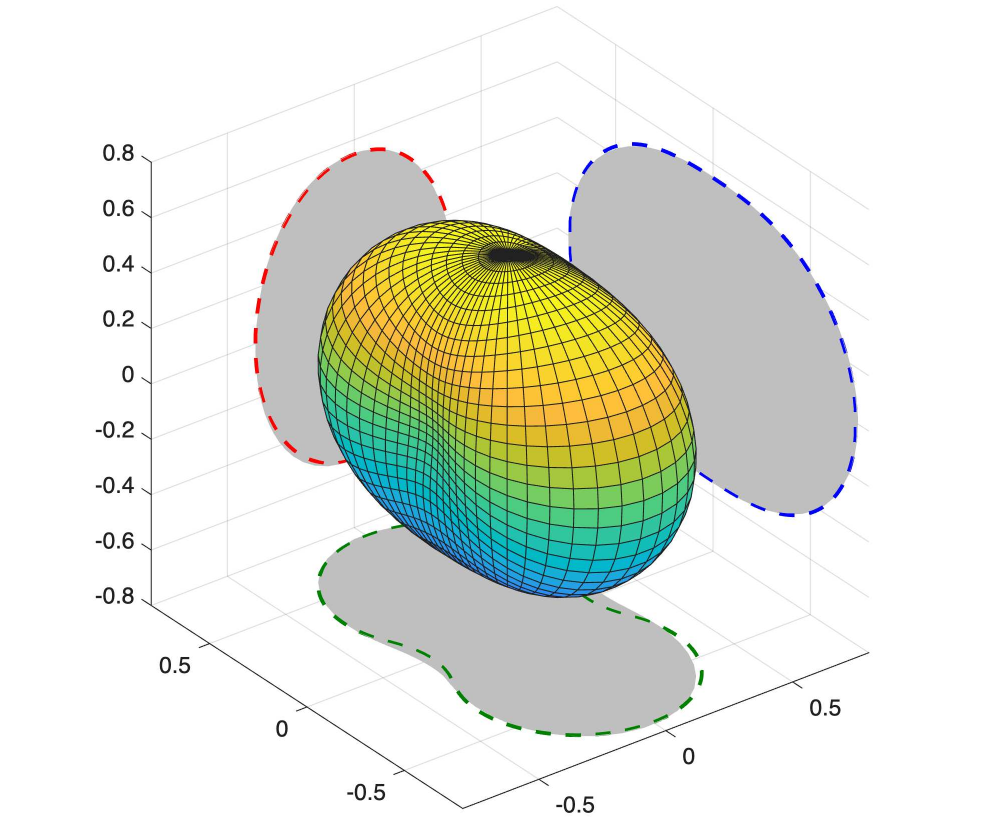}} &
        {\includegraphics[width=0.21\textwidth]{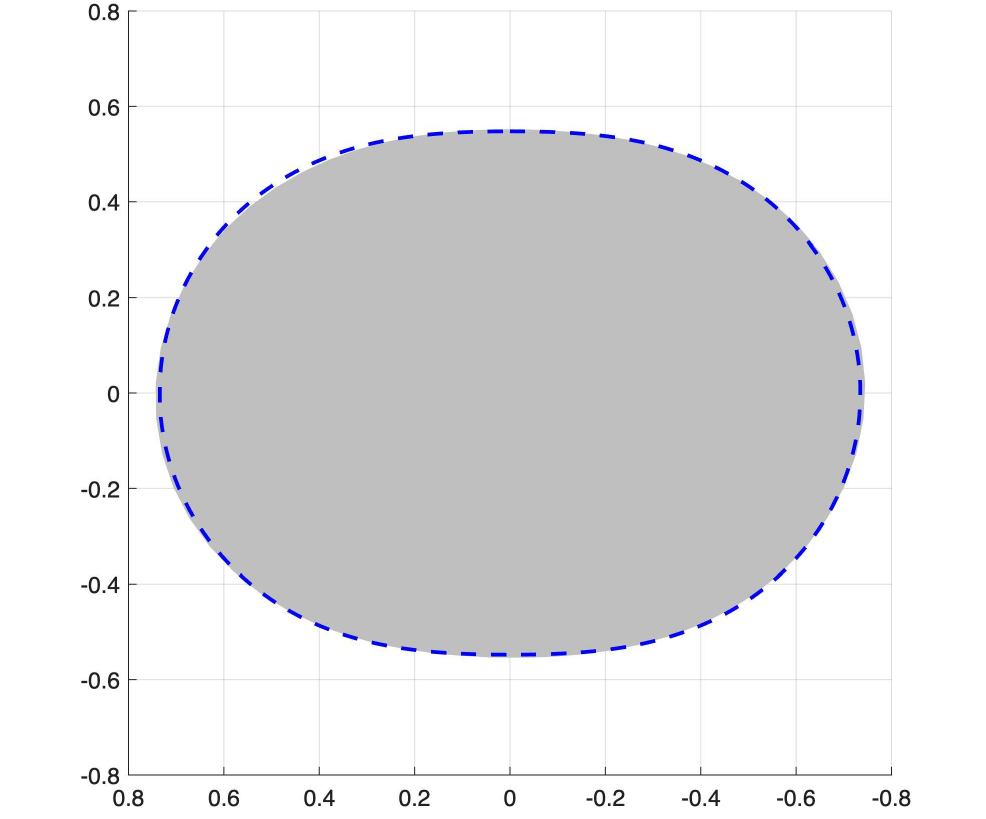}} &
        {\includegraphics[width=0.21\textwidth]{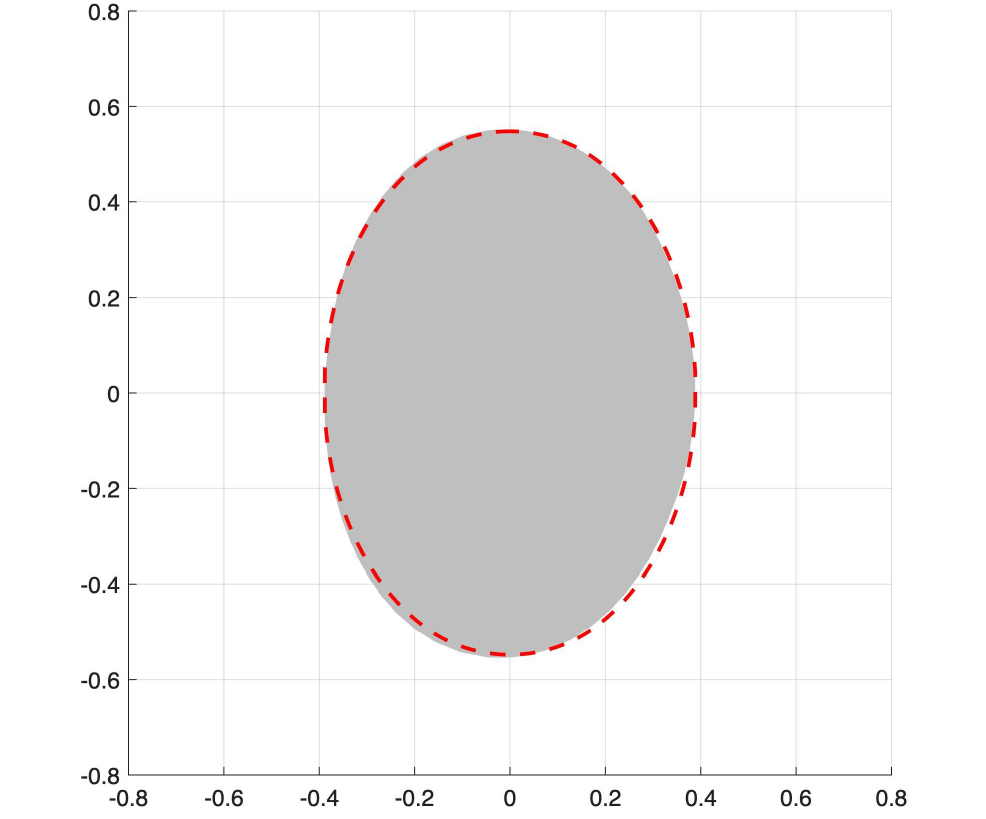}} &
        {\includegraphics[width=0.21\textwidth]{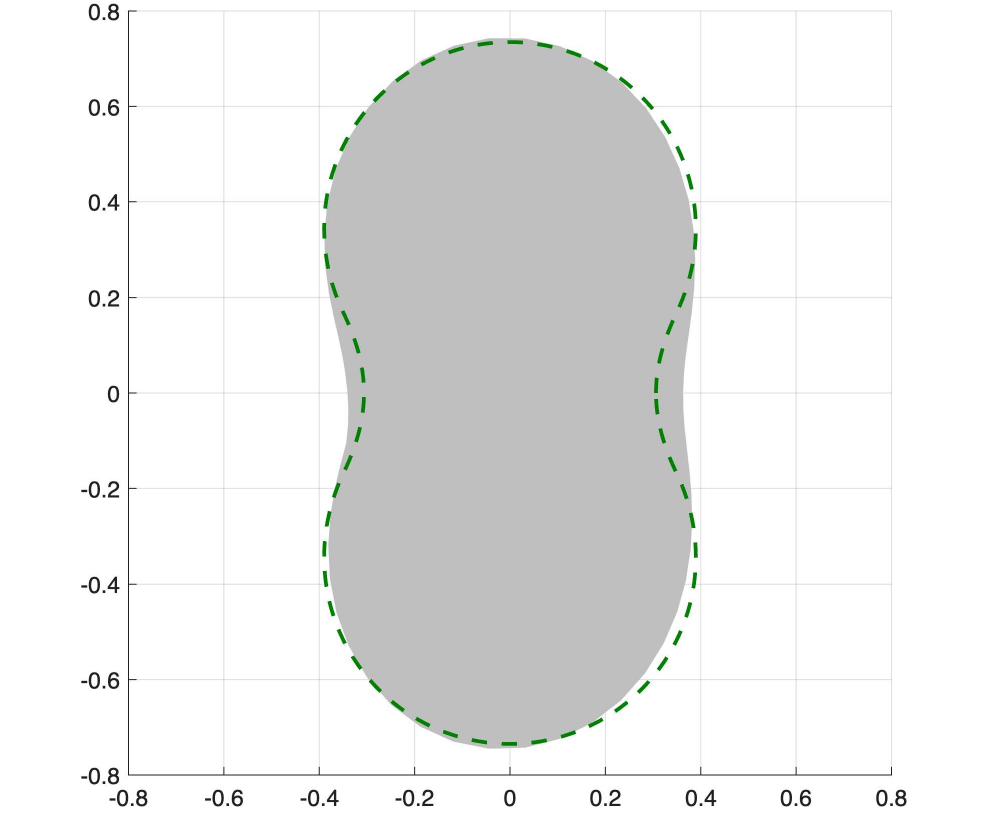}}
    \end{tabular}
}

\subfigure[$\pmb c^{(0)} = (-0.1, 0.6, 0.2)^\top$, $r^{(0)} = 0.3$, $\kappa=6$.]{
    \begin{tabular}{cccc}
        {\includegraphics[width=0.21\textwidth]{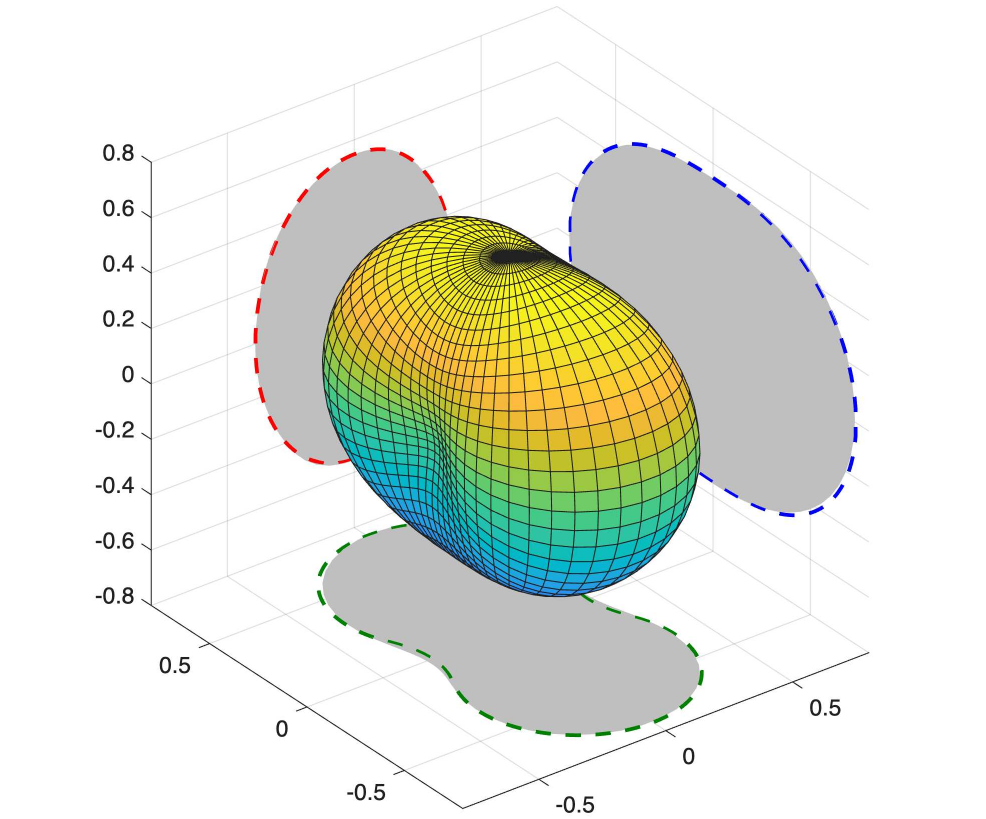}} &
        {\includegraphics[width=0.21\textwidth]{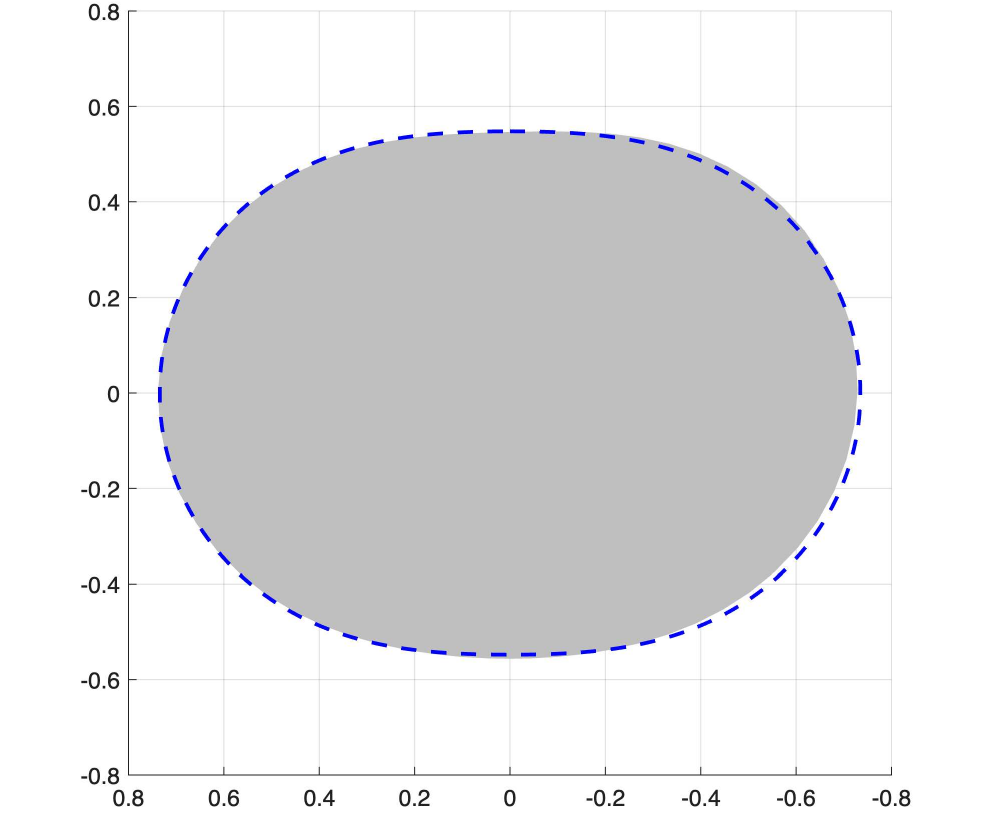}} &
        {\includegraphics[width=0.21\textwidth]{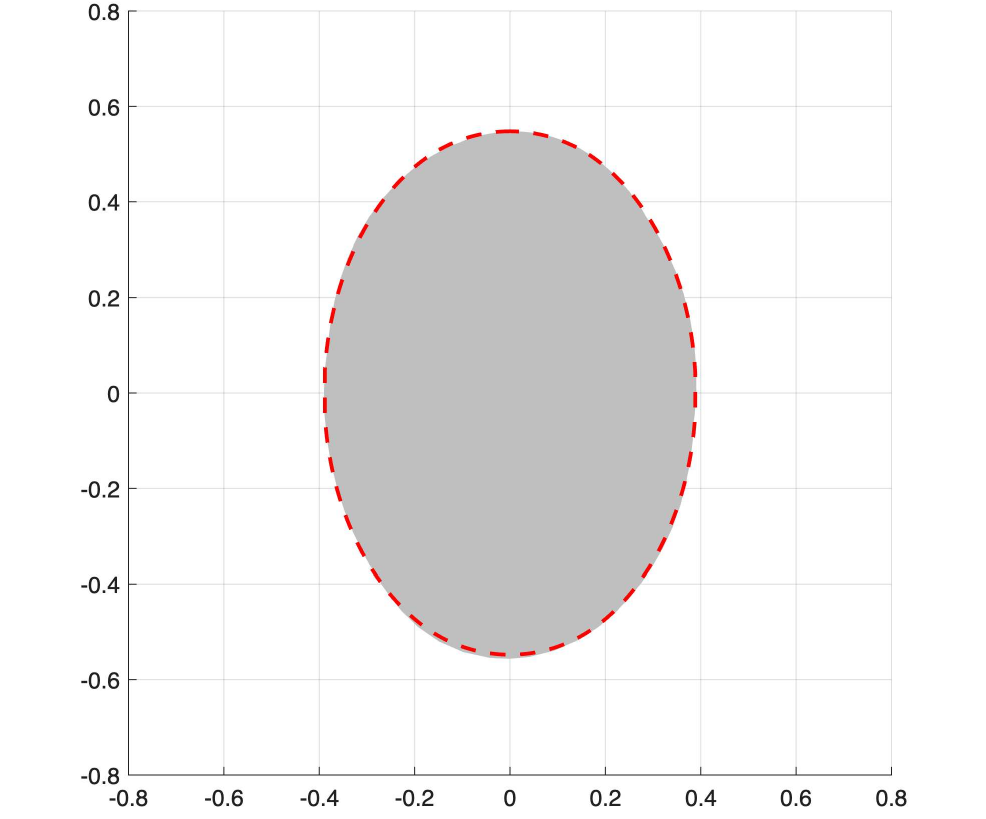}} &
        {\includegraphics[width=0.21\textwidth]{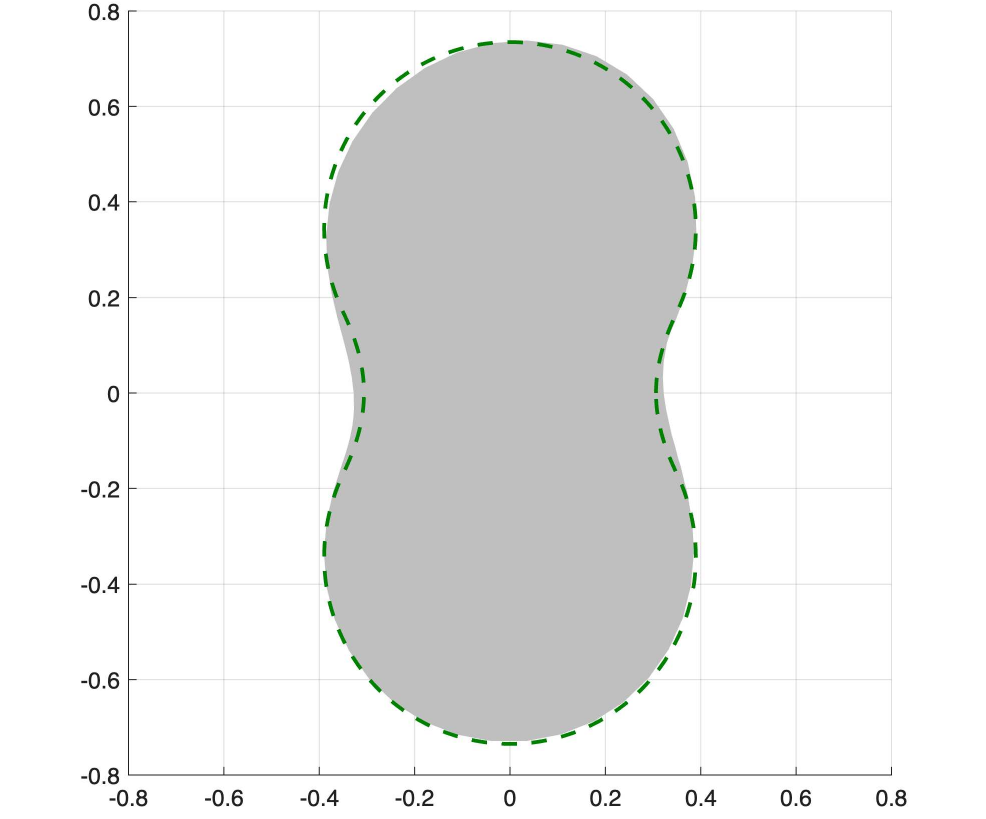}}
    \end{tabular}
}
\caption{Reconstruction of a pinched ball-shaped obstacle from phaseless far-field data generated by two point sources located at \((4,0,0)^\top\) and \((-4,0,0)^\top\) with $5\%$ noise, $\epsilon=0.015$.}\label{fig_ex3.3}
\end{figure}

\vspace{2ex}
{\noindent\bf Example 3: Reconstructions under different boundary conditions.}
\vspace{1ex}

We next examine whether the proposed method can be applied to boundary conditions other than the Dirichlet condition. To enable a clean side-by-side comparison among the three boundary conditions, all three cases use the same initial guess $\pmb c^{(0)}=(0,-0.1,0.1)^{\top}$, $r^{(0)}=0.6$ and the same wavenumber $\kappa=6.6$. Figure~\ref{fig_ex3} presents the reconstructions of the cushion-shaped obstacle from phaseless far-field data under Dirichlet, Neumann, and impedance boundary conditions, respectively.

\begin{figure}[!htbp]
\centering
\subfigure[Dirichlet boundary condition, $\kappa=6.6$, $\epsilon=0.004$.]{
    \begin{tabular}{cccc}
        {\includegraphics[width=0.21\textwidth]{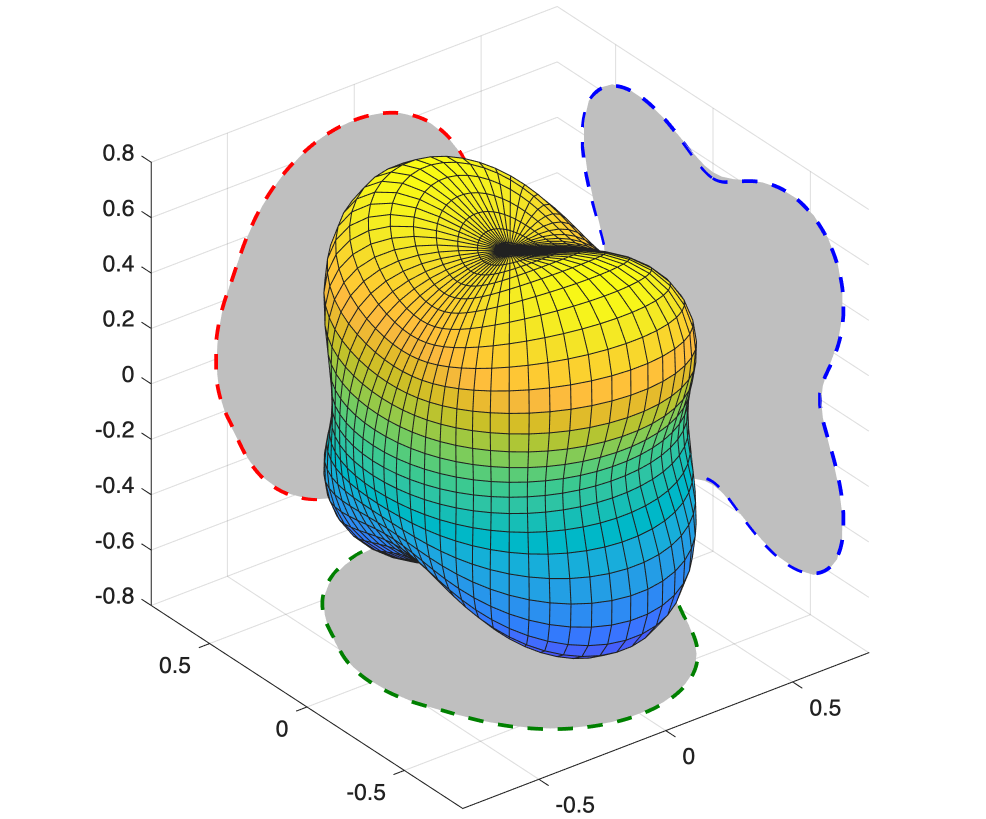}} &
        {\includegraphics[width=0.21\textwidth]{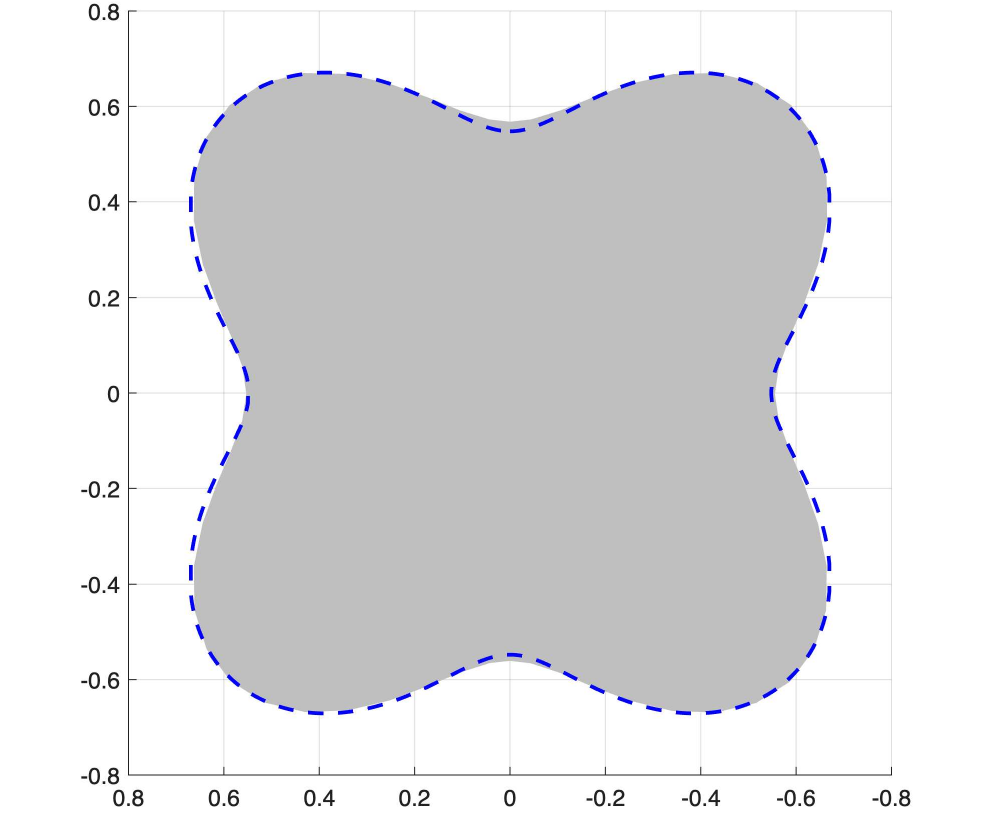}} &
        {\includegraphics[width=0.21\textwidth]{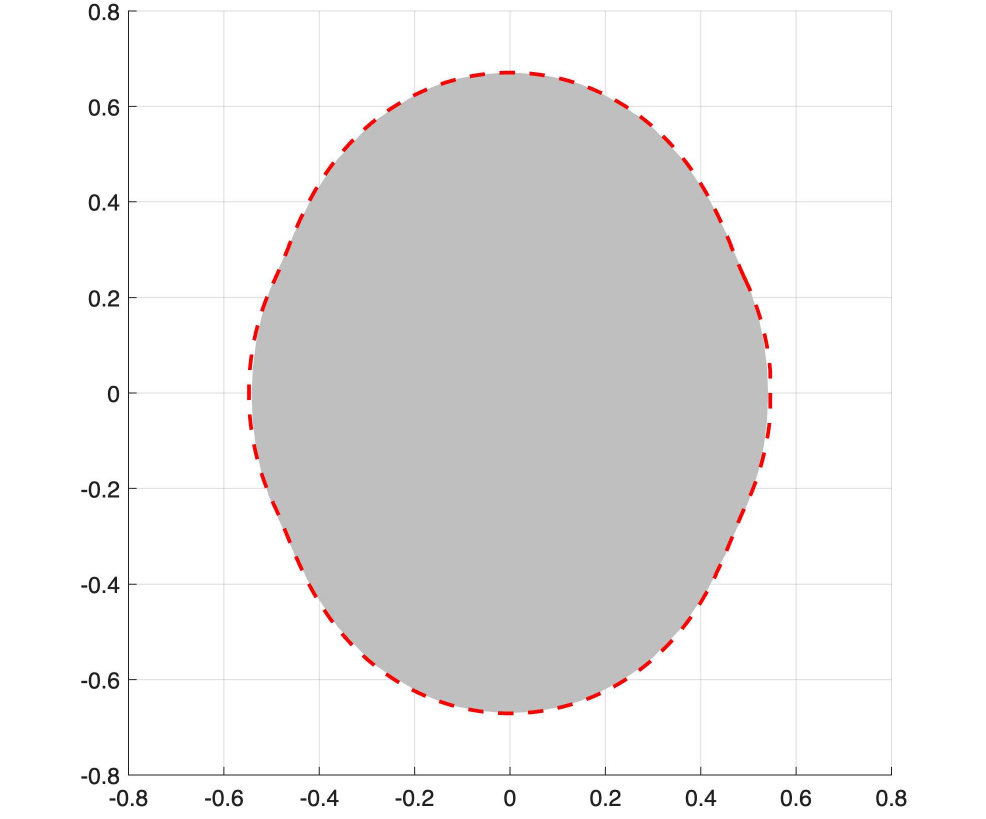}} &
        {\includegraphics[width=0.21\textwidth]{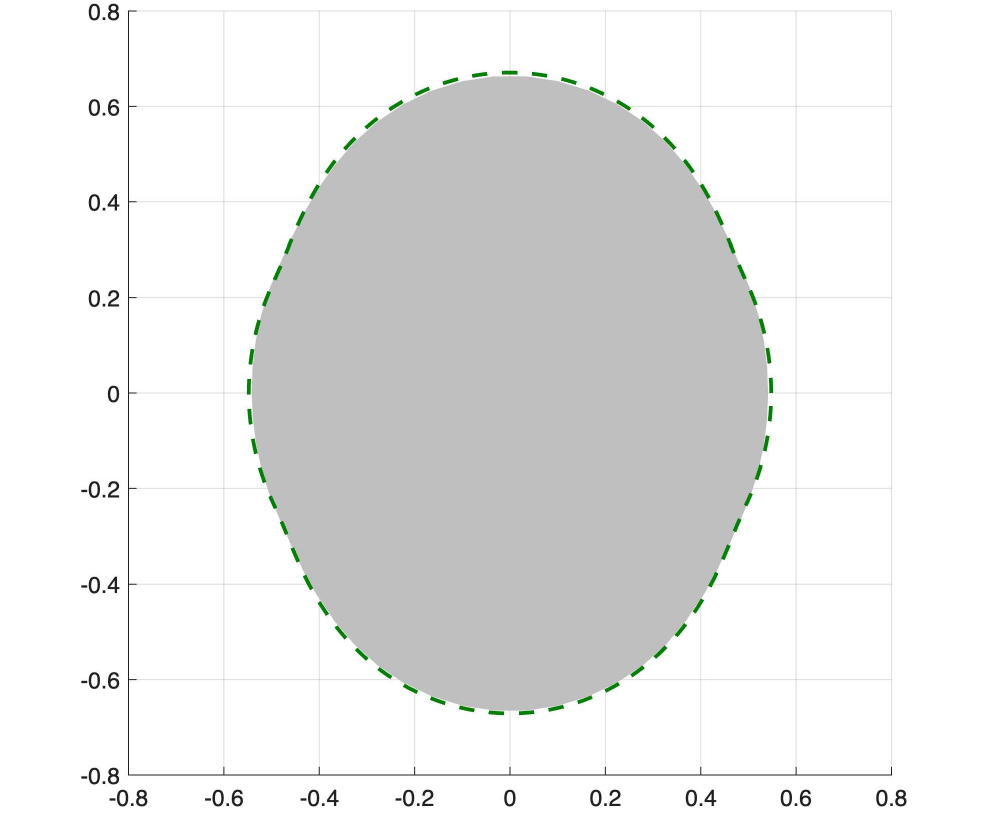}}
    \end{tabular}
}

\subfigure[Neumann boundary condition, $\kappa=6.6$, $\epsilon=0.018$.]{
    \begin{tabular}{cccc}
        {\includegraphics[width=0.21\textwidth]{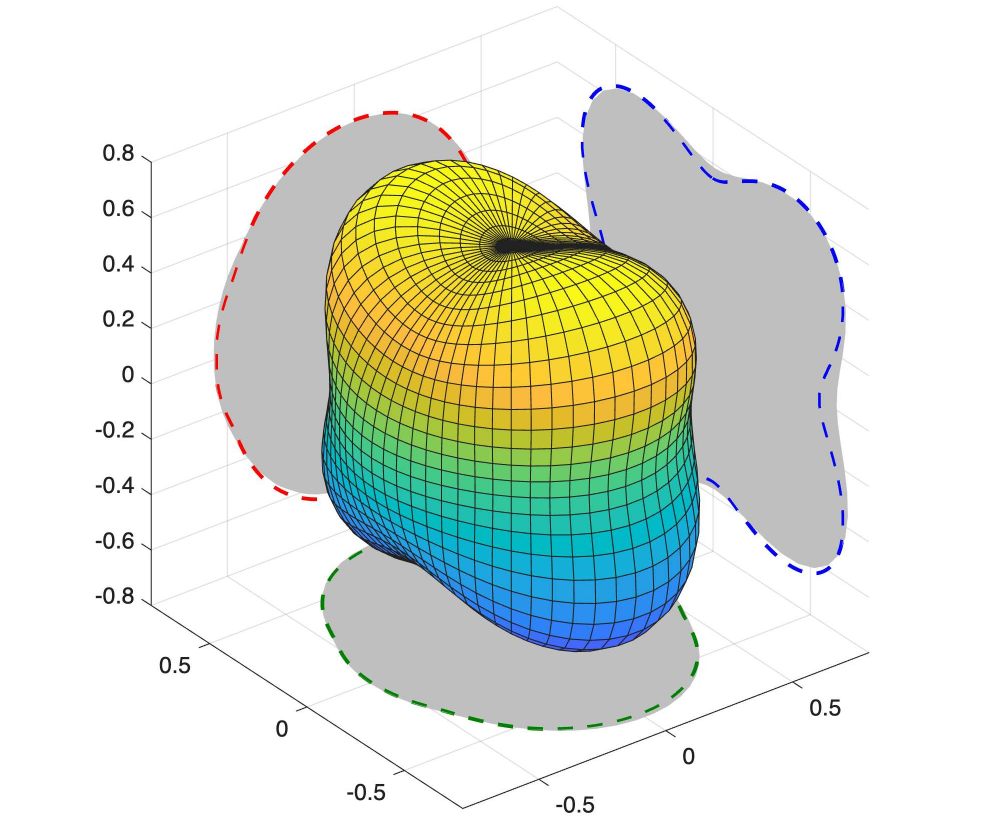}} &
        {\includegraphics[width=0.21\textwidth]{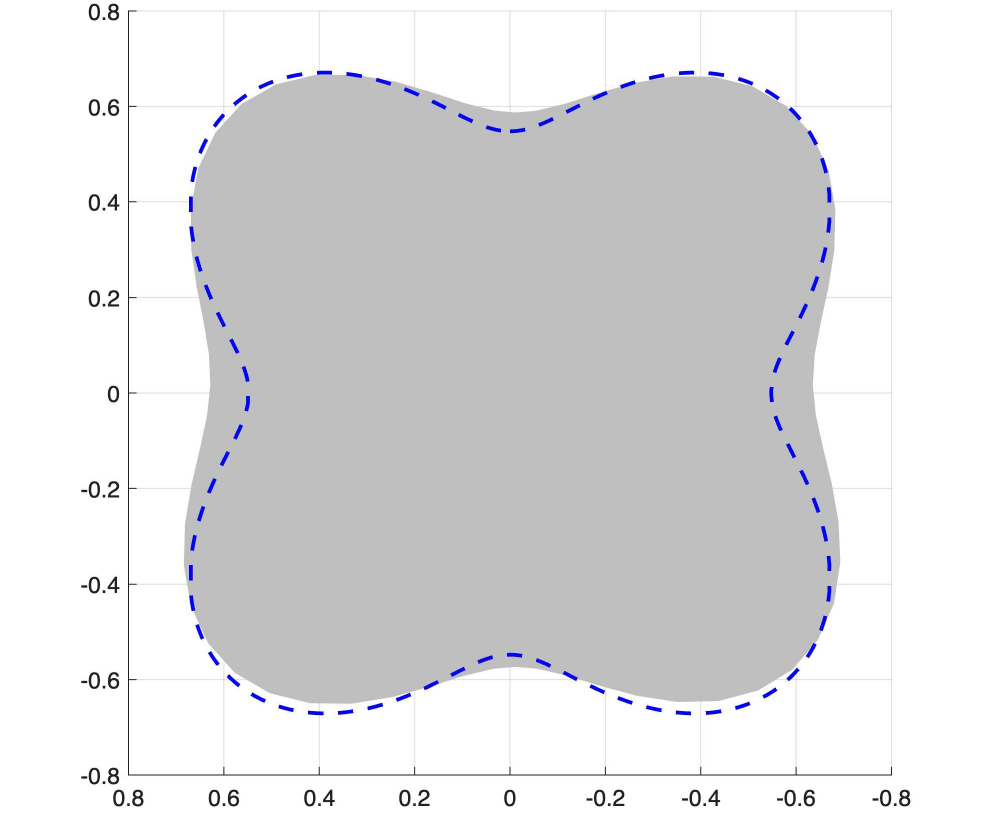}} &
        {\includegraphics[width=0.21\textwidth]{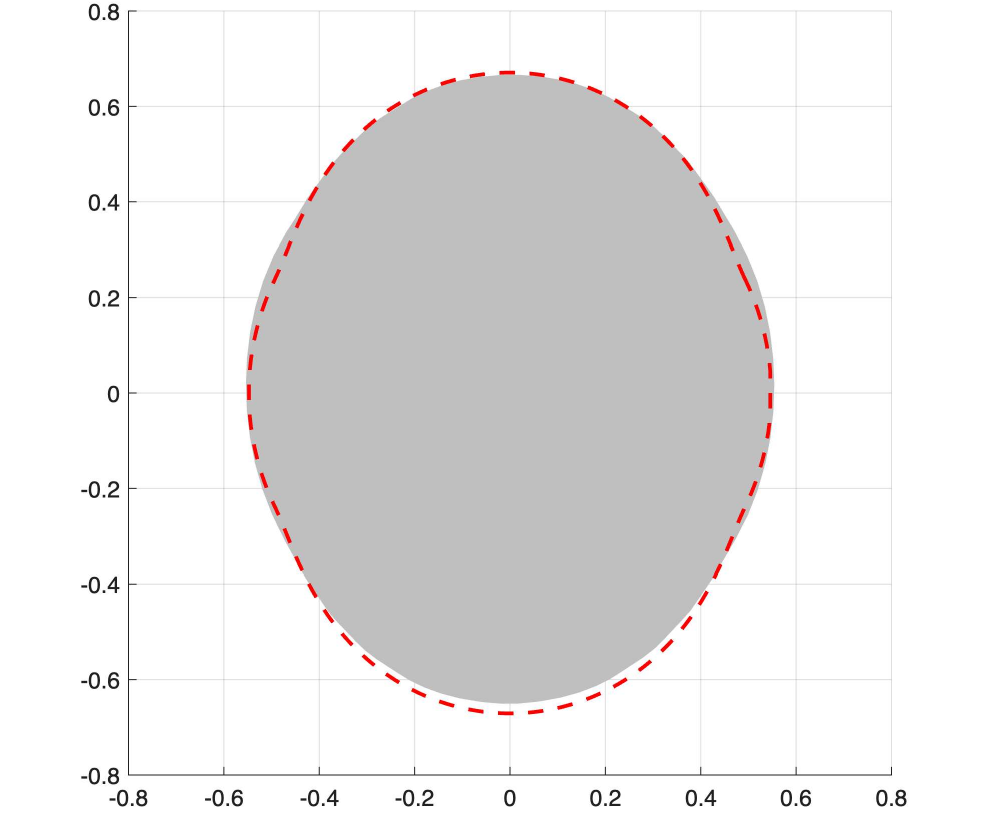}} &
        {\includegraphics[width=0.21\textwidth]{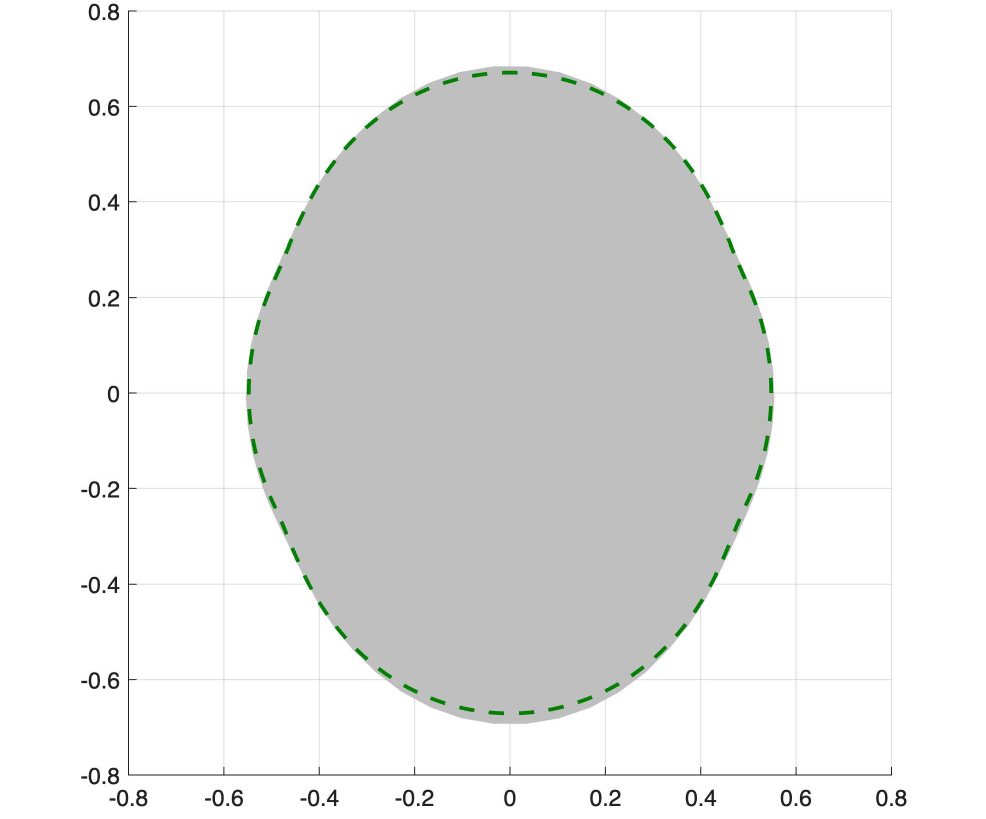}}
    \end{tabular}
}

\subfigure[Impedance boundary condition with $\eta=1+1{\rm i}$, $\kappa=6.6$, $\epsilon=0.008$.]{
    \begin{tabular}{cccc}
        {\includegraphics[width=0.21\textwidth]{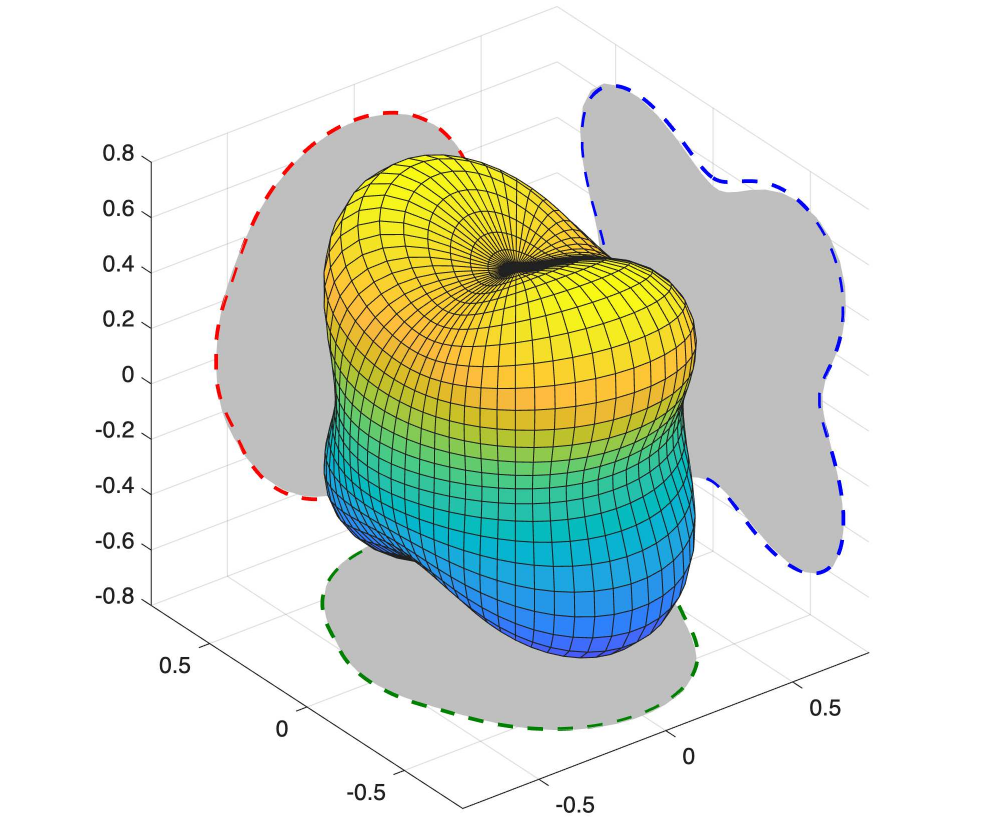}} &
        {\includegraphics[width=0.21\textwidth]{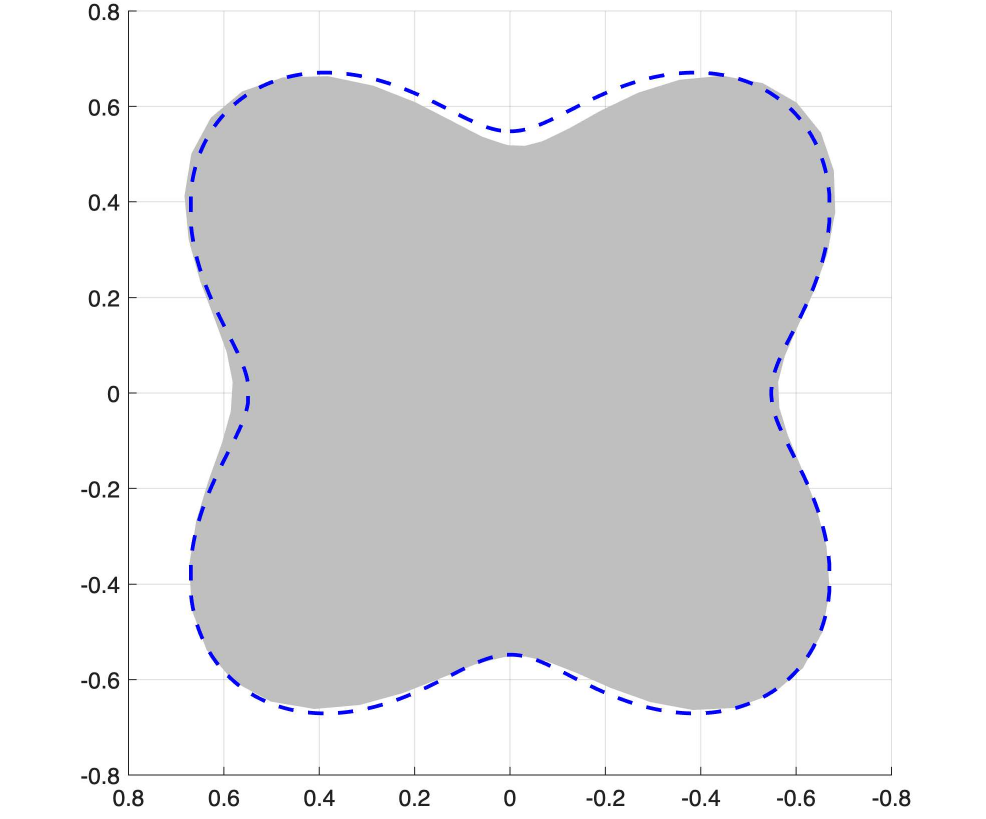}} &
        {\includegraphics[width=0.21\textwidth]{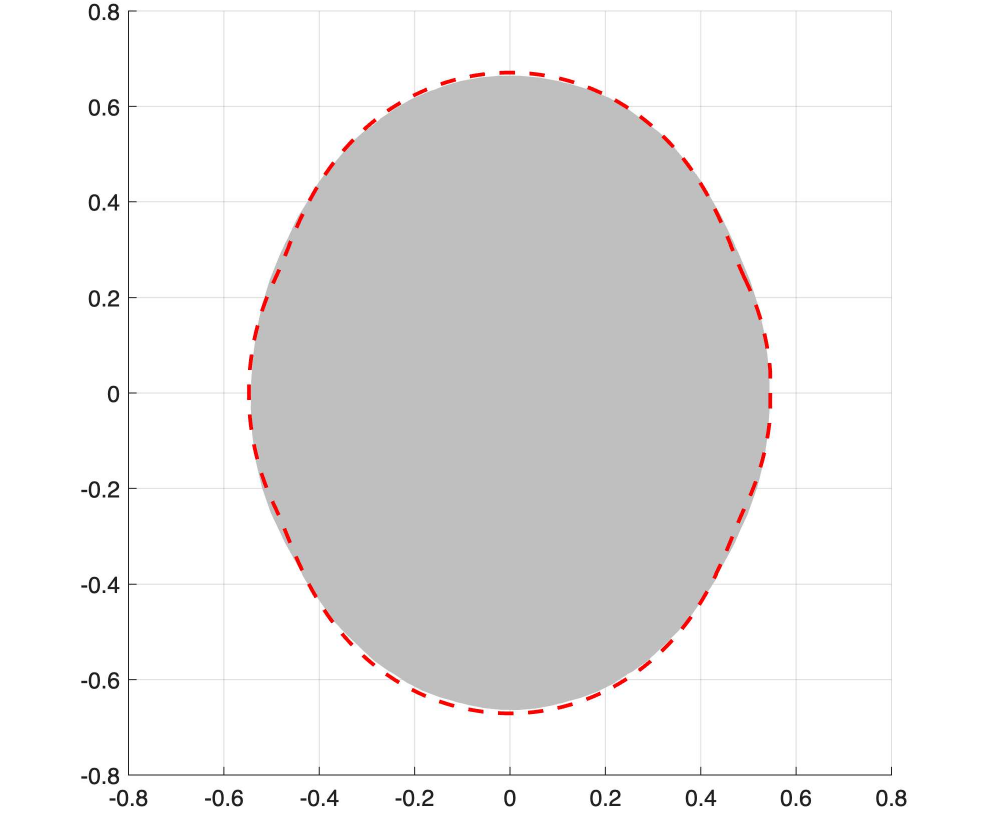}} &
        {\includegraphics[width=0.21\textwidth]{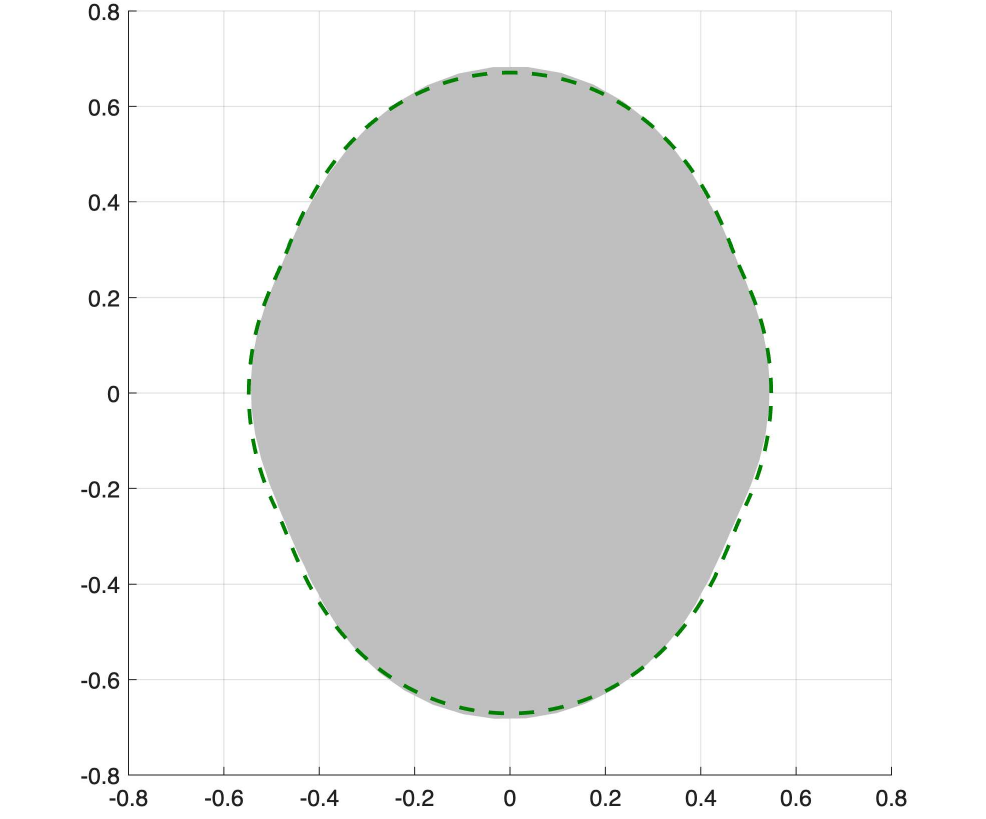}}
    \end{tabular}
}
\caption{Reconstructions of a cushion-shaped obstacle from phaseless far-field data under different boundary conditions with $1\%$ noise. The incident field is generated by point sources located at $(4,0,0)^\top$ and $(-4,0,0)^\top$, and the initial guess is a sphere with $\pmb{c}^{(0)}=(0,-0.1,0.1)^{\top}$ and $r^{(0)}=0.6$.}\label{fig_ex3}
\end{figure}

\begin{figure}[!htbp]
    \centering
    \begin{tabular}{cccc}
        \subfigure[Initial surface]{\includegraphics[width=0.217\textwidth]{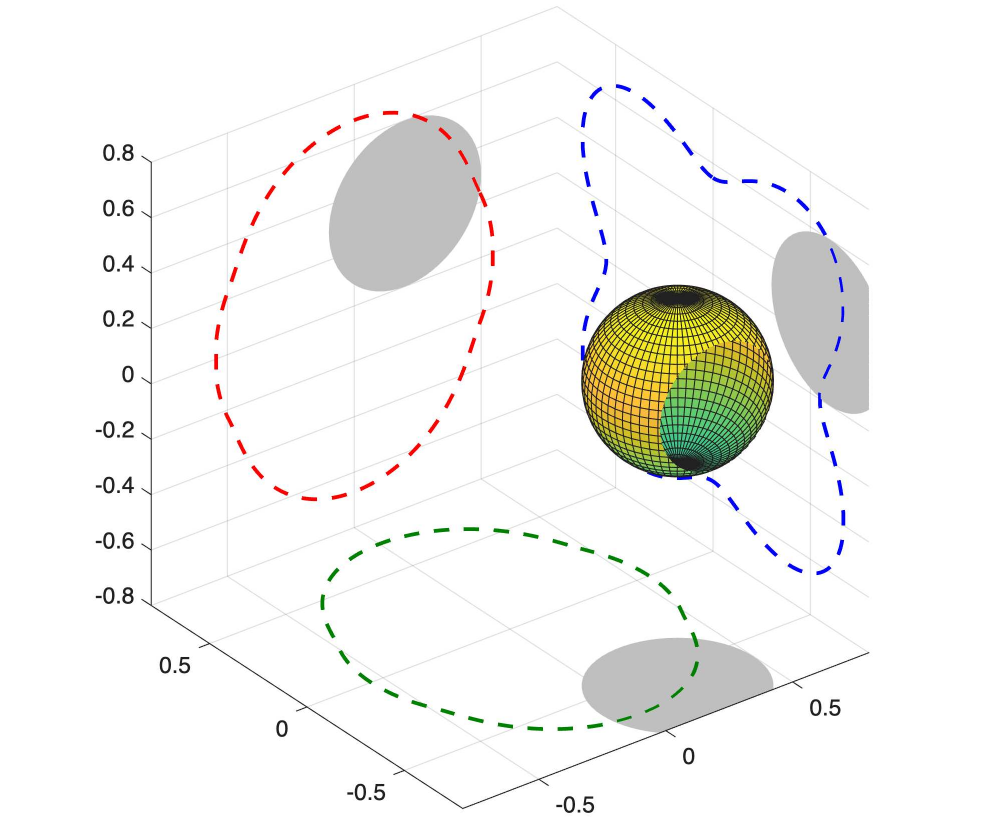}} &
        \subfigure[$M=1$]{\includegraphics[width=0.217\textwidth]{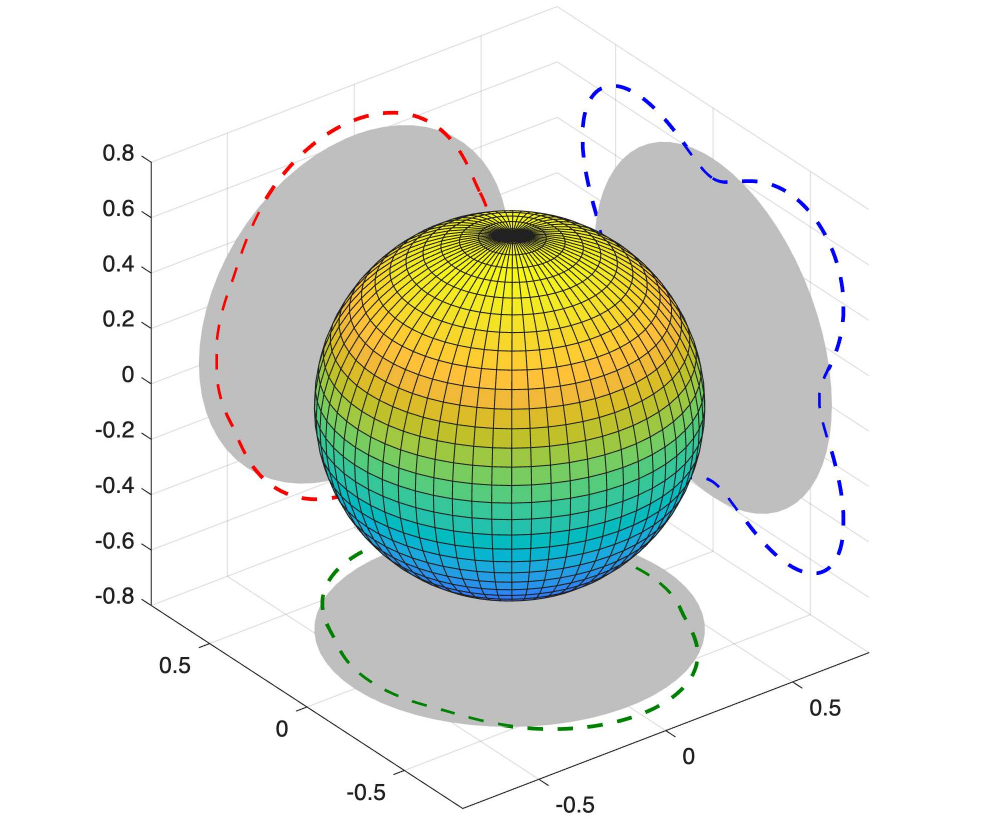}} &
        \subfigure[$M=3$]{\includegraphics[width=0.217\textwidth]{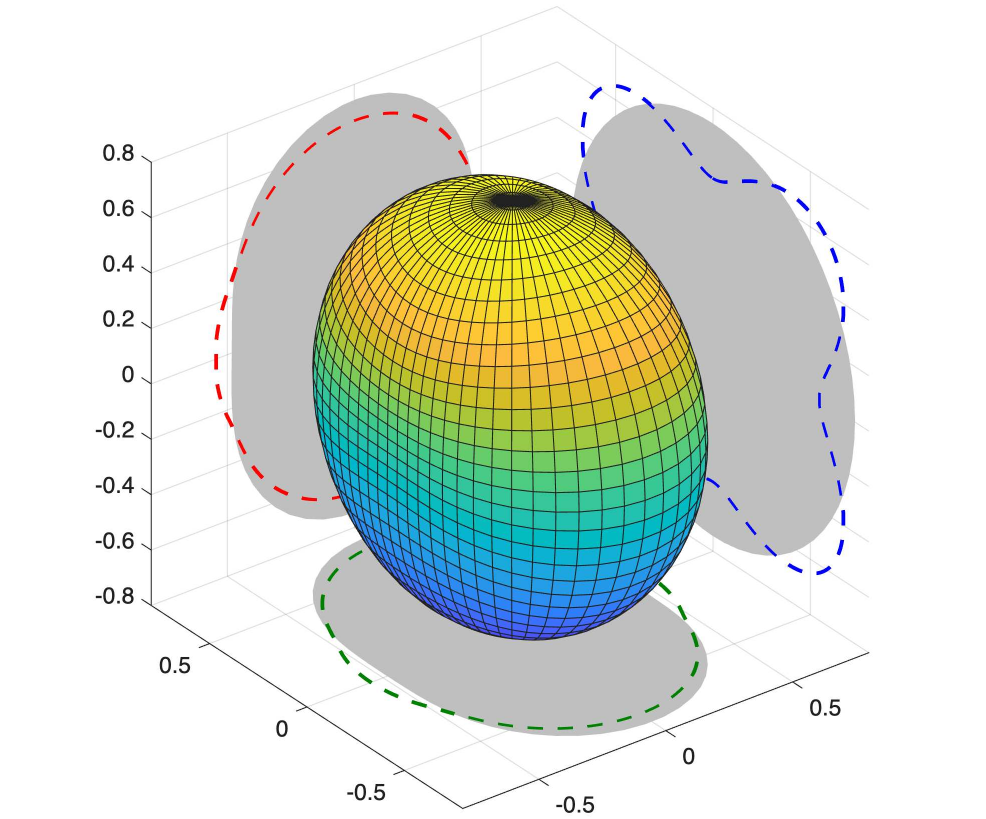}} &
        \subfigure[$M=5$]{\includegraphics[width=0.217\textwidth]{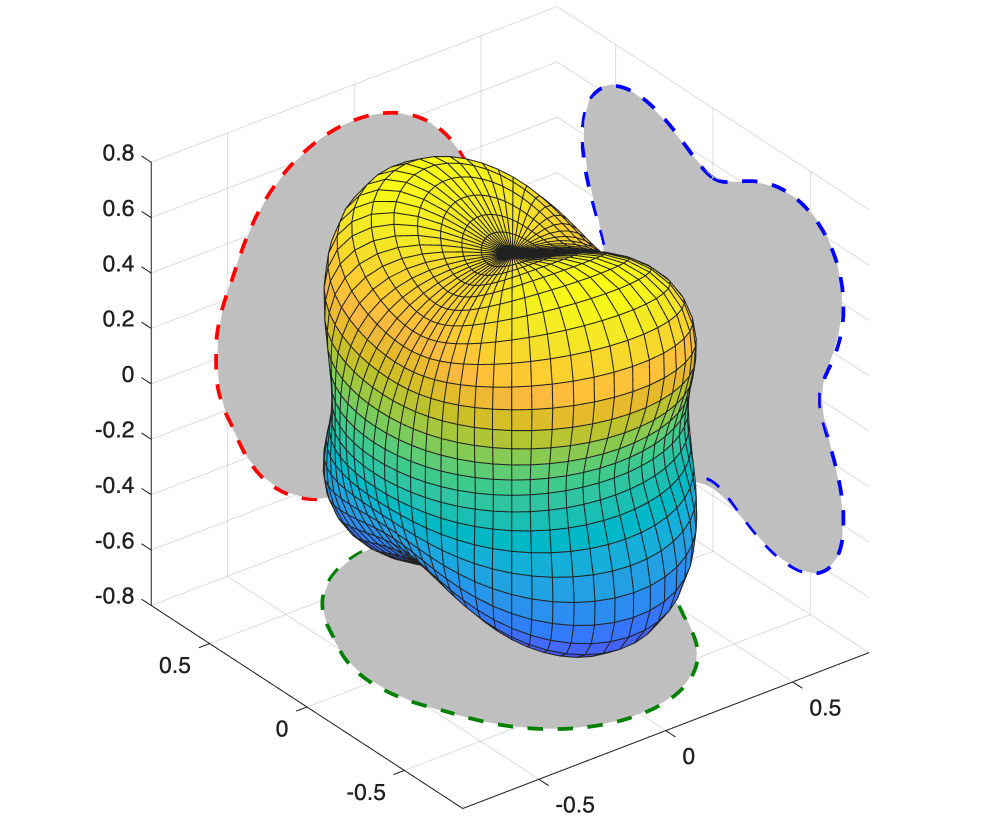}}
    \end{tabular}
    \caption{Reconstruction of a cushion-shaped obstacle from phased scattered field data with $1\%$ noise. $\kappa=4$, $\pmb d=(\pm 1,0,0)^{\top}$, $R=5$, $\pmb c^{(0)}=(0.2,-0.6,0.3)^{\top}$, $r^{(0)}=0.3$, $\epsilon=0.009$. Subfigures~(a)--(d) show the initial surface and the progressive reconstructions for truncation numbers $M=1,3,5$, respectively.}\label{fig_ex4}
\end{figure}

\vspace{2ex}
{\noindent\bf Example 4: Reconstruction process with increasing truncation numbers.}
\vspace{1ex}

Figure~\ref{fig_ex4} illustrates a process of reconstruction for a cushion-shaped obstacle from phased scattered field data with $1\%$ noise. 
Starting from the initial spherical surface in subfigure~(a), the process is shown for truncation numbers $M=1,3,5$ in subfigures~(b)--(d), respectively. 
The reconstructed surface gradually captures more geometric details of the true obstacle with increasing $M$, which demonstrates the effectiveness of iteratively increasing truncation numbers for improving reconstruction quality.

\vspace{2ex}
{\noindent\bf Example 5: Reconstruction of a non-star-shaped obstacle under the Dirichlet boundary condition.}
\vspace{1ex}

To further test the performance of the proposed method on a non-star-shaped geometry, 
we consider a bean-like surface obstacle under the Dirichlet boundary condition. 
The parametrization of the bean-like surface is given by
\[
z(\theta,\phi) =
\begin{pmatrix}
\begin{array}{l}
0.7\sqrt{1-0.1\cos(\pi\cos\theta)}\,\sin\theta\cos\phi \\[1ex]
0.7\left(\sqrt{1-0.4\cos(\pi\cos\theta)}\,\sin\theta\sin\phi + 0.3\cos(\pi\cos\theta)\right) \\[1ex]
0.7\cos\theta
\end{array}
\end{pmatrix}.
\]
The reconstruction results are shown in Figure~\ref{fig_ex5}. 
Subfigure~(a) displays the true shape from multiple viewpoints, 
in which the pronounced concave region along the $-y$ direction clearly reveals the nonconvex nature of the obstacle. 

Subfigure~(b) presents the reconstruction obtained from phaseless far-field data generated by two point sources located at $(4,0,0)^\top$ and $(-4,0,0)^\top$ with $1\%$ noise. The proposed approach successfully recovers the overall shape and location of the obstacle; however, since both sources illuminate the obstacle along the $\pm x$ direction, the concave region facing the $-y$ direction receives only weak indirect illumination, and the reconstructed surface in this region appears noticeably smoother than the true shape, failing to fully capture the depth of the indentation. 

To overcome this limitation, Subfigure~(c) shows the reconstruction obtained by adding two more point sources at $(0,4,0)^\top$ and $(0,-4,0)^\top$, so that the incident waves now directly probe the concave region from the $\pm y$ direction. With this enriched illumination, the reconstructed surface fits the true shape much more accurately, and in particular the concave feature is clearly resolved. This comparison demonstrates that the proposed method is capable of handling non-star-shaped obstacles, and that supplementing the incident sources along directions that directly illuminate the concave region significantly improves the reconstruction quality.

\begin{figure}[!htbp]
\centering 
\subfigure[True shape of bean-like surface obstacle.]{
    \begin{tabular}{cccc}
        {\includegraphics[width=0.21\textwidth]{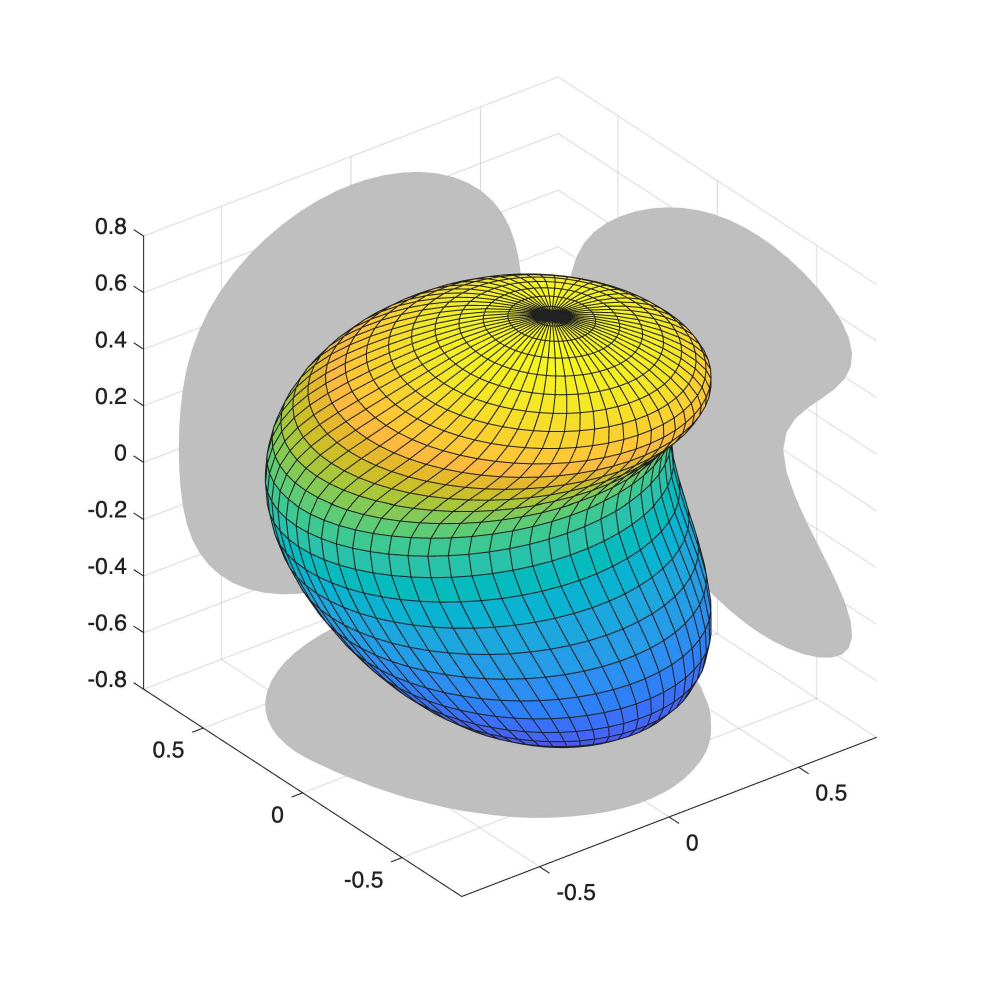}} &
        {\includegraphics[width=0.21\textwidth]{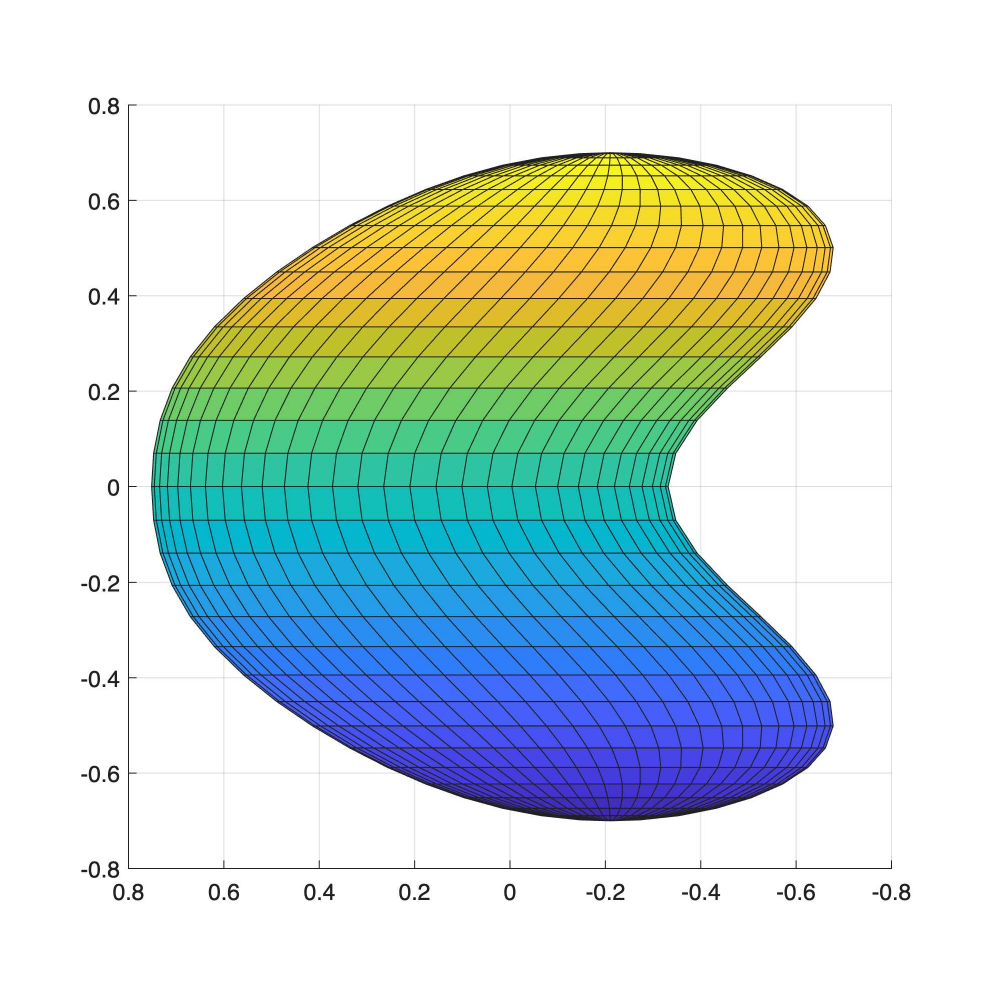}} &
        {\includegraphics[width=0.21\textwidth]{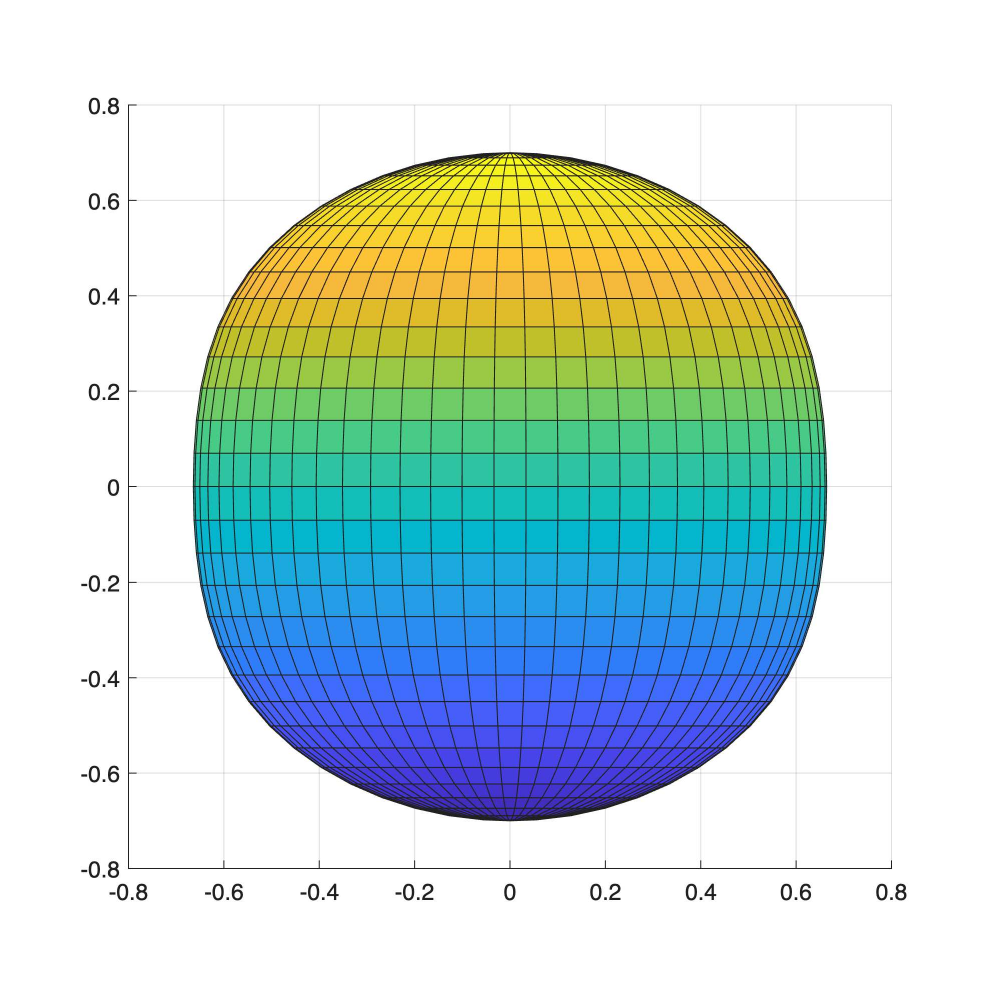}} &
        {\includegraphics[width=0.21\textwidth]{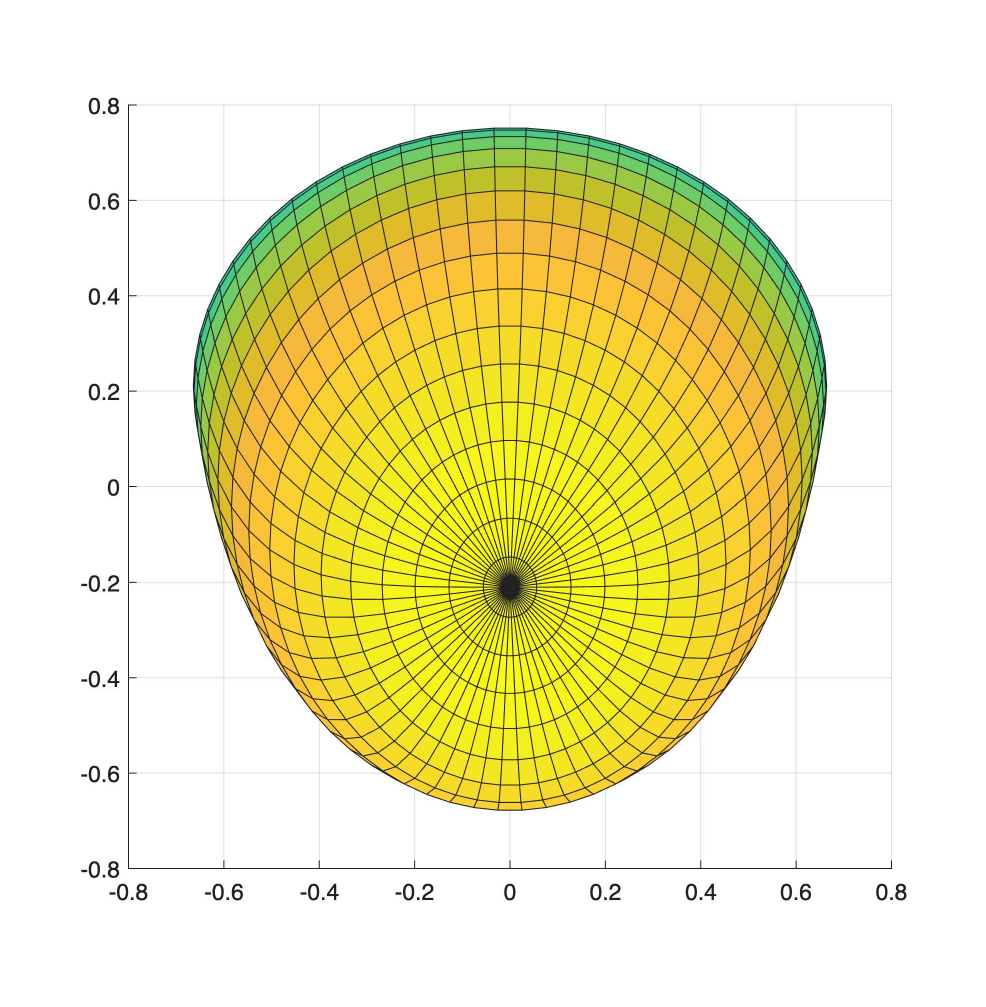}}
    \end{tabular}
}
\subfigure[Reconstructions from two point sources with $\epsilon=0.008$.]{
    \begin{tabular}{cccc}
        {\includegraphics[width=0.21\textwidth]{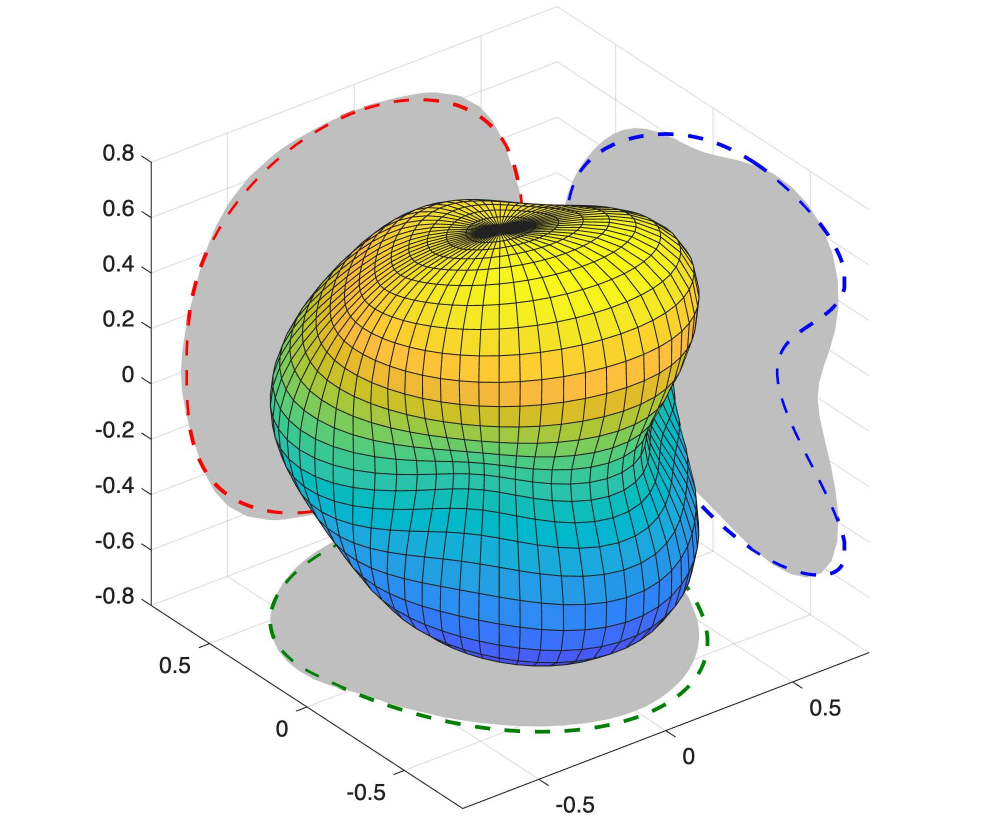}} &
        {\includegraphics[width=0.21\textwidth]{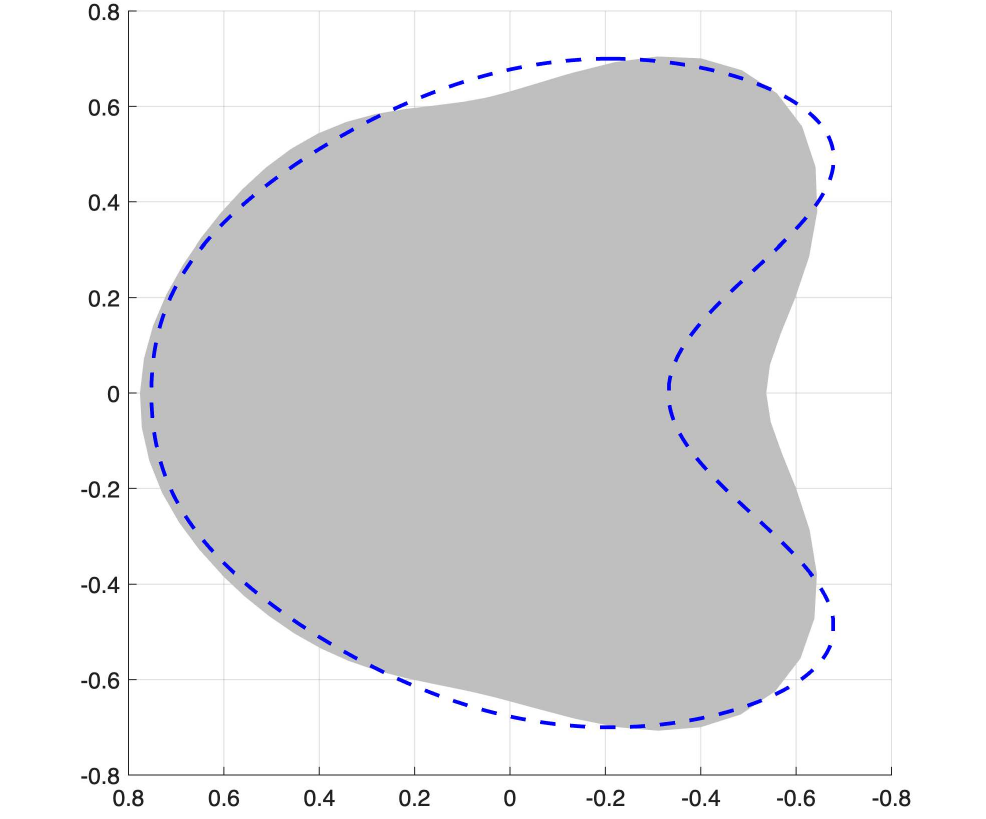}} &
        {\includegraphics[width=0.21\textwidth]{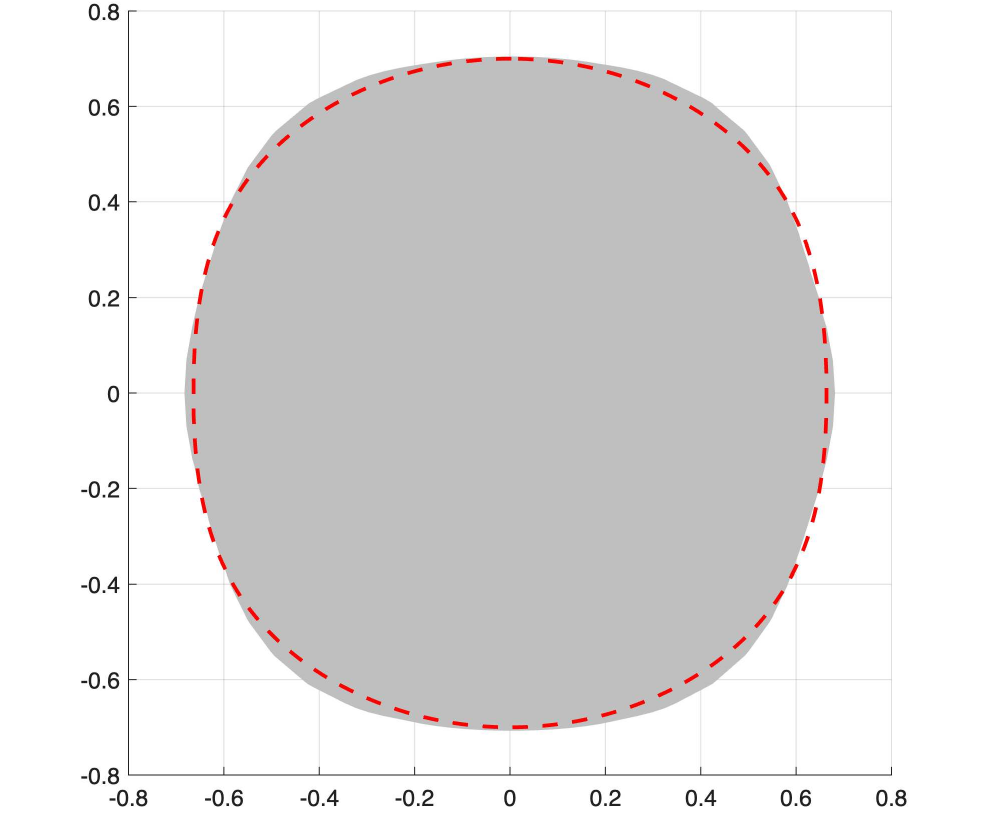}} &
        {\includegraphics[width=0.21\textwidth]{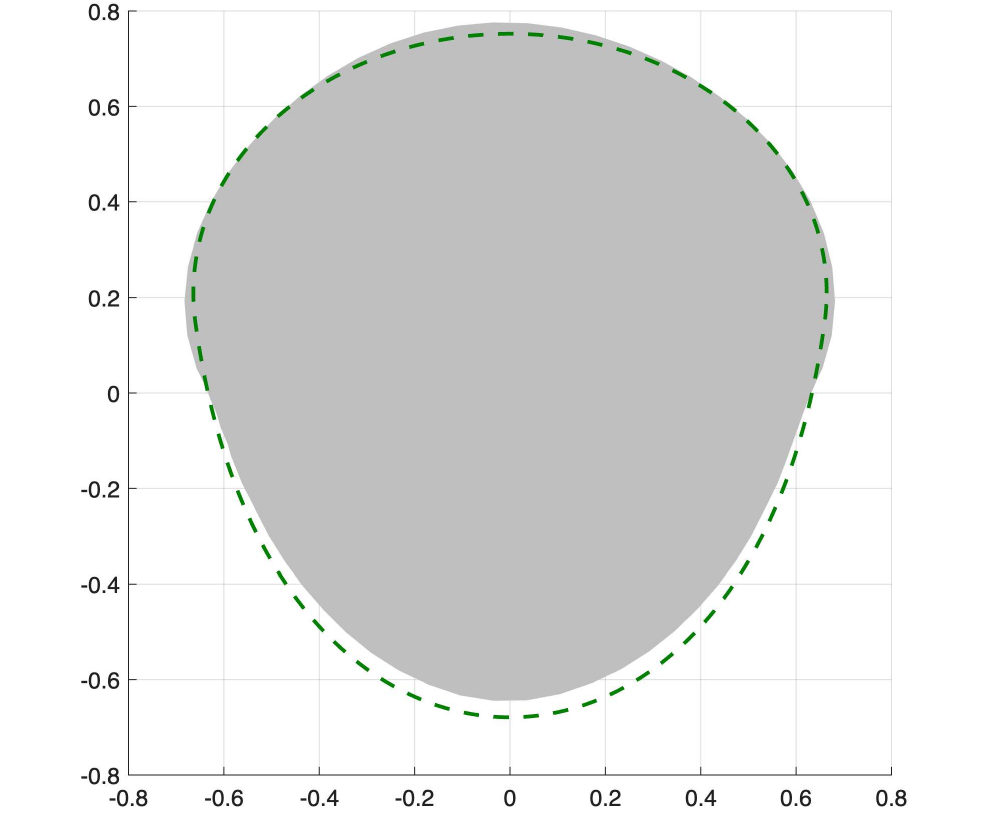}}
    \end{tabular}
}
\subfigure[Reconstructions from four point sources with $\epsilon=0.008$.]{
    \begin{tabular}{cccc}
        {\includegraphics[width=0.21\textwidth]{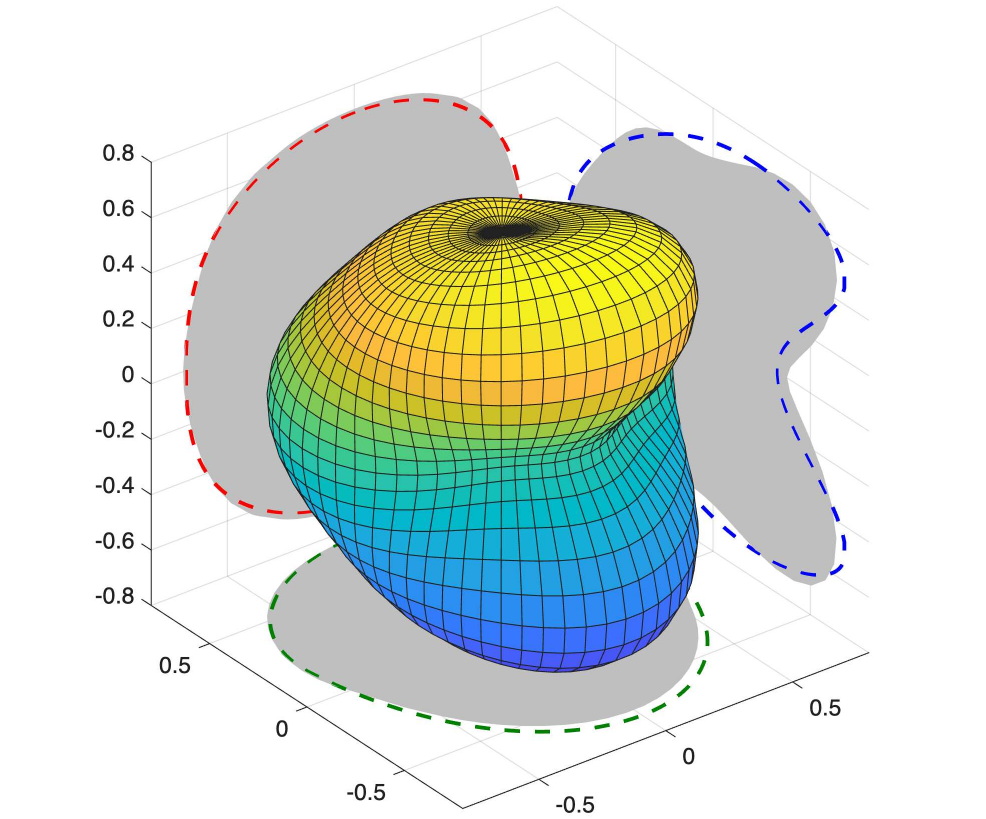}} &
        {\includegraphics[width=0.21\textwidth]{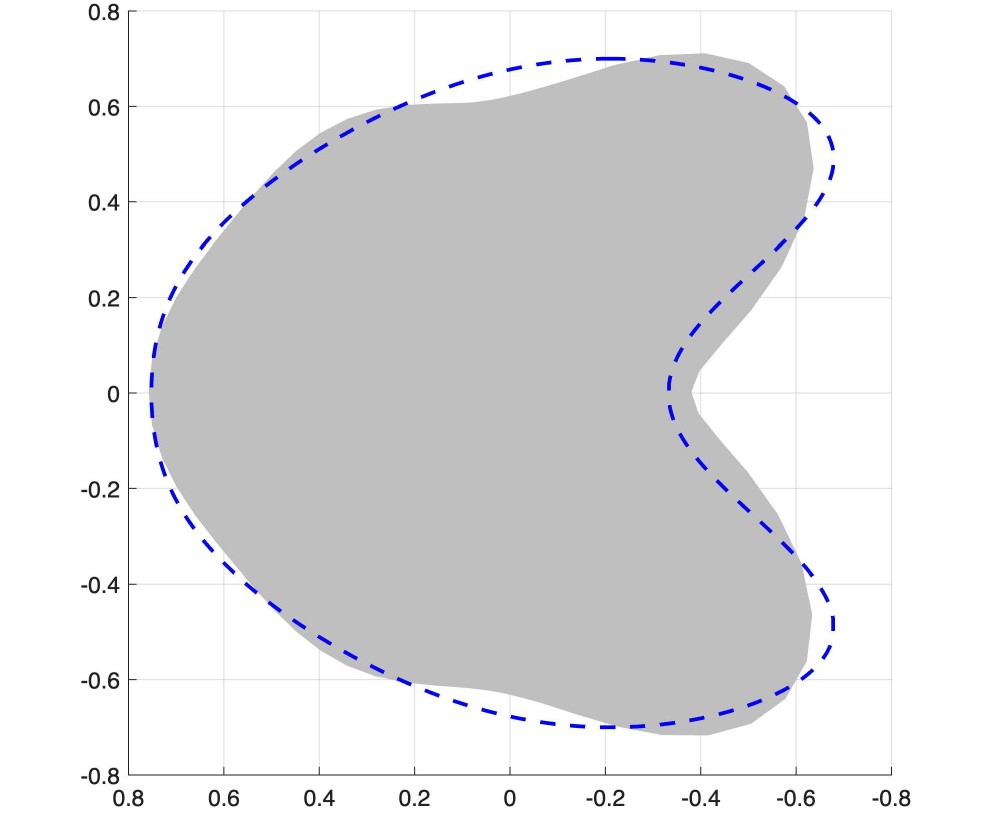}} &
        {\includegraphics[width=0.21\textwidth]{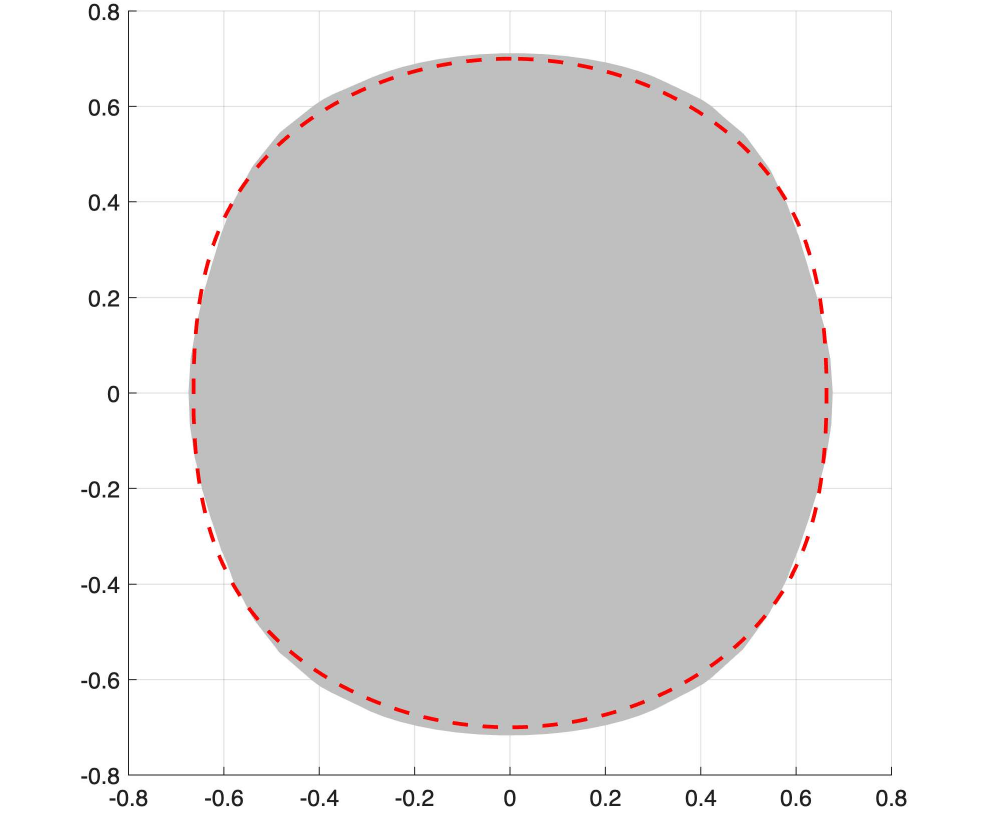}} &
        {\includegraphics[width=0.21\textwidth]{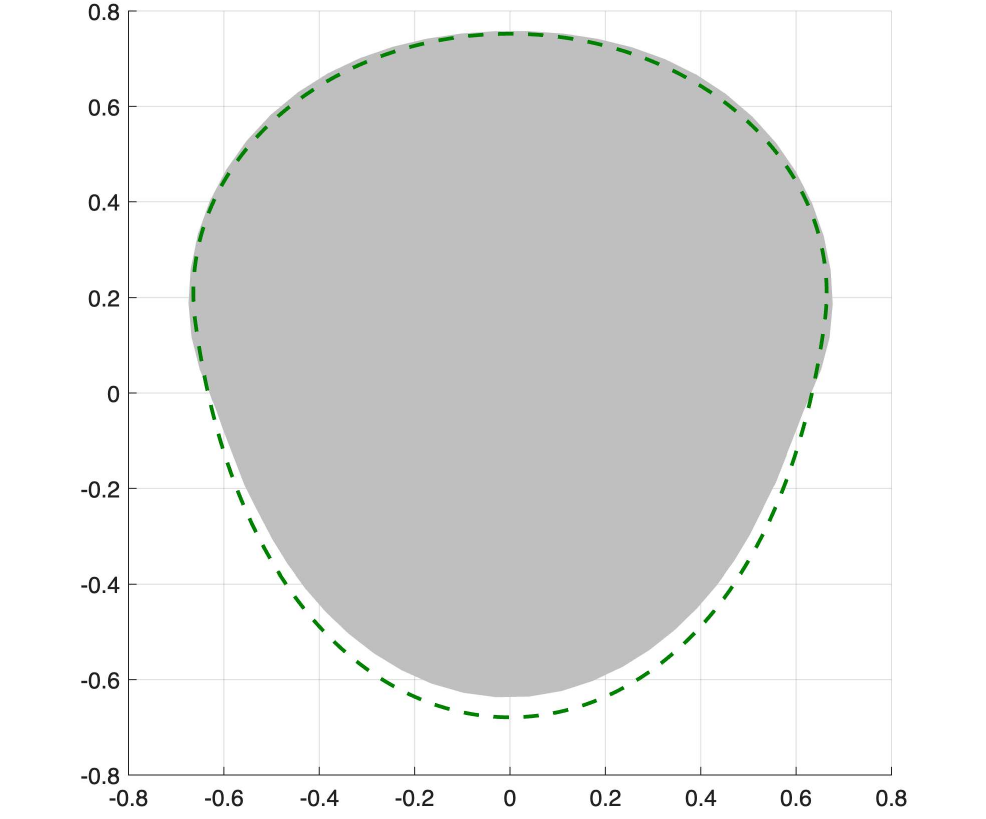}}
    \end{tabular}
}
\caption{Reconstruction of a bean-like surface obstacle under the Dirichlet boundary condition from phaseless far-field pattern with $1\%$ noise. (a) shows the true shape; (b) shows the reconstructions from two point sources located at $(4,0,0)^\top$ and $(-4,0,0)^\top$; (c) shows the reconstructions from four point sources located at $(\pm 4,0,0)^\top$ and $(0,\pm 4,0)^\top$. The initial guess is a sphere with $\pmb c^{(0)}=(0.1,-0.5,0.1)^{\top}$ and $r^{(0)}=0.3$, $\kappa=5$.}\label{fig_ex5}
\end{figure}

\vspace{2ex}
{\noindent\bf Example 6: Reconstructions from limited-aperture phaseless data under the Dirichlet boundary condition.}
\vspace{1ex}

To examine the performance of the proposed approach under incomplete angular information, we consider the reconstruction of the cushion-shaped obstacle under the Dirichlet boundary condition using phaseless far-field data measured only on a limited observation aperture.  
We consider limited-aperture measurements of the form
\[
\Gamma_{\mathrm{obs}}
=\big\{ (\theta,\phi):\ \theta\in I_\theta,\ \phi\in I_\phi \big\},
\]
where the intervals \(I_\theta\) and \(I_\phi\) specify the accessible angular ranges.  

Figure~\ref{fig_ex6} displays the reconstruction obtained from the limited-aperture phaseless data.  
Although the missing angular information leads to partial shadowing effects on the side opposite to the measurement aperture, 
the reconstruction still captures the global geometry and major features of the cushion-shaped obstacle.  
These results show that the proposed method remains stable and effective even when only limited-aperture phaseless measurements are available.

\begin{figure}[!htbp]
\centering 

\subfigure[ Reconstruction using limited-aperture data 1: \texorpdfstring{$I_\theta=[0,\pi/2]$}{I_theta=[0,pi/2]}, \texorpdfstring{$I_\phi=[0,2\pi]$}{I_theta=[0,2pi]} ]
{
    \begin{tabular}{cccc}
        {\includegraphics[width=0.21\textwidth]{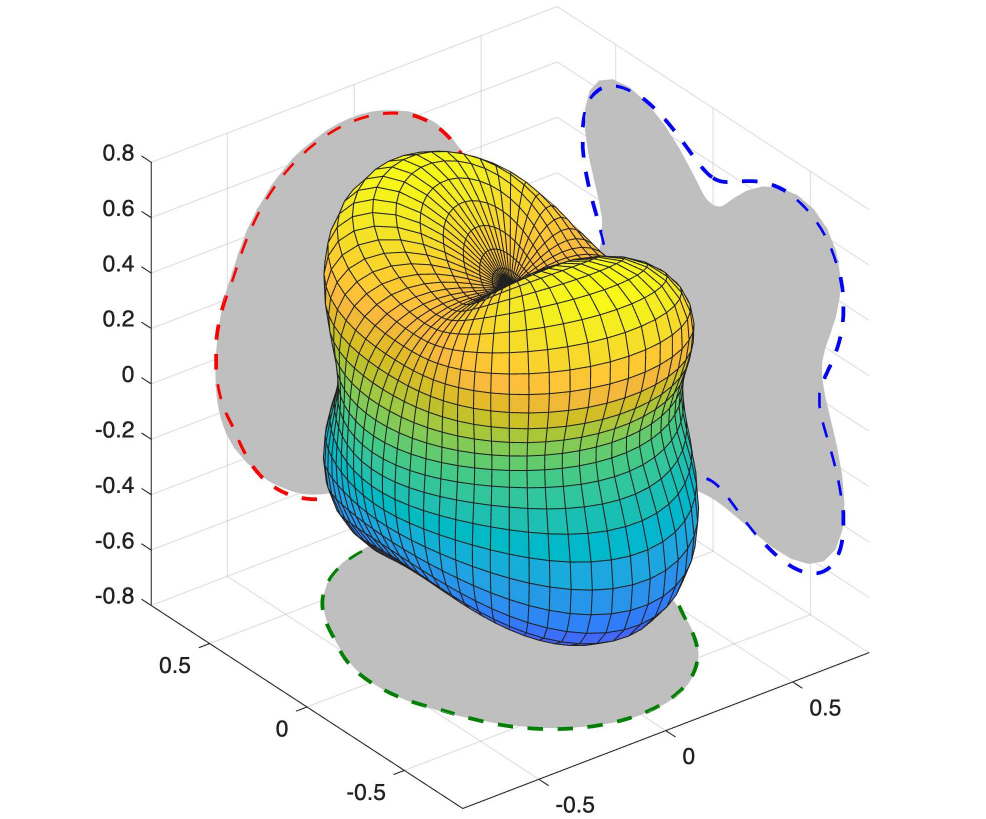}} &
        {\includegraphics[width=0.21\textwidth]{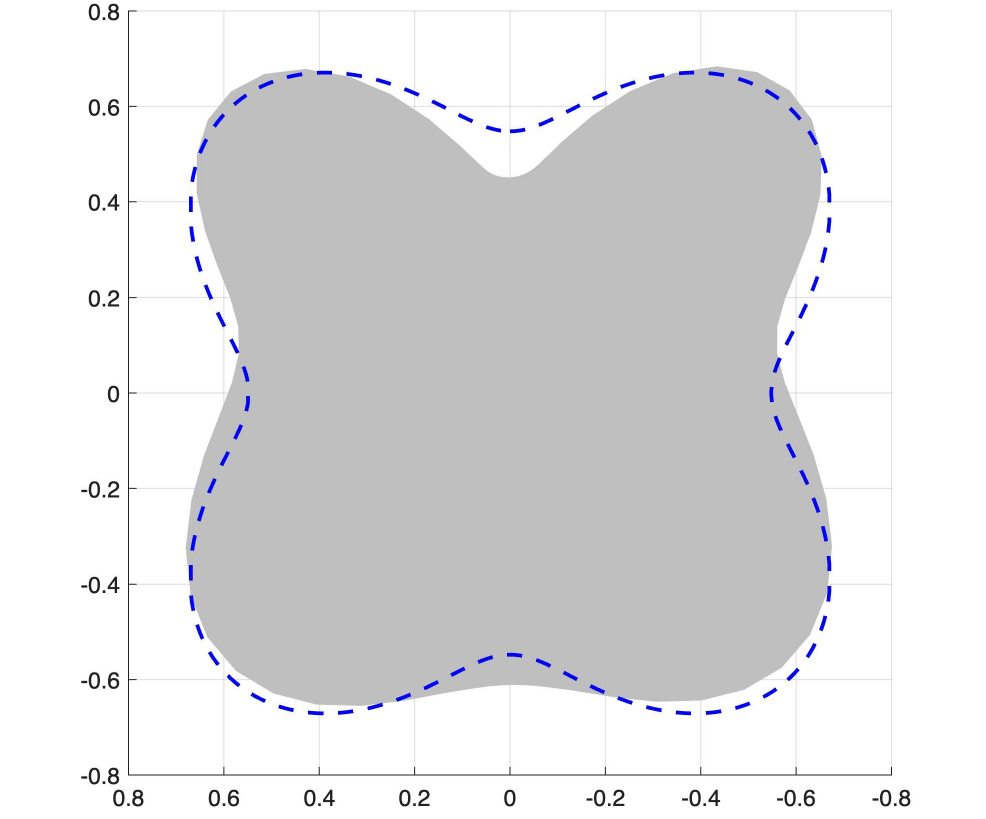}} &
        {\includegraphics[width=0.21\textwidth]{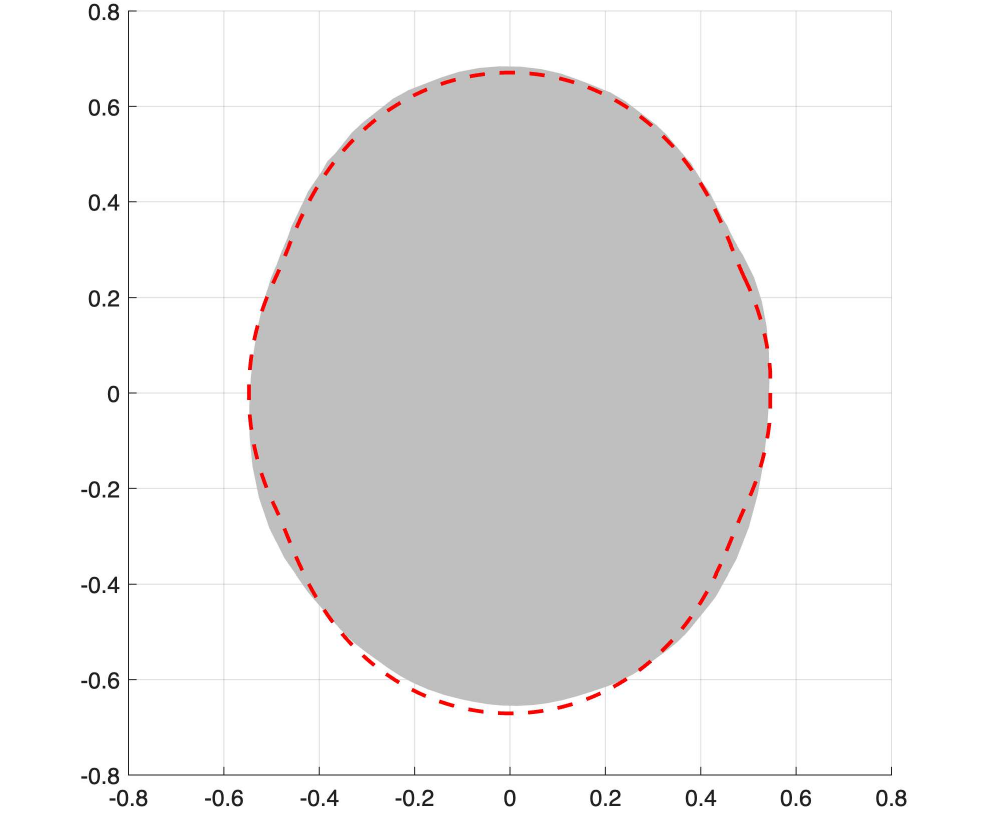}} &
        {\includegraphics[width=0.21\textwidth]{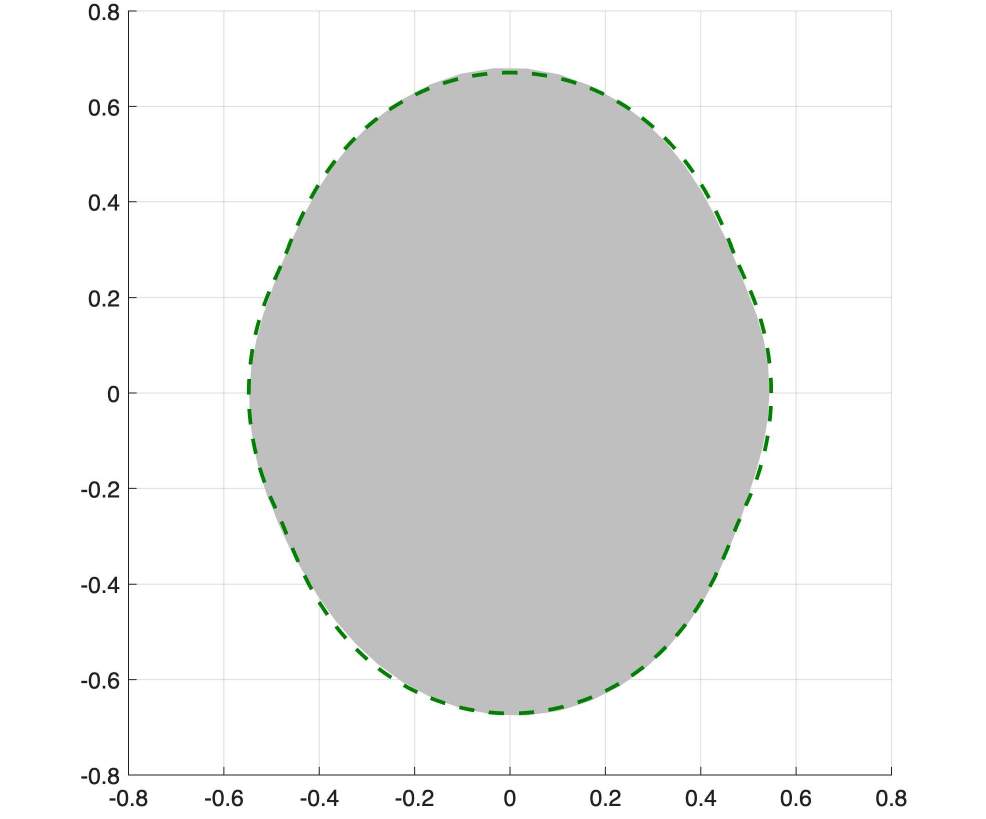}}
    \end{tabular} 
}\\[-1.2ex]

\subfigure[ Reconstruction using limited-aperture data 2: \texorpdfstring{$I_\theta=[0,\pi]$}{I_theta=[0,pi]}, \texorpdfstring{$I_\phi=[\pi/2,3\pi/2]$}{I_theta=[\pi/2,3pi/2]} ]
{
    \begin{tabular}{cccc}
        {\includegraphics[width=0.21\textwidth]{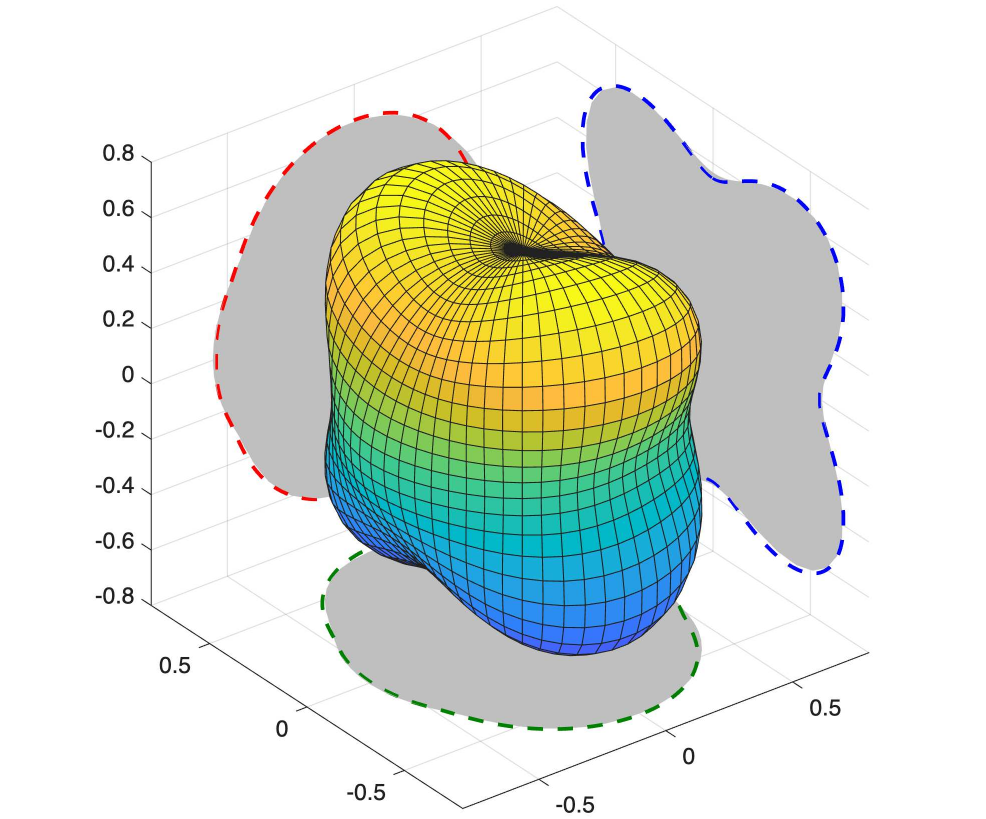}} &
        {\includegraphics[width=0.21\textwidth]{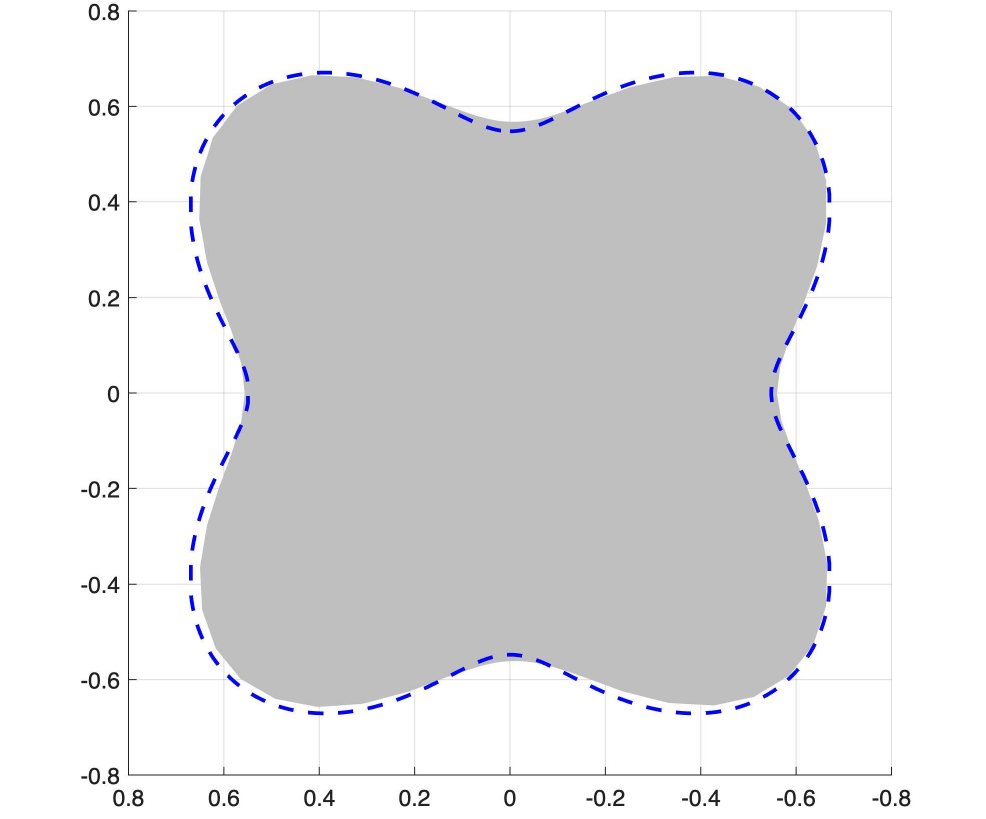}} &
        {\includegraphics[width=0.21\textwidth]{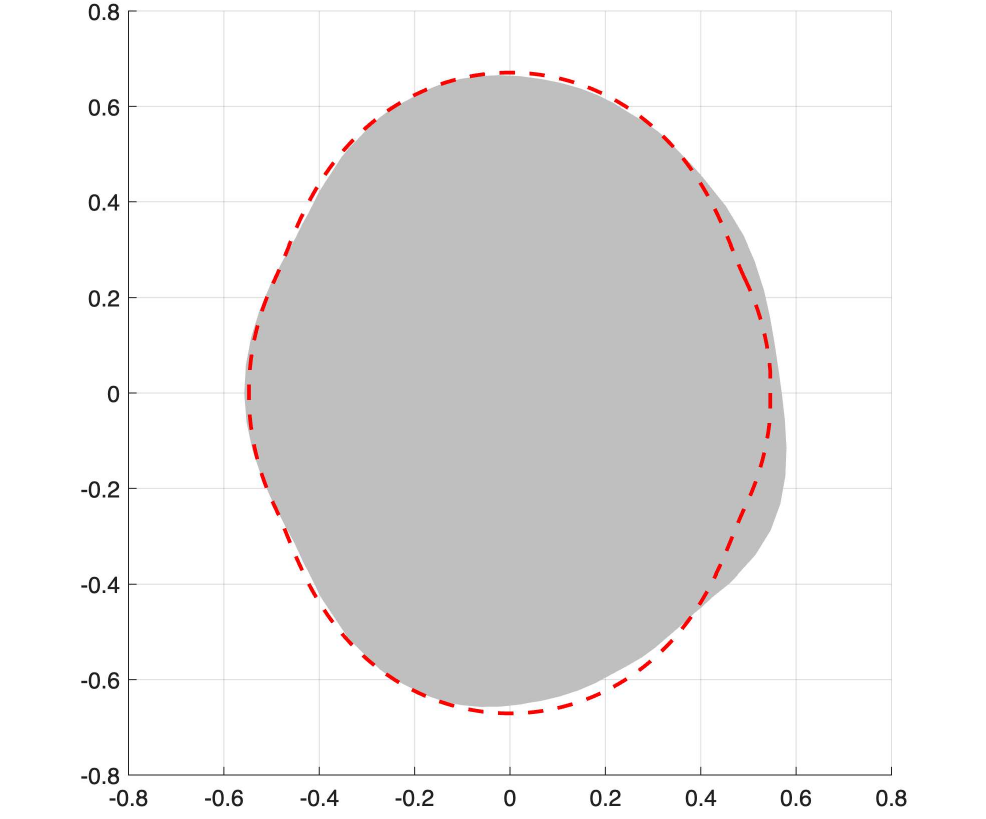}} &
        {\includegraphics[width=0.21\textwidth]{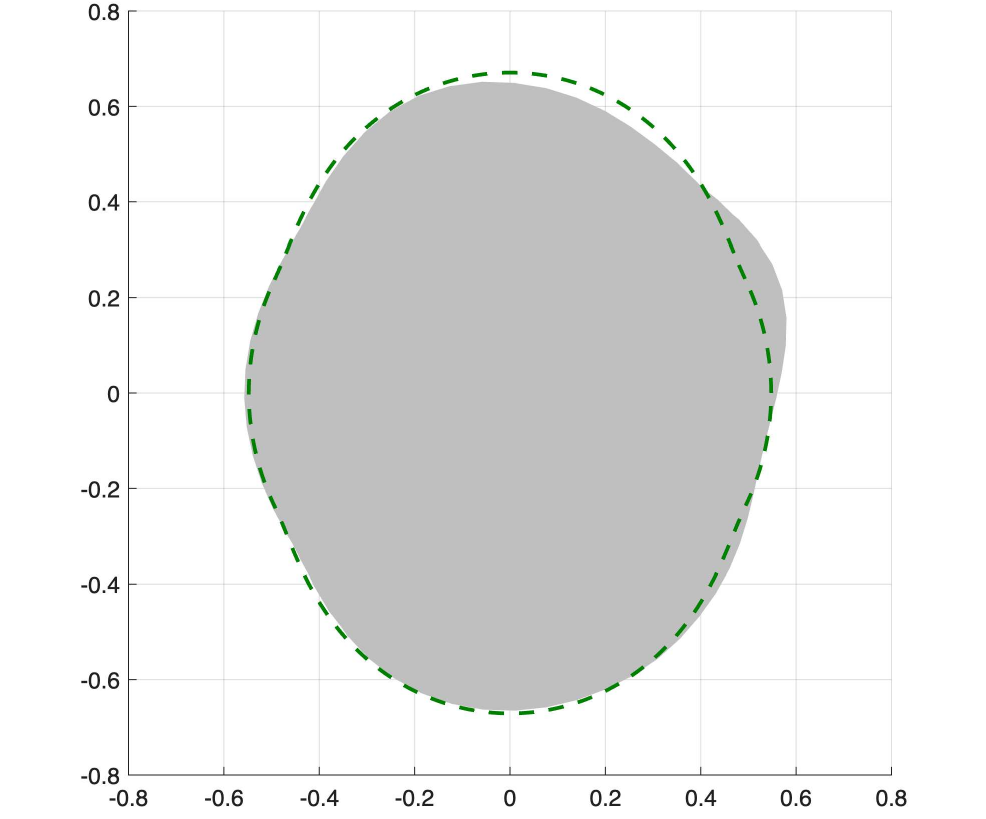}}
    \end{tabular} 
}\\[-1.2ex]

\subfigure[ Reconstruction using limited-aperture data 3: \texorpdfstring{$I_\theta=[\pi/4,3\pi/4]$}{I_theta=[\pi/4,3\pi/4]}, \texorpdfstring{$I_\phi=[\pi/2,3\pi/2]$}{I_theta=[\pi/2,3pi/2]} ]
{
    \begin{tabular}{cccc}
        {\includegraphics[width=0.21\textwidth]{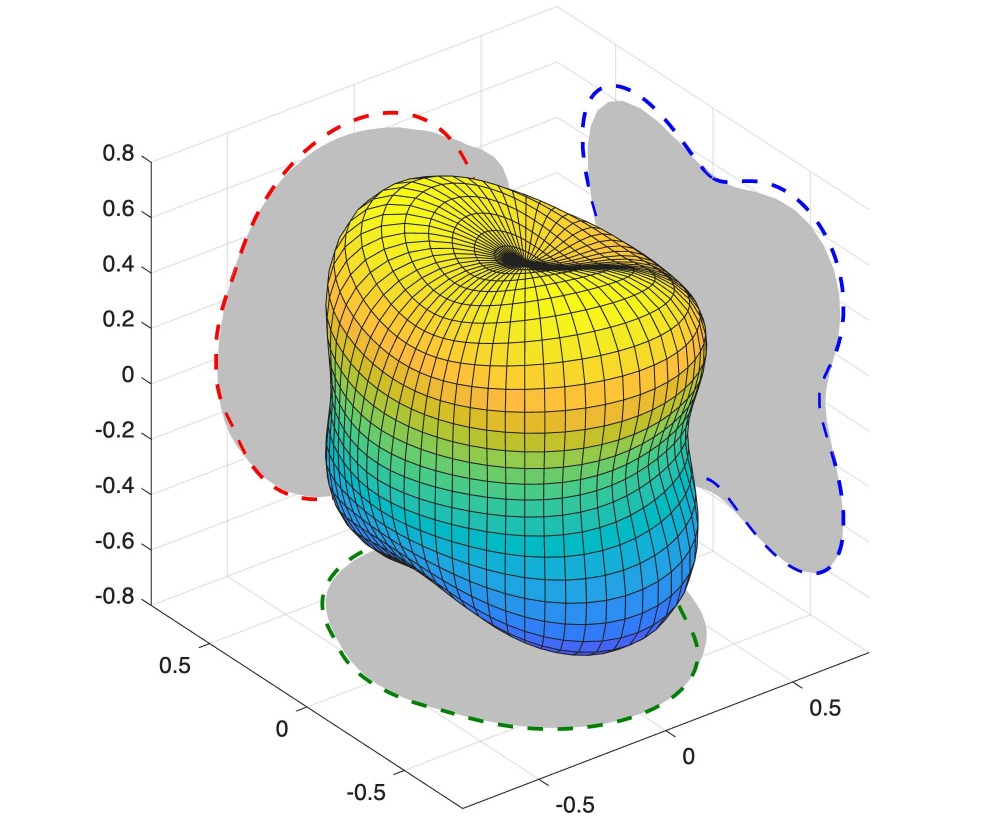}} &
        {\includegraphics[width=0.21\textwidth]{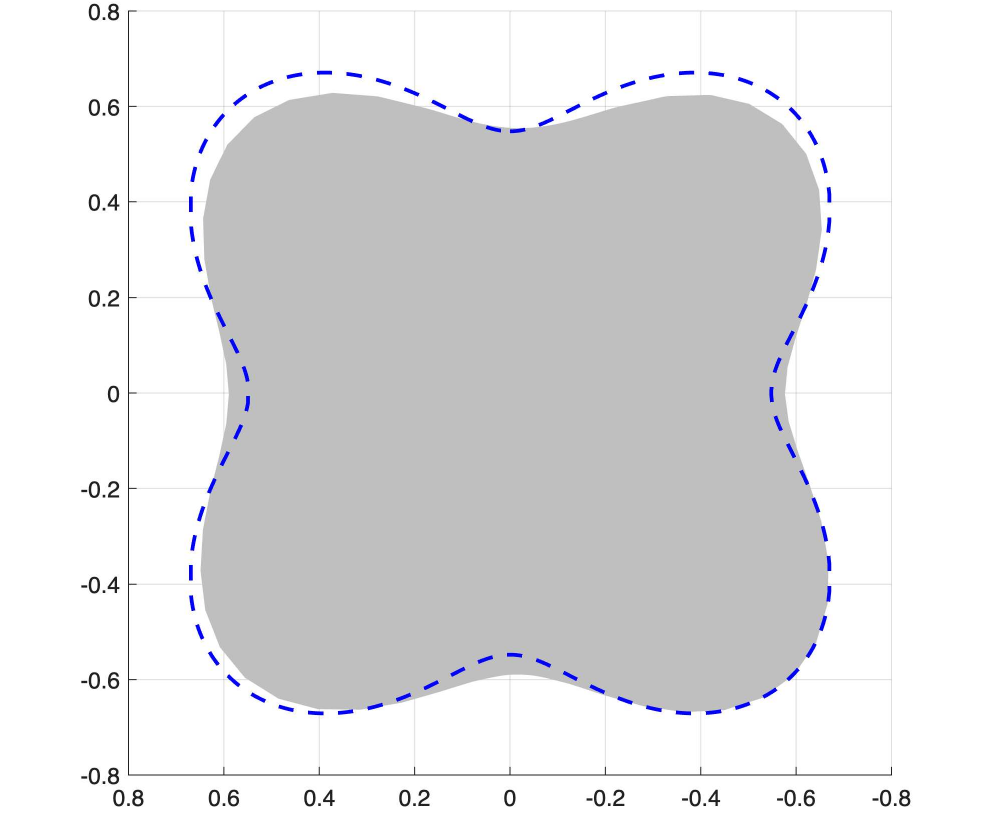}} &
        {\includegraphics[width=0.21\textwidth]{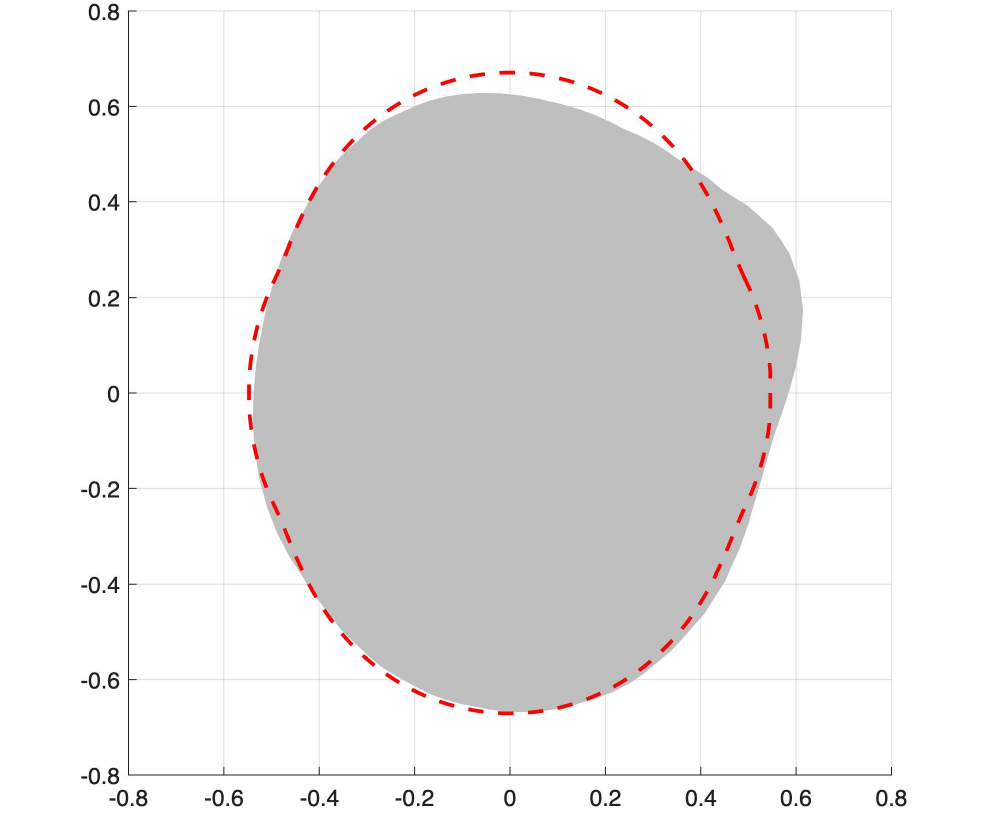}} &
        {\includegraphics[width=0.21\textwidth]{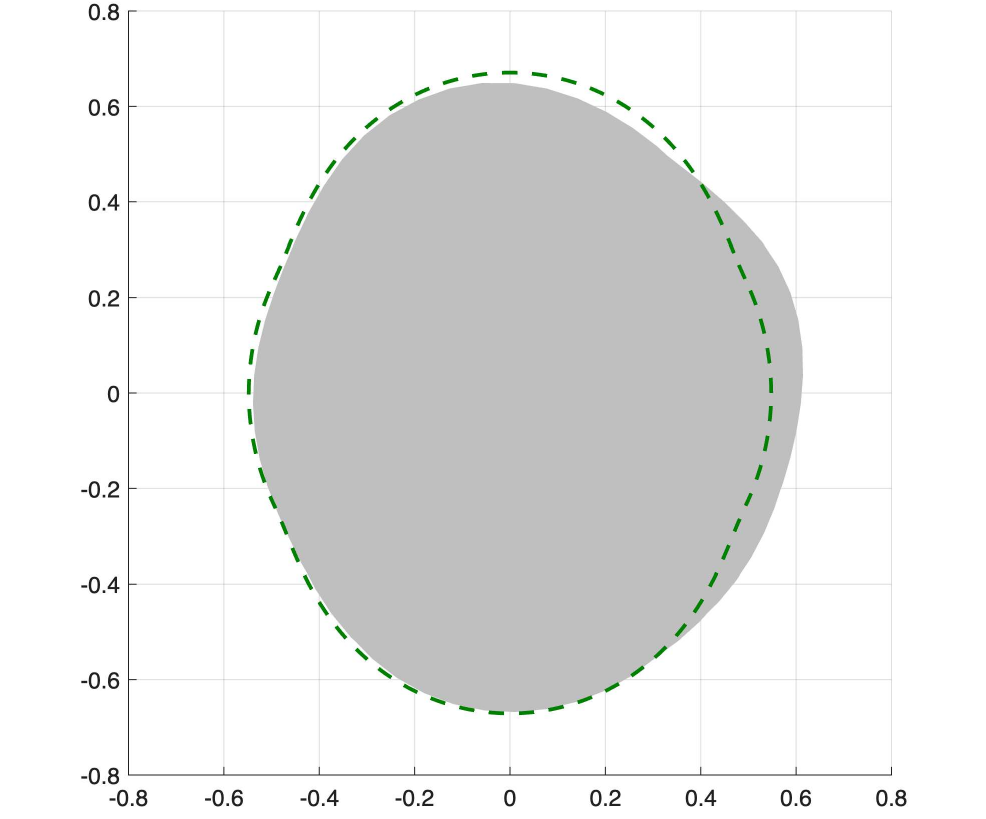}}
    \end{tabular} 
}\\[-1.2ex]
\caption{Reconstruction of a cushion-shaped obstacle under the Dirichlet boundary condition from limited-aperture phaseless far-field data generated by two point sources located at $(4,0,0)^\top$ and $(-4,0,0)^\top$ with $1\%$ noise. $\pmb c^{(0)}=(0.2,-0.3,-0.1)^{\top}$, $r^{(0)}=0.4$, $\kappa=4.5$, $\epsilon=(0.005,0.006,0.008)$.}\label{fig_ex6}
\end{figure}

\begin{figure}[!htbp]
\centering 

\subfigure[True shape of complex obstacle.]{
    \begin{tabular}{cccc}
        {\includegraphics[width=0.21\textwidth]{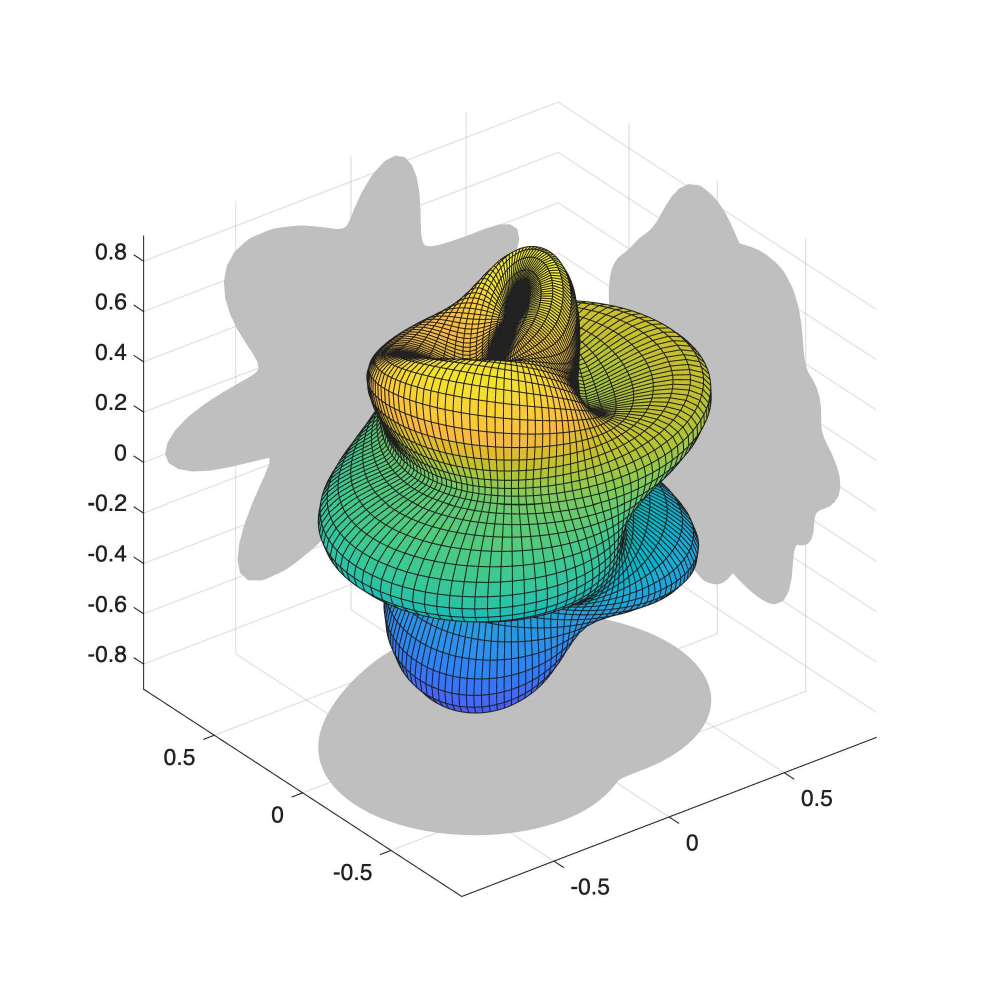}} &
        {\includegraphics[width=0.21\textwidth]{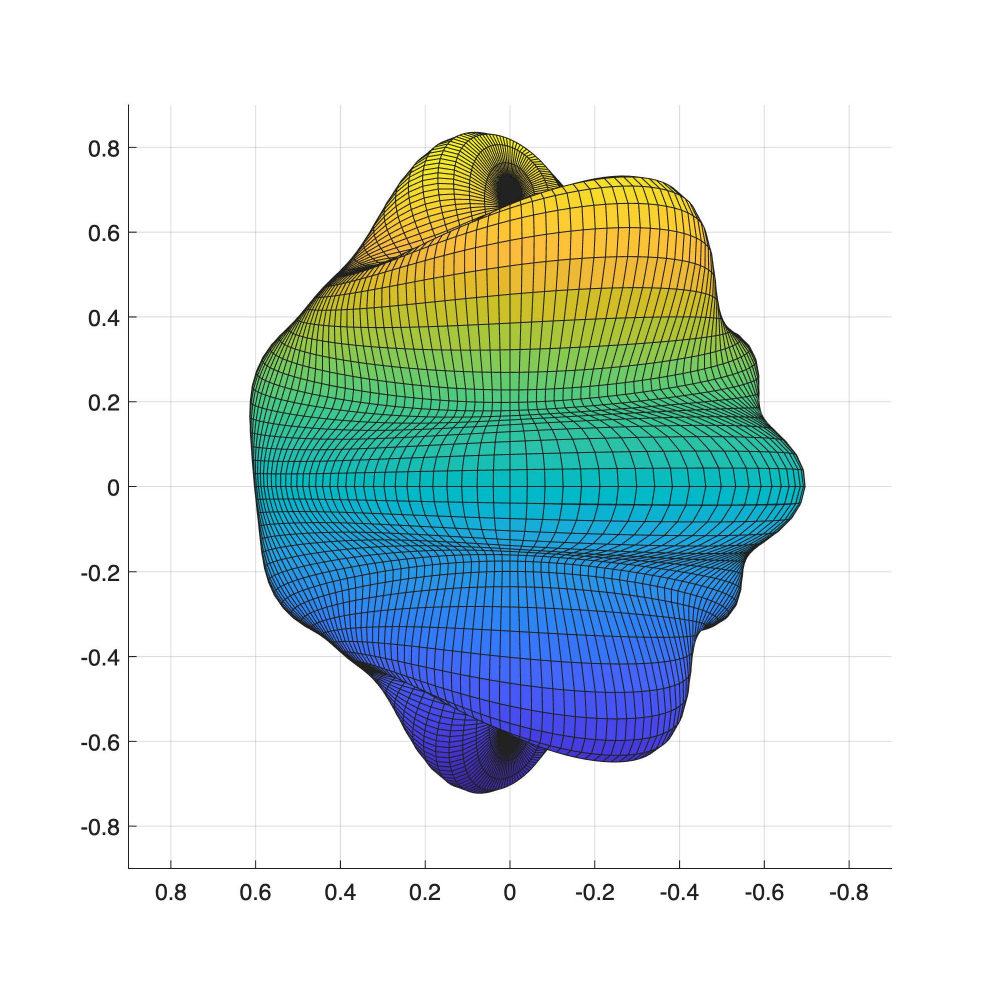}} &
        {\includegraphics[width=0.21\textwidth]{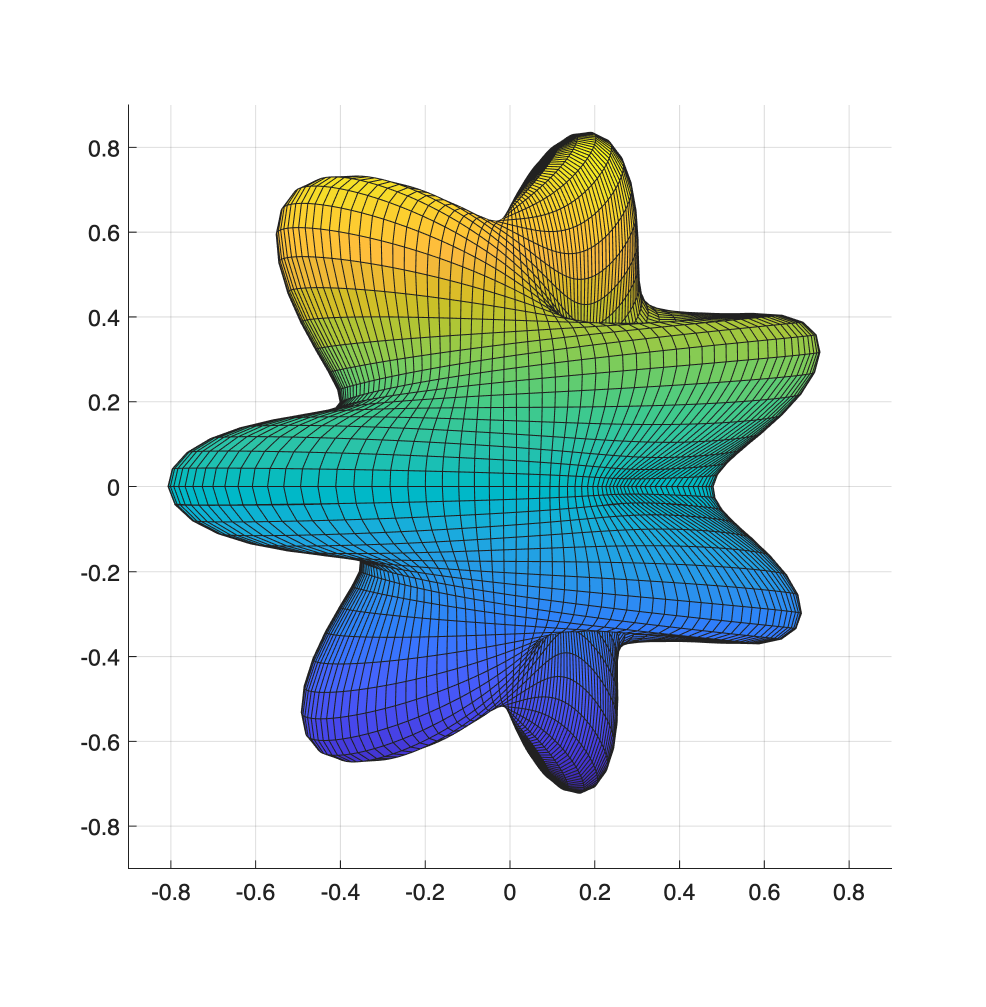}} &
        {\includegraphics[width=0.21\textwidth]{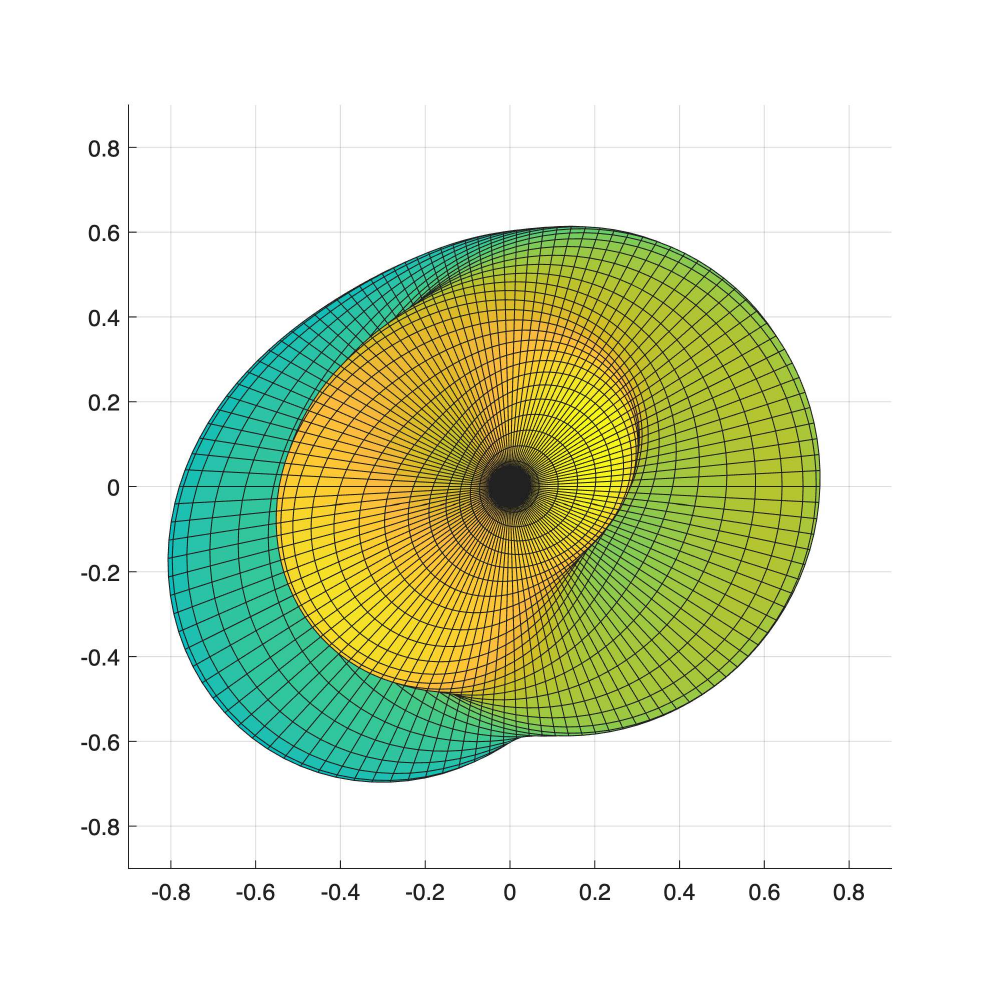}}
    \end{tabular}
}\\[-1.2ex]

\subfigure[Reconstruction with $1\%$ noise. $\epsilon=0.004$.]{
    \begin{tabular}{cccc}
        {\includegraphics[width=0.21\textwidth]{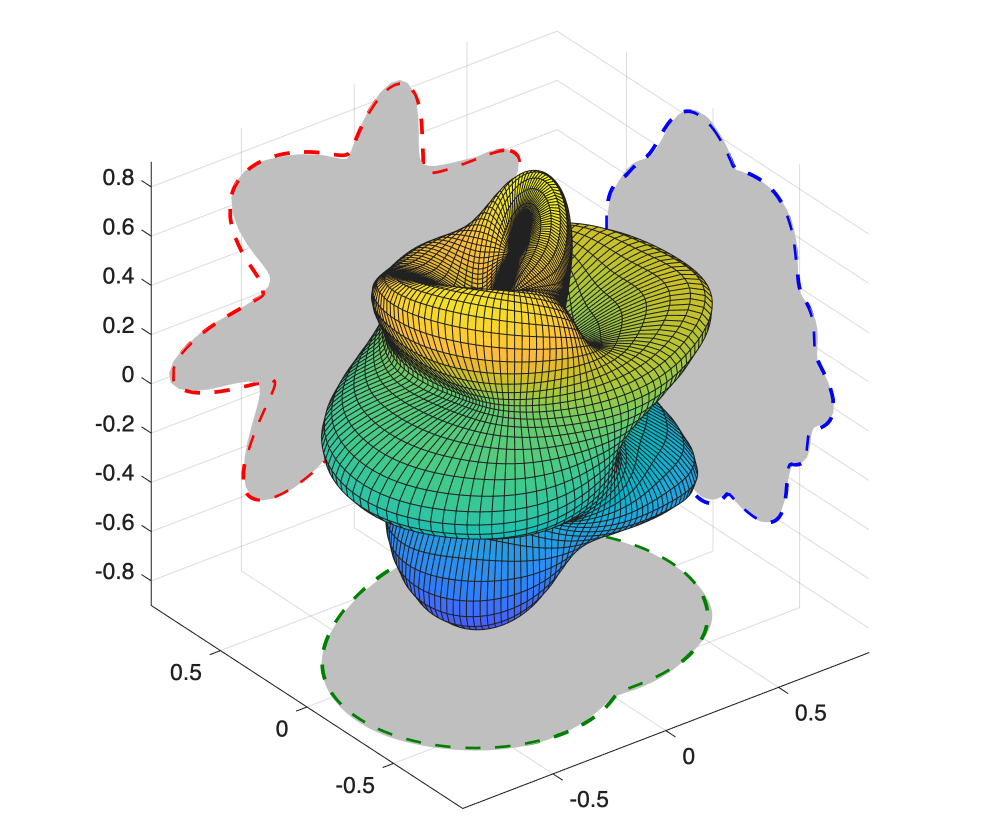}} &
        {\includegraphics[width=0.21\textwidth]{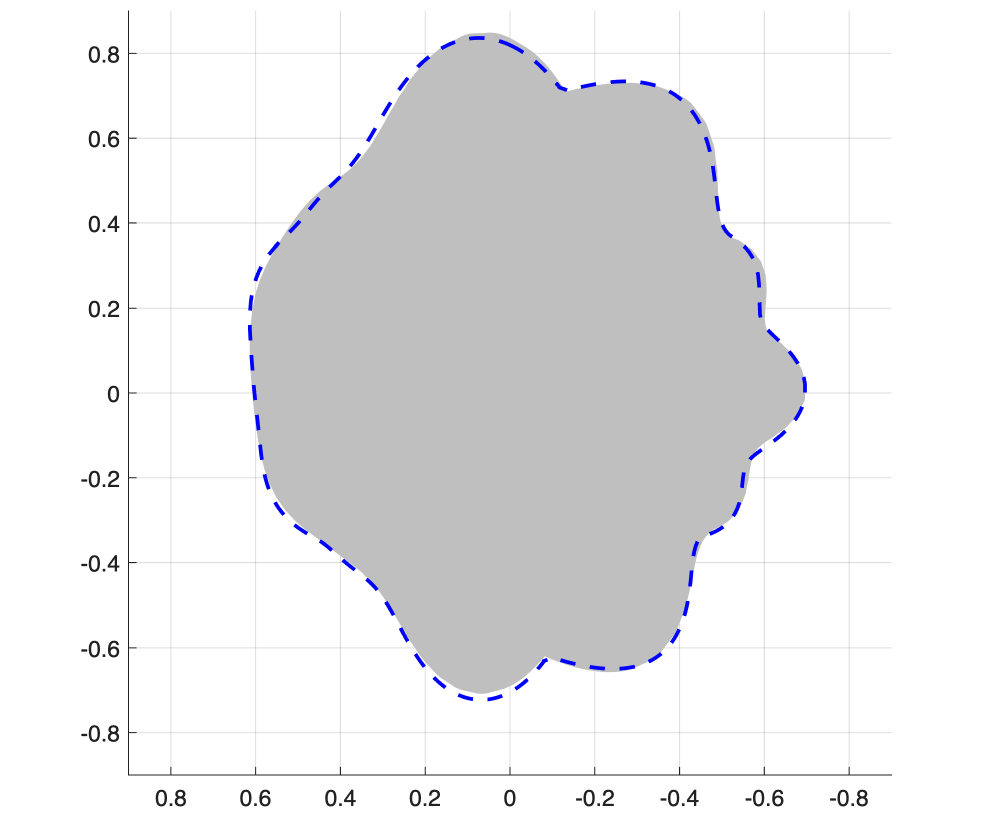}} &
        {\includegraphics[width=0.21\textwidth]{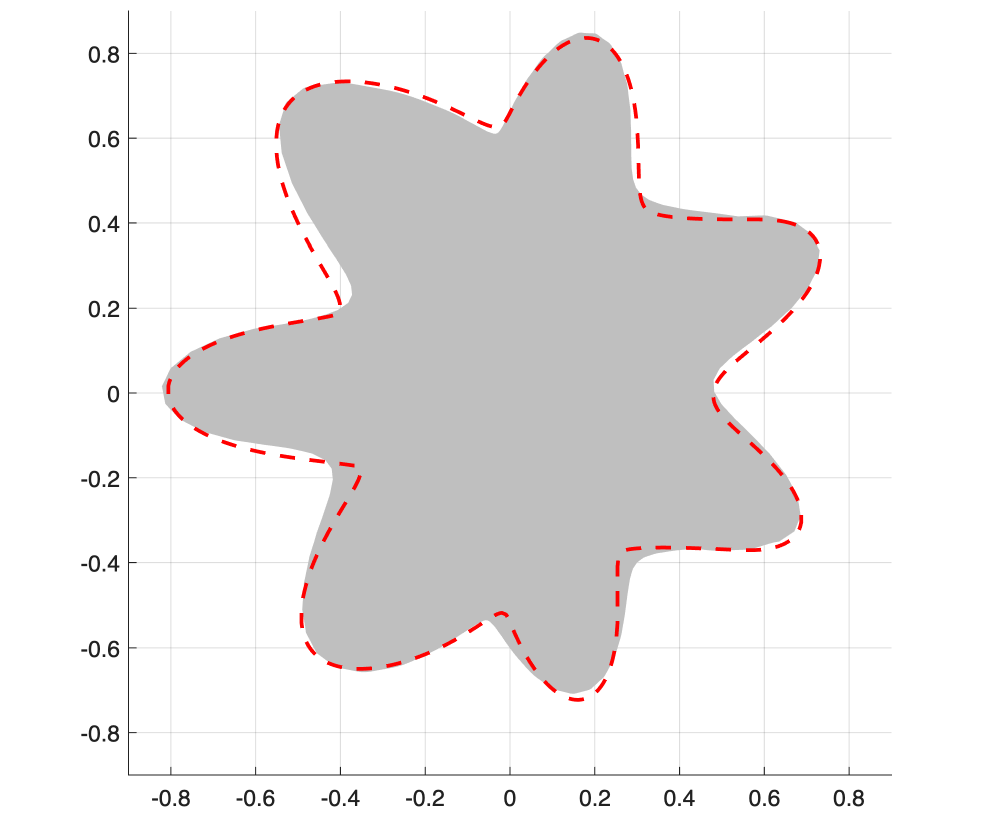}} &
        {\includegraphics[width=0.21\textwidth]{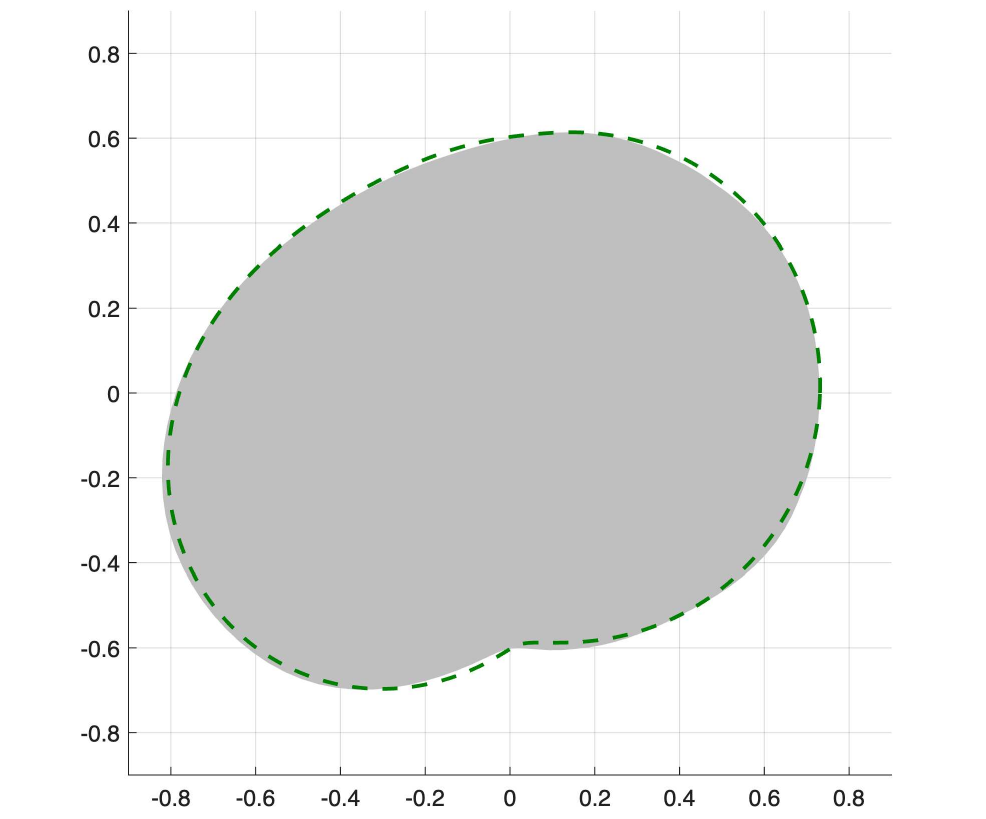}}
    \end{tabular} 
}\\[-1.2ex]

\subfigure[Reconstruction with $5\%$ noise. $\epsilon=0.015$.]{
    \begin{tabular}{cccc}
        {\includegraphics[width=0.21\textwidth]{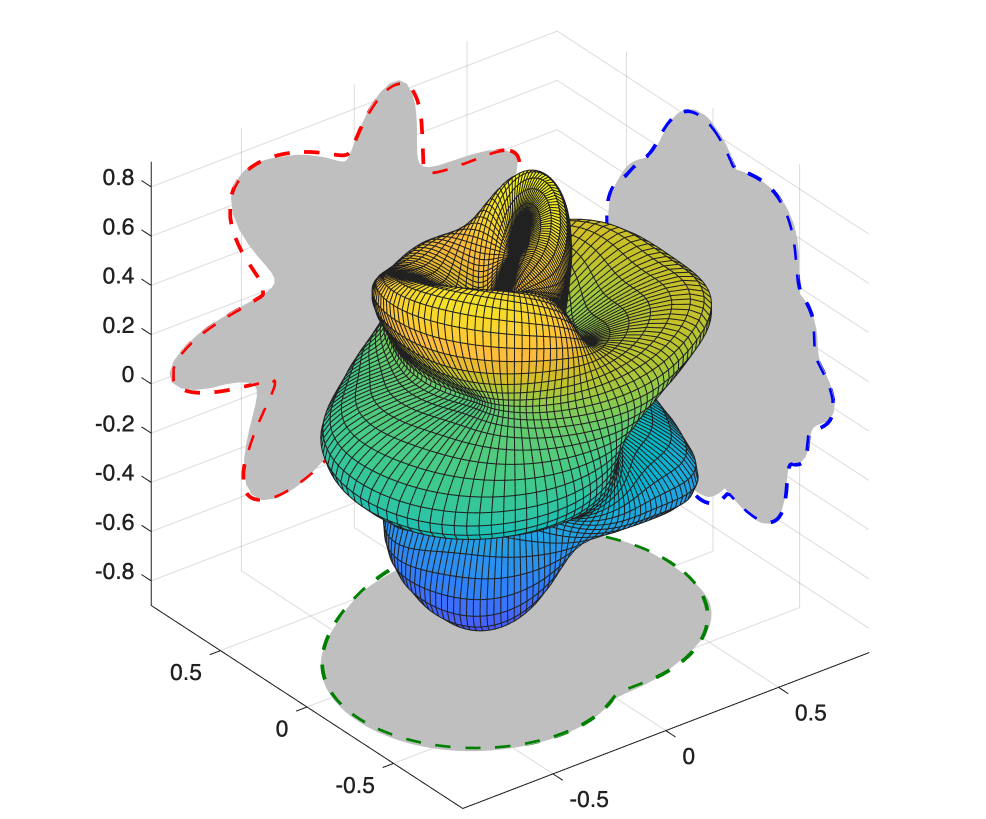}} &
        {\includegraphics[width=0.21\textwidth]{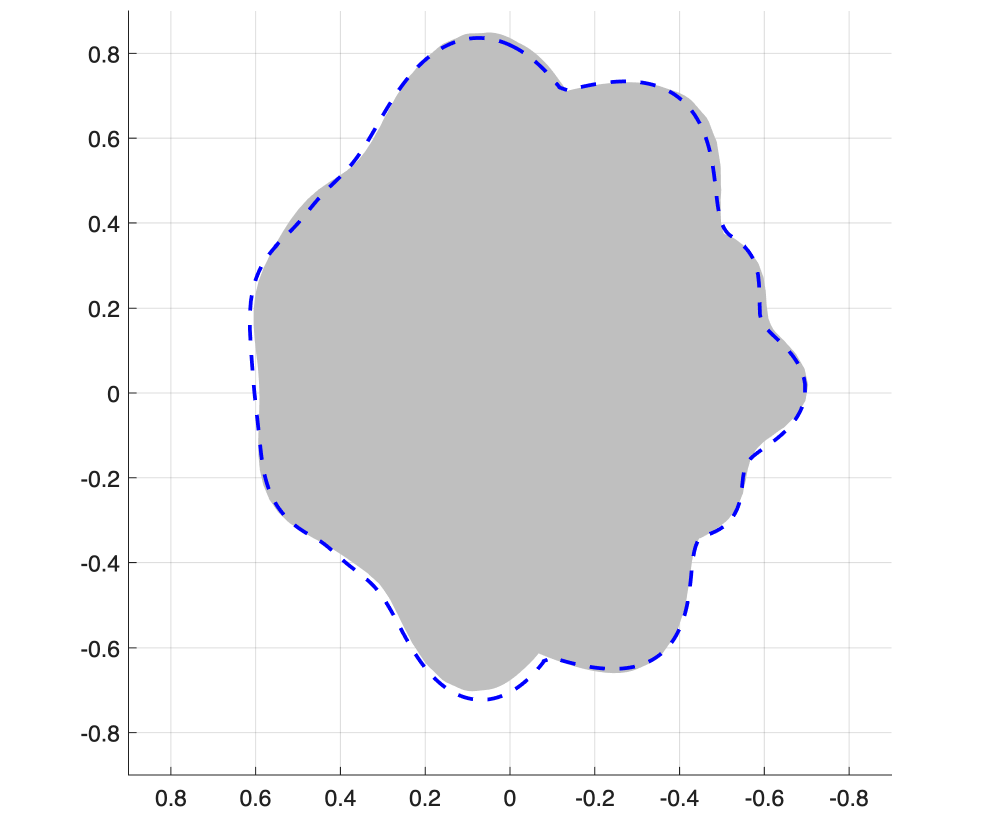}} &
        {\includegraphics[width=0.21\textwidth]{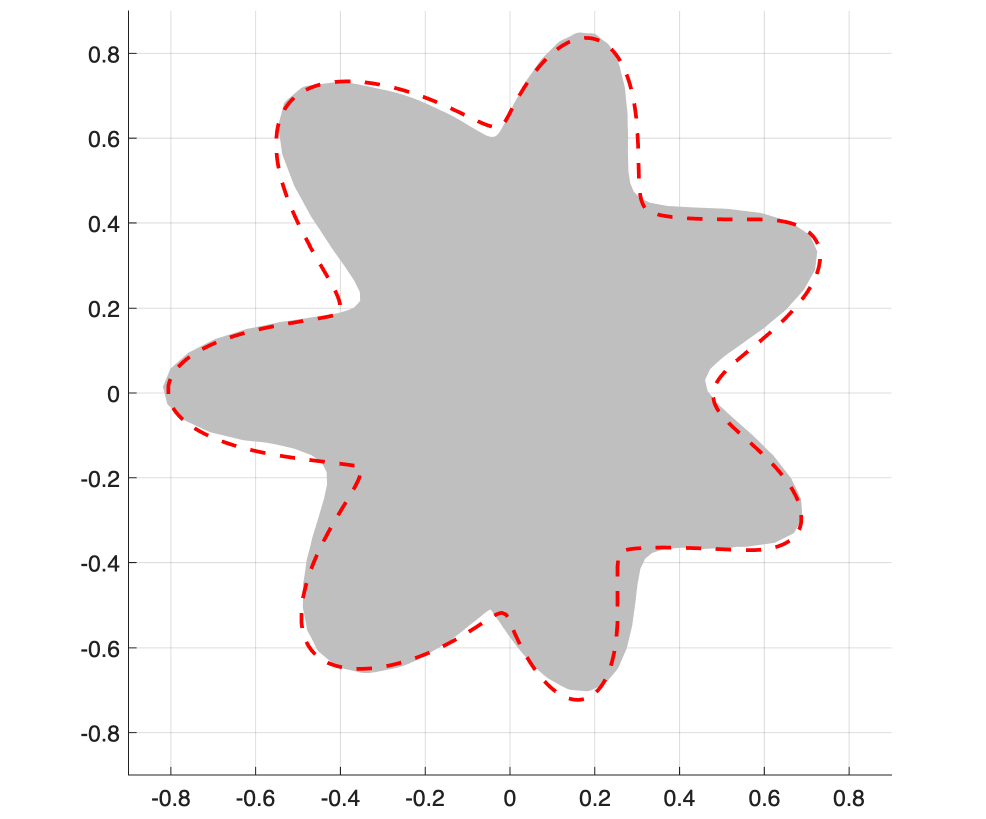}} &
        {\includegraphics[width=0.21\textwidth]{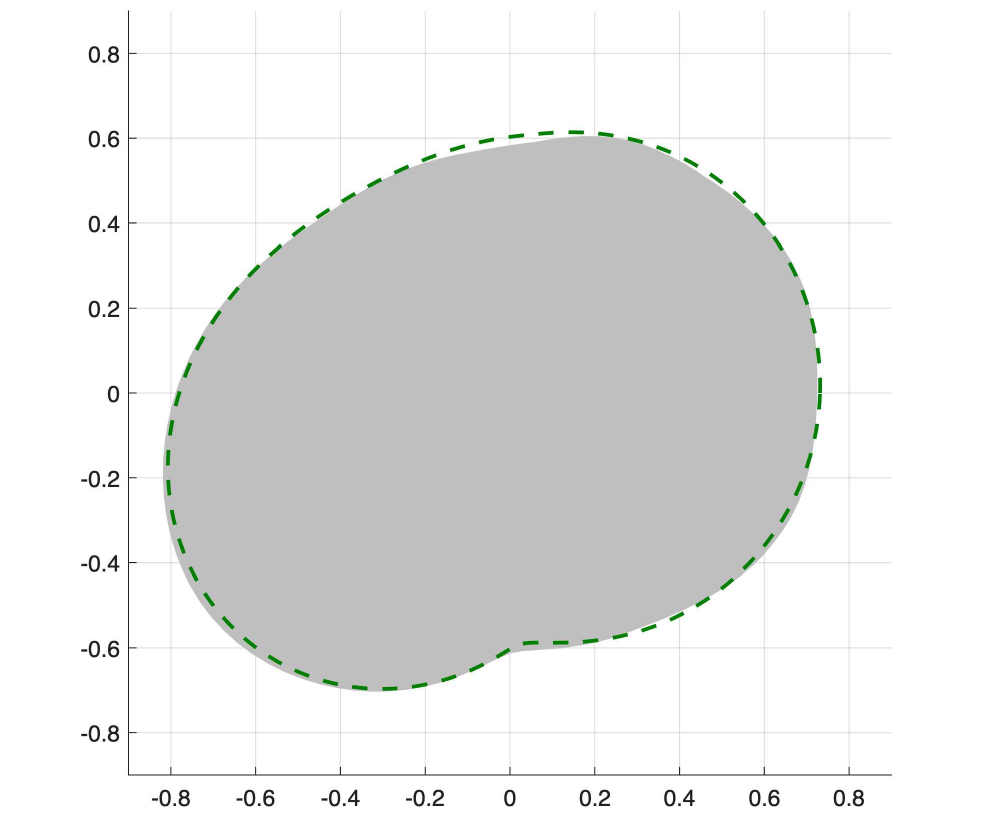}}
    \end{tabular} 
}\\[-1.2ex]

\subfigure[Reconstruction with $10\%$ noise. $\epsilon=0.03$.]{
    \begin{tabular}{cccc}
        {\includegraphics[width=0.21\textwidth]{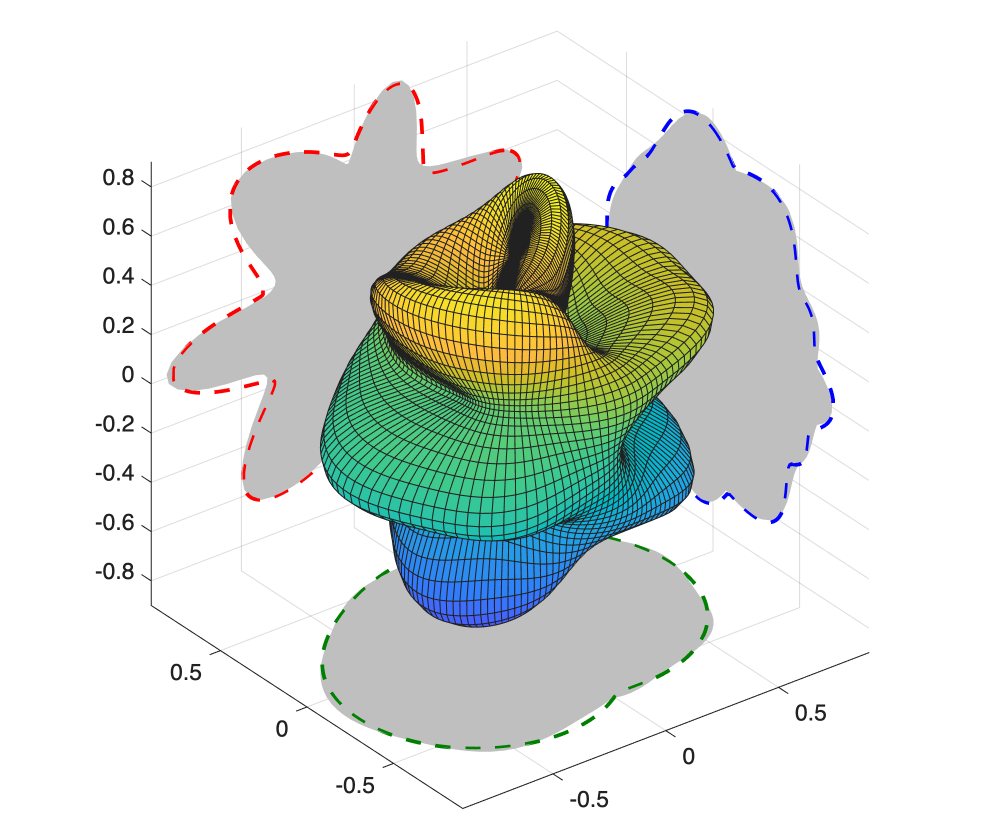}} &
        {\includegraphics[width=0.21\textwidth]{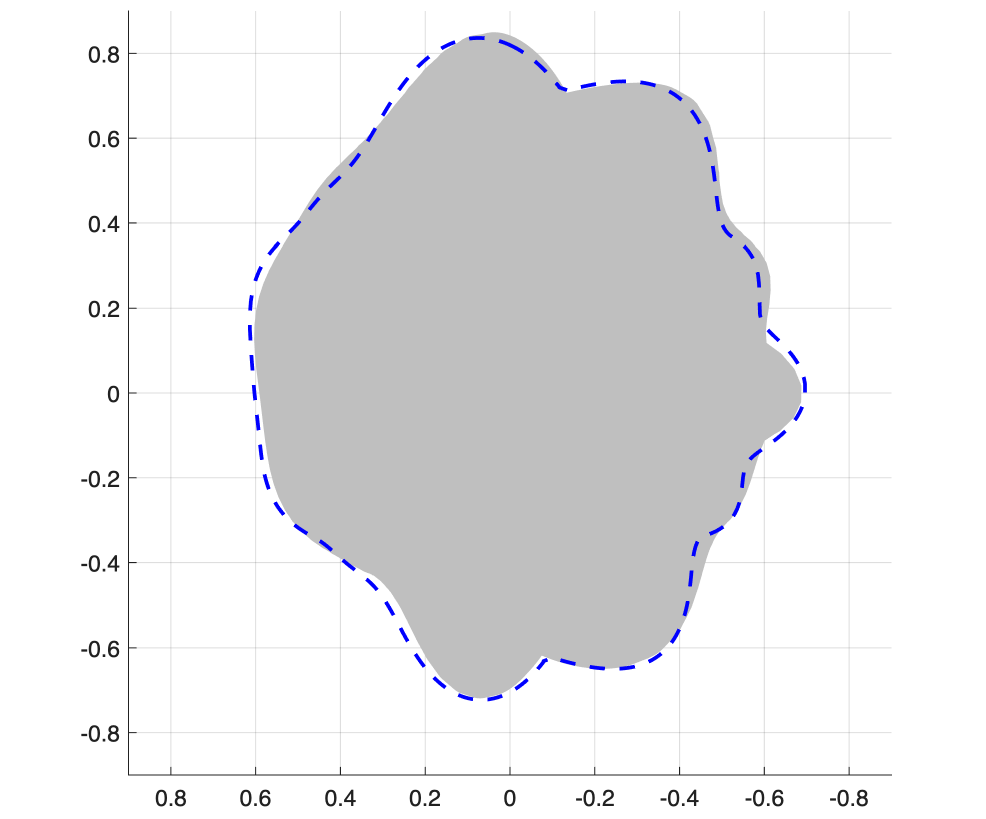}} &
        {\includegraphics[width=0.21\textwidth]{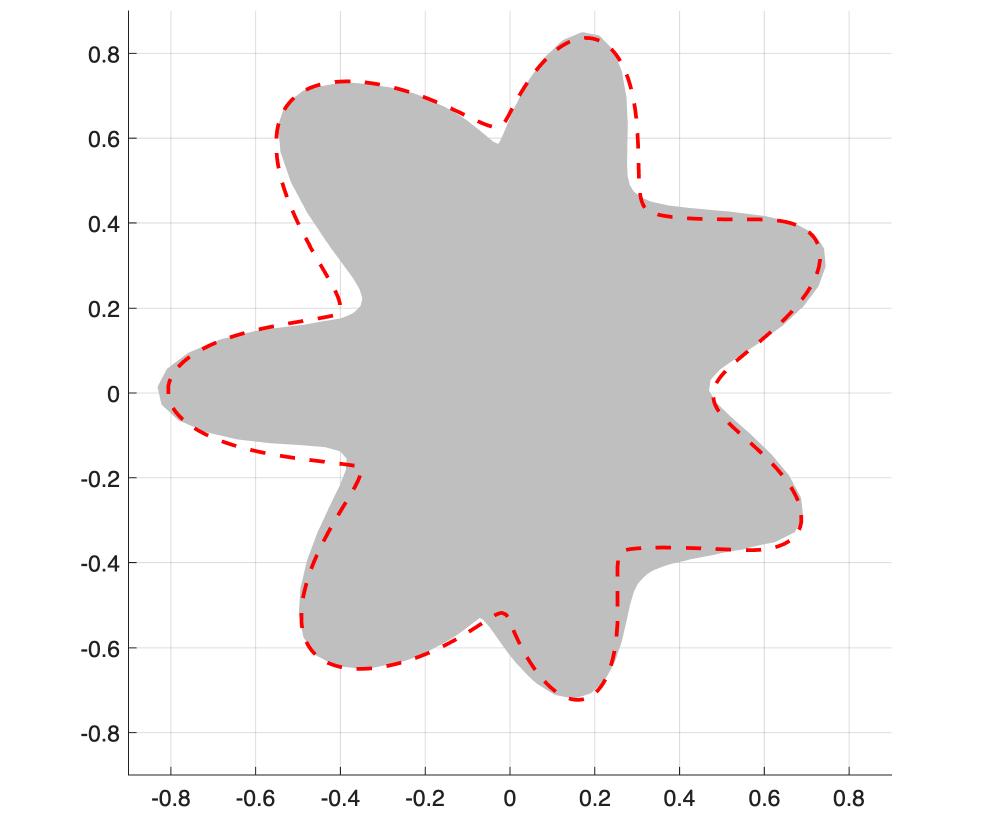}} &
        {\includegraphics[width=0.21\textwidth]{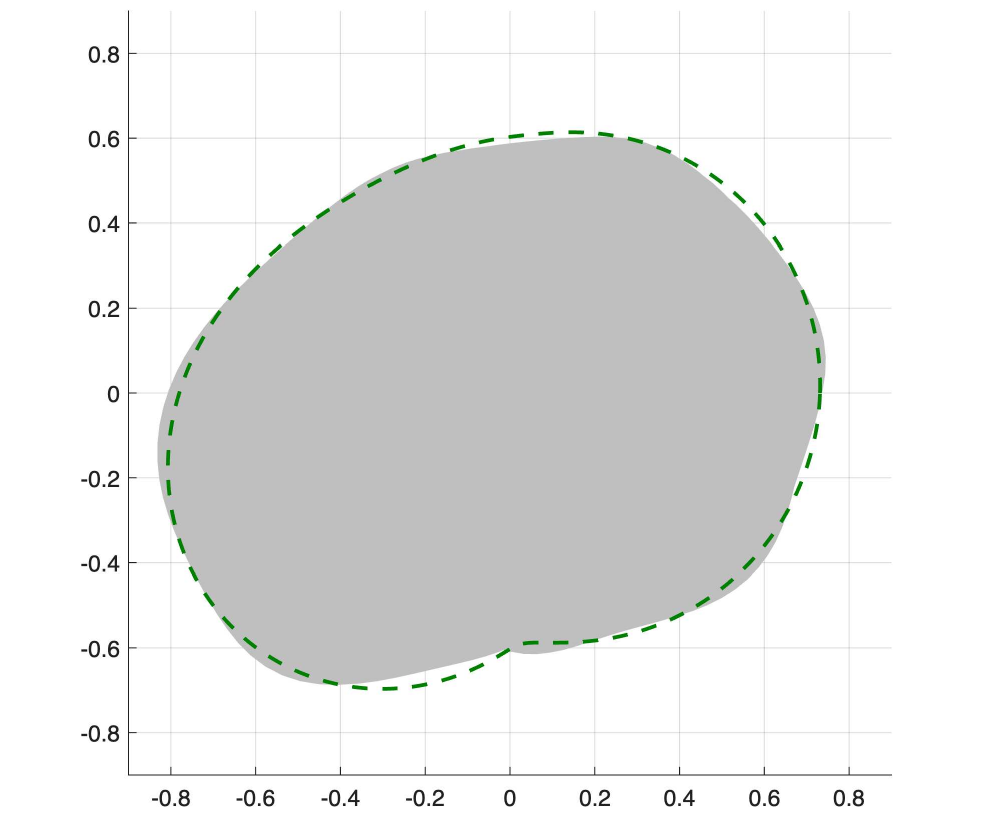}}
    \end{tabular} 
}\\[-1.2ex]
\caption{Reconstruction of a complex obstacle under the Dirichlet boundary condition from phaseless far-field data generated by four point sources located at $(0,0,\pm 4)^\top$ and $(\pm 4,0,0)^\top$ with different noise levels. $\pmb c^{(0)}=(0.1,-0.5,0.1)^{\top}$, $r^{(0)}=0.3$, $\kappa=5$.}\label{fig_ex7}
\end{figure}

\begin{figure}[!htbp]
    \centering
    \begin{tabular}{cccc}
        \subfigure[Initial surface]{\includegraphics[width=0.217\textwidth]{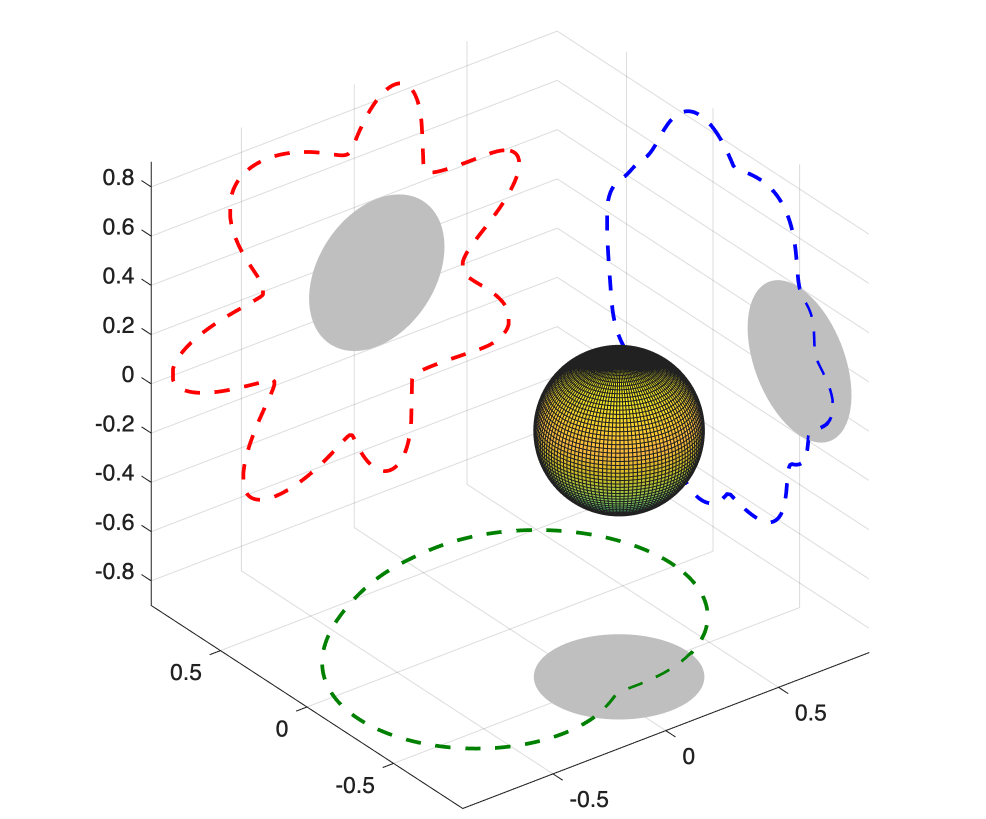}} &
        \subfigure[$M=0$]{\includegraphics[width=0.217\textwidth]{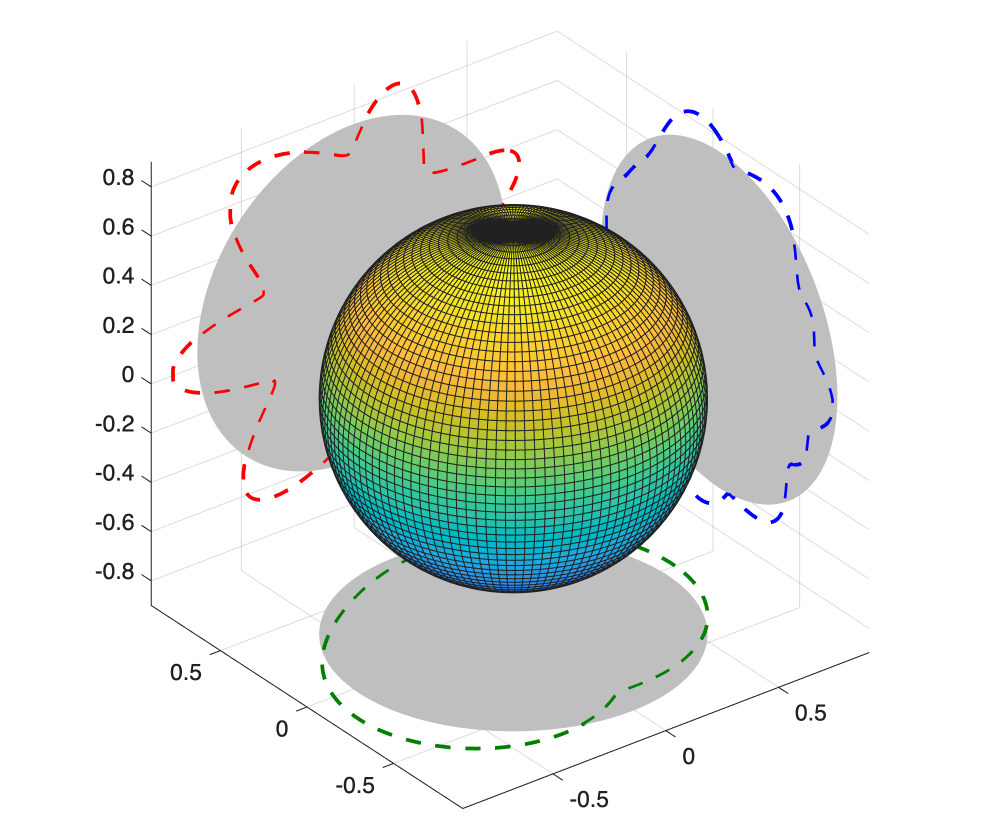}} &
        \subfigure[$M=3$]{\includegraphics[width=0.217\textwidth]{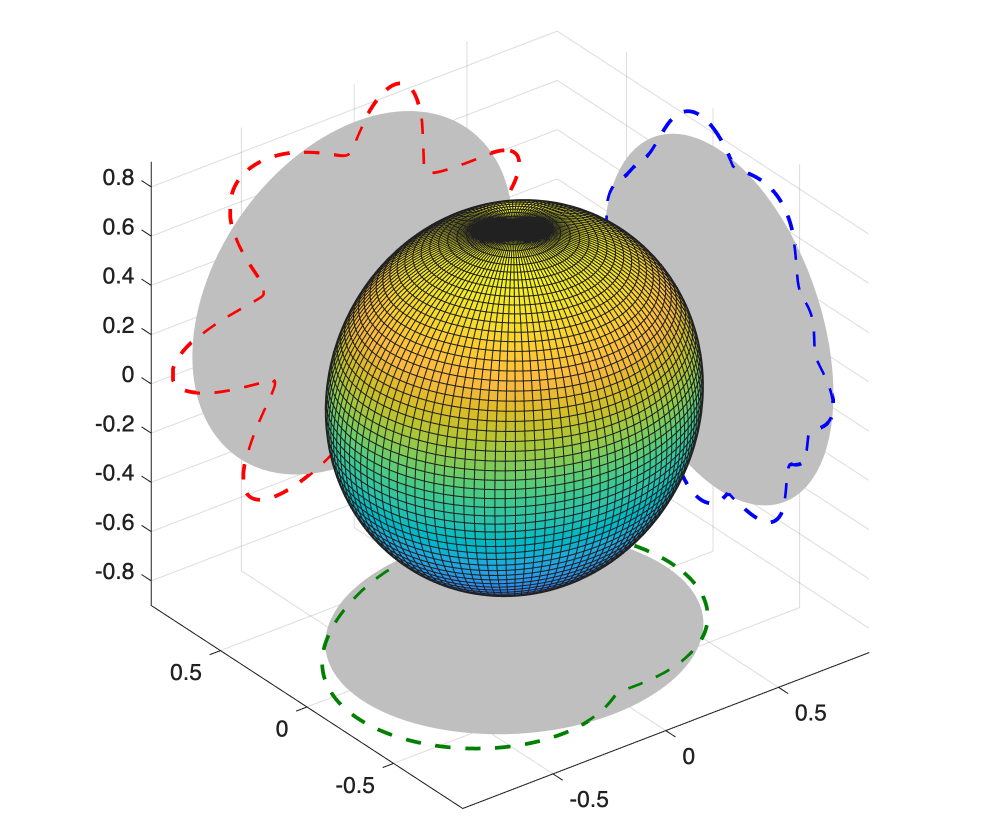}} &
        \subfigure[$M=6$]{\includegraphics[width=0.217\textwidth]{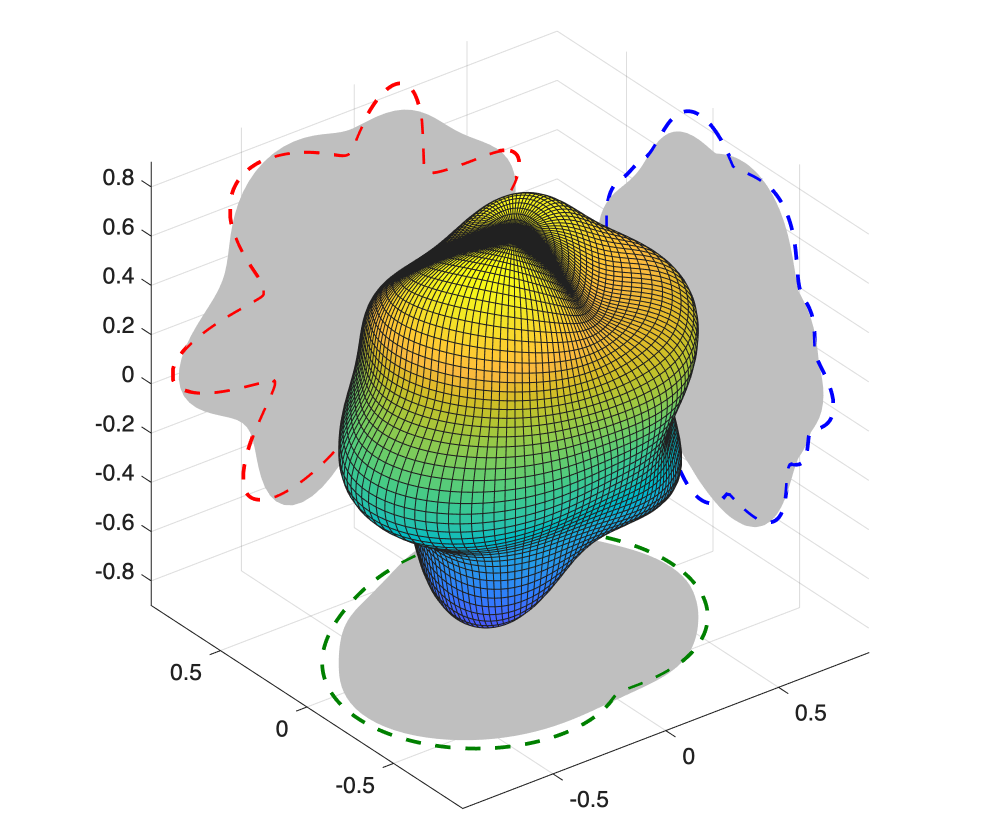}}
    \end{tabular}
    
    \begin{tabular}{cccc}
        \subfigure[$M=9$]{\includegraphics[width=0.217\textwidth]{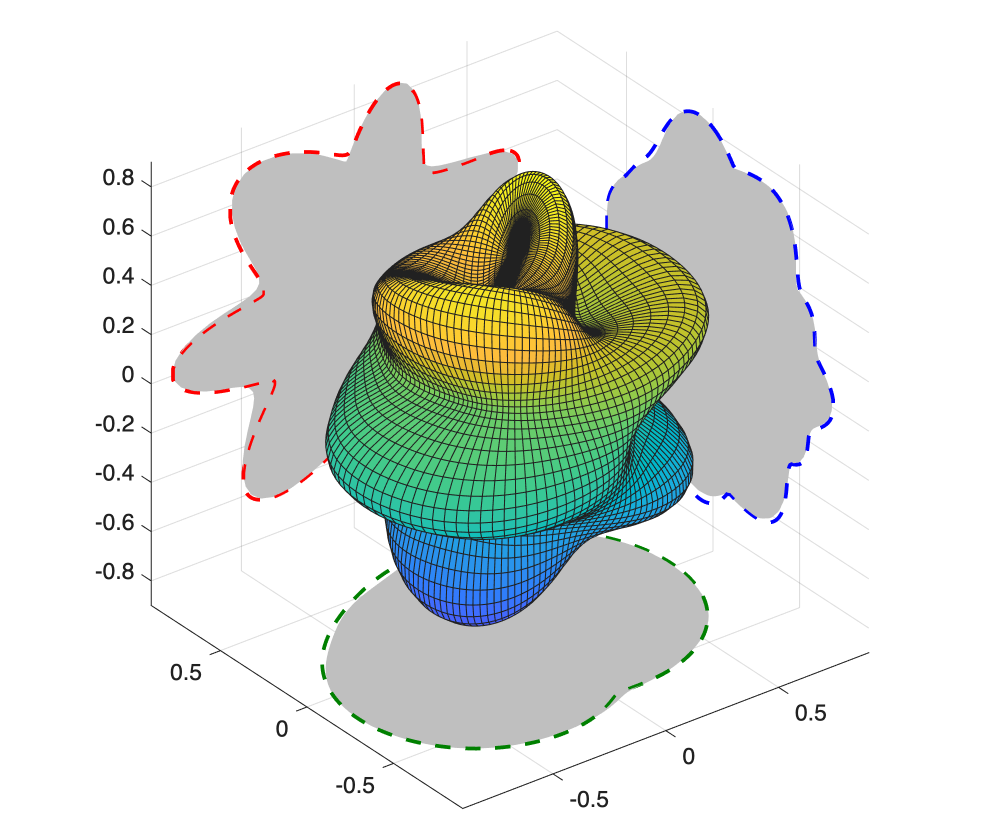}} &
        \subfigure[$M=12$]{\includegraphics[width=0.217\textwidth]{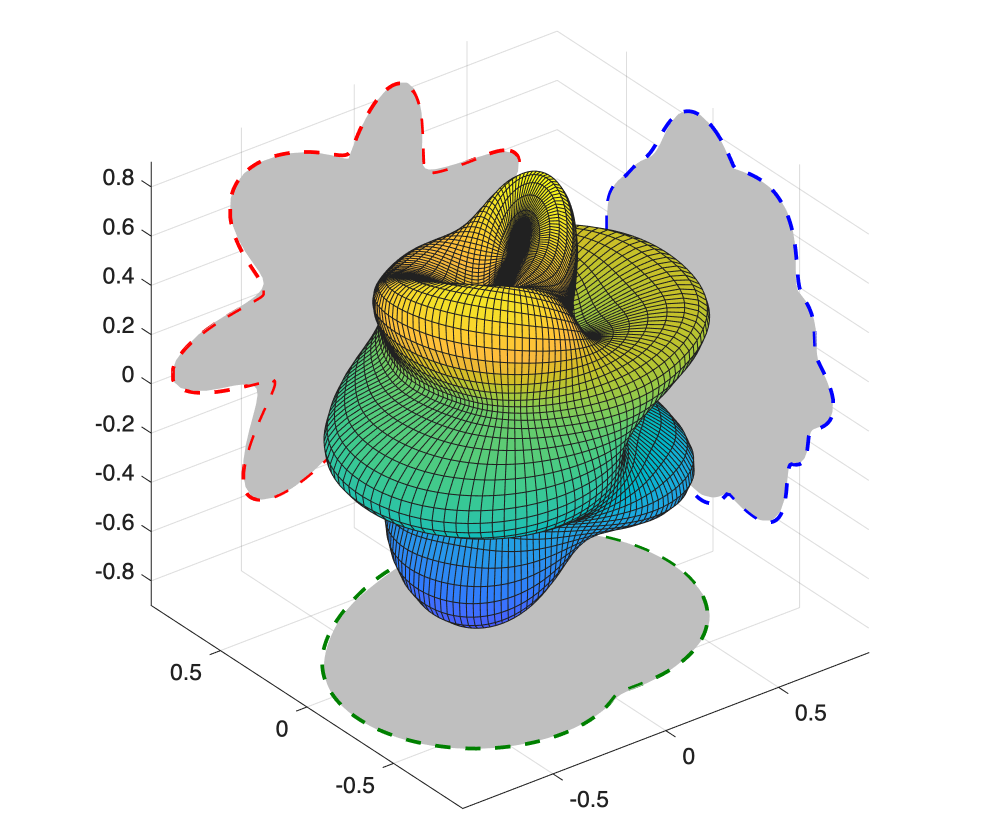}} &
        \subfigure[$M=15$]{\includegraphics[width=0.217\textwidth]{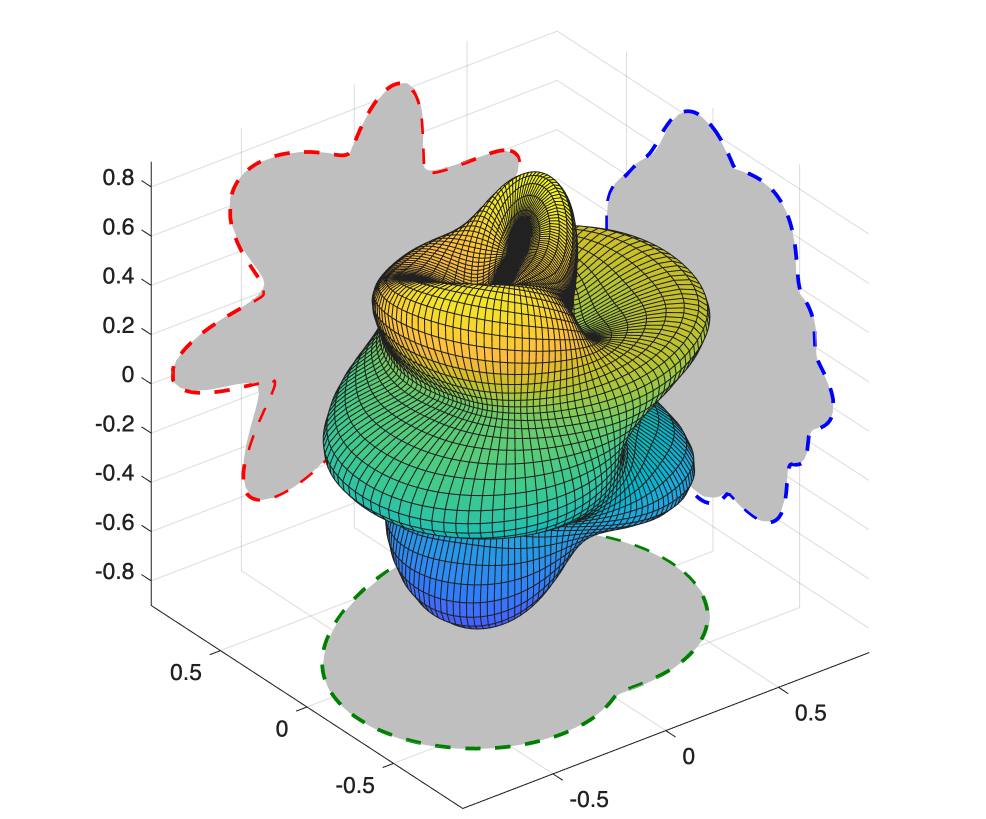}} &
        \subfigure[relative error]{\includegraphics[width=0.217\textwidth]{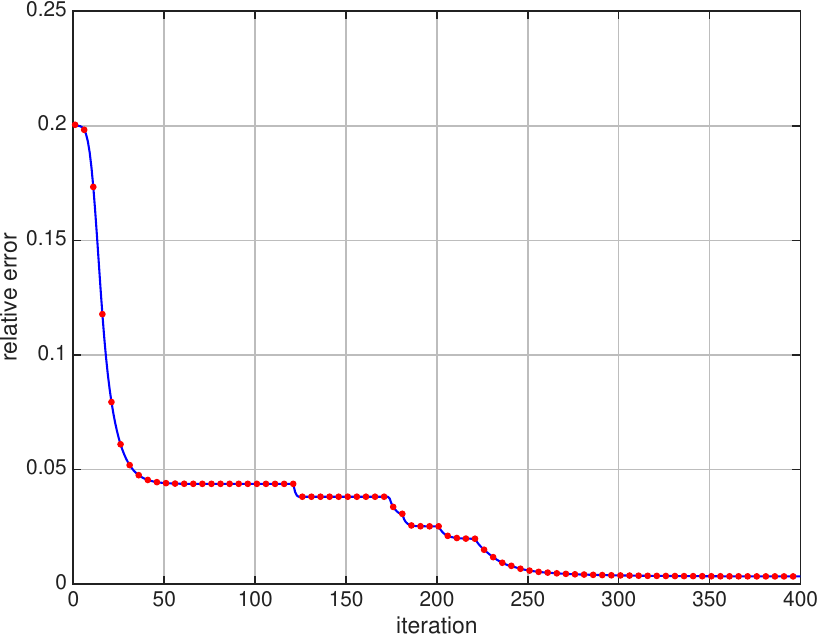}}
    \end{tabular}

    \caption{Reconstruction of a complex obstacle under the Dirichlet boundary condition from phaseless far-field data generated by four point sources located at $(0,0, \pm 4)^\top$ and $(\pm 4,0,0)^\top$ with $1\%$ noise. $\pmb c^{(0)}=(0.1,-0.5,0.1)^{\top}$, $r^{(0)}=0.3$, $\kappa=5$. Subfigures~(a)--(g) show the initial surface and the progressive reconstructions obtained as the truncation number $M$ increases. Subfigure~(h) reports the relative error with respect to the iteration index.}\label{fig_ex72}
\end{figure}

\vspace{2ex}
{\noindent\bf Example 7: Reconstructions of a complex obstacle under the Dirichlet boundary condition.}
\vspace{1ex}

To further assess the capability of the proposed approach for handling geometrically complicated three-dimensional obstacles, we consider the reconstruction of a complex obstacle 
under the Dirichlet boundary condition. Its boundary is described by the highly oscillatory star-shaped surface defined by the radial function
\[
r(\theta,\phi)=0.6\Big(1
+0.3\sin(7\theta)\cos\phi
+0.2\sin^{2}(3\theta)\sin(2\phi)
+0.1\cos\theta\Big).
\]
The surface exhibits strong angular oscillations in both $\theta$ and $\phi$, resulting in multiple localized peaks, dents, and nonconvex regions, making this example significantly more challenging than the previous ones.

To ensure the effectiveness of the reconstruction approach for this complex geometry, the computational parameters are refined as follows.  
The phaseless far-field measurements are taken at \(1800\) observation directions (i.e., \(\tilde{n}=30\)), and the corresponding data are generated by four point sources. 
The boundary integrals on \(\Gamma'\) are evaluated using \(1352\) quadrature points (i.e.\ \(n=25\)), and the same number of discrete nodes is employed on the current approximation of the boundary \(\Gamma\). 
The maximum truncation number is increased to \(M_{\max}=15\).

Figure~\ref{fig_ex7} presents the reconstructions obtained from phaseless far-field data with different noise levels.  
Subfigure~(a) shows the true obstacle from multiple viewpoints.  
Subfigures~(b)--(d) display the reconstructions corresponding to noise levels of $1\%$, $5\%$, and $10\%$, respectively.

In addition, Figure~\ref{fig_ex72} illustrates the reconstruction process with 1\% noise. Starting from the initial spherical guess, the approximation gradually improves as $M$ increases, capturing finer geometric features and progressively resolving the highly oscillatory components of the obstacle.
The relative error curve in Figure~\ref{fig_ex72} suggests a generally decreasing trend with respect to the iteration count, reflecting the gradual refinement of the reconstruction throughout the iterative process.

We also note that Example~7 involves substantially higher computational complexity due to the increased number of quadrature points, observation directions, and the larger truncation parameter $M_{\max}$, which together result in a significantly greater number of iterations.
As a consequence, this example requires noticeably more computational effort, taking about 40.34 seconds for 337 iterations, whereas the reconstructions in Examples~1--6 are completed within 1 second.

These results demonstrate the robustness of the proposed approach when applied to complex geometries, even under noisy phaseless measurements.

\vspace{2ex}
{\noindent\bf Example 8: Reconstruction of multiple obstacles with given initial centers.}
\vspace{1ex}

We consider the reconstruction of two sound-soft obstacles. The exact scatterer consists of two disjoint obstacles. The first one is a pinched ball-shaped obstacle centered at
\(
\pmb c_1=(-1.5,0.5,0.3)^\top,
\)
and the second one is a cushion-shaped obstacle centered at
\(
\pmb c_2=(1.5,-0.3,-0.2)^\top .
\)
The scattered-field data are generated at the wavenumber \(\kappa=5\) by six incident plane waves with directions
\(
\pmb d=(\pm1,0,0)^\top, (0,\pm1,0)^\top, (0,0,\pm1)^\top .
\)
The scattered-field data are measured on the observation sphere \(\Gamma_{B_R}\) with \(R=8\). We take \(\tilde n=20\), so that the number of measurement points is \(m=2\tilde n^2=800\). The relative noise level is \(\delta=1\%\).

 
For multiple obstacles, we assume that the number of obstacles and
initial estimates of their centers are available a priori. In this example,
we set \(\tilde N=2\) and choose
\[
\pmb c_1^{(0)}=(-1.3,\,0.4,\,0.2)^\top,
\qquad
\pmb c_2^{(0)}=(1.3,\,-0.2,\,-0.1)^\top
\]
as the initial guesses for the centers of the two obstacles.

For each obstacle, we introduce the homothetic surface
\[
\Gamma_i'=\pmb c_i+\varsigma(\Gamma_i-\pmb c_i), \qquad i=1,2 .
\]
For the \(t\)-th incident wave, the scattered field is represented as
\[
u^{\rm sc}_t(\pmb x)
=
\sum_{j=1}^{2}
\int_{\Gamma_j'}
G(\pmb x,\pmb y)g_{j,t}(\pmb y)\,ds(\pmb y),
\qquad t=1,\ldots,6 .
\]
Here \(g_{j,t}\) is the unknown density on \(\Gamma_j'\). For the sound-soft case, the coupled equivalent field equations are
\[
\begin{bmatrix}
S_{11} & S_{12}\\
S_{21} & S_{22}
\end{bmatrix}
\begin{bmatrix}
g_{1,t}\\
g_{2,t}
\end{bmatrix}
=
-
\begin{bmatrix}
u^{\rm inc}_t|_{\Gamma_1}\\
u^{\rm inc}_t|_{\Gamma_2}
\end{bmatrix},
\qquad t=1,\ldots,6,
\]
where
\[
(S_{ij}g_{j,t})(\pmb x)
=
\int_{\Gamma_j'}
G(\pmb x,\pmb y)g_{j,t}(\pmb y)\,ds(\pmb y),
\qquad
\pmb x\in\Gamma_i,\quad i,j=1,2 .
\]
The corresponding data equations on the observation sphere are
\[
K_1g_{1,t}+K_2g_{2,t}
=
u^{\rm sc}_{t,\delta}|_{\Gamma_{B_R}},
\qquad t=1,\ldots,6,
\]
with
\[
(K_jg_{j,t})(\pmb x)
=
\int_{\Gamma_j'}
G(\pmb x,\pmb y)g_{j,t}(\pmb y)\,ds(\pmb y),
\qquad
\pmb x\in\Gamma_{B_R},\quad j=1,2 .
\]

The initial guesses are two spheres of radius \(r^{(0)}=0.4\), centered at \(\pmb c_i^{(0)}\), \(i=1,2\). The parameters are
\(
\varsigma=0.9,
\alpha=10^{-8},
\lambda=10^{-8},
\rho=0.05 .
\)
The computation is divided into two stages. In the first stage, we set \(M=0\) and update only the centers for \(400\) iterations. In the second stage, the centers are fixed and the shapes are updated. The truncation number is increased from \(M=1\) to \(M=6\), and \(30\) iterations are used for each \(M\). The number of quadrature points on each \(\Gamma'_i\) is \(162\), and the number of discrete nodes on each current boundary \(\Gamma_i\) is \(338\).

Figure~\ref{fig_two_obstacles_recon} shows the final reconstruction. The result shows that, when the number of obstacles and rough initial centers are given, the coupled iterative method can recover both the locations and the shapes of the two obstacles from noisy scattered-field data.

\begin{figure}[!htbp]
\centering
\includegraphics[width=0.65\textwidth]{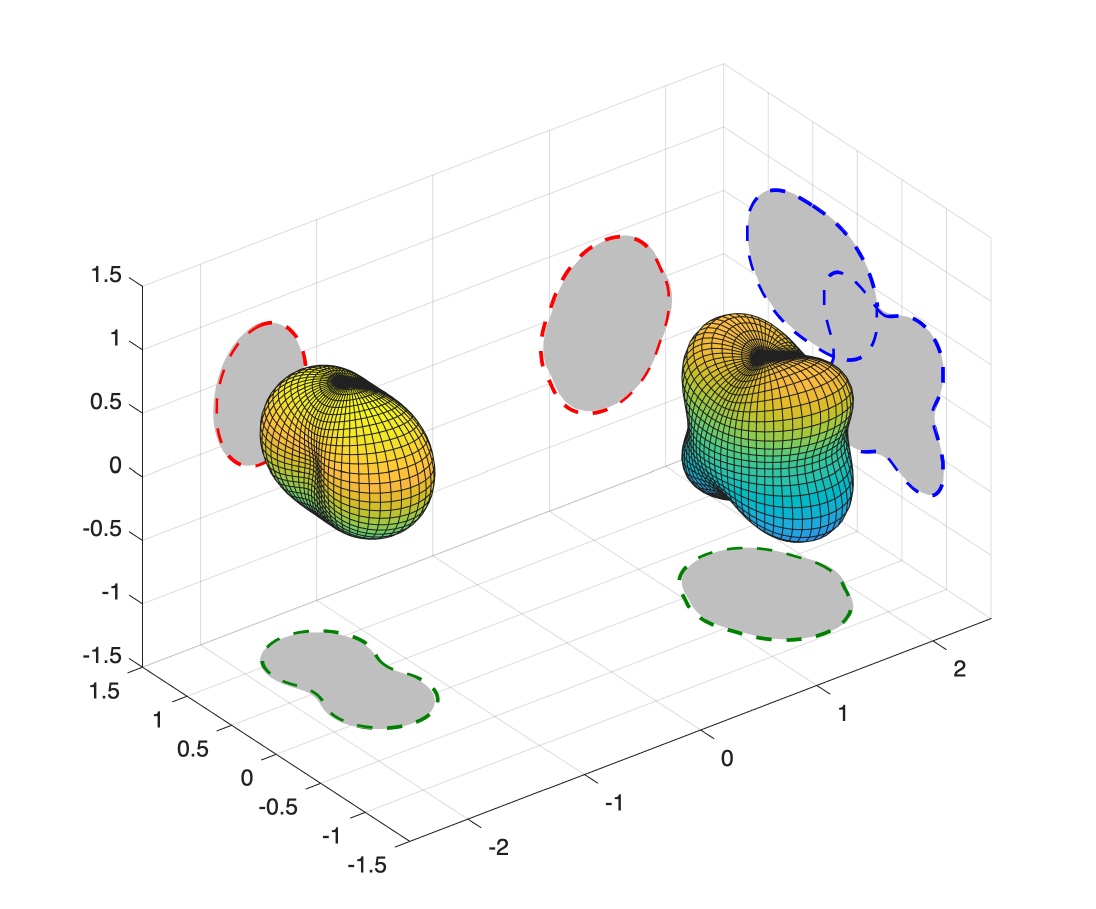}
\caption{Reconstruction of two sound-soft obstacles from \(1\%\) noisy
scattered-field data. We set \(\tilde N=2\) with initial centers
\(\pmb c_1^{(0)}=(-1.3,0.4,0.2)^\top\) and
\(\pmb c_2^{(0)}=(1.3,-0.2,-0.1)^\top\). The parameters are
\(\kappa=5\), \(r^{(0)}=0.4\), \(\varsigma=0.9\),
\(\alpha=\lambda=10^{-8}\), \(\rho=0.05\), \(M_{\rm max}=6\), and
\(\epsilon=0.0065\).}
\label{fig_two_obstacles_recon}
\end{figure}

\section{Conclusion}

This work has considered the three-dimensional inverse acoustic obstacle scattering problems with scattered field or phased/phaseless far-field data.
Based on the analytic continuation and the homothetic surface, we have proposed a highly efficient iterative approach that completely avoids dealing with singular integrals. 
It has been proved that the scattered field generated by the homothetic surface can arbitrarily approximate the exact one. 
Moreover, we established that the associated Fr$\acute{\rm e}$chet derivative is injective and has dense range, thereby ensuring the solvability of the inverse problem. 

Other possible directions include extending our method to the three-dimensional time-domain acoustic obstacle scattering problem, as well as to the inverse scattering problems for elastic waves in both frequency and time domains. 

Future work will extend the method to penetrable media with transmission conditions, non-smooth obstacles with corners, and configurations involving multiple scatterers.

\clearpage

\section*{Acknowledgements}
The work of H. Dong is supported in part by the NSFC Grant 12571454 and the National Key R\&D Program of China 2024YFA1012303.
The work of L. Zhao is supported in part by the NSFC Grant 12401565.


\begin{thebibliography}{99}
\bibitem{Ammari1}
H. Ammari, {\it An introduction to mathematics of emerging biomedical imaging}, Math\'ematiques \& Applications (Berlin), 62, Springer, Berlin, 2008;

\bibitem{Ammari2}
H. Ammari, G. Bao and J.~L. Fleming, An inverse source problem for Maxwell's equations in magnetoencephalography, SIAM J. Appl. Math. 62 (2002) 1369--1382.

\bibitem{bao05}
G. Bao, P. Li, Inverse medium scattering problems for electromagnetic waves, 
SIAM J. Appl. Math. 65 (2005) 2049--2066.

\bibitem{Cakoni}
M. Bonnet and F. Cakoni, Analysis of topological derivative as a tool for qualitative identification, Inverse Probl. 35 (2019) 104007.

\bibitem{Borden}
B. Borden, Mathematical problems in radar inverse scattering, Inverse Probl. 18 (2002) R1--R28.

\bibitem{borges20}
C. Borges, J. Lai, Inverse scattering reconstruction of a three-dimensional sound-soft axis-symmetric impenetrable object, 
Inverse Probl. 36 (2020) 105005.

\bibitem{gauss-newton}
A. Carpio, T. Dimiduk, F. Le~Lou\"er, M. Rap\'un, When topological derivatives met regularized Gauss-Newton iterations in holographic 3D imaging, J. Comput. Phys. 388 (2019) 224--251.


\bibitem{chang}
Y. Chang, Y. Guo, H. Liu, D. Zhang, Recovering source location, polarization, and shape of obstacle from elastic scattering data, J. Comput. Phys. 489 (2023) 112289.

\bibitem{RTM}
J. Chen, Z. Chen and G. Huang, Reverse time migration for extended obstacles: acoustic waves, Inverse Probl. 29 (2013) 085005.

\bibitem{chen2024}
J. Chen, B. Jin, H. Liu, Solving inverse obstacle scattering problem with latent surface representations, 
Inverse Probl. 40 (2024) 065013.


\bibitem{Colton}
D. Colton, R. Kress, Inverse Acoustic and Electromagnetic Scattering Theory, 
4th ed., Springer, New York, 2019.

\bibitem{Colton_ls}
D. Colton, J. Coyle, P. Monk, Recent developments in inverse acoustic scattering theory, 
SIAM Rev. 42 (2000) 369--414.

\bibitem{zhao22}
H. Dong, J. Lai, P. Li, Inverse obstacle scattering for elastic waves with phased or phaseless far-field data, SIAM J. Imaging Sci. 12 (2019) 809--838.

\bibitem{jcp22}
H. Dong, J. Lai and P. Li, A spectral boundary integral method for the elastic obstacle scattering problem in three dimensions, J. Comput. Phys. 469 (2022) 111546.

\bibitem{farhat02}
C. Farhat, R. Tezaur, R. Djellouli, On the solution of three-dimensional inverse obstacle acoustic scattering problems by a regularized Newton method, Inverse Probl. 18 (2002) 1229--1246.

\bibitem{ganesh2004}
M. Ganesh, I. Graham, A high-order algorithm for obstacle scattering in three dimensions, 
J. Comput. Phys. 198 (2004) 211--242.

\bibitem{Hagemann}
F. Hagemann, T. Arens, T. Betcke, F. Hettlich, Solving inverse electromagnetic scattering problems via domain derivatives, Inverse Probl. 35 (2019) 084005.

\bibitem{Hohage}
T. Hohage, Logarithmic convergence rates of the iteratively regularized Gauss-Newton method for an inverse potential and an inverse scattering problem, Inverse Probl. 13 (1997) 1279--1299.

\bibitem{Ivanyshyn2010}
O. Ivanyshyn, R. Kress, P. Serranho, Huygens' principle and iterative methods in inverse obstacle scattering, 
Adv. Comput. Math. 33 (2010) 413--429.

\bibitem{ivanshyn10}
O. Ivanyshyn, R. Kress, Identification of sound-soft 3D obstacles from phaseless data, Inverse Probl. Imaging 4 (2010) 131--149.

\bibitem{kang}
S. Kang and M. Lambert, Structure analysis of direct sampling method in 3D electromagnetic inverse problem: near- and far-field configuration, Inverse Probl. 37 (2021) 075002.

\bibitem{Kirsch_f}
A. Kirsch, Characterization of the shape of a scattering obstacle using the spectral data of the far field operator, 
Inverse Probl. 14 (1998) 1489--1512.

\bibitem{dec}
R. Kress and A. Zinn, On the numerical solution of the three-dimensional inverse obstacle scattering problem, J. Comput. Appl. Math. 42 (1992) 49--61.

\bibitem{Tikhonov}
R. Kress, Newton's method for inverse obstacle scattering meets the method of least squares, 
Inverse Probl. 19 (2003) S91--S104.

\bibitem{Louer1}
F. Le~Lou\"er and M.-L. Rap\'un, Detection of multiple impedance obstacles by non-iterative topological gradient based methods, J. Comput. Phys. 388 (2019) 534--560.

\bibitem{Louer2}
F. Le~Lou\"er and M.-L. Rap\'un, A boundary integral formulation and a topological energy-based method for an inverse 3D multiple scattering problem with sound-soft, sound-hard, penetrable, and absorbing objects, Comput. Methods Appl. Math. 22 (2022) no.~4, 915--943.

\bibitem{yuan}
P. Li and X. Yuan, Inverse obstacle scattering for elastic waves in three dimensions, Inverse Probl. Imaging 13 (2019) 545--573.

\bibitem{liukeyi24}
H. Liu and K. Liu, Direct imaging of inhomogeneities in a 3D shallow ocean waveguide with an icecap, J. Comput. Phys. 498 (2024) 112694.

\bibitem{Potthast_book}
R. Potthast, Point Sources and Multipoles in Inverse Scattering Theory, Chapman \& Hall/CRC, Boca Raton, FL, 2001.

\bibitem{serranho07}
P. Serranho, A hybrid method for inverse scattering for a sound-soft obstacle in $\mathbb{R}^3$, 
Inverse Probl. Imaging 1 (2007) 691--712.

\bibitem{Zou_ds}
K. Ito, B. Jin, J. Zou, A direct sampling method to an inverse medium scattering problem, 
Inverse Probl. 28 (2012) 025003.



\end{thebibliography}
\end{document}